\documentclass[11pt,leqno,oneside,letterpaper]{amsart}
\usepackage[letterpaper,left=1.75truein, top=1truein]{geometry}
\usepackage{amsmath,amsfonts,amssymb,amscd,amsthm,amsbsy,epsf,graphics}
\usepackage{xy}
\xyoption{all}

\usepackage{color}

\usepackage[colorlinks,linkcolor=blue,citecolor=blue]{hyperref}
\usepackage{epsfig}
\usepackage{graphicx}
\usepackage{ifpdf}

\usepackage{tikz-cd,calc}
\usetikzlibrary{decorations.markings}
\usetikzlibrary{arrows}
\tikzset{
  arrow/.pic={\path[tips,every arrow/.try,->,>=#1] (0,0) -- +(.1pt,0);},
  pics/arrow/.default={triangle 90}
}
\usetikzlibrary{knots}
\usetikzlibrary{decorations.markings}
\definecolor{light-gray}{gray}{0.92}
\definecolor{ultra-light-gray}{gray}{0.97}
\usepackage{mathrsfs}

\newtheorem{thm}{Theorem}[section]
\newtheorem{cor}[thm]{Corollary}
\newtheorem{lemma}[thm]{Lemma}
\newtheorem{prop}[thm]{Proposition}
\newtheorem{conj}[thm]{Conjecture}
\newtheorem{exa}[thm]{Example}
\newtheorem{que}[thm]{Question}
\newtheorem{rem}[thm]{Remark}

\newtheorem{claim}[thm]{Claim}

\newtheorem*{caseI}{Case I}
\newtheorem*{caseII}{Case II}

\newtheorem{assumps}[thm]{Assumptions}

\theoremstyle{definition}
\newtheorem{defn}[thm]{Definition}

\def\qed{{\hspace{2mm}{\small $\diamondsuit$}}}
\def\zed{{\mathbb Z}}
\def\D{{\mathcal D}}

\def\ep{\varepsilon}

\newtheoremstyle{cases}
  {12pt plus 6 pt}%       Space above
  {2pt}%       Space below
  {\bfseries}   %       Body font
  {}%          Indent amount (empty = no indent, \parindent = para indent)
  {\bfseries}% Thm head font
  {.}%         Punctuation after thm head
  {.5em}%      Space after thm head: " " = normal interword space;
  {}%          Thm head spec (can be left empty, meaning `normal')
\theoremstyle{cases}

\newtheorem{case}{Case}

\numberwithin{subcase}{case} \numberwithin{subsubcase}{subcase}
\numberwithin{equation}{subsection}
\usepackage{numbering}

\def\sfrac#1#2{\kern.1em\raise.5ex\hbox{$#1$}
    \kern-.1em/\kern-.05em\lower.25ex\hbox{$#2$}}
\def\zed{{\mathbb Z}}

\def\zed{{\mathbb Z}}

\def\sfrac#1#2{\kern.1em\raise.5ex\hbox{$#1$}
        \kern-.1em/\kern-.05em\lower.25ex\hbox{$#2$}}

\def\zed{{\mathbb Z}}

\def\G{{\Gamma}}
 \def\d{{\delta}}
 
 \def\e{{\epsilon}}
 \def\l{{\lambda}}
 \def\L{{\Lambda}}

  \def\O{{\Omega}}
   \def\s{{\sigma}}
 
 \def\a{{\alpha}}
 \def\b{{\beta}}
 \def\p{{\partial}}
 \def\r{{\rho}}
 \def\ra{{\rightarrow}}

 \def\g{{\gamma}}
 \def\D{{\Delta}}

 \def\c{{\mathbb C}}
 \def\z{{\mathbb Z}}
 \def\2{{\mathbb Z_2}}
 
 \def\t{{\tau}}

 \def\sl2{{SL(2,\mathbb C)}}

 \def\qed{{\hspace{2mm}{\small $\diamondsuit$}}}

 \def\pf{{\noindent{\bf Proof.\hspace{2mm}}}}

 \def\sl{{{\mbox{\tiny $\L$}}}}

\begin{document}

\title{Dehn fillings of knot manifolds containing essential twice-punctured tori}

\author[Steven Boyer]{Steven Boyer}
\thanks{Steven Boyer was partially supported by NSERC grant RGPIN 9446-2008}
\address{D\'epartement de Math\'ematiques, Universit\'e du Qu\'ebec \`a Montr\'eal, 201 avenue du Pr\'esident-Kennedy, Montr\'eal, QC H2X 3Y7.}
\email{boyer.steven@uqam.ca}
\urladdr{http://www.cirget.uqam.ca/boyer/boyer.html}

\author{Cameron McA. Gordon}
\thanks{Cameron Gordon was partially supported by NSF grant DMS-0906276.}
\address{Department of Mathematics, University of Texas at Austin, 1 University Station, Austin, TX 78712, USA.}
\email{gordon@math.utexas.edu}
\urladdr{http://www.ma.utexas.edu/text/webpages/gordon.html}

\author{Xingru Zhang}
\address{Department of Mathematics, University at Buffalo, Buffalo, NY, 14214-3093, USA.}
\email{xinzhang@buffalo.edu}
\urladdr{http://www.math.buffalo.edu/~xinzhang/}

\maketitle
\vspace{-.6cm}
\begin{center}
\today
\end{center}

\maketitle

\begin{abstract}
We show that if a hyperbolic knot manifold $M$
contains an essential twice-punctured torus $F$ with
boundary slope $\beta$ and admits a filling with slope
$\alpha$ producing a Seifert fibred space, then the
distance between the slopes $\alpha$ and $\beta$ is
less than or equal to $5$ unless $M$ is the exterior
of the figure eight knot. The result is sharp; the
bound of $5$ can be realized on infinitely many
hyperbolic knot manifolds. We also determine distance
bounds in the case that the fundamental group of the
$\alpha$-filling contains no non-abelian free group.
The proofs are divided into the four cases $F$ is a
semi-fibre, $F$ is a fibre, $F$ is non-separating but
not a fibre, and $F$ is separating but not a
semi-fibre, and we obtain refined bounds in each case.
\end{abstract}

\section{Introduction}

This is the third in a series of papers in which we investigate the following conjecture of the second named author (\cite[Conjecture 3.4]{Go2}). Throughout we assume that $M$ is a {\it hyperbolic knot manifold}. That is, $M$ is a compact, connected, orientable 3-manifold with torus boundary whose interior has a complete finite volume hyperbolic structure.

\begin{conj}\label{conj}{\rm  (C. McA. Gordon)} Suppose that $M$ is a hyperbolic knot manifold and $\alpha, \beta$ are slopes on $\partial M$ such that $M(\beta)$ is a toroidal manifold and $M(\alpha)$ is a Seifert manifold. If $\Delta(\alpha,\beta) > 5$, then $M$ is the figure eight knot exterior.\end{conj}

Recall that if $M$ is the exterior of the figure eight knot then $M(\gamma)$ is hyperbolic unless $\gamma \in \{\infty,0, \pm1,\pm2,\pm3,\pm4 \}$. Moreover, $M(\infty) \cong S^3$, $M(0)$ is a torus bundle, $M(\pm4)$ are toroidal, and $M(\pm1), M(\pm2)$ and $M(\pm3)$ are Seifert manifolds.

We note that 6 is the threshold distance characterising the figure eight knot exterior: there are infinitely many triples $(M;\alpha,\beta)$ where $M$ is a hyperbolic knot manifold with Seifert and toroidal filling slopes $\alpha$ and $\beta$ such that $\Delta(\alpha,\beta) = 5$.  See Example \ref{infinitely-many-5}.

Suppose that $\alpha$ and $\beta$ are slopes on $\partial M$ such that $M(\beta)$ is toroidal and $M(\alpha)$ is a Seifert manifold. In Proposition 3.3 of \cite{BGZ3} we showed that $\Delta(\alpha,\beta) \le 3$ if one of $M(\alpha)$ or $M(\beta)$ is reducible. Further, in Theorem 1.2 of \cite{BGZ3} we showed that if $M(\alpha)$ is toroidal then $\Delta(\alpha,\beta) \le 4$ (see Theorem \ref{thm:toroidal Seifert}  below). Thus we may assume that both $M(\alpha)$ and $M(\beta)$ are irreducible, and $M(\alpha)$ is atoroidal. In particular, $M(\alpha)$ is a {\it small Seifert manifold}, that is, it admits a Seifert structure with base orbifold of the form $S^2(a,b,c)$ where $a,b,c \ge 1$.

Since $M$ is hyperbolic and $M(\beta)$ is toroidal, $M$ contains an essential punctured torus $F$ with non-empty boundary of slope $\beta$. Let $m$ be the minimum value of $|\partial F|$ over all such $F$. In \cite{BGZ2} we verified the conjecture in the case that $m \geq 3$ and $M$ admits no punctured torus of boundary slope $\beta$ which is a fibre or semi-fibre. In \cite{BGZ3} we proved the conjecture when $M$ contains an essential once-punctured torus. In the present paper we consider the case $m = 2$. Note that this is the case that arises in the $\pm4$-surgeries on the figure eight knot.

Here is our main result.

\begin{thm}
\label{thm: twice-punctured precise}
  Let $M$ be a hyperbolic knot manifold which contains an essential twice-punctured torus $F$ with boundary slope $\beta$, and let $\alpha$ be a slope on $\partial M$ such that $M(\alpha)$ is an irreducible small Seifert manifold. If $\Delta(\alpha,\beta) > 5$ then $M$ is the exterior of the figure eight knot. More precisely:

$(1)$ If $F$ is a fibre in $M$, then $\Delta(\alpha,\beta) \le 3$.

$(2)$ If $F$ is a semi-fibre in $M$, then $\Delta(\alpha,\beta) \le 4$.

$(3)$ If $F$ is non-separating in $M$, though not a fibre, then $\Delta(\alpha,\beta) \le 5$.

$(4)$ If $F$ is separating in $M$, though not a semi-fibre, and if $\Delta(\alpha,\beta) > 5$, then $M$ is the exterior of the figure eight knot.
\end{thm}

Theorem \ref{thm: twice-punctured precise} combines with Proposition 3.3 and Theorems 1.2 and 1.3 of \cite{BGZ3}, and Theorem 2.7 of \cite{BGZ2}, to give:

\begin{cor}\label{cor:of main thm} Let $M$ be a hyperbolic knot manifold that is not an $m$-punctured torus bundle, $m \ge 3$, or an $m$-punctured torus semi-bundle, $m \ge 4$. Then Conjecture \ref{conj} holds for $M$.\end{cor}

The triples $(M;\alpha,\beta)$ where $M(\alpha)$ and $M(\beta)$ are toroidal and $\Delta(\alpha,\beta) \ge 4$ have been classified, by Gordon \cite{Go1} when $\Delta(\alpha,\beta) \ge 6$ and by Gordon and Wu \cite{GW} when $\Delta(\alpha, \beta) = 4$ or 5. Using this, the following result,  which deals with the case where $M(\beta)$ is toroidal and $M(\alpha)$ is toroidal Seifert, was proved in \cite{BGZ3}. Here $N$ is the 3-chain link (shown in Figure \ref{2 and 3 chain link}) exterior (also called the {\it magic manifold}), and slopes are parametrised in the usual way for exteriors of links in $S^3$.

\begin{thm}\label{thm:toroidal Seifert}{\rm  (\cite[Theorem 1.2]{BGZ3})}. Let $M$ be a hyperbolic knot manifold with slopes $\alpha$ and $\beta$ on $\partial M$ such that $M(\beta)$ is toroidal and $M(\alpha)$ is a toroidal Seifert manifold. Then

$(1)$ $\Delta(\alpha,\beta) \le 4$, and

$(2)$ $\Delta(\alpha,\beta) = 4$ if and only if $(M;\alpha,\beta) \cong (N(-1/2,-1/2);-4,0)$.\end{thm}

\begin{rem}{\rm  There is a mistake in the proof of \cite[Lemma 2.5]{BGZ3}, in the case of $N(-4) \cong M_3$. Namely, in the table just before Lemma 2.5, in the list of exceptional slopes for $N(-4)$, ``$-1/2$" should be ``$-5/2$". Similarly, in part (c) of Lemma 2.5, ``$N(-4,-1/2,\gamma)$" should be ``$N(-4,-5/2,\gamma)$", and ``$\gamma = -1/2$" should be ``no $\gamma$". Finally, in the first sentence after the proof of Lemma 2.5, ``parts (a) and (c)" should be ``part (a)", and the phrase ``and $N(-4)$, respectively" should be deleted.}
\end{rem}

We also consider the case that $M(\alpha)$ is {\it very small}, in other words, its fundamental group does not contain a non-abelian free group. In this case, consideration of the JSJ decomposition of $M(\alpha)$ shows that it is either a torus bundle over the circle, a torus semi-bundle over an interval, or a Seifert manifold whose base orbifold has non-negative Euler characteristic. As in the proof of Theorem \ref{thm:toroidal Seifert}, the results of \cite{Go1} and \cite{GW} can be used to deal with the cases where $M(\alpha)$ is a torus bundle or semi-bundle. These are described in Theorems \ref{thm:torus bundle} and \ref{thm:torus semi-bundle} below. In these theorems, $Wh$ denotes the exterior of the Whitehead link (shown in Figure \ref{2 and 3 chain link}), with slopes again parametrised in the usual way for link exteriors in $S^3$.

\begin{thm}\label{thm:torus bundle}
 Let $M$ be a hyperbolic knot manifold with slopes $\alpha$ and $\beta$ on $\partial M$ such that $M(\beta)$ is toroidal and $M(\alpha)$ is a torus bundle. Then

$(1)$ $\Delta(\alpha,\beta) \le 4$, and

$(2)$ $\Delta(\alpha,\beta) = 4$ if and only if $(M;\alpha,\beta) \cong (Wh(n);0,-4), \, n \in \zed, \, n \ne -4,-3,-2,-1,0$.\end{thm}

In part (2) of Theorem \ref{thm:torus bundle}, note that $Wh(1)$ is the figure eight knot exterior and $(Wh(1);0,4) \cong (Wh(1);0,-4)$.

\begin{thm}\label{thm:torus semi-bundle}
 Let $M$ be a hyperbolic knot manifold with slopes $\alpha$ and $\beta$ on $\partial M$ such that $M(\beta)$ is toroidal and $M(\alpha)$ is a torus semi-bundle. Then

$(1)$ $\Delta(\alpha,\beta) \le 4$, and

$(2)$ $\Delta(\alpha,\beta) = 4$ if and only if $(M;\alpha,\beta) \cong (Wh(-4n/(2n+1);-4,0), \, n \in \zed, \, n \ne -1,0$.\end{thm}

In Corollary \ref{cor:of main thm} we had to exclude the cases where $M$ is either an $m$-punctured torus bundle, $m \ge 3$, or an $m$-punctured torus semi-bundle, $m \ge 4$. If we assume that $M(\alpha)$ is very small Seifert, then we can include the latter case, and we can also include the former case if $M(\alpha)$ is of $C$- or $D$-type (whose definitions can be found at the beginning of \S \ref{subsec: semibundle}).

\begin{thm}\label{thm: very small cases}
 Let $M$ be a hyperbolic knot manifold with slopes $\alpha$ and $\beta$ on $\partial M$ such that $M(\beta)$ is toroidal and $M(\alpha)$ is a very small Seifert manifold. If $M(\alpha)$ is not of $C$- or $D$-type, assume that $M$ is not an $m$-punctured torus bundle, $m \ge 3$. Then $\Delta(\alpha,\beta) \le 5$.\end{thm}

Theorem \ref{thm: very small cases} is best possible; see Example \ref{example very small}.

The proofs of the four parts of Theorem \ref{thm: twice-punctured precise} are independent and are dealt with in separate parts of the paper.

Here is how the paper is organized.

In \S \ref{sec: examples} we show that Theorem \ref{thm: twice-punctured precise} is sharp by producing an infinite family hyperbolic knot
manifolds each of which contains an essential twice punctured torus of boundary slope $\beta$ and admits a small Seifert filling slope $\alpha$ such that
$\D(\alpha,\beta)=5$. Theorems \ref{thm:torus bundle} and \ref{thm:torus semi-bundle} are proven  in  \S\ref{sec:proofs of todoaidl bundle and semi-bundle cases}.
 Sections \ref{sec: initial reduction} and \ref{sec: cs theory} contain assumptions, reductions and background material to be applied later in the paper. In \S \ref{sec: amalgams} we construct curves in the $PSL_2(\mathbb C)$-character varieties of certain amalgams of triangle groups which will play an essential part in the proof of Theorem \ref{thm: twice-punctured precise}. The first two assertions of this theorem are proven in \S \ref{sec: twice-punctured semi-fibre} and \S \ref{sec: fibre}.   The relative JSJ method for studying Dehn fillings introduced in \cite{BCSZ1} and further developed in \cite{BGZ2} is outlined in \S \ref{sec: second reduction} and used to prove the third assertion of Theorem \ref{thm: twice-punctured precise} in \S \ref{non-sep not fibre}. Further background material for dealing with the last assertion is contained in \S \ref{sec: n-gons} and \S \ref{sec: background sep not semifibre} while \S \ref{sec: recogn the fig 8} provides sufficient conditions for recognizing that $M$ is the figure eight knot exterior when $\Delta(\alpha, \beta) > 5$. The proof of the final assertion of Theorem \ref{thm: twice-punctured precise} is then dealt with in \S \ref{sec: t1+ + t1- > 0}, \ref{sec: Delta 7 or more}, \ref{sec: delta = 6 d=1},   \ref{sec: sep not semifibre very small} and \ref{sec: sep not semifibre not very small}. The proof of Theorem \ref{thm: very small cases} is given in \S \ref{sec: very small cases}.

{\bf Acknowledgements}. The authors would like to
thank Bruno Martelli for graciously performing
computer calculations related to an earlier version of
the paper, and the referee for their careful reading
of the manuscript and especially for pointing out
flaws in our original proofs  of Lemmas
\ref{--quotient 4} and  \ref{F2-lens}.

\section{Examples}
\label{sec: examples}
In this section we examine the sharpness of the theorems in the introduction. We assume throughout that $M$ is a hyperbolic knot manifold which contains an essential twice-punctured torus $F$ of boundary slope $\beta$ and $\alpha$ is a slope on $\partial M$ such that $M(\alpha)$ either Seifert fibred or very small.

We use $N$ to denote the exterior of the hyperbolic $3$-chain link (shown  in Figure \ref{2 and 3 chain link}) with the coordinates on $\partial N$  as given in \cite{MP}.
\begin{figure}[!ht]
\includegraphics{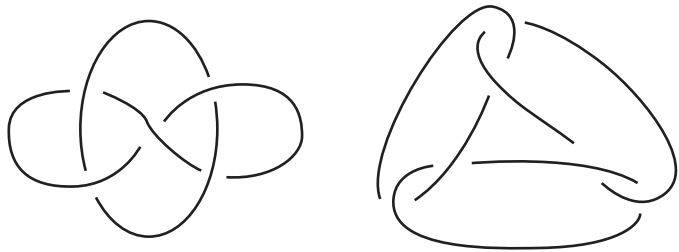}\caption{The Whitehead link and the 3 chain link}\label{2 and 3 chain link}
\end{figure}

The following example shows that $6$ is the threshold distance characterising the figure eight knot exterior in that there are infinitely many hyperbolic knot manifolds  which contain an essential twice punctured torus of boundary slope $\beta$ and admit a small Seifert filling slope $\alpha$ such that $\D(\alpha,\beta)=5$.

\begin{exa}
\label{infinitely-many-5}
{\rm
For each integer slope
$n\ne 0, -1, -2, -3$,
the manifold $N(-\frac{3}{2}, n)$
 is a hyperbolic knot manifold.
This fact follows from \cite[Theorems 1.1 and 1.2]{MP}.
Now $N(-\frac{3}{2}, n, -\frac{5}{2})$ is a toroidal manifold
 and $N(-\frac{3}{2}, n, 0)$ is a small Seifert fibred manifold, by \cite[Table A.6]{MP}. So we just need to
 verify that $N(-\frac{3}{2}, n)$ contains an essential twice-punctured torus of boundary slope $-\frac{5}{2}$.
To this end, first note that the manifold $N(-\frac{3}{2})$ is  hyperbolic by \cite[Theorem 1.1]{MP}.
Next note that $N(-\frac{3}{2},  -\frac{5}{2})$ and $N(-\frac{3}{2},  0)$
are both toroidal manifolds (\cite[Table 1]{MP}). Hence by \cite[Theorem 1.1]{GW2}, $N(-\frac{3}{2})$
is the exterior of the Whitehead sister link  (shown in \cite[Figure 7.1 (c)]{GW2})  since it is the only hyperbolic link exterior of two boundary components which admits two toroidal Dehn fillings  on one of its boundary components at distance $5$.
Furthermore by \cite[Table 1]{MP}, $N(-\frac{3}{2}, -\frac{5}{2})$ has a unique essential torus $T$  (up to isotopy) and this torus $T$ splits the manifold into two pieces, one is Seifert fibred over a disk with
two cone points and the other is Seifert fibred over an annulus with one cone point.
Now by \cite[Theorem 1.1]{GW2}, the essential torus $T$ in  $N(-\frac{3}{2},  -\frac{5}{2})$ can be arranged to intersect the third boundary torus of $N$ in two components of slope $-\frac{5}{2}$.
Moreover by the unique torus decomposition of the toroidal  manifold $N(-\frac{3}{2}, n, -\frac{5}{2})$ given in \cite[Table A.6]{MP}, one can see clearly that the torus $T$ remains essential in  $N(-\frac{3}{2}, n, -\frac{5}{2})$.
As $T$ intersects the boundary torus of $N(-\frac{3}{2}, n)$ in two components (minimally possible) of slope $-\frac{5}{2}$, its  restriction on $N(-\frac{3}{2}, n)$ is an essential twice-punctured torus of boundary slope $-\frac{5}{2}$.
}
\end{exa}

The following example shows that Theorem \ref{thm: very small cases} is sharp.

\begin{exa}\label{example very small}
{\rm Taking $n=-4$ or $-5$ in Example \ref{infinitely-many-5},
the Seifert filling on $N(-\frac{3}{2}, n)$ with slope $0$ has finite fundamental group (of $T$-type or $I$-type respectively) while the slope $0$ has distance $5$
to the essential twice-punctured torus of boundary slope $-\frac{5}{2}$. Note that $N(-\frac32, -4)$ is the figure eight sister manifold.
}
\end{exa}

\section{Proof of Theorems \ref{thm:torus bundle} and \ref{thm:torus semi-bundle}}\label{sec:proofs of todoaidl bundle and semi-bundle cases}

By \cite{Go1}, there are four triples $(M;\alpha,\beta)$ with $M(\alpha)$ and $M(\beta)$ toroidal and $\Delta(\alpha,\beta) \ge 6$. In all cases one readily checks that neither $M(\alpha)$ nor $M(\beta)$ is a torus bundle or semi-bundle. For example, one may proceed as follows. The manifolds $M$ are $Wh(1), Wh(-5), Wh(2)$ and $Wh(-5/2)$. Now $Wh \cong N(1)$, and with respect to the slope parametrizations described above, $Wh(r) \cong N(1,r+1)$. An examination of Tables A.2 and A.3 in \cite{MP}
 shows that none of the toroidal fillings in question is a torus bundle or semi-bundle.

For $\Delta = 4$ or 5, the triples $(M;\alpha,\beta)$ with $M(\alpha)$ and $M(\beta)$ toroidal are determined in \cite{GW}.
There are 14 hyperbolic manifolds $M_i, \, 1 \le i \le 14$, each with a pair of toroidal filling slopes $\alpha_i, \beta_i$ at distance 4 or 5, where $M_1, M_2, M_3$ and $M_{14}$ have two (torus) boundary components, and the others, one. It is shown in \cite{GW} that a hyperbolic knot manifold has two toroidal filling slopes $\alpha$ and $\beta$ at distance 4 or 5 if and only if either $(M;\alpha,\beta) \cong (M_i;\alpha_i,\beta_i)$ for some $4 \le i \le 13$, or $(M;\alpha,\beta) \cong (M_i(\gamma);\alpha_i,\beta_i)$, for $i = 1,2,3$ or 14 and some slope $\gamma$ on the second boundary component of $M_i$.

First, one sees from \cite[Lemma 22.2]{GW} that for $6 \le i \le 13$, neither $M_i(\alpha_i)$ nor $M_i(\beta_i)$ is a torus bundle or semi-bundle.

Next, the proof of \cite[Lemma 2.3]{BGZ3} shows that neither $M_4(\alpha_4)$ nor $M_4(\beta_4)$ is a torus bundle or semi-bundle.

For $M_5$, one can argue as follows. By \cite[Section 6]{L2}, $M_5 \cong N(1,-1/3)$, the toroidal filling slopes being $-4$ and 1. Then \cite[Table A.4]{MP} shows that neither $N(1,-1/3,-4)$ nor $N(1,-1/3,1)$ is a torus bundle or semi-bundle.

It remains to consider the cases $M_1, M_2, M_3$, and $M_{14}$. Recall that $M_1 \cong Wh \cong N(1), M_2 \cong N(-1/2)$, the exterior of the 2-bridge link with associated rational number 10/3, and $M_3 \cong N(-4)$, the exterior of the Whitehead sister link (see \cite{GW} and \cite[Table A.1]{MP}).

For $M_1 \cong N(1)$, the toroidal slopes are $-3$ and 1 (with respect to the standard slope coordinates for $N$, used in \cite{MP}). We must therefore  determine the slopes $\gamma$ such that $N(1,\gamma)$ is hyperbolic, i.e. $\gamma \notin \{\infty,-3,-2,-1,0,1 \}$ (see \cite[Table A.1]{MP}), $N(-3,1,\gamma)$ and $N(1,1,\gamma)$ are toroidal, and one is a torus bundle or semi-bundle. An examination of Tables 2, 3 and 4 of \cite{MP} shows that the only such slopes are (a) $\gamma = n \in \zed, \, n \ne -1,1$, and (b) $\gamma = -(2n+1)/(2n+3), \, n \in \zed, \, n\ne -2,-1$. Note that $N(-3,1,\gamma)$ is a torus bundle in case (a) and a torus semi-bundle in case (b). Translating the slopes into the standard coordinates for $Wh$ gives the examples in parts (2) of Theorems \ref{thm:torus bundle} and \ref{thm:torus semi-bundle}.

For $M_2 \cong N(-1/2)$ and $M_3 \cong N(-4)$ we do a similar analysis, the corresponding sets of exceptional slopes being $\{ \infty,-4,-3,-2,-1,0 \}$ and $\{ \infty,-3,-2,-1,-5/2,0 \}$, respectively, and toroidal slopes $\{ -4,0 \}$ and $\{ -5/2,0 \}$ (see \cite[Table A.1]{MP}). We conclude from Tables 2, 3 and 4 of \cite{MP} that no triples $(M;\alpha,\beta)$ of the required type arise.

Finally we take care of $M_{14}$. It follows from the classification in \cite{GW} of the hyperbolic knot manifolds with toroidal fillings at distance 4 that $M_{14} \cong Y$, where $Y$ is the manifold defined by Lee in \cite{L1}. The two toroidal fillings of $Y$ are $Y(0) \cong Y(4) \cong Q(2,2) \cup Wh$, where $Q(2,2)$ is the Seifert manifold with base orbifold $D^2(2,2)$. If $Q(2,2) \cup Wh(\gamma)$ were a torus bundle or semi-bundle then $\gamma$ would be an exceptional slope for $Wh$, and therefore would be in $\{\infty, -4,-3,-2,-1,0 \}$. But from \cite[Table A.1]{MP}  (making the appropriate change of slope coordinates) we see that for no such $\gamma$ is $Q(2,2) \cup Wh(\gamma)$ a torus bundle or semi-bundle.

\section{Initial assumptions and reductions} \label{sec: initial reduction}

We assume throughout the paper that $M$ is a hyperbolic knot manifold. A {\it slope} on the boundary of $M$ is a $\partial M$-isotopy class of essential simple closed curves. Slopes correspond bijectively with $\pm$ pairs of primitive elements of $H_1(\partial M)$ in the obvious way and we shall often represent a slope by a primitive element of $H_1(\partial M)$. The {\it rational longitude} on $\partial M$ is the unique slope $\lambda_M$ having the property that it represents a torsion element of $H_1(M)$.

To each slope $\gamma$ on $\partial M$ we associate the $\gamma$-Dehn filling of $M$, denoted by $M(\gamma)$, obtained by attaching a solid torus $V$ to $M$ in such a way that the meridional slope of $V$ is identified with $\gamma$. The resulting manifold is independent of all choices.

Given a closed, essential surface $S$ in $M$, we use $\mathcal{C}(S)$ to denote the set of slopes $\delta$ on $\partial M$ such that $S$ compresses in $M(\gamma)$.
A slope $\eta$ on $\partial M$ is called a {\it singular slope} for $S$ if $\eta \in \mathcal{C}(S)$ and $\Delta(\gamma, \eta) \leq 1$ for each $\gamma \in \mathcal{C}(S)$. For instance, if $M(\eta)$ is Seifert fibred with hyperbolic base orbifold other than a $2$-sphere with three cone points or a projective plane with two cone points, then $\eta$ is a singular slope of some closed, essential surface in $M$ (\cite[Theorem 1.7]{BGZ1}). A theorem of Wu (\cite{Wu1}) shows that if $S$ is a closed essential surface in $M$ for which $\mathcal{C}(S) \ne \emptyset$, then there is a singular slope for $S$.

\begin{prop}
\label{prop: sing slope exceptional}
{\rm (cf. \cite[Theorem 1.5]{BGZ1})}
If $\eta$ is a singular slope for some closed essential surface $S$ in $M,$ then for an arbitrary slope $\gamma$ we have
$$\Delta(\gamma, \eta) \leq
\left\{
\begin{array}{ll}
1 & \mbox{if $M(\gamma)$ is either  small Seifert or reducible} \\
1 & \mbox{if $M(\gamma)$ is Seifert fibred and $S$ does not separate}\\
2 & \mbox{if $M(\gamma)$ is toroidal and ${\mathcal C}(S)$ is infinite}\\
3 & \mbox{if $M(\gamma)$ is toroidal and ${\mathcal C}(S)$ is finite}.\\
\end{array} \right.$$
Consequently if $M(\gamma)$ is not hyperbolic, then
$\Delta(\gamma, \eta) \leq 3.$ \qed
\end{prop}

\begin{prop} \label{compresses}
Suppose that $F$ is an essential, properly embedded, twice-punctured torus of boundary slope $\beta$ which caps-off to a compressible torus in $M(\beta)$. If $\gamma$ is a slope on $\partial M$ such that $M(\gamma)$ is not hyperbolic, then $\Delta(\gamma, \beta) \leq 3$. If $M(\gamma)$ is a small Seifert manifold, then $\Delta(\gamma, \beta) \leq 1$.
\end{prop}

\begin{proof}
If $F$ is non-separating the proposition is a special case of \cite[Proposition 3.1]{BGZ3}. Assume then that $F$ is separating. The surface $F$ splits $M$ into two components which we denote by
 $X^{\pm}$, and $\partial F$ splits $\partial M$ into two annuli which we denote by $B^{\pm}$
 so that $F\cup B^{\pm}=\partial X^\pm$.
 Let $\widehat F$ be the corresponding torus in $M(\beta)$ obtained by capping off
 $\partial F$ with two disjoint meridian disks $D_1$ and $D_2$ of the filling solid torus $V_\b$.
 The disks $D_1$ and $D_2$ cut $V_\b$ into two
 components $H^\pm$ with $H^\pm$ attached to $X^\pm$ along $B^\pm$ as a $2$-handle.
 Let $\widehat X^\pm=X^\pm\cup_{B^\pm} H^\pm$.
 If  $\partial X^\epsilon$,  $\e\in\{+,-\}$,  is compressible in $X^\epsilon$, then by the   handle addition lemma (\cite[Theorem 2]{Ja2}),
 $\widehat F=\partial \widehat X^\epsilon$ is incompressible in $\widehat X^\epsilon$.
 Hence by the assumption of the proposition, $\partial X^\epsilon$
 is incompressible in $X^\epsilon$ for at least one $\epsilon$, say $\e=+$.
 Then $\partial X^+$ is also incompressible in $M$, since $F$ is (cf. \cite[\S1.2(5), page 15]{Ha}).
 Pushing  $\partial X^+$  slightly into the interior of $X^+$, we obtain
 an embedded closed separating incompressible surface $S$ in $M$.
Since $S$ contains an essential simple closed curve which is isotopic to the center circle
of the annulus  $B^+\subset \partial M$, $\beta$ is a singular slope
by \cite[Theorem 2.4.3]{CGLS}. The proposition now follows from Proposition \ref{prop: sing slope exceptional}.
\end{proof}

Here is a list of assumptions that will hold throughout the paper.

\begin{assumps}
\label{assumptions 0}  $\;$ \newline
{\begin{enumerate}

\vspace{-.85cm} \item $\alpha$ is a slope on $\partial M$ such that $M(\alpha)$ is Seifert fibred.

\vspace{.1cm} \item $\beta \ne \alpha$ is a slope on $\partial M$ of an essential, twice-punctured torus $F$.

\vspace{.1cm} \item $F$ caps-off to an essential torus in $M(\beta)$ {\rm (cf.~Proposition \ref{compresses})}.

\vspace{.1cm} \item Neither $\alpha$ nor $\beta$ is a singular slope of a closed essential surface in $M$ $(${\rm cf.~Proposition \ref{prop: sing slope exceptional}}$)$.

\vspace{.1cm} \item $M(\alpha)$ is an irreducible, atoroidal, small Seifert manifold with base orbifold $S^2(a,b,c)$ where $a,b,c \geq 1$ $($ {\rm cf.}~\cite[Theorem 1.2 and Proposition 3.3(b)]{BGZ3}$)$.

\vspace{.1cm} \item $M(\beta)$ is irreducible $(${\rm cf.}~\cite[Proposition 3.3(b)]{BGZ3}$)$.

\end{enumerate}
}
\end{assumps}
\vspace{-.3cm}
Assumption \ref{assumptions 0}(5) implies that $b_1(M)$, the first Betti number of $M$, is at most $2$ and if it is $2$, then $M(\alpha)$ fibres over the circle with fibre a horizontal surface.

Here are some additional notations which will be used throughout the paper.
For a set $\O$ and a subset $\s\subset \O$, $\O-\s$ denotes the complement of $\s$ in $\O$.
When $\O$ is a manifold and $\s\subset \O$ a submanifold, $\O\setminus \s$ denotes
the manifold which is the metric completion of $\O-\s$ with respect to the path metric
on $\O-\s$.

\section{Culler-Shalen theory} \label{sec: cs theory}

We collect various results on the $PSL_2(\mathbb C)$-character variety of $M$ which will be used in the paper. See \cite{CGLS}, \cite{LM}, \cite{BZ2} and \cite{BCSZ2} for the details.

Denote by $\mathcal{D}$ the abelian subgroup of $PSL_2(\mathbb C)$ consisting of diagonal matrices and by $\mathcal{N}$ the subgroup consisting of those matrices which are either diagonal or have diagonal coefficients $0$. Note that $\mathcal{D}$ has index $2$ in $\mathcal{N}$ and any element in $\mathcal{N} - \mathcal{D}$ has order $2$.

The action of $SL_2(\mathbb C)$ on $\mathbb C^2$ descends to one of $PSL_2(\mathbb C)$ on $\mathbb CP^1$. We call a representation $\rho$ with values in $PSL_2(\mathbb C)$ {\it irreducible} if the associated action on $\mathbb CP^1$ is fixed point free, otherwise we call
it {\it reducible}. We call it {\it strictly irreducible} if the action has no invariant subset in $\mathbb CP^1$ with fewer than three points. Equivalently, $\rho$ is strictly irreducible if it is irreducible and is not conjugate to a representation with values in $\mathcal{N}$.

Let $\Gamma$ be a finitely generated group. The set $R_{PSL_2}(\Gamma)$ of representations of $\Gamma$ with values in $PSL_2(\mathbb C)$ admits the structure of a $\mathbb C$-affine algebraic set \cite{LM} called the {\it $PSL_2(\mathbb C)$-representation variety} of $\Gamma$. The action of $PSL_2(\mathbb C)$ on $R_{PSL_2}(\Gamma)$ determines an algebro-geometric quotient $X_{PSL_2}(\Gamma)$ whose coordinate ring is $\mathbb C[R_{PSL_2}(\Gamma)]^{PSL_2(\mathbb C)}$ and a regular map $t: R_{PSL_2}(\Gamma) \to X_{PSL_2}(\Gamma)$ \cite{LM}. This quotient is called the {\it $PSL_2(\mathbb C)$-character variety} of $\Gamma$. For $\rho \in R_{PSL_2}(\Gamma)$, we denote $t(\rho)$ by $\chi_\rho$ and refer to it as the {\it character} of $\rho$. If $\chi_{\rho_1} = \chi_{\rho_2}$ and $\rho_1$ is irreducible, then $\rho_1$ and $\rho_2$ are conjugate representations. We can therefore call a character $\chi_\rho$ reducible, irreducible, or strictly irreducible if $\rho$ has that property.

When $\Gamma$ is the fundamental group of a path-connected space $Y$, we write $R_{PSL_2}(Y)$ rather than $R_{PSL_2}(\pi_1(Y))$, $X_{PSL_2}(Y)$ rather than $X_{PSL_2}(\pi_1(Y))$, and refer to them respectively as the $PSL_2(\mathbb C)$-representation variety of $Y$ and $PSL_2(\mathbb C)$-character variety of $Y$.

We call a curve in $X_{PSL_2}(\Gamma)$ {\it non-trivial}, respectively {\it strictly non-trivial}, if it contains an irreducible character, respectively strictly irreducible character. All but at most finitely many characters on a (strictly) non-trivial curve are (strictly) irreducible.

Each $\g \in \pi_1(M)$ determines an element $f_\g$ of the
coordinate ring $\mathbb C [X_{PSL_2}(M)]$ satisfying
$$f_\g(\chi_\rho) = (\mbox{trace}(\rho(\g)))^2 - 4$$
where $\rho \in R_{PSL_2}(M)$. Each $\d \in H_1(\partial M) = \pi_1(\partial M)$ determines an element of $\pi_1(M)$ well-defined up to conjugation and therefore an element $f_\d \in \mathbb C [X_{PSL_2}(M)]$. Similarly each slope $\d$ on $\partial M$ determines an element of $\pi_1(M)$ well-defined up to conjugation and taking inverse, and so defines $f_\d \in \mathbb C[X_{PSL_2}(M)]$.

To each curve $X_0$ in $X_{PSL_2}(M)$ we associate a function
$$\| \cdot \|_{X_0}: H_1(\partial M; \mathbb R) \to [0, \infty)$$
characterized by the fact that for each $\d \in H_1(\partial M)$
we have $\|\d\|_{X_0} =   \hbox{degree}(f_\d: X_0 \to \mathbb C)$.
It was shown in \cite{CGLS} that $\| \cdot \|_{X_0}$ is a seminorm, which we refer to as the {\it Culler-Shalen seminorm} of $X_0$.

Recall that $X_0$ admits an affine desingularisation
$X_0^\nu \stackrel{\nu}{\longrightarrow} X_0$ where $\nu$ is surjective and
regular. Moreover, the smooth
projective model $\widetilde X_0$ of $X_0$ is obtained by adding a finite number of
ideal points to $X_0^\nu$. Thus
$\widetilde X_0 = X_0^\nu \cup \mathcal{I}(X_0)$ where $\mathcal{I}(X_0)$ is the set of ideal
points of $X_0$. There are natural
identifications between the function fields of $X_0, X_0^\nu,$ and $\widetilde X_0$.
Thus to each $f \in \mathbb C(X_0)$ we
have corresponding $f^\nu \in \mathbb C(X_0^\nu) = \mathbb
C(X_0)$ and $\tilde f \in \mathbb C(\widetilde
X_0) = \mathbb C(X_0)$ where $f^\nu = f \circ \nu = \tilde f|X_0^\nu$.

For $x \in \widetilde X_0$ and $\gamma \in \pi_1(M)$ we use
$Z_x(\tilde f_\gamma)$, respectively $\Pi_x(\tilde f_\gamma)$, to denote the multiplicity of $x$
as a zero, respectively pole, of $\tilde f_\gamma$. From the definition of $\|\cdot\|_{X_0}$ we see that for each $\d \in H_1(\partial M)$ we have
\begin{equation}
\|\d\|_{X_0} = \sum_{x \in \widetilde X_0} Z_x(\tilde f_\d) =
\sum_{x \in \mathcal{I}(X_0)} \Pi_x(\tilde f_\d)
\end{equation}

For each group $G$, an epimorphism $\varphi: \pi_1(M) \to G$ determines a closed injective morphism
$\varphi^*: X_{PSL_2}(G) \to X_{PSL_2}(M), \; \chi_\rho \mapsto \chi_{\rho \circ \varphi}$. In particular, we can identify $X_{PSL_2}(M(\beta))$ with
an algebraic subset of $X_{PSL_2}(M)$.

Recall  that the triple ($M$, $\a$ , $\b$) always satisfies Assumptions \ref{assumptions 0}.
The purpose of this section is to prepare some
results to be applied later to deal with a special case when
$X_{PSL_2(\c)}(M(\b))$ is  positive dimensional and contains a nontrivial curve $X_0$.
Under this extra condition,  $X_0 \subset X_{PSL_2}(M(\beta))\subset X_{PSL_2}(M)$ yields a Culler-Shalen seminorm
for which $\|\beta\|_{X_0} = 0$ (since $f_\b$ is constantly equal to $0$ on $X_0$).
In fact, if $\beta^*$ is a dual slope to $\beta$ (i.e. $\Delta(\beta, \beta^*) = 1$) and we set $s_{X_0} = \|\beta^*\|_{X_0}$, then for each slope $\delta$ on $\partial M$ we have
\begin{equation}
\label{seminorm distance}
\|\delta\|_{X_0} = \Delta(\delta, \beta) s_{X_0}.
\end{equation}
It follows from the proof of \cite[Proposition 6.2]{Bo} that for each $x \in X_0^\nu$ and slope $\delta \ne \beta$ we have $Z_x(\widetilde f_\delta) \geq     Z_x(\widetilde f_{\beta^*})$.
Set
$$J_{X_0}(\alpha) = \{x \in \widetilde X_0 : Z_x(\widetilde f_\alpha) > Z_x(\widetilde f_{\delta}) \hbox{ for some } \delta \in H_1(\partial M) - \{0\} \}$$
Since  $X_0$ is non-trivial, each element of $X_0$ is the character of a representation which is either irreducible, or has non-abelian image, or has image $\{\pm I\}$. (See \cite[\S 2]{Bo}.) In the last case, we must have $b_1(M(\beta)) = 2$ (cf. \cite[Proposition 2.8]{Bo}), so $\beta = \lambda_M$, where $\l_M$ is the rational longitude of $M$. (Note that as remarked right after Assumptions \ref{assumptions 0}, $b_1(M)$ is at most $2$.)

It is shown in \cite[Lemma 4.1]{BZ2} that the inverse image in $R_{PSL_2}(M)$ of a non-trivial irreducible curve $X_0 \subset X_{PSL_2}(M)$ has a unique $4$-dimensional component $R_0$ which is conjugation invariant and maps onto $X_0$.

\begin{lemma}
\label{lemma: factors}
Suppose that Assumptions \ref{assumptions 0} hold and let $X_0 \subset X_{PSL_2}(M(\beta)) \subset X_{PSL_2}(M)$ be a non-trivial irreducible curve. Fix $x \in J_{X_0}(\alpha) \cap X_0^\nu$ and fix $\rho \in R_0$ satisfying $\nu(x) = \chi_\rho$.

$(1)$ $\rho$ factors through $\pi_1(M(\alpha)) \to \pi_1(S^2(a,b,c)) = \Delta(a,b,c)$ and can be chosen to be either irreducible or have non-abelian image.

$(2)$ $\chi_\rho$ is a simple point of $X_{PSL_2}(M)$.

$(3)$ If $\nu(x)$ is irreducible and $X_0$ is strictly non-trivial, then
$$Z_x(\widetilde f_\alpha) = \left\{ \begin{array}{ll} 1 & \mbox{ if the image of $\rho$ is a dihedral group of order $6$ or more } \\ 2 & \mbox{ if $\rho$ is strictly irreducible} \end{array} \right.$$
\end{lemma}

\begin{proof} It follows from the discussion  preceding the lemma,
the condition $x\in  J_{X_0}(\alpha) \cap X_0^\nu$ implies  $Z_x(\tilde f_\alpha)> Z_x(\tilde f_{\beta^*})$.
Thus it follows from part (1)(b) of \cite[Proposition 6.2]{Bo}  that  $\rho(\pi_1(\partial M))$ is a  nontrivial  finite cyclic group.
In particular $\nu(x)$ is not the character of the trivial representation.
It also follows from \cite[Proposition 1.5.4]{CGLS} that $\rho(\alpha) = \pm I$, and so $\rho$ induces a homomorphism $\pi_1(M(\alpha)) \to PSL_2(\mathbb C)$.  As noted just before the statement of the lemma, we can choose $\rho$ to be irreducible or have non-abelian image. In either case, it further factors through $\pi_1(S^2(a,b,c)) = \Delta(a,b,c)$ by Lemma 3.1 of \cite{BeBo}, thus proving (1).
It then follows from \cite[Proposition and Theorem 1.1]{BeBo} that $\chi_\rho$ is a simple point of $X_{PSL_2}(M)$, so (2) holds. Finally, (3) is a consequence of the method of the proof of \cite[Lemma 6.1]{BZ2} combined with (2), which completes the proof.
\end{proof}
Set
$$\mathcal{N}_{X_0}(\alpha) = \{ x \in J_{X_0}(\alpha)\cap X_0^\nu\; | \;  \nu(x) = \chi_\rho \hbox{ where $\rho$ is irreducible and takes values in } \mathcal{N}\}$$

\begin{prop}
\label{prop: boundary values}
Suppose that Assumptions \ref{assumptions 0} hold and let $X_0$ be a curve in $X_{PSL_2}(M(\beta)) \subset X_{PSL_2}(M)$
and $\beta^* \in H_1(\partial M)$ be a dual class to $\beta$.

$(1)$  If $\tilde f_{\beta^*}$ has poles at each ideal point of $\tilde X_0$ and
$n > 1$ divides $\Delta(\alpha, \beta)$, then there is a point $x \in J_{X_0}(\alpha) \cap X_0^\nu$
such that $\rho(\pi_1(\p M)) = \mathbb Z/n$  for any $\rho$ with $\nu(x) = \chi_\r$.

$(2)$ If  $X_0$ is strictly non-trivial and $s_{X_0} \ne 0$, then
$$\Delta(\alpha, \beta) = 1 + \frac{1}{s_{X_0}} \big(2|J_{X_0}(\alpha)| - |\mathcal{N}_{X_0}(\alpha)| \big)$$
as long as   $J_{X_0}(\alpha)\subset X_0^\nu$,  $\nu(x)$ is irreducible for each $x\in J_{X_0}(\alpha)$ and
no element of $\mathcal{N}_{X_0}(\alpha)$ corresponds to the character
of a representation with image a dihedral group of order $4$.
\end{prop}

\begin{proof}
Part  (1) follows from part  (2) of \cite[Proposition 6.2]{Bo}.

Under the conditions of part  (2) of this proposition, note that a point $x$ of $\widetilde X_0$ belongs to  $ J_{X_0}(\alpha)$ if and only if
$Z_x(\widetilde f_\alpha)>Z_x(\widetilde f_{\beta^*})$.
Thus  $$\D(\alpha,\beta) s_{X_0}=\|\alpha\|_{X_0}=\sum_{x \in \widetilde X_0} Z_x(\tilde f_\alpha) = \sum_{x\in \widetilde X_0} Z_x(\tilde f_{\beta^*})+
\sum_{x\in J_{X_0}(\alpha)} Z_x(\tilde f_\alpha)=s_{X_0}+\sum_{x\in J_{X_0}(\alpha)} Z_x(\tilde f_\alpha).$$
The formula in part (2) now follows from Lemma \ref{lemma: factors}(3).
\end{proof}

\section{Bending characters of triangle group amalgams} \label{sec: amalgams}

\subsection{Curves of characters of free products of cyclic groups}
\label{subsec: prod cyclics}
In this section we construct curves in the character varieties of certain amalgams of triangle groups to be used to
prove the cases of Theorem \ref{thm: twice-punctured precise} examined in \S \ref{sec: twice-punctured semi-fibre},
\S \ref{subsec: refinements nonsep very small}, \S \ref{sec: sep not semifibre very small}, and \S \ref{sec: sep not semifibre not very small}.

Fix integers $q \geq p \geq 2$ and write $\mathbb Z/p* \mathbb Z/q = \langle a, b : a^p = b^q= 1 \rangle$.
It was shown in Example 3.2 of \cite{BZ2} that $X_{PSL_2}(\mathbb Z/p * \mathbb Z/q)$ is a disjoint union of a finite number of isolated
points and $\lfloor \frac{p}{2} \rfloor \lfloor \frac{q}{2} \rfloor$ non-trivial curves, each isomorphic to a complex line. Explicit parametrisations of these curves can be given as follows.  For integers
$j, k$ with $1\leq j \leq \lfloor \frac{p}{2} \rfloor$ and $1 \leq k \leq \lfloor \frac{q}{2} \rfloor$, set
$$\lambda = e^{\pi ij/p}, \mu = e^{\pi ik/q},\tau = \mu + \mu^{-1}.$$
For $z \in \mathbb C$ define $\rho_z \in R_{PSL_2}(\mathbb Z/p * \mathbb Z/q)$ by
$$\rho_z(a)=\pm \left(\begin{array}{cc} \lambda & 0\\0& \lambda^{-1} \end{array}\right),
\rho_z(b)=\pm \left(\begin{array}{cc}
z & 1\\z(\tau -z) -1 & \tau -z \end{array}\right).$$
The characters of the representations $\rho_z$ parameterize a curve $X(j,k) \subset X_{PSL_2}(\mathbb Z/p * \mathbb Z/q).$
Moreover, it is shown in \cite[Example 3.2]{BZ2} that the holomorphic map
$$\Psi_{(j,k)}: \mathbb C \to X(j,k), z \mapsto \chi_{\rho_z},$$
is bijective if $j < \frac p2 $ and $k < \frac q2 $, and a 2-1 branched
cover otherwise. Since $(\hbox{trace}(\rho_z(ab)))^2 - 4  = ((\lambda - \lambda^{-1})z  + \lambda^{-1} \tau)^2 - 4$, the map
$$f_{a b}: X(j,k) \to \mathbb C, \chi_{\rho_z} \mapsto (\hbox{trace}(\rho_z(ab)))^2 - 4$$
has degree $1$ if $j = \frac{p}{2}$ or $k = \frac{q}{2}$, and $2$ otherwise.

\begin{lemma}
\label{lemma: product curves}
Fix a positive integer $d > 2$ and an element $A$ of order $d$ in $PSL_2(\mathbb C)$.

$(1)$ If $1 \leq j < \frac{p}{2}$ and $1 \leq k < \frac{q}{2}$, then there is an irreducible character
$\chi_\rho \in X(j,k)$ such that $\rho(ab) = A$.

$(2)$ If $j = \frac p2$ or $k = \frac q2$,  then there is an irreducible character
$\chi_\rho \in X(j,k)$ such that $\rho(ab) = A$ for all but at most one curve $X(j,k) \subset X_{PSL_2}(\mathbb Z/p* \mathbb Z/q)$.

$(3)(a)$  If $(p,q,d) \ne (2,3,6)$, there is an irreducible representation $\rho: \mathbb Z/p * \mathbb Z/q \to PSL_2(\mathbb C)$ such that $\rho(ab) = A$.

$(b)$   If $(p,q,d) \ne (2,3,6), (2,6,3), (2,4,4)$, there is an irreducible representation $\rho: \mathbb Z/p * \mathbb Z/q \to PSL_2(\mathbb C)$ such that $\rho(a)$ has order $p$, $\rho(b)$ has order $q$, and $\rho(ab) = A$.\end{lemma}

\begin{proof}
Fix a curve $X(j,k) \subset X_{PSL_2}(\mathbb Z/p * \mathbb Z/q)$ and consider the parametrisation $\Psi_{(j,k)}: \mathbb C \to X(j,k), z \mapsto \chi_{\rho_z}$ described above. If $\omega =( \hbox{trace}(A))^2 - 4$, then $\rho_z(ab)$ is conjugate to $A$ if and only if $f_{a b}(\chi_{\rho_z}) = \omega$. Hence we must show that under the hypotheses of the lemma, we can find a $z \in \mathbb C$ such that $\rho_z$ is irreducible and $f_{a b}(\chi_{\rho_z}) = \omega$.  It is clear that $\rho_z$ is reducible if and only if $z^2 - \tau z + 1 = 0$, or equivalently, $z = \mu, \mu^{-1}$.

Let $g_\omega$ be the polynomial $g_\omega(z) = f_{a b}(\chi_{\rho_z}) - \omega = ((\lambda - \lambda^{-1})z  + \lambda^{-1} \tau)^2 - (\omega + 4)$. By construction, $\omega \in (-4, 0)$, and the reader will verify that this implies that $g$ has simple roots.
 Thus if there is no irreducible character $\chi_\rho \in X(j,k)$ such that $\rho(ab) = A$, then the two roots of $g_\omega$ are $\mu$ and $\mu^{-1}$. On the other hand, if $z$ is a root of $g_\omega$, then $\rho_z(ab)$ is conjugate to $A$ in $PSL_2(\mathbb C)$, so the eigenvalues of $\rho_z(ab)$ coincide with those of $A$, at least up to sign. Hence the eigenvalues of $\rho_\mu(ab)$ and $\rho_{\mu^{-1}}(ab)$ coincide up to sign. It follows that $\lambda \mu^{-1} \in \{\pm \lambda \mu, \pm \lambda^{-1} \mu^{-1}\}$. But this implies that either $\lambda^2 = -1$ or $\mu^2 = -1$, i.e. $j=p/2$ or $k=q/2$.
Part (1) of the lemma thus holds.

Now assume that $j = \frac p2$, so $\lambda = i$. If there is no irreducible character $\chi_\rho \in X(\frac p2,k)$ such that $\rho(ab) = A$, then by  the argument  in the preceding paragraph  the eigenvalues of $A$ are $\pm i \mu_k, \pm i \mu_k^{-1}$. (We write $\mu_k$ here to underline the dependence of $\mu$ on $k$.) Suppose as well that there is no irreducible character $\chi_\rho \in X(\frac p2,k')$ such that $\rho(ab) = A$, for some $1 \leq k' \leq \lfloor \frac q2 \rfloor$. Then the eigenvalues of $A$ are  $\pm i \mu_{k'}, \pm i \mu_{k'}^{-1}$. Thus $\mu_{k'} \in \{\pm \mu_k, \pm \mu_k^{-1}\}$. Hence $e^{2\pi i(k \pm k')/q} = 1$, or equivalently, $k ' \equiv \pm k$ (mod $q$). Our constraints on $k$ and $k'$ then imply that $k ' = k$.

The case that $k  = \frac q2$ can be handled similarly, which completes the proof of (2).

It follows from (1) and (2) that  if there is no irreducible representation $\rho: \mathbb Z/p * \mathbb Z/q \to PSL_2(\mathbb C)$ such that $\rho(ab) = A$,  then $p = 2$ and $\lfloor \frac q2 \rfloor = 1$, so $q = 2$ or $3$. The case $q = 2$ is easily ruled out using elementary properties of the dihedral group $\Delta(2, 2, d)$. Suppose then that $q = 3$. The proofs of (1) and (2) show that some eigenvalue of $A$ is conjugate to $\pm \lambda \mu = \pm i e^{\pm \pi i/3}$. But then $d = 6$, so (3)(a) holds.

To prove part (3)(b), suppose that  there is no irreducible representation $\rho: \mathbb Z/p* \mathbb Z/q \to PSL_2(\mathbb C)$ such that $\rho(a)$ has order $p$, $\rho(b)$ has order $q$, and $\rho(ab) = A$.
By (1) and (2), we see that  $p = 2$. Further, (2) implies that the value of Euler's totient function at $q$ is at most $2$, so $q$ is either $2, 3, 4$, or $6$. We have already seen that the case that $q = 2$ can be ruled out and that $q = 3$ implies that $(p, q, d) = (2, 3, 6)$. If $q = 4$ or $6$,
the argument of the previous paragraph shows that $(p, q, d) = (2, 4, 4)$ or $(2, 6, 3)$ or $(2, 6, 6)$. Since $(2, 6, 6)$ is a hyperbolic triple, there is a discrete faithful irreducible representation $\rho: \Delta(2, 6, 6) \to PSL_2(\mathbb C)$, and since there is a unique conjugacy class of elements of order $6$ in $PSL_2(\mathbb C)$, we can assume that $\rho(ab) = A$, a contradiction. This completes the proof.\end{proof}

%\begin{lemma}\label{lemma: d=2 case} Let $A$ be an order $2$ element in  $PSL_2(\mathbb C)$.

%$(1)$ If  there is no irreducible character
%$\chi_\rho \in X(j,k)$ such that $\rho(ab) = A$, then at least one of $p$ and $q$ is even.

%$(2)$ There is an irreducible  representation $\rho: \mathbb Z/p * \mathbb Z/q \to PSL_2(\mathbb C)$ such that $\rho(ab) = A$.
%\end{lemma}

%
%\begin{proof} As in the proof of  Lemma \ref{lemma: product curves}, we like to find an irreducible $\rho_z$ such that $\chi_{\rho_z}\in X(j, k)$ and
% $(\hbox{trace}(\rho_z(ab)))^2=((\lambda - \lambda^{-1})z  + \lambda^{-1} \tau)^2=0$.
%So $z=-\l^{-1}\tau/(\l-\l^{-1})$. If  this $z$ value does not give an irreducible representation, then $ -\l^{-1}\tau/(\l-\l^{-1})=\m$ or $\m^{-1}$,
%i.e.  $(\m\l)^2=-1$ or $(\m^{-1}\l)^2=-1$, which implies that at least one of  $p$ and  $q$ is even.
%So (1) holds.

%To prove (2), we may assume by part (1) that at least one of $p$ and $q$, say $p$, is even.
%We may then take  $\rho: \mathbb Z/p* \mathbb Z/q  \to PSL_2(\mathbb C)$
%to be a dihedral representation   such that $\rho(a)$ has order $2$, $\rho(b)$ has order $q$, and $\rho(ab) = A$.
%\end{proof}
%

\subsection{Amalgamated products}  \label{afp}
Fix positive integers $p_+ \leq q_+, p_- \leq q_-$ and $d$, each at least $2$, and consider the triangle groups
$$\Delta(p_+, q_+, d) = \langle a_+, b_+ : a_+^{p_+} = b_+^{q_+} = (a_+b_+)^{d}= 1\rangle$$
$$\Delta(p_-, q_-, d) = \langle a_-, b_- : a_-^{p_-} = b_-^{q_-} = (a_-b_-)^{d}= 1\rangle$$
For each $\epsilon \in \{\pm\}$, $\mathbb Z/d \cong \langle a_\epsilon b_\epsilon \rangle \leq \Delta(p_\epsilon, q_\epsilon, d)$.
Let $\psi: \langle a_+ b_+ \rangle \to \langle a_- b_- \rangle$ be an isomorphism. Then
$$\psi(a_+ b_+) = (a_- b_-)^s$$
where $1 \leq s < d$ and $\gcd(s,d) = 1$. We consider the amalgamated free product
$$\Delta(p_+, q_+, d) *_{\psi} \Delta(p_-, q_-, d)$$

\begin{lemma}
\label{lemma: homs 1}
There is a homomorphism $\rho: \Delta(p_+, q_+, d) *_{\psi}  \Delta(p_-, q_-, d) \to PSL_2(\mathbb C)$ satisfying:

$(1)$ The restriction $\rho_\epsilon$ of $\rho$ to $\Delta(p_\epsilon, q_\epsilon, d)$ is irreducible for both values of $\epsilon$.

$(2)$ If $\rho_+$  is not faithful then either $(p_+, q_+, d)$ is a Euclidean triple or $(p_-, q_-, d) = (2,3,6)$. Further,

$(a)$ if $(p_+, q_+, d)$ is a Euclidean triple, then $\rho_+$ has image isomorphic to
\vspace{-.2cm}
\begin{itemize}

\item $\Delta(2,3,3)$ when $(p_+, q_+, d)$ is a permutation of $(2,3,6)$ or $(3,3,3)$;

\vspace{.2cm} \item $\Delta(2,2,4)$ when $(p_+, q_+, d)$ is a permutation of $(2,4,4)$;

\end{itemize}
\vspace{-.2cm}
$(b)$  if $(p_+, q_+, d)$ is a spherical or hyperbolic triple and $(p_{-}, q_{-}, d) = (2,3,6)$, $\rho_+$ can be taken to factor through a representation of $\Delta(p_+, q_+, 3)$ which is faithful if $(p_+, q_+, 3)$ is spherical or hyperbolic.

$(3)$ $\rho_+(\langle a_+ b_+ \rangle) \cong \left\{ \begin{array}{ll} \mathbb Z/d & \mbox{ if $(p_{\epsilon}, q_{\epsilon}, d) \ne (2,3,6)$ for both $\epsilon$} \\
\mathbb Z/3 & \mbox{ if $(p_{\epsilon}, q_{\epsilon}, d) = (2,3,6)$ for some $\epsilon$} \end{array} \right. $

$(4)(a)$  If $d>2$, then $\rho_-(a_-)$ has order $p_-$ and either $\rho_-(b_-)$ has order $q_-$ or
\vspace{-.2cm}
\begin{itemize}

\item $(p_+, q_+, d) = (2, 3, 6), (p_-, q_-, d) = (2,6,6)$ and $\rho_-(b_-)$ has order $3$,

\vspace{.2cm} \item $(p_+, q_+, d) = (3, 3, 3), (p_-, q_-, d) = (2,6, 3)$ and $\rho_-(b_-)$ has order $3$,

\vspace{.2cm} \item $(p_-, q_-, d) = (2, 4, 4)$ and $\rho_-(b_-)$ has order $2$.
\end{itemize}

$(b)$ If  $d=2$,   $\rho_-$
is a discrete faithful representation when $(p_-, q_-, 2)$ is a spherical or hyperbolic triple, or
 an epimorphism $\Delta(3, 6, 2) \to \Delta(3, 3, 2)$,  or
 an epimorphism $\Delta(4, 4, 2) \to \Delta(4, 2, 2)$.

$(5)$ If $\rho$ is conjugate to a representation with values in $\mathcal{N}$, then $(p_\epsilon, q_\epsilon, d)$ is a permutation of either $(2, 4, 4)$ or some $(2, 2, n)$ for both values of $\epsilon$.

\end{lemma}

\begin{proof}
We construct $\rho$ by piecing together representations $\rho_+: \Delta(p_+, q_+, d) \to PSL_2(\mathbb C)$ and
$\rho_-: \Delta(p_-, q_-, d) \to PSL_2(\mathbb C)$ which agree on $\langle a_+b_+ \rangle \equiv_\psi \langle a_-b_- \rangle$.

First suppose that $(p_+, q_+, d)$ is spherical or hyperbolic. If
\vspace{-.2cm}
\begin{itemize}

\item $(p_-, q_-, d) \ne (2, 3, 6)$, we choose $\rho_+$ to be discrete and faithful;

\vspace{.2cm} \item $(p_-, q_-, d) = (2, 3, 6)$ and $(p_+, q_+, 3)$ is spherical or hyperbolic, we take $\rho_+$ to be a discrete, faithful representation of $\Delta(p_+, q_+, 3)$;

\vspace{.2cm} \item $(p_-, q_-, d) = (2, 3, 6)$ and $(p_+, q_+, 3)$ is Euclidean, then $(p_+, q_+, 3) = (2, 6, 3)$ or $(3,3,3)$. In either case we take $\rho_+$ to be an epimorphism $\Delta(p_+, q_+, 3) \to \Delta(2,3,3)$.

\end{itemize}
\vspace{-.2cm}

There are several cases to consider when $(p_+, q_+, d)$ is a Euclidean triple:
\vspace{-.2cm}
\begin{itemize}

\item if $(p_+, q_+, d) = (3,3,3)$, we take $\rho_+$ to be an epimorphism $\Delta(3,3,3) \to \Delta(2,3,3) \subset SO(3) \subset PSL_2(\mathbb C)$. Note that each element of order $3$ in $\Delta(3,3,3)$ is sent to an element of order $3$ in $\Delta(2,3,3)$.

\vspace{.2cm} \item if $(p_+, q_+, d)$ is a permutation of $(2,4,4)$, we take $\rho_+$ to be an epimorphism of $\Delta(p_+, q_+, d)$ onto $\Delta(2,2,4) \subset SO(3) \subset PSL_2(\mathbb C)$ which send $a_+b_+$ to an element of order $d$.

\vspace{.2cm} \item if $(p_+, q_+, d)$ is a permutation of $(2,3, 6)$, we take $\rho_+$ to be an  epimorphism of $\Delta(p_+, q_+, d)$ onto $\Delta(2,3,3)$

\end{itemize}
\vspace{-.2cm}
The reader will verify that assertion (3) holds with these choices.

Next we construct $\rho_-$.

In the case that  $\rho_+(\langle a_+ b_+ \rangle) \cong \mathbb Z/d$ and $d>2$, neither $(p_+, q_+, d)$ nor $(p_-, q_-, d)$ is $(2, 3, 6)$.
Lemma \ref{lemma: product curves}(3) implies that if $(p_-, q_-, d) \ne (2, 6, 3), (2, 4, 4)$, we can find an irreducible representation $\rho_-: \Delta(p_-, q_-, d) \to PSL_2(\mathbb C)$ such that $\rho_-(a_-)$ has order $p_-$, $\rho_-(b_-)$ has order $q_-$, and
$\rho_-(a_-b_-) = \rho_+(a_+b_+)^m$ where $m$ is an integer satisfying $sm+dl=1$ for some integer $l$.
So we have $\rho_-(a_-b_-)^s = \rho_+(a_+b_+)$.
This  identity shows that $\rho_+$ pieces together with $\rho_-$ to form a homomorphism $\rho$ satisfying assertions (1) through (4).
If  $(p_-, q_-, d) =(2, 6, 3)$ or $(2, 4, 4)$, we take $\rho_-$ to be an epimorphism $\Delta(2, 6, 3) \to \Delta(2, 3, 3)$ or
an epimorphism $\Delta(2, 4, 4) \to \Delta(2, 2, 4)$. A conjugate of $\r_-$ will then piece together with $\r_+$ since there is a unique
conjugacy class of elements of order $3$ in $PSL_2(\c)$ and a unique conjugacy class of elements of order $4$.

In the case that  $\rho_+(\langle a_+ b_+ \rangle) \cong \mathbb Z/2$, we choose  $\rho_-$
to be a discrete faithful representation when $(p_-, q_-, 2)$ is a spherical or hyperbolic triple. Otherwise we choose
 $\r_-$ to be an epimorphism $\Delta(3, 6, 2) \to \Delta(3, 3, 2)$  or
 an epimorphism $\Delta(4, 4, 2) \to \Delta(2, 4, 2)$. Since there is a unique
conjugacy class of elements of order $2$ in $PSL_2(\c)$, a conjugate of $\r_-$ will piece together with $\r_+$ to produce a homomorphism
$\rho$ satisfying assertions (1) through (4).

If $\rho_+(\langle a_+ b_+ \rangle) \not \cong \mathbb Z/d$, then $\rho_+(\langle a_+ b_+ \rangle) \cong \mathbb Z/3$ and some $(p_\epsilon, q_\epsilon, d)$ is $(2, 3, 6)$. If
\vspace{-.2cm}
\begin{itemize}

\item $(p_+, q_+, d) = (2, 3, 6)$ and $(p_-, q_-, 3)$ is spherical or hyperbolic, we take $\rho_-$ to be the composition of the obvious epimorphism
$\Delta(p_-, q_-, 6) \to \Delta(p_-, q_-, 3)$ with a faithful representation
$\Delta(p_-, q_-, 3) \to PSL_2(\mathbb C)$;

\vspace{.2cm} \item if $(p_+, q_+, d) = (2, 3, 6)$ and $(p_-, q_-, 3)$ is Euclidean, then $(p_-, q_-, 3) = (2, 6, 3)$ or $(3,3,3)$ and
we take $\rho_-$ to be the composition of the quotient homomorphism $\Delta(p_-, q_-, 6) \to \Delta(p_-, q_-, 3)$
with an epimorphism $\Delta(p_-, q_-, 3) \to \Delta(2, 3, 3)$;

\vspace{.2cm} \item $(p_+, q_+, d) \ne (2, 3, 6)$, then $(p_-, q_-, d) = (2, 3, 6)$ and we take $\rho_-$ to be an epimorphism
$\Delta(2, 3, 6) \to \Delta(2, 3, 3)$.

\end{itemize}
\vspace{-.2cm}
As above, a conjugate of $\rho_-$ pieces together with $\rho_+$ to yield a homomorphism satisfying assertions (1) through (4).

To complete the proof, we verify that assertion (5) holds.

If   $(p_\epsilon, q_\epsilon, d)$ is a permutation of $(2, 3, 6)$ for some $\epsilon$, then the image of $\rho_\epsilon$ is $\Delta(2,3,3)$, which does not conjugate into $\mathcal{N}$.

Suppose that $(p_\epsilon, q_\epsilon, d)$ is not a permutation of $(2, 3, 6)$ for either choice of $\epsilon$. Then
$\rho_+(\langle a_+ b_+ \rangle) \cong \mathbb Z/d$ and the reader will verify that the image of $\rho_+$ conjugates into $\mathcal{N}$ if and only if
$(p_+, q_+, d)$ is a permutation of $(2, 4, 4)$ or some $(2, 2, n)$.

Suppose that $(p_-, q_-, d)$ is not a permutation of $(2, 4, 4)$ and $d>2$. Then (3) and (4) show that $\rho_-(a_-)$ has order $p_-$, $\rho_-(b_-)$ has order $q_-$, and $\rho_-(a_-b_-)$ has order $d$. Then as $\Delta(p_-, q_-, d)$ is generated by any two of $a_-, b_-, a_-b_-$, the image of $\rho_-$ is conjugate into $\mathcal{N}$ if and only if two $p_-, q_-, d$ are $2$.
If $(p_-, q_-, 2)$ is not $(4, 4, 2)$ or $(2, n, 2)$, then our choice of  $\r_-$ guarantees that its image cannot be conjugate into
$\mathcal{N}$, which completes the proof.
\end{proof}

Fix a homomorphism $\rho: \Delta(p_+, q_+, d) *_{\psi} \Delta(p_-, q_-, d) \to PSL_2(\mathbb C)$ and denote by $\rho_\epsilon$ the restriction of $\rho$ to $\Delta(p_\epsilon, q_\epsilon, d)$. Let $Z(\rho(\langle a_+ b_+ \rangle))$ denote the centraliser of $\rho(\langle a_+ b_+ \rangle)$ in $PSL_2(\mathbb C)$ and $Z^0(\rho(\langle a_+ b_+ \rangle))$ its component of the identity. For each $S \in Z(\rho_+(\langle a_+ b_+\rangle))$, define $\rho_S: \Delta(p_+, q_+, d) *_{\psi} \Delta(p_-, q_-, d)  \to PSL_2(\mathbb C)$ to be the homomorphism determined by the push-out diagram
\begin{center}
\begin{tikzpicture}[scale=0.8]
\node at (12, 6) {$\Delta(p_+, q_+, d)$};
\node at (6, 4.5) {$ \langle a_+ b_+ \rangle$};
\node at (18, 4.5) {$PSL_2(\mathbb{C})$};
\node at (12, 3) {$\Delta(p_-, q_-, d)$};

\node at (15.55, 5.7) {$\rho_+$};
\node at (15.55, 3.2) {$S \rho_- S^{-1}$};

\draw [ ->] (6.9, 4.8) --(10.5, 5.9);
\draw [ ->] (6.9,4.2) -- (10.5, 3.15);
\draw [ ->] (13.45, 5.9) -- (16.9,4.8);
\draw [ ->] (13.45, 3.15) -- (16.9,4.3);
\end{tikzpicture}
\end{center}

We say that the character $\chi_{\rho_S}$ is obtained by {\it bending} $\chi_\rho$ by $S$. The {\it bending function} of $\rho$ is the map
$$\beta_\rho: Z^0_{PSL_2}(\rho(\langle a_+ b_+ \rangle)) \to X_{PSL_2}(\Delta(p_+, q_+, d) *_{\psi} \Delta(p_-, q_-, d)), \; S \mapsto \chi_{\rho_S}.$$
It is shown in \cite[Lemma C.1]{BoiBo} that $\beta_\rho$ is constant if and only if one of the following two conditions holds:
\begin{itemize}

\vspace{-.2cm} \item  $\rho_+(\langle a_+ b_+ \rangle) = \{\pm I\}$ and either $\rho_+(\Delta(p_+, q_+, d)) = \{\pm I\}$ or $\rho_-(\Delta(p_-, q_-, d)) = \{\pm I\}$;

\vspace{.2cm} \item $\rho_+(\langle a_+ b_+ \rangle) \ne \{\pm I\}$ and either $\rho_+(\Delta(p_+, q_+, d)$ is abelian and reducible, or $\rho_-(\Delta(p_-, q_-, d))$  is abelian and reducible, or $\rho$ is reducible.

\end{itemize}
\vspace{-.2cm}
In particular, if $\rho$ is irreducible and $\beta_{\rho}$ is constant, then $\rho_\epsilon(\Delta(p_\epsilon, q_\epsilon, d))$ is abelian and reducible for some $\epsilon$.

We say that $\rho$ can be {\it bent non-trivially} if $\beta_\rho$ is non-constant. In this case, the image of $\beta_{\rho}$ is contained in a curve in $X_{PSL_2}(\Delta(p_+, q_+, d) *_{\psi} \Delta(p_-, q_-, d) )$.

\begin{lemma}
\label{lemma: bending homs sep}
Let $\rho: \Delta(p_+, q_+, d) *_{\psi} \Delta(p_-, q_-, d) \to PSL_2(\mathbb C)$ be an irreducible homomorphism as constructed in Lemma \ref{lemma: homs 1}. Then

$(1)$ $\rho$ can be bent non-trivially to produce a non-trivial curve $Y_0 \subset X_{PSL_2}(\Delta(p_+, q_+, d) *_{\psi} \Delta(p_-, q_-, d))$.

$(2)$ $Y_0$ is strictly non-trivial if and only if $(p_\epsilon, q_\epsilon, d) \not \in \{(2,2,d), (2,4,4)\}$ for some $\epsilon$.

$(3)$  $Y_0$ has exactly two ideal points unless $p_\epsilon = q_\epsilon = d = 2$ for some $\epsilon$, in which case it has one. Further, for each integer $k$, the map $\tilde f_{c_+ (a_+b_+)^k c_-}$ has a pole at each ideal point of $Y_0$ where $c_\pm \in \{a_\pm, a_\pm^{-1}, b_\pm, b_\pm^{-1}\}$.

\end{lemma}

\begin{proof}
By construction, $\rho_\epsilon = \rho|\Delta(p_\epsilon, q_\epsilon, d)$ is irreducible for both $\epsilon$ and so it can be bent non-trivially to produce a
curve $Y_0$ in $X_{PSL_2}(\Delta(p_+, q_+, d) *_{\psi} \Delta(p_-, q_-, d) )$ determined by bending $\rho$.

If $Y_0$ is {\it not} strictly non-trivial, it is a curve of characters of representations with values in $\mathcal{N}$. Then Lemma \ref{lemma: homs 1} implies that for each $\epsilon$, $(p_\epsilon, q_\epsilon, d)$ is a permutation of either $(2,4,4)$ or a triple of the form $(2,2,n)$. Without loss of generality, we can assume that the images of both $\rho_+$ and $\rho_-$ are contained in $\mathcal{N}$.

If, for some $\epsilon$, we have $(p_\epsilon, q_\epsilon, d) = (2,n,2)$ for $n > 2$ or $(4,4,2)$, then $\rho_+(a_+b_+) \in \mathcal{N} - \mathcal{D}$. After conjugating by an element of $\mathcal{D}$, we can assume that $\langle \rho_+(a_+b_+) \rangle = \{\pm I, \pm \left(\begin{smallmatrix} 0 & 1 \\ -1 & 0   \end{smallmatrix}\right)\}$. Then $Z^0(\rho(\langle a_+ b_+ \rangle)) = \{\pm \left(\begin{smallmatrix} x & y \\ -y & x   \end{smallmatrix}\right) \; | \; x^2 + y^2 = 1\}$. Let $\rho_{(x,y)}$ denote the representation obtained by piecing together $\rho_+$ and $\left(\begin{smallmatrix} x & y \\ -y & x   \end{smallmatrix}\right) \rho_- \left(\begin{smallmatrix} x & y \\ -y & x   \end{smallmatrix}\right)^{-1}$. Since $(p_\epsilon, q_\epsilon) \ne (2,2)$, the only conjugate of $\rho_{(x,y)}$ which takes values in $\mathcal{N}$ is $\rho_{(x,y)}$ itself when $\epsilon = +$, and $\left(\begin{smallmatrix} x & y \\ -y & x   \end{smallmatrix}\right)^{-1} \rho_{(x,y)} \left(\begin{smallmatrix} x & y \\ -y & x  \end{smallmatrix}\right)$ otherwise. In either case, the fact that $\left(\begin{smallmatrix} x & y \\ -y & x   \end{smallmatrix}\right)$ conjugates an irreducible subgroup of $\mathcal{N}$ into $\mathcal{N}$ implies that $\pm \left(\begin{smallmatrix} x & y \\ -y & x   \end{smallmatrix}\right) \in \mathcal{N}$. Thus $x = 0$ or $y = 0$. It follows that for $xy \ne 0$, $\rho_{(x,y)}$ does not conjugate into $\mathcal{N}$, contrary to our assumptions. Thus $(p_\epsilon, q_\epsilon, d) \in \{(2,2,d), (2,4,4)\}$ for both $\epsilon$. This proves the forward implication of (2).

For the reverse implication, suppose that $(p_+, q_+, d)$ and $(p_-, q_-, d)$ are of the form $(2,2,d)$ or $(2,4,4)$.
As a first case, suppose that $(p_\epsilon, q_\epsilon, d)$ is either $(2,2,d)$ where $d > 2$ or $(2,4,4)$ for some $\epsilon$. Lemma \ref{lemma: homs 1} then implies that the image of $\rho_\epsilon$ is isomorphic to $\Delta(2,2,d)$ for both $\epsilon$, and we can suppose it is contained in $\mathcal{N}$. As $d > 2$, $\rho_+(a_+b_+) \in \mathcal{D}$. Hence $Z^0(\rho(\langle a_+ b_+ \rangle)) = \mathcal{D}$ and it is then easy to see that the image of $\rho_S$ is contained in $\mathcal{N}$ for all $S \in \mathcal{D}$. Hence $Y_0$ contains no characters of strictly irreducible representations.

Suppose next that $(p_\epsilon, q_\epsilon, d) = (2,2,2)$ for both $\epsilon$. Then we can conjugate $\rho$ so that its image lies in $\mathcal{N}$ and $\rho_+(a_+b_+) \in \mathcal{D}$. Hence, again we have that $Y_0$ contains no characters of strictly irreducible representations, which completes the proof of (2).

Finally we consider assertion (3) for the case that $c_+ = a_+$ and $c_- = a_-$. The other cases are treated similarly.

Without loss of generality we suppose that $\rho_+(a_+b_+) \in \mathcal{D}$ and therefore $Z^0(\rho(\langle a_+ b_+ \rangle)) = \mathcal{D}$. There is a morphism
$$g: \mathbb C^* \to Y_0, \; t \mapsto \chi_{\rho_t}$$
where $\rho_{t-} = \left(\begin{smallmatrix} t & 0 \\ 0 & t^{-1}  \end{smallmatrix}\right) \rho_- \left(\begin{smallmatrix} t^{-1} & 0 \\ 0 & t  \end{smallmatrix}\right)$. Then $\rho_t(a_+ (a_+b_+)^k a_-) = \rho_+(a_+) \rho_+(a_+b_+)^k \rho_{t-}(a_-)$. By construction, $\rho_+$ and $\rho_-$ are irreducible and $\rho_+(a_+b_+)$ is diagonal. Thus we can write
$$\rho_+(a_+) = \pm \left(\begin{matrix} a & b \\ c & d  \end{matrix}\right) \;\;\;\;\; \rho_+(a_+b_+) = \pm \left(\begin{matrix} s & 0 \\ 0 & s^{-1}  \end{matrix}\right) \;\;\;\;\; \rho_-(a_-) = \pm \left(\begin{matrix} x & y \\ z & w  \end{matrix}\right)$$
where $bc \ne 0, yz \ne 0$ and $s \ne 0$. Then
$\rho_+(a_+ (a_+b_+)^k) =  \pm \left(\begin{smallmatrix} a_0 & b_0 \\ c_0 & d_0  \end{smallmatrix}\right)$ where $b_0c_0 \ne 0$ and $\rho_t(a_-) =  \pm \left(\begin{smallmatrix} x & t^2y \\ t^{-2}z & w \end{smallmatrix}\right)$ so that
$$\rho_t(a_+ (a_+b_+)^k a_-) =  \pm \left(\begin{matrix} a_0 x + t^{-2} b_0z & * \\ * & d_0 w + t^{2} c_0y   \end{matrix}\right)$$
The trace of $\rho_t(a_+ (a_+b_+)^k a_-)$ is therefore $\pm (c_0 y t^2 +(a_0 x + d_0w) + b_0 z t^{-2})$, which diverges as $t$ tends to either $0$ or infinity since $b_0c_0yz \ne 0$.  Thus the two ideal points of $\mathbb C^*$ are sent to ideal points of $Y_0$ and $\tilde f_{a_+ (a_+b_+)^k a_-}$ has a pole at each of them. Thus $Y_0$ has at most two ideal points. It has only one if the two ideal points of $\mathbb C^*$ are sent to the same ideal point of $Y_0$.

To complete the proof, suppose that $g(t) = g(t')$. Then there is a $B \in PSL_2(\mathbb C)$ such that $\rho_{t'} = B \rho_t B^{-1}$. In particular, $B \rho_+B^{-1} = \rho_+$ and $B \rho_-^{t} B^{-1} = \rho_-^{t'}$. Since
\begin{itemize}

\vspace{-.2cm} \item  $\rho_+$ is irreducible, either $B = \pm I$ or, up to conjugation, the image of $\rho_+$ is $\mathcal{K} = \{\pm I, \pm \left(\begin{smallmatrix} i & 0 \\ 0 & -i   \end{smallmatrix}\right), \pm \left(\begin{smallmatrix} 0 & 1 \\ -1 & 0   \end{smallmatrix}\right),  \pm \left(\begin{smallmatrix} 0 & i \\ i & 0   \end{smallmatrix}\right)\}$ and $B \in \mathcal{K}$.

\vspace{.2cm} \item $\rho_-$ is irreducible, either $B = \pm \left(\begin{smallmatrix} t'/t & 0 \\ 0 & t/t'  \end{smallmatrix}\right)$ or, up to conjugation, the image of $\rho_-$ is $\mathcal{K}$ and $\pm \left(\begin{smallmatrix} 1/t' & 0 \\ 0 & t'  \end{smallmatrix}\right) B \left(\begin{smallmatrix} t & 0 \\ 0 & 1/t  \end{smallmatrix}\right)\in \mathcal{K}$.

\end{itemize}
\vspace{-.2cm}
Hence if neither $\rho_+$ nor $\rho_-$ has image isomorphic to $\Delta(2,2,2) \cong \mathcal{K}$, then $\pm I = B = \pm \left(\begin{smallmatrix} t'/t & 0 \\ 0 & t/t'  \end{smallmatrix}\right)$. Hence $t' = \pm t$. It follows that the two ideal points of $\mathbb C^*$ are sent to distinct ideal points of $Y_0$.

On the other hand, suppose that $p_\epsilon = q_\epsilon = d = 2$ for some $\epsilon$. By Lemma \ref{lemma: homs 1}, the image of $\rho_\epsilon$ is isomorphic to $\Delta(2,2,2)$ and we can suppose that this image is $\mathcal{K}$. In this case, $\rho_+(a_+ b_+) = \left(\begin{smallmatrix} i & 0 \\ 0 & -i  \end{smallmatrix}\right)\in \mathcal{K}$. If $\epsilon = +$, the reader will verify that conjugation by $B = \pm \left(\begin{smallmatrix} 0 & 1 \\ -1 & 0   \end{smallmatrix}\right)$ takes $\rho_t$ to $\rho_{\frac{1}{t}}$. Thus $g(t) = g(\frac{1}{t})$. It follows that $g$ sends the two ideal points of $\mathbb C^*$ to the same ideal point of $Y_0$. Similarly, if $\epsilon = -$, conjugation by $B = \pm \left(\begin{smallmatrix} 0 & 1/t \\ -t & 0   \end{smallmatrix}\right)$ takes $\rho_t$ to $\rho_{\frac{1}{t}}$. Thus $g(t) = g(\frac{1}{t})$, so again $g$ sends the two ideal points of $\mathbb C^*$ to the same ideal point of $Y_0$. This completes the proof.
\end{proof}

\subsection{HNN extensions} \label{hnn}

Fix integers $d, n \geq 2$ and consider the presentation
$\Delta(d, n, d) = \langle a, b : a^{d} = b^{n} = (ab)^{d}= 1\rangle$.
Let $\psi: \langle a \rangle \to \langle ab \rangle$ be an isomorphism. Then
$$\psi(a) = (ab)^s$$
where $1 \leq s \leq d - 1$ and $\gcd(s,d) = 1$. Form the HNN extension
$$\Delta(d,n,d)*_{\psi} :=  \langle a, b, t : a^{d} = b^{n} = (ab)^{d}= 1, t a t^{-1} = (ab)^s \rangle$$

A homomorphism $\rho: \Delta(d,n,d)*_{\psi} \to PSL_2(\mathbb C)$ can be thought of as a pair $(\theta, \pm A)$ where $\theta = \rho|\Delta(d,n,d)$ and $\pm A = \rho(t) \in PSL_2(\mathbb C)$ satisfies $A \theta(a) A^{-1} = \theta(ab)^s$. Conversely, a pair $(\theta, \pm A)$ where $\theta: \Delta(d,n,d) \to PSL_2(\mathbb C)$ is a homomorphism and $A \in PSL_2(\mathbb C)$ satisfies $A \theta(a) A^{-1} = \theta(ab)^s$ defines a homomorphism $\rho: \Delta(d,n,d)*_{\psi} \to PSL_2(\mathbb C)$ in the obvious way. We shall write
$$\rho = (\theta, A)$$
Note that $(\theta, A), (\theta, B) \in R_{PSL_2}(\Delta(d,n,d)*_{\psi})$ if and only if $B = AS$ for some $S \in Z_{PSL_2}(\rho(\langle a \rangle))$.
Let $Z^0_{PSL_2}(\rho(\langle a \rangle))$ denote the component of the identity of $Z_{PSL_2}(\rho(\langle a \rangle))$.
The map
$$\beta_{(\theta, A)}: Z^0_{PSL_2}(\rho(\langle a \rangle)) \to X_{PSL_2}(\Delta(d,n,d)*_{\psi}), S \mapsto \chi_{(\theta, AS)}$$
is called the {\it bending function} of $(\theta, A)$. It is shown in \cite[Lemma C.3]{BoiBo} that if $\beta_{(\theta, A)}$ is constant, then $\theta(\langle a \rangle) \ne \{\pm I\}$ and,  after a possible conjugation, either
\begin{itemize}

\vspace{-.2cm} \item  $\theta(\Delta(d,n,d)) \subset \mathcal{D}$ and $A = \pm \left(\begin{smallmatrix} 0 & 1 \\ -1 & 0 \end{smallmatrix}\right)$, or

\vspace{.2cm} \item $\theta(\Delta(d,n,d))$ and $A$ are contained in the group of upper-triangular matrices.

\end{itemize}
\vspace{-.2cm}
In particular, if $\beta_{(\theta, A)}$ is constant, then $\theta$ is reducible and $(\theta, A)$ is either reducible or conjugate into $\mathcal{N}$.

We say that $(\theta, A)$ can be {\it bent non-trivially} if $\beta_{(\theta, A)}$ is non-constant. In this case, the image of $\beta_{(\theta, A)}$ is contained in a curve in $X_{PSL_2}(\Delta(d,n,d)*_{\psi})$.

\begin{lemma}
\label{lemma: bending homs non-sep}
There is a homomorphism
$(\theta, A): \Delta(d,n,d)*_{\psi} \to PSL_2(\mathbb C)$
for which

$(1)$ $\theta$ is irreducible.

$(2)$ $a, b$ and $ab$ are sent to elements of order $d,n$ and $d$ respectively unless $(d,n) = (4,2)$. In this case $a, b$ and $ab$ are sent to elements of order $2$.

$(3)(a)$ $(\theta, A)$ can be bent non-trivially to produce a curve $Y_0 \subset X_{PSL_2}(\Delta(d,n,d)*_{\psi})$.

$(b)$ $Y_0$ is strictly non-trivial.

$(c)$ $Y_0$ has two ideal points unless $n = 2$ and $d \in\{2,4\}$ $($so $\theta(\Delta(d,n,d)) \cong \Delta(2,2,2)$$)$, in which case it has one ideal point.  For each $j, l \in \mathbb Z$, the function $\tilde f_{t a^jta^l}$ has a pole at each ideal point of $Y_0$.

\end{lemma}

\begin{proof}
If $(d,n,d)$ is a spherical triple, it is either $(2,n,2)$ or $(3,2,3)$. In either case we take $\theta$ to be injective with image $\Delta(d,n,d)$. Then $s \equiv \pm 1$ (mod $d$), and as there is a unique conjugacy class of elements of order $d \in \{2,3\}$ in $PSL_2(\mathbb C)$ we can find an $A \in PSL_2(\mathbb C)$ such that $A{\theta}(a)A^{-1} = {\theta}(ab)^s$. Then $({\theta}, A)$ satisfies (1), (2), and (3)(a) (cf. \cite[Lemma C.3]{BoiBo}).

If $(d,n,d)$ is a Euclidean triple, it is either $(3,3,3)$ or $(4,2,4)$. In the first case we take $\theta$ to have image $\Delta(3,2,3)$ and note that each element of order $3$ in $\Delta(3,3,3)$ is sent to an element of order $3$ in $\Delta(3,2,3)$. In the second case we take $\theta$ to have image $\Delta(2,2,2)$. As in the previous paragraph, there is some $A \in PSL_2(\mathbb C)$ for which $(\theta, A)$ is a homomorphism satisfying (1), (2), and (3)(a).

Assume then that $(d,n,d)$ is a hyperbolic triple. In particular $d>2$. Lemma \ref{lemma: product curves} implies that we can find an irreducible representation $\theta: \mathbb Z/d * \mathbb Z/n \to PSL_2(\mathbb C)$ such that $\theta(a)$ has order $d$, $\theta(b)$ has order $n$ and
$\theta(ab)$ is conjugate to $\theta(a)^m$, where $m$ is a chosen  integer such that $ms+dl=1$ for some integer $l$.
Then $\theta(a)$ is conjugate to $\theta(ab)^s$ and we can therefore construct a representation $(\theta, A): \Delta(d,n,d)*_{\psi} \to PSL_2(\mathbb C)$ satisfying (1), (2), and (3)(a).

Next we prove (3)(b).
Suppose that the image of $(\theta, A)$ is conjugate into $\mathcal{N}$. Without loss of generality we can suppose that its image is contained in $\mathcal{N}$. Then $A \in \mathcal{N}$ and as $\theta$ is irreducible, part (2) of the lemma implies that either $d = 2$ or $(d,n,d) = (4,2,4)$. In the first case the image of $\theta$ is the dihedral group $D_n$ while in the second it is $D_2$. In either case we can suppose that
$$\theta(\langle a \rangle) = \{\pm I, \pm \left(\begin{smallmatrix} 0 & 1 \\ -1 & 0   \end{smallmatrix}\right)\}$$
so, as we noted above, the component of the identity of the centraliser of $\theta(\langle a \rangle)$ is given by $\{\pm \left(\begin{smallmatrix} z & w \\ -w & z   \end{smallmatrix}\right) : z, w \in \mathbb C \hbox { and } z^2 + w^2 = 1\}$. Set $\rho_{(z,w)} = (\theta, A  \left(\begin{smallmatrix} z & w \\ -w & z   \end{smallmatrix}\right))$ and suppose that the image of $B\rho_{(z,w)}B^{-1}$ is contained in $\mathcal{N}$ for some $B \in PSL_2(\mathbb C)$.

If $n > 2$, then $\rho_{(z,w)}(b) = \theta(b) \in \mathcal{D}$ and so as $B\rho_{(z,w)}(b)B^{-1} \in \mathcal{N}$, we must have $B\rho_{(z,w)}(b)B^{-1} = \theta(b)^{\pm 1}$. It follows that $B \in \mathcal{N}$. But then as $BA  \left(\begin{smallmatrix} z & w \\ -w & z   \end{smallmatrix}\right)B^{-1} \in \mathcal{N}$, so $\left(\begin{smallmatrix} z & w \\ -w & z   \end{smallmatrix}\right) \in \mathcal{N}$ and therefore $z = 0$ or $w = 0$. Hence if $zw \ne 0$ and $n > 2$, $\rho_{(z,w)}$ does not conjugate to a representation with values in $\mathcal{N}$ and so $Y_0$ is strctly non-trivial.

Suppose that $n = 2$. Then the image of $\theta$ is $D_2$. If the image of $\rho_{(z,w)}$ contains a diagonal element of order different than $2$, we can proceed as in the case $n > 2$ to complete the proof that $Y_0$ is strictly non-trivial. Otherwise the image of $\rho_{(z,w)}$ coincides with the image of $\theta$, which we can take to be $\mathcal{K} = \{\pm I, \pm \left(\begin{smallmatrix} i & 0 \\ 0 & -i   \end{smallmatrix}\right), \pm \left(\begin{smallmatrix} 0 & 1 \\ -1 & 0   \end{smallmatrix}\right),  \pm \left(\begin{smallmatrix} 0 & i \\ i & 0   \end{smallmatrix}\right)\}$. Hence $A$ is an element of the abelian group $\mathcal{K}$ and so $\theta(ab)^s = A \theta(a) A^{-1} = \theta(a)$. Since $\theta(a)$ has order $2$ and $s$ is odd,
we have $\theta(ab) = \theta(a)$. But this is impossible as $\theta(a)$ and $\theta(ab)$ generate $\mathcal{K}$. This completes the proof of (3)(b).

Finally we prove (3)(c). We assume, without loss of generality, that $\theta(a)$ is diagonal, say
$$\theta(a) = \pm \left(\begin{matrix} u & 0 \\ 0 & u^{-1}  \end{matrix}\right),$$
so $Z^0_{PSL_2}(\theta(\langle a \rangle)) = \mathcal{D} \cong \mathbb C^*$. Then there is a regular map
$$g: \mathbb C^* \to Y_0, r \mapsto \chi_{\rho_r}, \; \rho_r = (\theta, \pm A  \left(\begin{smallmatrix} r & 0 \\ 0 & r^{-1}   \end{smallmatrix}\right))$$
which factors through $\mathbb C^* / \{\pm 1\} \cong \mathbb C^*$.

Write $A = \pm \left(\begin{matrix} x & y \\ z & w  \end{matrix}\right)$
and fix an integer $k$ so that $sk \equiv 1$ (mod $d$). Then as $\theta$ is irreducible and $A \theta(a)^k A^{-1} = \theta(ab)^{sk} = \theta(ab)$, the reader will verify that $x \ne 0$ and $w \ne 0$. A short calculation then shows that
$$\rho_r(t a^jta^l) = \pm \left(\begin{matrix} u^{(j+l)}x^2 r^2 + u^{(l-j)}yz  & * \\ * & u^{-(j+l)} w^2 r^{-2} + u^{-(l-j)}yz  \end{matrix}\right)$$
Hence, up to sign, the trace of $\rho_r(t a^jta^l)$ is $(u^{(j+l)}x^2) r^2 + (u^{(l-j)} + u^{-(l-j)})yz + (u^{-(j+l)} w^2) r^{-2}$, which diverges as $r$ tends to either $0$ or infinity since $x \ne 0$ and $w \ne 0$. Thus, each of the two ideal points of $\mathbb C^*$ is sent to an ideal point of $Y_0$ under the map $g$. Further, $\tilde f_{t a^jta^l}$ has a pole at each of them.

If $g(r) = g(r')$, there is a $B \in PSL_2(\mathbb C)$ such that $\rho_{r'} = B \rho_r B^{-1}$. In particular, $\theta = B \theta B^{-1}$ and $B A  \left(\begin{smallmatrix} r' & 0 \\ 0 & (r')^{-1}   \end{smallmatrix}\right) B^{-1} = A  \left(\begin{smallmatrix} r & 0 \\ 0 & r^{-1}   \end{smallmatrix}\right)$. Since $\theta$ is irreducible, either $B = \pm I$ or, up to conjugation, $\theta(\Delta(d,n,d)) = \mathcal{K}$ and $B \in \mathcal{K}$. In the former case, a simple calculation shows that $r' = \pm r$,
while in the latter, $(d, n)$ is either $(2,2)$ or $(4, 2)$.  In particular, if $(d, n)$ is neither $(2,2)$ nor $(4, 2)$, the two ideal points of $\mathbb C^*$ are sent to distinct ideal points of $Y_0$.

Suppose then that $(d, n)$ is either $(2,2)$ or $(4, 2)$ and therefore $\theta(\Delta(d,n,d)) = \mathcal{K}$. We have assumed that $\theta(a)$ is diagonal, so $\theta(a) = \pm \left(\begin{smallmatrix} i & 0 \\ 0 & -i   \end{smallmatrix}\right)$, and up to conjugation by $\pm \left(\begin{smallmatrix} \sqrt{i} & 0 \\ 0 & 1/\sqrt{i}  \end{smallmatrix}\right)$ we can suppose that $\theta(ab) = \pm \left(\begin{smallmatrix} 0 & 1 \\-1 & 0   \end{smallmatrix}\right)$. It is easy to verify that if $v \ne 0$ and we set
$$A_v = \pm \left(\begin{smallmatrix} v & i/2v \\iv & 1/2v   \end{smallmatrix}\right) = \pm \left(\begin{smallmatrix} 1 & i/2 \\ i & 1/2   \end{smallmatrix}\right) \left(\begin{smallmatrix} v & 0 \\ 0 & 1/v   \end{smallmatrix}\right),$$
then $A_v \theta(a) A_v^{-1} = \theta(ab) = \theta(ab)^s$, since $s$ is odd. Further, $B = \pm \left(\begin{smallmatrix} 0 & i \\ i & 0   \end{smallmatrix}\right)$ conjugates $(\theta, A_v)$ to $(\theta, A_{\frac{1}{2v}})$.
It follows that $g$ sends the two ideal points of $\mathbb C^*$ to the same ideal point of $Y_0$. Thus $Y_0$ has a unique ideal point.
\end{proof}

\section{The proof of Theorem \ref{thm: twice-punctured precise} when $F$ is a semi-fibre} \label{sec: twice-punctured semi-fibre}

In this section we prove the semi-fibre case of Theorem \ref{thm: twice-punctured precise}.

\begin{prop}
\label{prop: F2-semi-fibre}
Suppose that $M$ is a hyperbolic knot manifold which contains an essential twice-punctured torus $F$ of boundary slope $\beta$ and let $\alpha$ be a slope on $\partial M$ such that $M(\alpha)$ is an irreducible small Seifert manifold. If $F$ is a semi-fibre in $M$, then
$\Delta(\alpha, \beta) \leq 4$.
\end{prop}

We divide the proof of Proposition \ref{prop: F2-semi-fibre} into two cases.

\subsection{Proof of Proposition \ref{prop: F2-semi-fibre} when $M(\alpha)$ is very small}\label{subsec: semibundle}
In this subsection we prove refined versions of the very small case of Proposition \ref{prop: F2-semi-fibre}.
In fact in this case, we allow the semi-fibre $F$ to have arbitrary number of punctures.

Recall that the base orbifold of $M(\alpha)$ is $S^2(a,b,c)$ where $a,b,c$ are positive integers. The condition that
$M(\alpha)$ be very small corresponds to the requirement that either $\min\{a,b,c\} = 1$ or $(a,b,c)$ is a Platonic or a Euclidean triple.

If $\min\{a,b,c\} = 1$ the fundamental group of $\pi_1(M(\alpha))$ is cyclic and we say that $\alpha$ is a {\it $C$-type filling slope}.

If $(a,b,c)$ is a Platonic triple we
say that $\alpha$ is, respectively, a {\it $D$-type, $T$-type, $O$-type} or {\it $I$-type filling slope} if up to permutation $(a,b,c)$ is, respectively, $(2,2,n), (2,3,3), (2,3,4)$, or $(2,3,5)$.

Finally, $(a,b,c)$ is a Euclidean triple when $(a,b,c)$ is one of $(2,3,6), (2,4,4)$, or $(3,3,3)$.

 \begin{prop}
 \label{prop: semi very small}
Let $M$ be a hyperbolic knot manifold that contains an $m$-punctured torus semi-fibre with boundary slope $\beta$, and let $\alpha$ be a slope on $\partial M$ such that $M(\alpha)$ is an irreducible  very small Seifert manifold. Then

$(1)$ $M(\alpha)$ is not of $T$- or $I$-type.

$(2)$ If $M(\alpha)$ is of $C$-type or has base orbifold $S^2(3,3,3)$ then $\Delta(\alpha,\beta) = 1$ and $m=2$.

$(3)$ If $M(\alpha)$ is of $O$-type or has base orbifold $S^2(2,3,6)$ then either
\vspace{-.2cm}
\begin{itemize}

\item $\Delta(\alpha,\beta) = 1$ and $m = 2$ or $6$, or

\vspace{.2cm} \item $\Delta(\alpha,\beta) = 3$ and $m = 2$.

\end{itemize}

$(4)$  If $M(\alpha)$ has base orbifold $S^2(2,4,4)$ then either
\vspace{-.2cm}
\begin{itemize}

\item $\Delta(\alpha,\beta) = 1$ and $m = 2, 4$ or $8$, or

\vspace{.2cm} \item $\Delta(\alpha,\beta) = 2$ and $m = 2$ or $4$, or

\vspace{.2cm} \item $\Delta(\alpha,\beta) = 4$ and $m = 2$.

\end{itemize}

$(5)$ If $M(\a)$ is of $D$-type, then $\D(\a,\b)\leq 3$.
\end{prop}

\pf The semi-fibre $F$ separates $M$ into two components $X^+$ and $X^-$ each a twisted $I$-bundle over an $\frac{m}{2}$-punctured Klein bottle.
Hence $\pi_1(F)$ is an index two subgroup of $\pi_1(X^\e)$ for each $\e\in\{\pm\}$
and there is an epimorphism $\varphi: \pi_1(M) \to \mathbb Z/2* \mathbb Z/2$ whose kernel is $\pi_1(F)$.
Clearly $\varphi(\beta)=1$. We use $x$ and $y$ to denote the generators of the two $\mathbb Z/2$ factors.
Let $\b^*$ be a dual slope to $\b$ on $\p M$, then we may assume that each component of $\b^*\cap X^\e$ is an $I$-fibre of $X^\e$.
Choose a disk region in $F$, which contains $F\cap \beta^*$, as a fat base point for each of $F, X^\e, M, M(\b), M(\a)$.
Then each component of $\b^*\cap X^\e$ represents  an element of $\pi_1(X^\e)$ which is not contained in $\pi_1(F)$, and thus
its image under $\varphi$ is the generator of one of the $\mathbb Z/2$ factors.
It follows that $\varphi(\b^*)=(xy)^{m/2}$, at least up to exchanging $x$ and $y$.

Write
 $\alpha=(\beta^*)^p\beta^q$ where $p=\D(\alpha,\beta)$.
Then  $\varphi(\alpha) = \varphi(\beta^*)^p = (xy)^{mp/2}$.
 It follows that $\varphi$ induces  a surjective homomorphism
 $\overline{\varphi}:\pi_1(M(\alpha))\ra (\z/2*\z/2)/\langle \langle (xy)^{mp/2}\rangle \rangle = \langle x,y;x^2=y^2=(xy)^{mp/2}=1\rangle =D_{mp/2}$.
Set $n=mp/2$.
As $H_1(D_n)\cong\mathbb Z/2$ ($n$ odd) and $\mathbb Z/2\oplus\mathbb Z/2$ ($n$ even)
while a very small Seifert manifold of type $T$ or $I$ has odd order  first homology, part (1) of the proposition holds.
When $n>1$,  $D_n$ is  an irreducible subgroup of $PSL_2(\c)$. So if $M(\a)$ is of $C$-type, then $n=1$ which means
 $p=\D(\a,\b)=1$ and $m=2$.
Hence when $\D(\a,\b)>1$, the base orbifold of $M(\a)$ is $S^2(a,b,c)$ with $a,b,c\geq 2$
and  the surjective homomorphism $\overline{\varphi}$ must factor
through $\pi_1(S^2(a,b,c))=\D(a,b,c)$ (\cite[Lemma 3.1]{BeBo}).
But then at least two of $a, b, c$ are even as otherwise the abelianisation of $\D(a,b,c)$
is of odd order and so there is no surjective homomorphism
$\D(a,b,c) \to D_n$. In particular $(a,b,c) \ne (3,3,3)$. So (2) holds.

In the case that $\alpha$ is of $O$-type (i.e. $(a,b,c) = (2,3,4)$) or $(a, b, c) = (2, 3, 6)$, then $n=1$ or $3$
since $D_3$ is the only irreducible dihedral group which is a quotient of either $\D(2,3,4)$ or $\D(2,3,6)$ (cf. \cite[Lemma 5.3]{BZ1},
\cite[Proposition 5.4]{Bo}). This yields (3).

Similarly, if $(a,b,c) = (2,4,4)$, then $n=1, 2$ or $4$ (cf. \cite[Proposition 5.3]{Bo}), which gives (4).

Finally, suppose that $\alpha$ is of $D$-type. Then $(a,b,c) = (2,2,l)$ for some $l \geq 2$ and therefore $\alpha$ is a Klein bottle filling slope on $\partial M$. By assumption $\beta$ is also a Klein bottle filling slope. It then follows from \cite[Theorem 1.1]{MS} that $\Delta(\alpha, \beta) \leq 4$.  Further, such triples $(M;\alpha,\beta)$ with $\Delta(\alpha,\beta) = 4$ are determined by Lee in \cite{L1}. One checks that in all cases both $M(\alpha)$ and $M(\beta)$ are toroidal. Here are the details.

Lee (\cite{L1}) defines two infinite families of hyperbolic knot manifolds:
\vspace{-.2cm}
\begin{itemize}

\item $X_n, \, n \in \zed$, each having a pair of Klein bottle Dehn filling slopes $\alpha_n$ and $\beta_n$ with $\Delta(\alpha_n,\beta_n) = 4$,

\vspace{.2cm} \item $Y(r), \,  r \notin \{\infty,0,4\}$, where $Y(r)$ is the $r$-Dehn filling of the exterior $Y$ of a 2-component link in $S^3$ for which $Y(r)(0)$ and $Y(r)(4)$ contain Klein bottles,

\end{itemize}
\vspace{-.2cm}
and shows that any triple $(M;\alpha,\beta)$ such that $M(\alpha)$ and $M(\beta)$ contain Klein bottles with $\Delta(\alpha,\beta) = 4$ is homeomorphic to either $(X_n;\alpha_n,\beta_n)$ for some $n \in \zed$, or $(Y(r);0,4)$ for some $r \ne \infty, 0,$ or $ 4$.

Now $X_n$ is obtained by Dehn filling along one boundary component of $M_2 \cong N(-1/2)$, the exterior of the 10/3 2-bridge link (\cite{L1}). Comparing the parametrisations of the exceptional slopes $\{\infty,-2,-1,0,1,2 \}$ in \cite{L1} and $\{\infty,-4,-3,-2,-1,0 \}$ in \cite[Table A.1]{MP}, we see that $X_n \cong N(-1/2,(2n-5)/2) \cong N(-1/2,-(2n+3)/2)$, where the second homeomorphism follows from
\cite[Proposition 1.5(4)]{MP}. An examination of Table A.8 and (for the case $N(-1/2,-5/2) \cong N(-1/2,-3/2)$) Table A.4 in \cite{MP} now shows that $\{ \alpha_n,\beta_n \} = \{-4,0 \}$ and $X_n(\alpha_n)$ and $X_n(\beta_n)$ are toroidal.

Finally we consider the manifolds $Y(r)$. By \cite{L1}, $Y(0) \cong Y(4) \cong Q(2,2) \cup Wh$, and the Klein bottle fillings on $Y(r)$ are $Y(r)(0) \cong Y(0,r)$ and $Y(r)(4) \cong Y(0,4-r)$ (see \cite[Lemma 2.2]{L1}). Now $Y(0,r) \cong Q(2,2) \cup Wh(r)$, with the induced parametrisation of slopes on $Wh$. From Lemma 2.4 of \cite{L1} we have that $Wh(2) \cong Q(2,4)$ and $Wh(3) \cong Q(2,3)$, so from Table A.1 of \cite{MP} it follows that this parametrisation of slopes on $Wh$ is related to the standard parametrisation on $N(-1)$ by $Wh(r) \cong N(-1,r-3)$. Since $r \ne \infty$, Table A.1 of \cite{MP} shows that $Wh(r)$ has incompressible boundary, and hence $Y(0,r)$ is toroidal.
Part (5) is proved
\qed

We can refine the bounds given in Proposition \ref{prop: semi very small} in certain situations, which we describe next.

The semi-fibre $F$ splits $M$ into two components $X^+, X^-$ and if $\widehat F$ denotes the closed surface in $M(\beta)$
obtained by attaching disjoint meridian disks of the $\beta$-filling solid torus to $F$, then $\widehat F$ splits $M(\beta)$
into two twisted $I$-bundles $\widehat X^+$ and $\widehat X^-$ over the Klein bottle. Let $\phi_\epsilon$ denote the slope on $\widehat F$ of the Seifert fibre
structure on $\widehat X^\epsilon$ with base orbifold $D^2(2,2)$ and set
$$d = \Delta(\phi_+, \phi_-)$$
We consider constraints on $d$, $M(\alpha)$ when $\Delta(\alpha, \beta) > 1$. Given the latter condition, Proposition \ref{prop: semi very small} implies that $\alpha$ has type $D$ or $O$, or $(a,b,c)$ is either $(2,3,6)$ or $(2,4,4)$.

\begin{prop}
\label{prop: very small d not 1}
Suppose that $F$ is an $m$-punctured semi-fibre and that $\alpha$ has type $O$, so $(a,b,c) = (2,3,4)$, or $(a,b,c)$ is either $(2,3,6)$ or $(2,4,4)$. Suppose further that $d \ne 1$ and $\Delta(\alpha, \beta) > 1$. Then $(a,b,c) = (2,4,4)$ and $\Delta(\alpha, \beta) = 2$.
\end{prop}

\pf
We saw in the proof of Proposition \ref{prop: semi very small} that there is an epimorphism
$$\varphi_1: \pi_1(M(\beta)) \to \mathbb Z/2* \mathbb Z/2$$
whose kernel is $\pi_1(\widehat F)$. Pulling back the curve $X(1,1) \subset X_{PSL_2}(\mathbb Z/2* \mathbb Z/2)$ of \S \ref{subsec: prod cyclics} produces a curve
$$X_1 = (\varphi_1)^*(X(1,1)) \subset X_{PSL_2}(M(\beta)) \subset X_{PSL_2}(M)$$
More generally, for each positive integer $d'$ dividing $d$ we can construct a curve $X_{d'}$ in the character variety of $M$ as follows.

Killing the fibre classes $\phi_+, \phi_-$ in $\pi_1(M(\beta))$ yields a surjection
$$\varphi_d: \pi_1(M(\beta)) \to \Delta(2, 2, d) *_{\mathbb Z / d} \Delta(2, 2, d)$$
where in the case that $d = 0$ we take $\Delta(2, 2, d)$ to be $\mathbb Z / 2 * \mathbb Z / 2$ and $\mathbb Z / d$ to be the central $\mathbb Z$ of index $2$ in each of the $(\mathbb Z / 2 * \mathbb Z / 2)$-factors.

If $d \geq 2$ let $Y_d \subset X_{PSL_2}(\Delta(2, 2, d) *_{\psi} \Delta(2, 2, d))$ be a curve constructed by bending a representation
$\Delta(2, 2, d) *_{\psi} \Delta(2, 2, d) \to PSL(2, \mathbb C)$ which restricts to an irreducible representation with image $D_d$
on each of the two copies of $\Delta(2, 2, d)$. (See Lemma \ref{lemma: bending homs sep}.) Set
$$X_d = (\varphi_d)^*(Y_d) \subset X_{PSL_2}(M(\beta)) \subset X_{PSL_2}(M)$$
Lemma \ref{lemma: bending homs sep}(3) shows that $X_d$ has exactly one ideal point if $d = 2$ and two if $d > 2$. Further, $\tilde f_{\beta^*}$ has a pole at each ideal point of $X_d$ so that $s_{X_d} \geq 1$ if $d = 2$ and $s_{X_d} \geq 2$ if $d > 2$.

In the case that $d = 0$, one can use the method of proof of Lemmas \ref{lemma: product curves}, \ref{lemma: homs 1}, and \ref{lemma: bending homs sep} to produce a non-trivial curve $Y_0$ in $X_{PSL_2}((\mathbb Z / 2 * \mathbb Z / 2) *_\mathbb Z (\mathbb Z / 2* \mathbb Z / 2))$ obtained by bending an irreducible  homomorphism $(\mathbb Z / 2 * \mathbb Z / 2) *_\mathbb Z (\mathbb Z / 2 * \mathbb Z / 2) \to PSL_2(\mathbb C)$ for which $\rho(a_+ b_+)$ is an element of infinite order in $PSL_2(\mathbb C)$. In particular, the image of each representation with character in $Y_0$ is irreducible and infinite. As above,
$$X_0 = (\varphi_0)^*(Y_0) \subset X_{PSL_2}(M(\beta)) \subset X_{PSL_2}(M)$$
has two ideal points and $\tilde f_{\beta^*}$ has a pole at each of them. Hence $s_{X_0} \geq 2$.

Similarly, if $d' > 1$ is a positive integer dividing $d$ we can use the obvious quotient homomorphism
$$\pi_1(M(\beta)) \to \Delta(2, 2, d) *_{\mathbb Z / d} \Delta(2, 2, d) \to \Delta(2, 2, d') *_{\mathbb Z / d'} \Delta(2, 2, d')$$
to construct a curve
$$X_{d'} \subset X_{PSL_2}(M(\beta)) \subset X_{PSL_2}(M)$$
each of whose characters is irreducible. Further,
$s_{X_d'} \geq 1$ if $d' = 2$ and $s_{X_d'} \geq 2$ if $d' > 2$.

By construction, the image of each representation whose character lies on some $X_{d'}$ contains a normal subgroup isomorphic to $\mathbb Z/  d'$.

It is shown in \cite[Lemma 3.1]{BeBo} that if $\rho: \pi_1(M(\alpha)) \to PSL_2(\mathbb C)$ has non-diagonalisable image and has a non-trivial character, then it factors through the triangle group $\Delta(a, b, c)$. Therefore by Lemma 5.3 of \cite{BZ1} and Propositions 5.3 and 5.4 of \cite{Bo}, $\rho$ has image
$$\mbox{image}(\rho) \cong \left\{
\begin{array}{ll}
D_3 \mbox{ or } O_{24} = \Delta(2,3,4) & \mbox{ if $\alpha$ is of $O$-type} \\
D_3 \mbox{ or } T_{12} = \Delta(2,3,3) & \mbox{ if $(a,b,c) = (2,3,6)$} \\
D_2 \mbox{ or } D_4 & \mbox{ if $(a,b,c) = (2,4,4)$}
\end{array} \right.
$$
In particular, the image of $\rho$ is a finite group whose non-trivial elements have order $2, 3$, or $4$.

No $X_{d'}$ contains the character of a representation with image $T_{12}$ or $O_{24}$. This is clear when $d' = 1$ since neither of these groups is generated by two elements of order $2$, and it is true when $d' > 1$ since neither $T_{12}$ nor $O_{24}$ contain non-trivial cyclic normal subgroups. Thus the only representations of $\pi_1(M(\alpha))$ whose characters lies on some $X_{d'}$ have image $D_2, D_3$, or $D_4$.

Each of the curves $X_{d'}$ we constructed above has $s_{X_{d'}} > 0$ and so as
$$\|\alpha\|_{X_{d'}} = \Delta(\alpha, \beta) s_{X_{d'}}$$
(cf. (\ref{seminorm distance})) and $\Delta(\alpha, \beta) > 1$, $J_{X_{d'}}(\alpha) \ne \emptyset$ for each $d'$. Note as well that since $\tilde f_{\beta^*}$ has a pole at each ideal point of $X_{d'}$, so does $\tilde f_{\alpha} = \tilde f_{(\beta^*)^p}$ where $p = \Delta(\alpha, \beta) > 1$. Hence, any element of $J_{X_{d'}}(\alpha)$ is contained in $X_{d'}^\nu$ and if its image in $X_{d'}$ is the character $\chi_\rho$, then $\chi_\rho$ is a smooth point of $X_{PSL_2}(M)$ (Lemma \ref{lemma: factors}). In particular, $X_{d'}$ is the unique curve in $X_{PSL_2}(M)$ containing it. We know, again by Lemma \ref{lemma: factors}, that $\rho(\alpha) = \pm I$, and so $\rho$ induces irreducible homomorphism
$$\bar \rho: \pi_1(M(\alpha)) \to PSL(2, \mathbb C)$$
Thus the image of $\rho$ is  $D_2, D_3$, or $D_4$. This excludes the possibility that $d' = 0$ as all representations with characters on $X_0$ have infinite image. It also excludes the possibility that $d' > 4$. Thus $d \in \{2, 3, 4\}$.

Since $d > 1$, there are at least two curves $X_1$ and $X_d$, and as $\Delta(\alpha, \beta) > 1$, each has $J_{X_{d'}}(\alpha) \ne \emptyset$ and therefore contains the character of a representation with dihedral image. This rules out the cases that $(a, b, c)$ is either $(2,3,4)$ or $(2,3,6)$, since these groups admit a unique character of an irreducible representation with dihedral image. Thus $\Delta(\alpha, \beta) \leq 1 \mbox{ when } d \ne 1 \mbox{ and } (a,b,c) = (2,3,4) \mbox{ or }(2,3,6)$.

Finally suppose that $(a, b, c) = (2, 4, 4)$. Then $\pi_1(M(\alpha))$ admits exactly three irreducible characters
of representations with values in $PSL(2, \mathbb C)$, one with image $D_2$ and two with image $D_4$, by Proposition 5.3 of \cite{BeBo}. Thus $d$ is either $2$ or $4$, as is $\Delta(\alpha, \beta)$ by Proposition \ref{prop: boundary values}. Further, the latter result implies that if $\Delta(\alpha, \beta) = 4$, each of disjoint $X_1, X_2$ and $X_4$ contains at least two elements of $J_{X_{d'}}(\alpha)$, which is impossible. Thus $\Delta(\alpha, \beta) = 2$.
\qed

\subsection{Proof of Proposition \ref{prop: F2-semi-fibre} when $M(\alpha)$ is not very small}
\label{subsec: 7.2}

In this subsection we verify Proposition \ref{prop: F2-semi-fibre} in the case that $M(\alpha)$ is not very small.

There is a unique $2$-fold cover $\widetilde M \to M$ where $\widetilde M$ is an $F$-bundle over $S^1$. Indeed, using the notation of the proof of Lemma \ref{prop: semi very small}, it corresponds to the composition of $\varphi$ with the epimorphism $\mathbb Z/2 * \mathbb Z/2 \to \mathbb Z/2$ with kernel $\langle xy \rangle$.
Let $\theta$ denote the covering involution of this cover. The reader will verify that $\partial \widetilde M$ has two boundary components $T_1$ and $T_2$, so the covering map restricts to a homeomorphism between $T_i$ and $\partial M$ for each $i$. Hence each slope $\gamma$ on $\partial M$ lifts to a slope $\gamma_i$ on $T_i$ ($i = 1, 2$) such that $\theta(\gamma_1) = \gamma_2$. Thus there is an induced $2$-fold cover $\widetilde M (\gamma_1, \gamma_2) \to M(\gamma)$.

Pull back the Seifert structure on $M(\alpha)$ to $\widetilde M(\alpha_1, \alpha_2)$. The base orbifold $\mathcal{B}$ of $\widetilde M(\alpha_1, \alpha_2)$ covers that of $M(\alpha)$ with degree $1$ or $2$, and so is either $S^2(a,b,c)$ or $S^2(\frac{a}{2}, \frac{b}{2}, c,c)$, at least up to a permutation of $a, b, c$. Since $M(\alpha)$ is not very small, $\mathcal{B}$ has at least three cone points. In particular, $\widetilde M(\alpha_1, \alpha_2)$ is not a lens space.

\begin{lemma}
\label{lemma: c}
Suppose that $M$ is a hyperbolic knot manifold which contains an essential twice-punctured torus $F$ of boundary slope $\beta$ and let $\alpha$ be a slope on $\partial M$ such that $M(\alpha)$ is an irreducible small Seifert manifold. Suppose as well that $F$ is a semi-fibre in $M$ and $M(\alpha)$ is not very small. If $\Delta(\alpha, \beta) > 1$, then the cover $\mathcal{B} \to S^2(a,b,c)$ has degree $2$ so for some permutation of $a,b,c$ we can suppose that $a$ and $b$ are even and that $\mathcal{B} \cong S^2(a/2, b/2,c,c)$. Further, $c \geq 2\Delta(\alpha, \beta) - 2$.
\end{lemma}

\begin{proof}
Let $X_0 \subset X_{PSL_2}(M)$ be the $\mathbb Z/2 * \mathbb Z/2$-curve $X_0$ constructed in the proof of Lemma \ref{prop: semi very small} . From the description of $\widetilde M \to M$ above, the image of $X_0$ in $X_{PSL_2}(\widetilde M)$ is a curve of characters of representations which factor through the cyclic subgroup
$\langle xy \rangle$ of $\mathbb Z/2 * \mathbb Z/2$. Hence the image is a curve of reducible characters.
We saw in the proof of Lemma \ref{prop: semi very small}  that $J_{X_0}(\alpha) \subset X_0^\nu = X_0$ and that each element of $J_{X_0}(\alpha)$ is irreducible.

Suppose that $\Delta(\alpha, \beta) > 1$. If $\chi_\rho \in J_{X_0}(\alpha)$, then $\rho$ factors through a homomorphism $\bar \rho: \pi_1(S^2(a,b,c)) \to PSL_2(\mathbb C)$ (Lemma \ref{lemma: factors}), and therefore the restriction of $\rho$ to $\pi_1(\widetilde M)$ factors through $\bar \rho|\pi_1(\mathcal{B})$. The image of $\rho$, and therefore that of $\bar \rho$, is a dihedral group $D_n$ where $n \geq 2$ divides $\Delta(\alpha, \beta)$. Since $\rho$ is irreducible, but $\rho|\pi_1(\widetilde M)$ is reducible,
$\rho(\pi_1(\widetilde M)) = \bar \rho(\pi_1(\mathcal{B}))$ is the index $2$ cyclic subgroup $C_n$ of $D_n$. Thus the cover $\mathcal{B} \to S^2(a,b,c)$ must have degree $2$. It follows that after possibly permuting $a,b,c$, we can suppose that $a$ and $b$ are even and that $\mathcal{B} \to S^2(a,b,c)$ is the obvious cover $S^2(a/2, b/2,c,c) \to S^2(a,b,c)$. This proves the first assertion of the lemma.

In order to prove the second assertion, note that by the previous paragraph the generators of order $a$ and $b$ in $\pi_1(S^2(a,b,c))$ are sent to $D_n - C_n$. Thus $\bar \rho$ factors through the dihedral group $\Delta(2,2,c) \cong D_c$ in the obvious way. Hence, for each $\chi_{\rho'} \in J_{X_0}(\alpha)$, $\rho'$ factors as a composition $\pi_1(M) \to \pi_1(M(\alpha)) \to \pi_1(S^2(a,b,c)) \to \Delta(2,2,c) \to PSL_2(\mathbb C)$. Since there are $\lfloor \frac{c}{2} \rfloor$ characters of irreducible representations $\Delta(2,2,c) \to PSL_2(\mathbb C)$ and the image of each one conjugates into $\mathcal{N}$, Proposition \ref{lemma: factors}(3) implies that $\Delta(\alpha, \beta) \leq 1 + \lfloor \frac{c}{2} \rfloor$. Hence $c \geq 2\Delta(\alpha, \beta) - 2$, which completes the proof.
\end{proof}

We will make use of the involution $\tau_F$ depicted in Figure
\ref{bgz5-fig25}. It is central in the mapping class group of $F$. The
quotient $F / \tau_F$ is the $2$-orbifold $D^2(2,2,2,2)$.

\begin{figure}[!ht]
\includegraphics{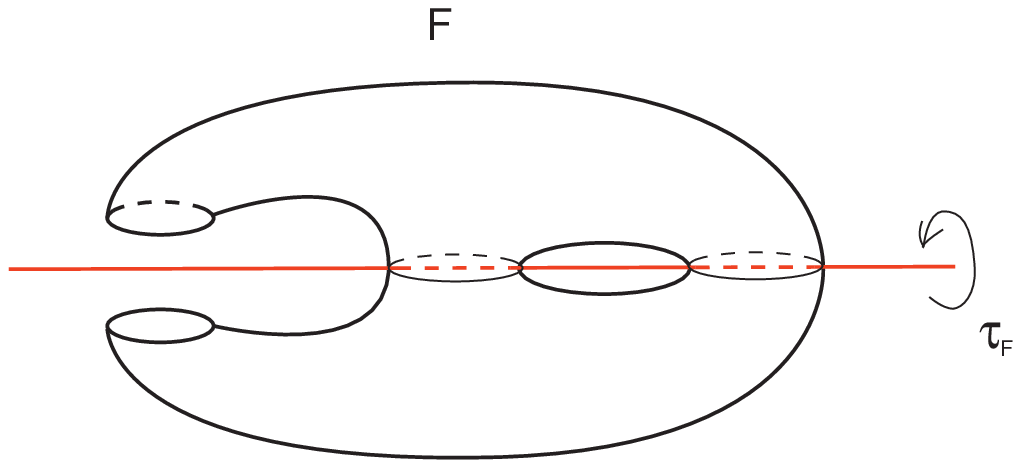} \caption{ }\label{bgz5-fig25}
\end{figure}

There is an involution $\tau$ on $\widetilde M$ induced by $\tau_F$. The quotient $\widetilde M / \tau$ is homeomorphic to $(V, L)$ where $V$ is a solid torus and $L$ is a $4$-braid. We claim that for each slope $\gamma$ on $\partial M$, $\tau$ induces an involution $\tau_\gamma$ of each $\widetilde M (\gamma_1, \gamma_2)$. To see this, note that $\tau$ is the lift of an involution $\tau_0$ of $M$ obtained by applying $\tau_F$ semi-fibre by semi-fibre. Then $\tau_0$ restricts to a hyperelliptic involution of $\partial M$, and so acts on $H_1(\partial M)$ by multiplication by $-1$. Since $\tau$ is a lift of $\tau_0$, the isomorphism $\tau_*: H_1(T_1) \to H_1(T_2)$ is multiplication by $-1$ under the identification induced by $\theta$. In particular, $\tau$ sends the slope $\gamma_1$ to $\gamma_2$, which is what we needed to prove.

The restriction of the quotient map $\widetilde M \to V$ to each $T_i$ yields a homeomorphism $T_i \to \partial V$ such that for any slope $\gamma$ on $\partial M$, $\gamma_1$ and $\gamma_2$ determine the same slope $\bar \gamma$ on $\partial V$. Clearly $\widetilde M(\gamma_1, \gamma_2)/\tau_\gamma = V(\bar \gamma)$. Further, the branch set $L_\gamma \subset V(\bar \gamma)$ of this quotient is $L \subset V \subset V(\bar \gamma)$.

Note that $\bar \beta$ is a meridian of $V$ while if $\beta^*$ is a dual class to $\beta$, $\bar \beta^*$ is a longitude of $V$. Thus after possibly changing the sign of $\alpha$ we can write
\begin{equation}\label{alpha 4}
\text{\em $\alpha = p \beta^* + q \beta$}
\end{equation}
where $p, q$ are coprime. After possibly changing the signs of $\beta^*$ and $\beta$ we may assume that
\begin{equation}\label{p 4}
\text{\em $p = \Delta(\alpha, \beta)$}
\end{equation}
By construction
\begin{equation}\label{Valpha}
\text{\em $V(\bar \alpha) \cong L(p, q)$}
\end{equation}
As in the proof of \cite[Lemma 4.1]{BGZ3}, we can assume that $\tau_\alpha$ preserves the Seifert structure on $\widetilde M(\alpha_1, \alpha_2)$.

\begin{lemma} \label{--quotient 4}
Suppose that $M$ is a hyperbolic knot manifold which contains an essential twice-punctured torus $F$ of boundary slope $\beta$ and let $\alpha$ be a slope on $\partial M$ such that $M(\alpha)$ is an irreducible small Seifert manifold. Suppose as well that $F$ is a semi-fibre in $M$ and $M(\alpha)$ is not very small. If $\tau_\alpha$ reverses the orientation of the Seifert fibres of $\widetilde M(\alpha_1, \alpha_2)$, then $\Delta(\alpha, \beta) \leq 4$.
\end{lemma}

\pf Suppose that $\tau_\alpha$ reverses the orientation of the fibres of $\widetilde M(\alpha_1, \alpha_2)$. The method of proof of \cite[Lemma 4.4]{BGZ3} shows that either

\indent \hspace{.3cm} $(a)$ $\widetilde M(\alpha_1, \alpha_2)$ has base orbifold $S^2(p,p,m)$ for some $m \geq 2$ where the branch set $L_\alpha$ in $L(p,q)$  \\ \indent \hspace{.85cm}   is a closed $m/n$-rational tangle in a Heegaard solid torus as shown in  Figure \ref{bgz5-fig24} (the example  \\ \indent \hspace{.85cm}  shown in Figure \ref{bgz5-rational tangles} should suffice to illustrate our convention for rational tangles), or

\indent \hspace{.3cm} $(b)$ $\widetilde M(\alpha_1, \alpha_2)$ has base orbifold $S^2(p,p,m,s)$ for some $m, s \geq 2$ where the branch set $L_\alpha$ in \\ \indent \hspace{.85cm}  $L(p,q)$ is a closed two tangle Montesinos link with rational tangles $m/n$ and $s/t$.

Pull $L_\alpha$ back to the universal cover $S^3$ of
$L(p,q)$ and call the resulting link $\widetilde
L_\alpha$.  Then $\tilde L_\alpha$ is a Montesinos
link with $r \ge p$ rational tangles. The 2-fold cover
$\Sigma_2(\tilde L_\alpha)$ of $S^3$ branched over
$\tilde L_\alpha$ is Seifert with base orbifold a
2-sphere with $r$ cone points. By \cite[Theorem
5.3]{S} its Heegaard genus is realized by a vertical
splitting, of genus $r-1$, unless it is as in case (2)
of that theorem, in which case $r$ is even and $\ge 4$
and its Heegaard genus is realized by a horizontal
splitting of genus $r-2$. On the other hand,
$L_\alpha$ is the image of $L$ in $V(\bar \alpha)$, so
$\widetilde L_\alpha$ is the closure of a $4$-braid.
Hence the Heegaard genus of $\Sigma_2(\widetilde
L_\alpha)$ is at most $3$. Hence either $r-1 \le 3$,
or $r \ge 4$ is even and $r-2 \le 3$. It follows that
$\Delta(\alpha, \beta) = p \le r \le 4$. \qed

\begin{figure}[!ht]
\centerline{\includegraphics{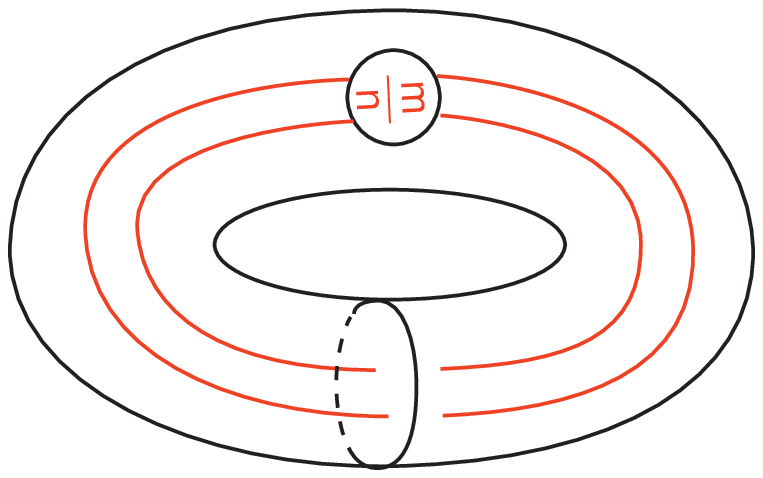}} \caption{}\label{bgz5-fig24}
\end{figure}

\begin{figure}[!ht]
\centerline{\includegraphics{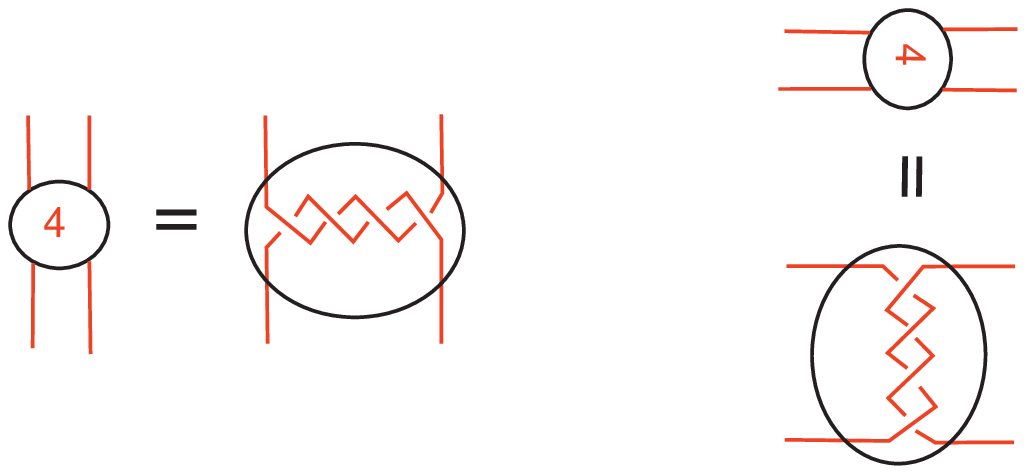}} \caption{}\label{bgz5-rational tangles}
\end{figure}

For the rest of this section we suppose that $\tau_\alpha$ preserves the orientation of the fibres of $\widetilde M(\alpha_1, \alpha_2)$. As in \cite[Lemma 4.3]{BGZ3} we have

\begin{lemma} \label{+-quotient 5}
Suppose that $M$ is a hyperbolic knot manifold which contains an essential twice-punctured torus $F$ of boundary slope $\beta$ and let $\alpha$ be a slope on $\partial M$ such that $M(\alpha)$ is an irreducible small Seifert manifold. Suppose as well that $F$ is a semi-fibre in $M$ and $M(\alpha)$ is not very small. If $\tau_\alpha$ preserves the orientation of the Seifert fibres of $\widetilde M(\alpha_1, \alpha_2)$ and $\Delta(\alpha, \beta) \geq 4$, then

$(1)$ there is an induced Seifert structure on $L(p, q) \cong \widetilde M(\alpha_1, \alpha_2)/\tau_\alpha$ whose branched set $L$ is a union of at most four Seifert fibres.

$(2)$ the induced Seifert structure on $L(p, q)$ has at least one cone point of order $\Delta(\alpha, \beta) - 1$.

$(3)$ $L$ contains at least one regular fibre of $L(p, q)$ and if only one, $L(p,q)$ has no exceptional fibre of multiplicity $2$.
\end{lemma}

\begin{proof}
By Lemma \ref{lemma: c}, $\widetilde M(\alpha_1, \alpha_2)$ has base orbifold $S^2(\bar a, \bar b, c, c)$ and $c \geq 2(\Delta(\alpha, \beta) - 1)$. As $M(\alpha)$ is not very small, $S^2(\bar a, \bar b, c, c)$ is hyperbolic.

Let $\bar \tau_\alpha$ be the map of $S^2(\bar a, \bar b, c, c)$ induced by $\tau_\alpha$.  Either $\bar \tau_\alpha$ is the identity or it is an orientation-preserving involution with exactly two fixed points. In the former case, a fibre of multiplicity $j$ in $\widetilde M(\alpha_1, \alpha_2)$ is sent to a fibre of multiplicity $\bar j = \frac{j}{\gcd(j,2)}$ in $L(p,q)$. Further, the fixed point set of $\tau_\alpha$ is the union of exceptional fibres of $\widetilde M(\alpha_1, \alpha_2)$ of even multiplicity. Hence $L$ has at most four components. Further, the base orbifold $\overline{\mathcal{B}}$ of $L(p,q)$ in the induced Seifert structure is $S^2(\bar{\bar a}, \bar{\bar b}, \bar c, \bar c)$. Since $\bar c \geq \Delta(\alpha, \beta) - 1 \geq 3$ (Lemma \ref{lemma: c}) we must have $\bar{\bar a} = \bar{\bar b} = 1$. Thus $a, b \in \{2, 4\}$ and $\overline{\mathcal{B}} = S^2(\bar c, \bar c)$. Further, since $S^2(a,  b, c)$ is hyperbolic, at least one of $a, b$ is $4$. It follows that $L$ contains at least one regular fibre. If there is only one, then $a = 2$ (say) and $b = 4$. It is easy to see that the lemma holds in this situation.

Next suppose that $\bar \tau_\alpha$ is not the identity. In this case $L$ has at most two components. If $a = 2$, then $\bar b > 1$ and $\mathcal{B} = S^2(\bar b, c, c)$. As $\bar \tau_\alpha$ is not the identity, it fixes exactly one of the cone points of $\mathcal{B}$, which necessarily has order $\bar b$. Denote by $\phi_1$ the $\tau_\alpha$-invariant regular fibre and $\phi_{\bar b}$ the $\tau_\alpha$-invariant exceptional fibre. The image in $L(p,q)$ of $\phi_1$ is a fibre of multiplicity $1$ if $\phi_1$ lies in the fixed-point set of $\tau_\alpha$, and is $2$ otherwise. Similarly the image in $L(p,q)$ of $\phi_{\bar b}$ is a fibre of multiplicity $\bar b > 1$ if $\phi_{\bar b}$ lies in the fixed-point set of $\tau_\alpha$, and is $2 \bar b$ otherwise. As $\overline{\mathcal{B}}$ can have only two cone points, it follows that $\phi_1$ lies in the fixed point set of $\tau_\alpha$ and that its image in $L(p,q)$ is a regular fibre contained in $L$. Further, $\overline{\mathcal{B}} = S^2(\bar b, c)$ if $\phi_{\bar b}$ is contained in the fixed-point set of $\tau_\alpha$ and $S^2(2 \bar b, c)$ otherwise. It is easy to see that the lemma holds in this situation. A similar argument shows that the lemma holds if $b = 2$.

Finally suppose that $\bar \tau_\alpha$ is not the identity and $a , b > 2$ and define $\mathcal{F}$ to be the fixed point set of $\bar \tau_\alpha$. In this case $S^2(\bar a, \bar b, c,c)$ contains four cone points. Then $\mathcal{F}$ contains either two cone points or none. The former case does not occur as otherwise $\overline{\mathcal{B}}$ would have three cone points (since $a, b \geq 4$). Thus $\mathcal{F}$ contains no cone points and the two associated regular fibres must lie in the fixed point set of $\tau_\alpha$. Thus $L$ contains two regular fibres and $\overline{\mathcal{B}} = S^2(\bar b, c)$. This completes the proof in this final case.
\end{proof}

For the rest of this section we suppose that $\Delta = \Delta(\alpha, \beta) \geq 5$ and show that
this leads to a contradiction.

Recall that $L$ is a 4-braid in the Heegaard solid torus $V$ in $L(p,q)$, and
consists of Seifert fibres in the induced Seifert structure on $L(p,q)$.
Let $X$ be the exterior of $L$ in $L(p,q)$ and $Y$ the exterior of $L$ in $V$.
Thus $X$ inherits a Seifert structure.
By Lemma \ref{+-quotient 5}(3)
at least one component of $L$ is a regular fibre in $L(p,q)$.
We distinguish two cases.

\begin{caseI}
At least two components of $L$  are regular fibres.
\end{caseI}

\begin{caseII}
Exactly one component of $L$ is a regular fibre.
\end{caseII}

\begin{lemma}\label{lem4.7}
Let $K$ be  a component of $L$ that is a regular fibre of $L(p,q)$.
Then
\vspace{-.2cm}
\begin{itemize}

\item[(1)] the winding number of $K$ in $V$ is greater than $1$;

\vspace{.2cm} \item[(2)] in Case II the winding number of $K$ in $V$ is greater than $2$.

\end{itemize}
\end{lemma}

\begin{proof} Note that because $L$ is a braid in $V$, $K$ is too.
Since the Seifert structure on $L(p,q)$ has two exceptional fibres by Lemma \ref{+-quotient 5}(2),
the exterior of a regular fibre has a Seifert structure with base orbifold
$D^2 (a,b)$, $a,b\ge 2$.
If the winding number of $K$ in $V$ is 1 then the exterior of $K$ in $L(p,q)$ is
a solid torus.
This proves (1).

To prove (2), suppose $K$ has winding number 2 in $V$; thus $K$ is a 2-braid in $V$.
Then the exterior of $K$ in $L(p,q)$ has a Seifert structure with base orbifold of the
form $D^2 (2,c)$.
Hence $a$ or $b=2$, contradicting Lemma \ref{+-quotient 5}(3).
\end{proof}

It follows from Lemma~\ref{lem4.7} that in Case~I $L$ has exactly two components,
$K_1$ and $K_2$, say, each with winding number~2 in $V$.
Let $T_1$ and $T_2$ be the corresponding boundary components of $X$.
There is a vertical annulus $A$ in $X$ with boundary components
$a_1\subset T_1$, $a_2 \subset T_2$.

In Case II, let $K$ be the component of $L$ that is a regular fibre.
By Lemma~\ref{lem4.7}, either $K=L$ or $L$ has two components $K$ and $K'$,
with winding numbers 3 and 1, respectively, in $V$.
Let $T_K$ be the boundary component of $X$ corresponding to $K$.
There is a vertical essential annulus $A$ in $X$, with $\partial A\subset T_K$,
separating $X$ into two components $X_1$ and $X_2$.
Note that either $X_1$ and $X_2$ are both solid tori, or one is a solid torus
and the other is homeomorphic to $T^2\times I$.

In both Cases I and II, choose $A$ among all annuli with the stated properties
to have minimal intersection with the complementary solid torus to $V$ in $L(p,q)$.
We adopt the construction and terminology described in
\cite[Section 6]{BGZ3}.
Thus a meridian disk $D$ of $V$ gives a properly embedded 4-punctured
disk $P$ in $Y$, and from $A$ we get an essential $n$-punctured annulus $Q$ in $Y$.
The intersection of $P$ with $Q$ defines graphs $\Gamma_P$, $\Gamma_Q$ with
vertices $d_V$, $c_1$, $c_2$, $c_3$, $c_4$ and $a_1,a_2,b_1,\ldots,b_n$,
respectively.
Note that in Case~II $n$ is even.
Also, since $L$ is  a braid we may orient the components of $L$ coherently
in $V$; then the vertices $c_1,c_2,c_3,c_4$ all have the same sign.
By the remarks after Lemma~\ref{lem4.7}, each vertex $c_j$ of $\Gamma_P$
has valency~1 in Case~I, and in Case~II, valency 0 or 2, with at least three
having valency~2.

\begin{lemma}\label{lem4.8}
In Case I, $\Gamma_P$ does not contain a $D$-edge Scharlemann cycle.
\end{lemma}

\begin{proof}
Such a Scharlemann cycle could be used to construct an annulus
$A'\subset X$ with $\partial A' = \partial A$ and having fewer intersections
with $T_V$.
\end{proof}

\begin{lemma}\label{lem4.9}
In Case II, suppose $\Gamma_P$ contains a $D$-edge Scharlemann cycle
of order $m=2$ or $3$, lying in $X_i$.
Then $X_i$ is a solid torus and the core of $X_i$ is an exceptional fibre
of $X$ of multiplicity $m$.
\end{lemma}

\begin{proof}
Let the face of $\Gamma_P$ bounded by the Scharlemann cycle be $f$.
Note that the edges of $\partial f$ in $\Gamma_Q$ do not lie in a disk in $A$,
for otherwise $X_i$ would contain a punctured lens space.

Let $H$ be the component of the intersection of the filling solid torus
$\overline{L(p,q)-V}$ with $X_i$ corresponding to the label-pair of the
Scharlemann cycle.
Define $W = N(A\cup H\cup f)$.
It is easy to see that $W$ is a solid torus in which $A$ has winding number $m$.
Therefore $X_i$ is not homeomorphic to $T^2\times I$.
Hence $X_i$ is a solid torus, and so $\overline{X_i-W}$ is a solid torus in
which the annulus that is the frontier of $W$ in $X_i$ is longitudinal.
It follows that the core of $X_i$ is an exceptional fibre of $X$ of multiplicity $m$.
\end{proof}

\begin{cor}\label{cor4.10}
In both Cases I and II, $\Gamma_P$ does not contain a $D$-edge $S$-cycle.
\end{cor}

\begin{proof}
In Case I this follows from Lemma~\ref{lem4.8}.
In Case~II it follows from Lemma~\ref{lem4.9} together with Lemma \ref{+-quotient 5}(3).
\end{proof}

\begin{lemma}\label{lem4.11}
In Case II, $\Gamma_P$ does not contain $D$-edge Scharlemann cycles
$f_1$ and $f_2$ on distinct label-pairs that lie in the same component $X_i$,
where $f_j$ has order~$2$ or $3$, $j=1,2$.
\end{lemma}

\begin{proof}
Let $H_1,H_2\subset X_i$ be the components of
the intersection of the filling solid torus $\overline{L(p,q)-V}$ with $X_i$
corresponding to the label-pairs of $f_1,f_2$, respectively.
Let $W_1 = N(A\cup H_1\cup f_1)$.
By the proof of Lemma~\ref{lem4.9}, $W_1$ is a solid torus and the annulus
$A_1$ that is the frontier of $W_1$ in $X_i$ is longitudinal in the solid torus
$U = \overline{X_i-W}$.
Let $W_2 = N(A_1 \cup H'_2 \cup f'_2)\subset U$, where $H'_2 = H_2\cap U$
and $f'_2 = f_2\cap U$.
Again by the proof  of Lemma~\ref{lem4.9},
$A_1$ has winding number 2 or 3
(the order of the Scharlemann cycle $f_2$) in the solid torus $W_2$,
and hence in $U$.
This is a contradiction.
\end{proof}

We will first dispose of Case I, and Case~II when $n\ge 4$.
Note that the valency of $d_V$ in $\Gamma_P$ is $\Delta n \ge 5n$ (since $\D\geq 5$ by assumption).

For an edge $\bar e$ in the reduced graph $\bar \Gamma_P$, we denote
by $wt(\bar e)$ the weight of $\bar e$, i.e. the number of parallel edges
represented by $\bar e$.
Let $\bar e$ be a $D$-edge of the reduced graph $\bar\Gamma_P$.
We say $\bar e$ is of {\em type $O$} if $\bar e$ cuts off a subdisk of $D$
that contains a single vertex $c_j$.
Otherwise, $\bar e$ is of {\em type $N$}.
Define
\begin{equation*}
\lambda (\bar e) = \begin{cases}
2\;wt (\bar e) \ ,\ \text{if $\bar e$ is of type $N$};\\
\noalign{\vskip6pt}
2\;wt (\bar e) + (\text{number of edges of $\Gamma_P$ incident to the
vertex $c_j$), if $\bar e$ of type $O$}.
\end{cases}
\end{equation*}

By Corollary~\ref{cor4.10} and the fact that (by the parity rule) no $D$-edge
can have the same label at both endpoints, it is easy to see that, assuming
$n\ge 4$ in Case~II, we have $\lambda (\bar e)\le n$ for every $D$-edge
$\bar e$ of $\bar\Gamma_P$.

Let $k$ be the number of $D$-edges of $\bar\Gamma_P$.
If $k=0$, then $5n \le 4$ in Case~I, and $5n\le 8$ in Case~II, both contradictions.
So we assume $k\ge 1$.
Also, if $n=1$ then $k=0$, so we also assume $n>1$.
Let $k_0$ be the number of $D$-edges of $\bar\Gamma_P$ of type $O$.
It is easy to see that $k\le 5$ and $k-1 \le k_0 \le 4$.
Note that there are $(4-k_0)$ vertices $c_j$ that are not associated with
$D$-edges of type~$O$.
It follows that
\begin{equation*}
5n \le \D n\le kn + \begin{cases}
(4-k_0)\ ,&\text{in Case I};\\
\noalign{\vskip6pt}
2(4-k_0)\ ,&\text{in Case II, if $n\ge 4$}
\end{cases}
\end{equation*}

\begin{lemma}\label{lem4.12}
Case I is impossible.
\end{lemma}

\begin{proof}
We have $5n \le kn + (4-k_0) \le kn + (4-(k-1))$, and hence
$5(n-1) \le k(n-1)$.
Since $n>1$, this implies that $k=5$ and $k_0=4$.
Thus $\bar \Gamma_P$ has four $D$-edges of type~$O$ and one of type $N$.

Now if $\bar e$ is of type $N$ then
\begin{equation*}
\lambda (\bar e) \le \begin{cases}
n\ ,&\text{$n$ even};\\
\noalign{\vskip6pt}
n-1\ ,&\text{$n$ odd};
\end{cases}
\end{equation*}
and if $\bar e$ is of type $O$ then (since each $c_j$ has valency 1)
\begin{equation*}
\lambda (\bar e) \le \begin{cases}
n-1\ ,&\text{$n$ even};\\
\noalign{\vskip6pt}
n\ ,&\text{$n$ odd}.
\end{cases}
\end{equation*}
Thus we get $5(n-1) < 5(n-1)$, a contradiction.
\end{proof}

\begin{lemma}\label{lem4.13}
In Case II, $n=2$.
\end{lemma}

\begin{proof}
Assume $n\ge 4$.
We have $5n \le kn +2 (4-k_0)\le kn+2 (4-(k-1))$, giving $(5-k)n \le 2(5-k)$.
This is a contradiction unless $\Delta = k=5$ and $\lambda (\bar e) =n$ for each of
the five $D$-edges of $\bar \Gamma_P$.
The $D$-edges of $\bar\Gamma_P$ are as shown in Figure \ref{bgz5-fig1}.
(Each $c_j$ has valency~0 or 2 in $\Gamma_P$; the $CD$-edges
of $\bar\Gamma_P$ are not shown in Figure \ref{bgz5-fig1}.)
Corresponding to the faces $\bar f_1,\bar f_2$ of $\bar\Gamma_P$ shown
in Figure \ref{bgz5-fig1} are 3-gon $D$-edge faces $f_1,f_2$ of $\Gamma_P$.
Since $\lambda (\bar e)=n$ for each $D$-edge $\bar e$ of $\bar\Gamma_P$,
$f_1$ and $f_2$ are Scharlemann cycles of order~3.

\begin{figure}[!ht]
\centerline{\includegraphics{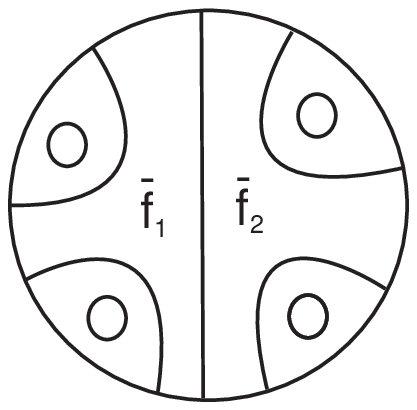}} \caption{ }\label{bgz5-fig1}
\end{figure}

If $n\equiv 2$ $(\text{mod }4)$ then $f_1$ (say) $\subset X_1$ and $f_2\subset X_2$.
Hence each of the two exceptional fibres of $L(p,q)$ has multiplicity~3.
But this contradicts Lemma \ref{+-quotient 5}(2).

If $n\equiv 0$ $(\text{mod }4)$ then $f_1$ and $f_2$ both lie in $X_1$, say.
Since the label-pairs at the corners of $f_1$ and $f_2$ are distinct, this
contradicts Lemma~\ref{lem4.11}.
\end{proof}

{From} now on we will assume that we are in Case~II and $n=2$.
Note that there are no parallel $D$-edges in $\Gamma_P$ by Corollary~\ref{cor4.10}.

Since all the $c_j$'s have the same sign, and $a_1,a_2$ (resp. $b_1,b_2$) have
opposite signs, we may assume that $a_1,a_2$ and $b_1,b_2$ are numbered
so that any $AB$-edge of $\Gamma_Q$ joins $a_i$ to $b_i$, $i=1$ or 2;
equivalently, every $CD$-edge of $\Gamma_P$ has the same label (1 or 2)
at both its endpoints.

Let $f$ be a face of $\Gamma_P$.
Then $\partial f$ consists of edges of $\Gamma_P$ together with {\em corners\/}
at the vertices, i.e. subarcs of the corresponding boundary components of $P$.
We call a corner at the vertex $d_V$ a {\em $d_V$-corner}, and a corner at
a vertex $c_j$ a {\em $c$-corner}.

\begin{lemma}\label{lem4.14}
Suppose $\Gamma_P$ has a disk face with at least two $d_V$-corners,
contained in $X_i$.
Then $X_i$ is a solid torus.
\end{lemma}

\begin{proof}
Suppose $X_i$ is not a solid torus; then $X_i\cong T^2\times I$.
Let $T_i,T'_i$ be the boundary components of $X_i$, where $A\subset T_i$.
Let $f$ be a disk face of $\Gamma_P$ as in the statement of the lemma.
Let $H_i = (\,\overline{L(p,q)-V}\,)\cap X_i$, and define $W = N(T_i\cup H_i\cup f)$.
Then $\partial W = T_i  \cup T''_i$, say.
%Take a fat base-point containing $H_i\cap A$ and let $x$ be the homotopy class
%of the core of the 1-handle $H_i$.
%Then $\pi_1 (W) \cong (\pi_1(T_i) * \zed)/ \llangle [\partial f]\rrangle$,
%where $\zed$ is generated by $x$.
%Suppose $f$ has $m\ge2$ $d_V$-corners.
%Then $[\partial f]$ is of the form $xw_1xw_2\ldots xw_m$,
%where $w_j\in \pi_1(T_i)$, $1\le j\le m$.
%Therefore $\pi_1(W)/\llangle \pi_1(T_i)\rrangle \cong \zed_m$.
Let $E$ be a disk in $A$ containing $H_i \cap A$ and let $x\in H_1(W,E)\cong H_1 (W)$
be the class of the core of the 1-handle $H_i$.
Then $H_1(W) \cong (H_1(T_i) \oplus\zed)/([\partial f])$, where $\zed$ is
generated by $x$.
If $f$ has $m\ge 2$ $d_V$-corners then $[\partial f] = (y,mx)$ for some
$y\in H_1 (T_i)$.
Therefore $H_1(W,T_i) \cong\zed/m$.
But this contradicts the fact that, since $T''_i$ is homologous to $T_i$ in
$X_i \cong T^2\times I$, $W\cong T^2\times I$.
\end{proof}

Let $f$ be a face of $\Gamma_P$.
We say $f$ is {\em of type $(a,b)$} if $f$ is a disk and has $a$ $d_V$-corners and
$b$ $c$-corners.

% \begin{lemma}\label{lem4.15}
% Suppose $\Gamma_P$ has a face of type $(1,1)$ that lies in $X_i$.
% Then there is no face of $\Gamma_P$ that lies in $X_i$ of type $(2,1)$ or $(3,2)$.
% \end{lemma}
%
% \begin{proof}
% Let $f$ be a face of $\Gamma_P$ of type $(1,1)$, and $g$ a face of one of the
% other two types listed in the lemma, both lying in $X_i$.
% Note that $f$ has no $D$-edges and $g$ has exactly one $D$-edge.
% The existence of $g$ implies that $X_i$ is a solid torus by Lemma~\ref{lem4.14}.
% For computations in $\pi_1(X_i)$ we take as ``base-point'' a disk in
% $A\subset\partial X_i$ containing the vertices $b_1$ and $b_2$, all the
% $AB$-edges of $\Gamma_Q$, and the $D$-edge of $g$.
%
% Let $A_i$ be the annulus $\overline{\partial X_i -A}$.

 \begin{lemma}\label{lem4.15}
 $\Gamma_P$ does not have faces of types $(a,b)$ and $(c,d)$, where $c\ge2$
 and $ad-bc = \pm1$, that both lie in the same $X_i$.
 \end{lemma}

 \begin{proof}
 Let $f,g$ be faces of type $(a,b)$, $(c,d)$, respectively, lying in $X_i$.
 The existence of $g$ implies that $X_i$ is a solid torus by Lemma~\ref{lem4.14}.

 Let $A_i$ be the annulus $\overline{\partial X_i -A}$.
 Then $H_1 (X_i,A)$ is generated by $t$, where $t$ is represented by a cocore
 arc of $A_i$ running from $a_2$ to $a_1$.
 Also, $H_1(X_i,A)\cong \zed/m$, where $m$ is the multiplicity of the
 exceptional fibre that is the core of $X_i$.
 Let $H_i = (\,\overline{L(p,q)-V}\,) \cap X_i$, and $\partial_0 H_i =
 \partial H_i\cap T_V$.
 The boundary of the face $f$, as it lies in $X_i$, consists of $a$ cocore arcs of the
 annulus $\partial_0 H_i$ (coming from the $d_V$-corners of $f$),
 $b$ cocore arcs   of the annulus $A_i$,
 (coming from the $c$-corners of $f$),
 together with edges of $\Gamma_Q$, lying in $A$; similarly for $g$.
 Let $x$ be the element of $H_1(X_i,A)$ represented by the core of the
 1-handle $H_i$, oriented from $b_1$ to $b_2$.
 Then $f$ and $g$ give the relations
 \begin{align*}
 ax + bt & = 0\\
 cx + dt & = 0
 \end{align*}
 in $H_1 (X_i,A)$.
 Since $ad-bc = \pm1$, this implies that $t=0$ and therefore
 $H_1 (X_i,A)=0$, a contradiction.
 \end{proof}

\begin{cor}\label{cor4.16}
$\Gamma_P$ does not have a pair of faces of the following types lying in $X_i$:
\begin{itemize}
\item[(1)] $(1,1)$ and $(2,1)$;
\item[(2)] $(1,1)$ and $(3,2)$;
\item[(3)] $(2,1)$ and $(3,2)$.
\end{itemize}
\end{cor}

\begin{lemma}\label{lem4.17}
In $\Gamma_Q$, the endpoints of the $AB$-edges incident to $b_i$ are
consecutive around $b_i$, $i=1,2$.
\end{lemma}

\begin{proof}
This follows immediately from the fact that all the $AB$-edges incident to $b_i$
have their other endpoint on $a_i$, $i=1,2$.
\end{proof}

We will use Lemma~\ref{lem4.17} in conjunction with the following observation.
For $i=1$ or 2, consider the edge-endpoints labeled $i$ around $d_V$ in $\Gamma_P$.
In $\Gamma_Q$ these correspond to the edge-endpoints around vertex $b_i$.
Number these $p_0,p_1,\ldots,p_{\Delta-1}$ in order around $b_i$.
Then around $d_V$ the corresponding points with label~$i$ occur in the order
$p_0,p_d,\ldots,p_{(\Delta-1)d}$, for some integer $d$ coprime to $\Delta$.
In particular, if $\Delta=6$ then $d = \pm1$.
Since $AB$-edges in $\Gamma_Q$ correspond to $CD$-edges in $\Gamma_P$
we have the following corollary to Lemma~\ref{lem4.17}.

\begin{cor}\label{cor4.18}
If $\Delta =\ 6$ then the endpoints with label $i$ on $d_V$ of the $CD$-edges
of $\Gamma_P$ are consecutive among all the edge-endpoints with label~$i$
on $d_V$, $i=1,2$.
\end{cor}

\begin{lemma}\label{lem4.19}
Case II, $n=2$, is impossible.
\end{lemma}

\begin{proof}
Let $k$ be the number of $D$-edges of $\Gamma_P$, and $\ell$ the number of
vertices $c_j$ with valency~2.
Thus $1\le k\le 5$ and $\ell=3$ or 4.
\medskip

\noindent $\underline{k=5}$.
The $D$-edges of $\Gamma_P$ are as shown in Figure \ref{bgz5-fig2}.
Then each of $X_1$ and $X_2$ contains a $D$-edge Scharlemann cycle of order~3.
By Lemma~\ref{lem4.9} $L(p,q)$ has two exceptional fibres of multiplicity~3.
This contradicts Lemma \ref{+-quotient 5}(2).
\medskip

\begin{figure}[!ht]
\centerline{\includegraphics{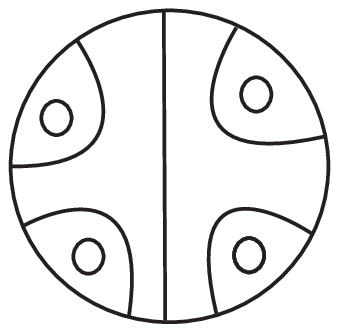}} \caption{ }\label{bgz5-fig2}
\end{figure}

\noindent $\underline{k=4}$.
There are two cases, (a) and (b), where the $D$-edges of $\Gamma_P$ are as
shown in Figure \ref{bgz5-fig3}(a) and (b), respectively.
\medskip

\begin{figure}[!ht]
\centerline{\includegraphics{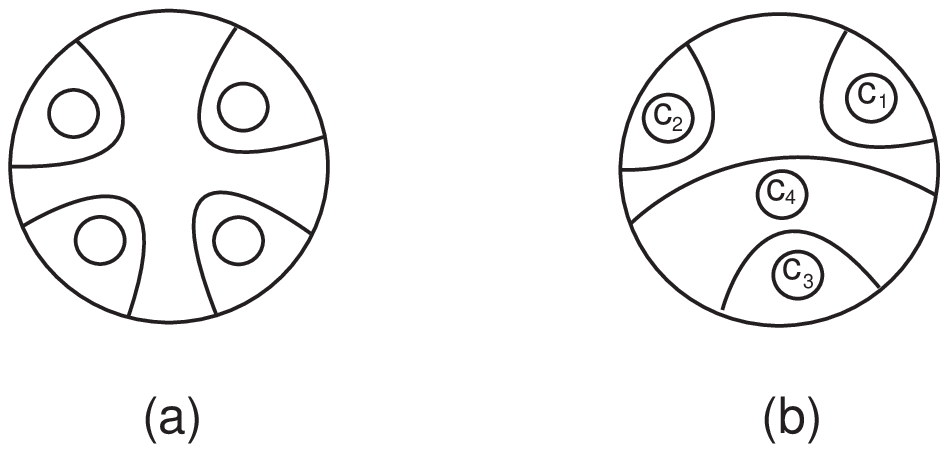}} \caption{ }\label{bgz5-fig3}
\end{figure}

\noindent  {\bf Case (a).}
A vertex $c_j$ of valency 2 gives rise to a face of type $(2,1)$.
It follows from Lemma~\ref{lem4.14} that $\ell=4$, and hence $\Gamma_P$ is
as shown in Figure \ref{bgz5-fig4}.
In particular $\Delta =8$.
Let the edge-endpoints around vertex $b_1$ in $\Gamma_Q$ be
$p_0,p_1,\ldots,p_7$, numbered in order around the vertex.
By Lemma~\ref{lem4.17} we may assume that $p_0,p_1,p_2,p_3$ are
endpoints of $AB$-edges and $p_4,p_5,p_6,p_7$ endpoints of $B$-edges.
By the remarks after Lemma~\ref{lem4.17} the corresponding points with label~1
on $d_V$ in $\Gamma_P$ appear in the order $p_0,p_d,\ldots,p_{7d}$ for some $d$
coprime to 8.
Since we see from Figure~4 that the $CD$-edges and $D$-edges with label~1
alternate around $d_V$ it is clear that no such integer $d$ exists.
\medskip

\begin{figure}[!ht]
\centerline{\includegraphics{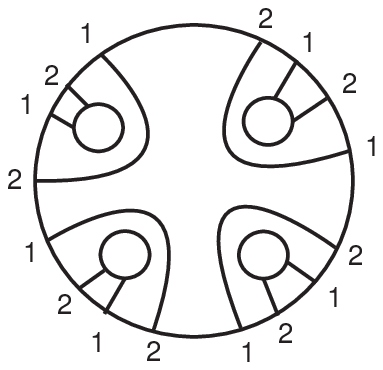}} \caption{ }\label{bgz5-fig4}
\end{figure}

\noindent  {\bf Case (b).}
First note that at least one of $c_1,c_2$ has valency $2$, and hence both do by
Lemma \ref{lem4.14}.
Similarly, Lemma \ref{lem4.14} implies that $c_4$ has valency $2$.
There are two possibilities for the edges incident to $c_4$, shown in
Figure \ref{bgz5-fig5}(i) and (ii).

\begin{figure}[!ht]
\centerline{\includegraphics{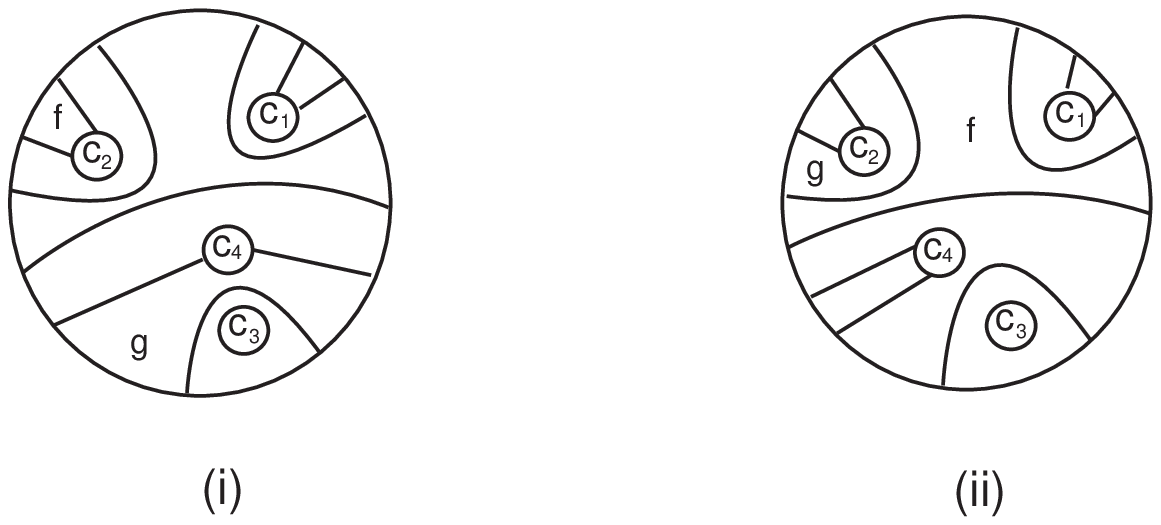}} \caption{ }\label{bgz5-fig5}
\end{figure}

In (i) the faces $f$ and $g$ contradict Corollary~\ref{cor4.16}(1).

In (ii), the face $f$ implies, by Lemma~\ref{lem4.14}, that $c_3$ has valency~2.
Hence there is a face of type (1,1) on the same side as the face $g$,
contradicting Corollary~\ref{cor4.16}(1).
(Here, and in the sequel, by ``on the same side'' we shall mean on the same
side of $A$ in $X$, i.e. in the same component $X_1$ or $X_2$.)
\medskip

\noindent $\underline{k=3}$.
There are three cases, (a), (b)  and (c), where the $D$-edges of $\Gamma_P$ are as
shown in Figure \ref{bgz5-fig6}(a), (b) and (c), respectively.
\medskip

\begin{figure}[!ht]
\centerline{\includegraphics{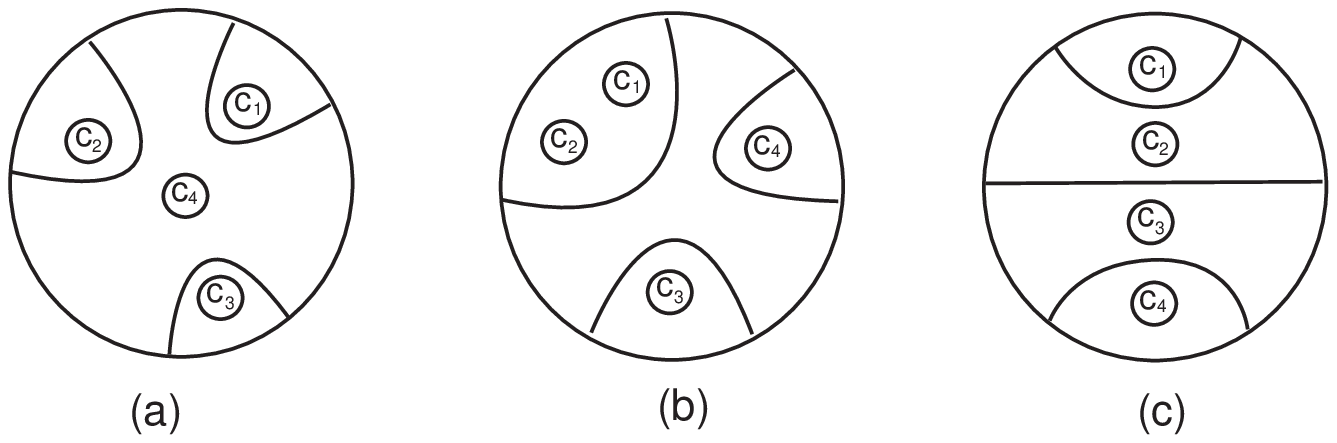}} \caption{ }\label{bgz5-fig6}
\end{figure}

\noindent  {\bf Case (a).}
If $c_4$ has valency 0 then $\Gamma_P$ is as shown in Figure \ref{bgz5-fig7}.
Hence $\Delta =6$.
Since the 1-labels on $d_V$ belonging to $CD$-edges are not consecutive
around $d_V$, this contradicts Corollary~\ref{cor4.18}.

\begin{figure}[!ht]
\centerline{\includegraphics{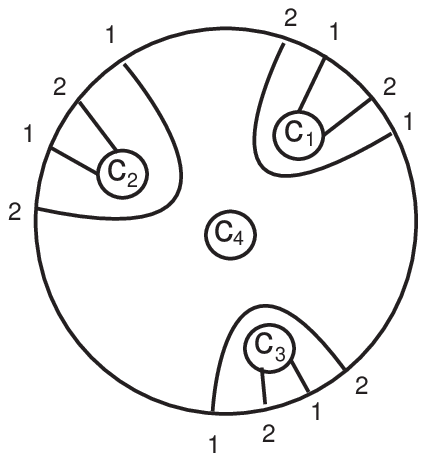}} \caption{ }\label{bgz5-fig7}
\end{figure}

So suppose $c_4$ has valency 2.
The two possibilities for the edges incident to $c_4$ are shown in Figure \ref{bgz5-fig8}(i)
and (ii).

\begin{figure}[!ht]
\centerline{\includegraphics{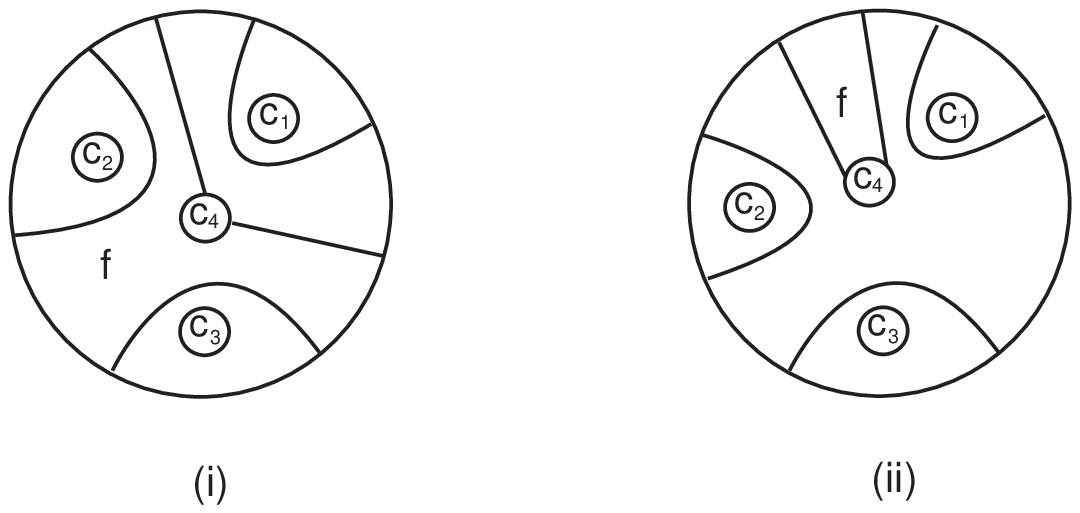}} \caption{ }\label{bgz5-fig8}
\end{figure}

In (i), the face $f$ implies that $c_1$ has valency~2, by Lemma~\ref{lem4.14}.
Hence there is a face $g$ of type (2,1) on the same side as $f$.
Also, at least one of $c_2,c_3$ has valency~2; this gives a face of type (1,1)
on the same side as $g$, contradicting Corollary~\ref{cor4.16}(1).

In subcase (ii), since some $c_j$, $j=1,2,3$, has valency~2, we get a face of
type (2,1) on the same side as $f$, contradicting Corollary~\ref{cor4.16}(1).
\medskip

\noindent  {\bf Case (b).}
At least one of $c_3,c_4$ has valency~2, and hence all the $c_j$'s have
valency~2 by Lemma~\ref{lem4.14}.
The four possibilities for $\Gamma_P$ are shown in Figure \ref{bgz5-fig9}(i), (ii), (iii) and (iv).

\begin{figure}[!ht]
\centerline{\includegraphics{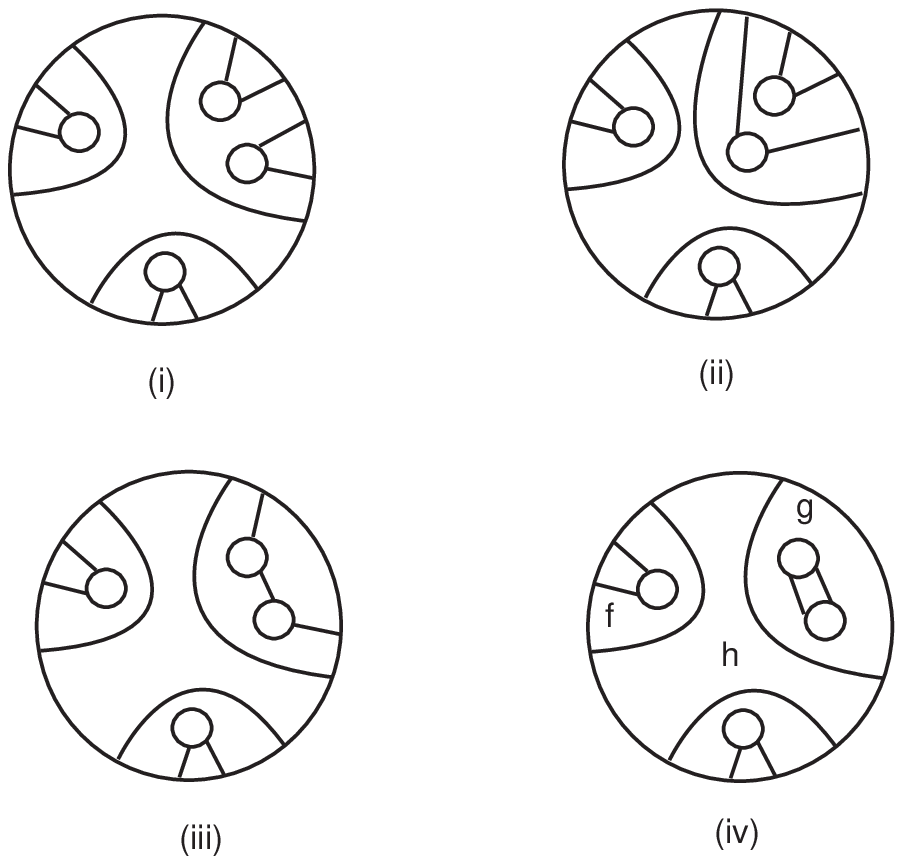}} \caption{ }\label{bgz5-fig9}
\end{figure}

(i) contradicts Corollary~\ref{cor4.16}(3), and (ii) contradicts Corollary~\ref{cor4.16}(1).

In (iii), $\Delta=6$ and the labeling around $d_V$ contradicts Corollary~\ref{cor4.18}.

In (iv), suppose the face $f$ lies in $X_1$.
Then $X_1$ is a solid torus by Lemma~\ref{lem4.14}, and $f$ gives the relation $2x+t=0$
in $H_1(X_1,A)$.
(See the proof of Lemma~\ref{lem4.15}  for notation.)
The annulus face $g$ also lies in $X_1$ and gives the relation $x+2t =0$.
These relations  give $3t=0$, and hence $H_1(X_1,A) \cong\zed/3$.
This implies that the core of $X_1$ is an exceptional fibre of multiplicity~3.
The $D$-edge Scharlemann cycle $h$ implies that the core of $X_2$ is also
an exceptional fibre of multiplicity~3, by Lemma~\ref{lem4.9}.
This contradicts Lemma \ref{+-quotient 5}(2).

Case (c). Since $\ell\ge 3$, either $c_1$ and $c_2$ both have valency 2, or
$c_3$ and $c_4$ both have valency 2. Assume the former without loss of
generality. The fact that $c_1$ has valency 2 gives a face $f$ of type
(1,1) and a face $g$ of type (2,1) on opposite sides. There are two
possibilities for the configuration of the $CD$-edges incident to $c_2$. In
one case we get a face of type (1,1) on the same side as $g$, and in the
other case a face of type (2,1) on the same side as $f$. These both
contradict Corollary 4.16(1).
\medskip

\noindent $\underline{k=2}$.
There are two cases (a) and (b), illustrated in Figure \ref{bgz5-fig10}(a) and (b).
\medskip

\begin{figure}[!ht]
\centerline{\includegraphics{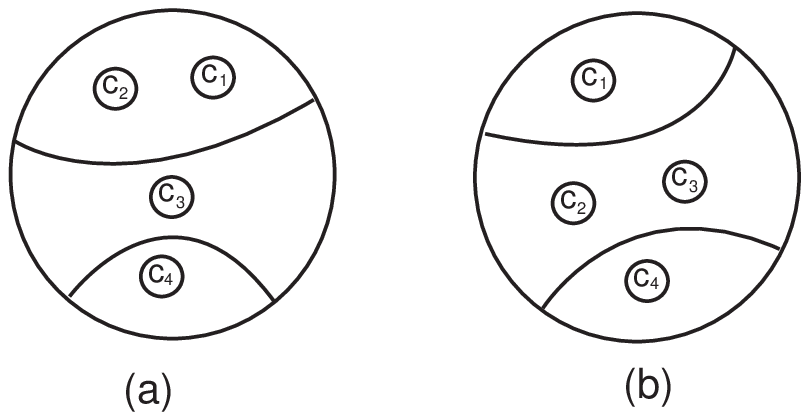}} \caption{ }\label{bgz5-fig10}
\end{figure}

\noindent  {\bf Case (a).}
First suppose that $c_3$ has valency $0$.
Then $c_1,c_2$ and $c_4$ have valency $2$, and since $\Delta \ge 5$ the
edges incident to $c_1,c_2$ and $c_4$ are all $CD$-edges.
This gives faces of type (2,1) and (3,2) on the same side, contradicting
Corollary~\ref{cor4.16}(3).

So suppose $c_3$ has valency~2.
The two possible arrangements of the edges incident to $c_3$ are shown
in Figure \ref{bgz5-fig11}(i) and (ii).

\begin{figure}[!ht]
\centerline{\includegraphics{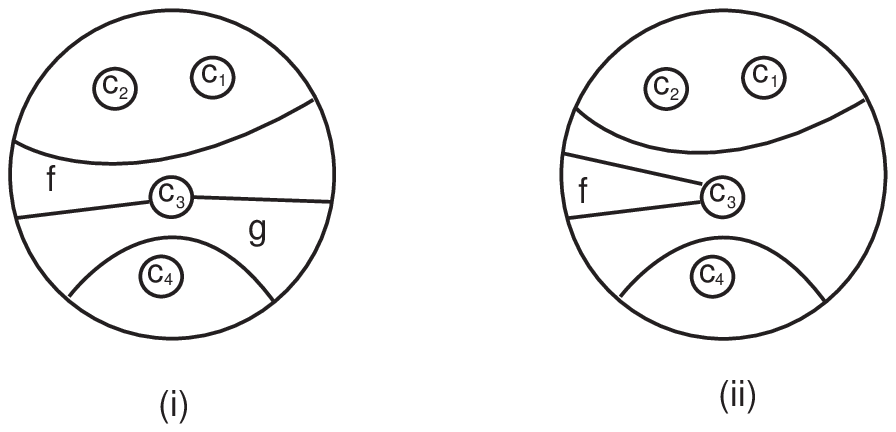}} \caption{ }\label{bgz5-fig11}
\end{figure}

In (i), the face $f$, together with Lemma~\ref{lem4.14}, implies that $c_4$
has valency~2.
This gives a face of type (1,1) on the same side as $g$, contradicting
Corollary~\ref{cor4.16}(1).

In subcase (ii), if $c_4$ has valency~2 then we get a face of type (2,1)
on the same side as $f$, contradicting Corollary~\ref{cor4.16}(1).
If $c_4$ has valency~0  then $c_1$ and $c_2$ have valency~2, and
since $\Delta \ge 5$ the edges incident to $c_1$ and $c_2$ are $CD$-edges.
The resulting face of type (3,2) contradicts Lemma~\ref{lem4.14}.
\medskip

\noindent  {\bf Case (b).}
First suppose that one of $c_2,c_3$ has valency~0.
Note that since $\Delta \ge 5$ $\Gamma_P$ has no $C$-edges.
Therefore there are three possibilities for $\Gamma_P$, illustrated
in Figure \ref{bgz5-fig12}(i), (ii) and (iii).

\begin{figure}[!ht]
\centerline{\includegraphics{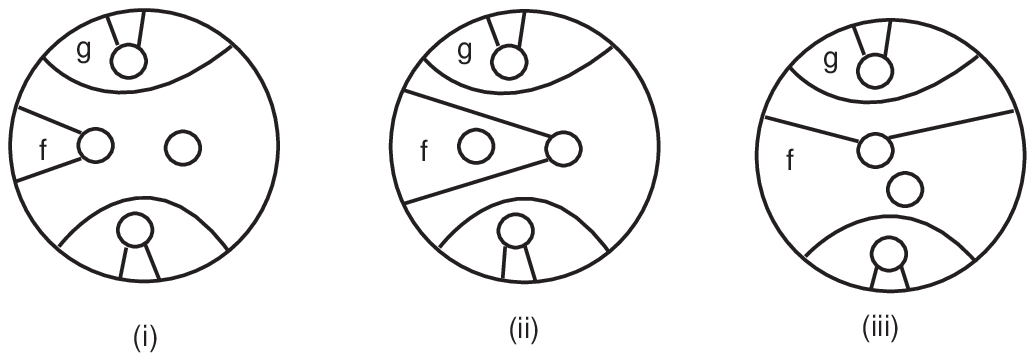}} \caption{ }\label{bgz5-fig12}
\end{figure}

In (i) the faces $f$ and $g$ contradict Corollary~\ref{cor4.16}(1).

In (ii) and (iii) the faces $f$ and $g$ contradict Lemma~\ref{lem4.14}.

So suppose that both $c_2$ and $c_3$ have valency~2.
Since $\Delta \ge 5$ there cannot be two $C$-edges joining $c_2$ and $c_3$.
There are therefore seven possibilities for the edges incident to $c_2$ and $c_3$,
shown in Figure \ref{bgz5-fig13}(i)--(vii).

\begin{figure}[!ht]
\centerline{\includegraphics{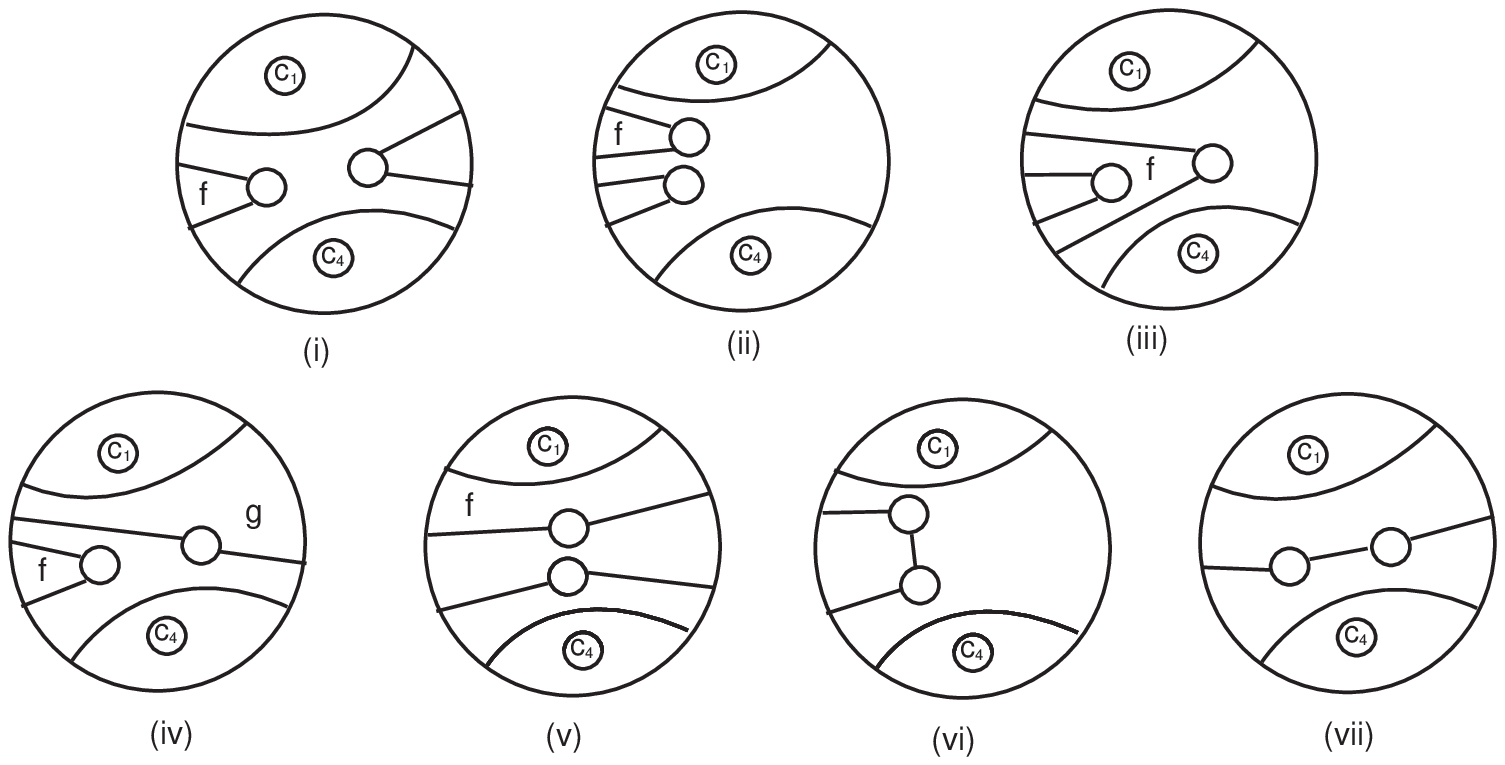}} \caption{ }\label{bgz5-fig13}
\end{figure}

In (i) and (ii), at least one of $c_1$ and $c_4$ has valency~2.
This gives a face of type (2,1) on the same side as $f$, contradicting
Corollary~\ref{cor4.16}(1).

In (iii), the face $f$ implies, by Lemma~\ref{lem4.14}, that both $c_1$ and $c_4$
have valency~2.
Hence $\Gamma_P$ is as illustrated in Figure \ref{bgz5-fig14}.
In particular $\Delta =6$.
But the 1-labels on $d_V$ of the $CD$-edges are not consecutive on $d_V$,
contradicting Corollary~\ref{cor4.18}.

\begin{figure}[!ht]
\centerline{\includegraphics{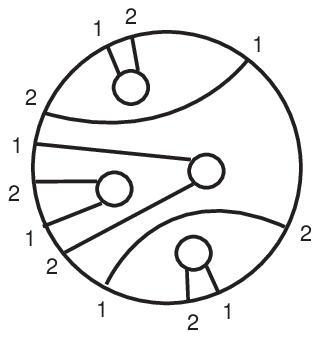}} \caption{ }\label{bgz5-fig14}
\end{figure}

In (iv), the faces $f$ and $g$ contradict Corollary~\ref{cor4.16}(1).

In (v), at least one of $c_1$ and $c_4$ has valency~2, giving a face of type
(1,1) on the same side as $f$.
This contradicts Corollary~\ref{cor4.16}(1).

In (vi) and (vii) we must have $\ell=4$ since $\Delta \ge 5$.
Then in (vi) we get faces of type (1,1) and (3,2) on the same side, and
in (vii) faces of type (1,1) and (2,1) on the same side, contradicting
Corollary~\ref{cor4.16}(2) and (1) respectively.
\medskip

\noindent $\underline{k=1}$.
The two possibilities for the $D$-edge of $\Gamma_P$ are shown in Figure \ref{bgz5-fig15}
(a) and (b).
Note that since $\Delta \ge 5$, $\Gamma_P$ has no $C$-edges and $\ell=4$.
\medskip

\begin{figure}[!ht]
\centerline{\includegraphics{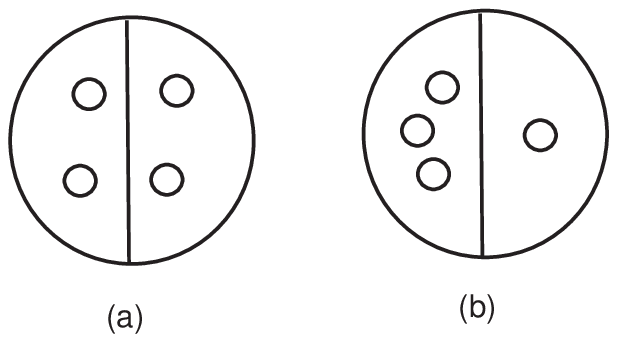}} \caption{ }\label{bgz5-fig15}
\end{figure}

\noindent  {\bf Case (a).}
Either $\Gamma_P$ is as shown in Figure \ref{bgz5-fig16}(i), or the $CD$-edges on one side
of the $D$-edge are as shown in Figure \ref{bgz5-fig16}(ii).

\begin{figure}[!ht]
\centerline{\includegraphics{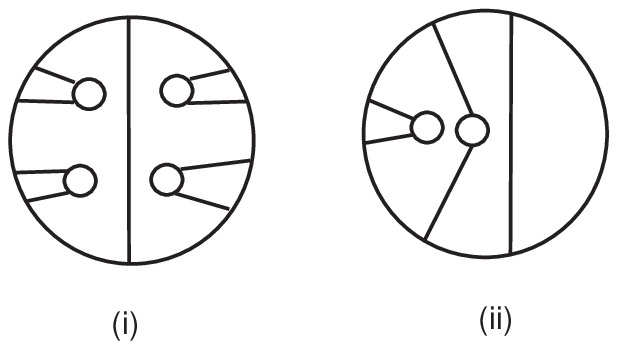}} \caption{ }\label{bgz5-fig16}
\end{figure}

In (i) we have faces of type (1,1) and (3,2) on the same side, contradicting
Corollary~\ref{cor4.16}(2).

In (ii) we have faces of type (1,1) and (2,1) on the same side,
contradicting Corollary~\ref{cor4.16}(1).
\medskip

\noindent  {\bf Case (b).}
It is easy to see that for all possible configurations of the $CD$-edges we get
faces of type (1,1) and (2,1) on the same side, contradicting Corollary~\ref{cor4.16}(1).

This completes the proof of Lemma~\ref{lem4.19}.\qed

The not very small case of Proposition \ref{prop: F2-semi-fibre} now follows from Lemmas~\ref{lem4.12}, \ref{lem4.13}
and \ref{lem4.19}.
\end{proof}

\section{The proof of Theorem \ref{thm: twice-punctured precise} when $F$ is a fibre} \label{sec: fibre}

In this section we prove the fibre case of Theorem \ref{thm: twice-punctured precise}.

\begin{prop}
\label{prop: F2fibre}
Suppose that $M$ is a hyperbolic knot manifold which contains an essential twice-punctured torus $F$ of boundary slope $\beta$ and let $\alpha$ be a slope on $\partial M$ such that $M(\alpha)$ is an irreducible small Seifert manifold. If $F$ is a fibre in $M$, then
$\Delta(\alpha, \beta) \leq 3$.
\end{prop}

Recall the involution $\tau_F$ depicted in Figure
\ref{bgz5-fig25} and let $\tau$ be its fibre-by-fibre extension to $M$. Then $M / \tau = (V, L)$ where $V$ is a solid torus and $L$ is a $4$-braid. The reader will verify that there is curve representing a dual class $\beta^*$ to $\beta$ on $\partial M$ on which $\tau$ acts by rotation by $\pi$. Write
\begin{equation}\label{alpha 2}
\text{\em $\alpha = p \beta^* + q \beta$}
\end{equation}
where $p, q$ are coprime. After possibly changing the signs of $\beta^*$ and $\beta$ we may assume that
\begin{equation}\label{p 2}
\text{\em $p = \Delta(\alpha, \beta)$}
\end{equation}
Note that $\beta$ is sent to the meridian slope $\bar \beta$ on $\partial V$, so $V (\bar \beta) \cong S^1 \times S^2$, while $\beta^*$ is sent to $2 \bar \beta^*$ where $\bar \beta^*$ is a longitude of $V$.

For each slope $\gamma$ on $\partial M$, $\tau$ extends to an involution $\tau_\gamma: M(\gamma) \to M(\gamma)$. Moreover, if $\widetilde U_\gamma$ denotes the filling torus in $M(\g)$ and $\widetilde K_\gamma$ its core, then
\begin{equation}\label{fix 2}
\text{\em $\hbox{Fix}(\tau_\gamma) = \left\{ \begin{array}{ll}
\widetilde L & \hbox{ if } \Delta(\gamma, \beta^*) \hbox{ is odd} \\
\widetilde L \cup \widetilde K_\gamma & \hbox{ if } \Delta(\gamma, \beta^*)  \hbox{ is even}
\end{array} \right.$}
\end{equation}
It is clear that $\widetilde U_\gamma/\tau_\g$ is a solid torus $U_\gamma$. Denote its core $\widetilde K_\gamma / \tau_\gamma$ by $K_\gamma$. Thus $M(\gamma) / \tau_\gamma = V \cup_{\bar \g} U_\gamma$ is a lens space. Indeed, if $\gamma = r\beta^* + s \beta$, then under the double cover $\partial M \to \partial V$ we have $\gamma \mapsto 2r \bar \beta^*  + s \bar \beta$. Let $\bar \gamma = \frac{1}{\gcd(2,s)}(2r \bar \beta^*  + s\bar \beta)$ denote the associated slope and $L_\gamma$ the branch set in $M(\gamma) / \tau_\gamma$. Then
$$(M(\gamma) / \tau_\gamma, L_\gamma) = (V(\bar \gamma), L_\g) \cong \left\{ \begin{array}{ll}
(L(2r, s), L) & \hbox{ if } s  \hbox{ is odd} \\
(L(r, \frac{s}{2}), L \cup K_\gamma)  & \hbox{ if } s  \hbox{ is even}
\end{array} \right.$$
We are interested in the case $\gamma = \alpha$. Set
\begin{equation}\label{pbar 2}
\text{\em $\bar p = \frac{2p}{\gcd(2,q)}  \;\;\;\;\;  \hbox{ and } \;\;\;\;\; \bar q = \frac{q}{\gcd(2,q)}$}
\end{equation}
so that $\bar \alpha = \bar p \bar \beta^*  + \bar q \bar \beta$ and
$$M(\alpha) / \tau_\alpha \cong L(\bar p, \bar q)$$
From \ref{fix 2} we see that
\begin{equation}\label{comp 2}
\text{\em $|L_\alpha| = \left\{ \begin{array}{ll}
|L| & \hbox{ if } q  \hbox{ is odd} \\
|L| + 1 & \hbox{ if } q \hbox{ is even}
\end{array} \right.$}
\end{equation}
By \cite[Lemma 4.1]{BGZ3} there is a $\tau_\alpha$-invariant Seifert structure of $M(\alpha)$ with base orbifold of the form $S^2(a,b,c)$ where $a, b, c \geq 1$. Let $\bar \tau_\alpha$ be the induced map on $S^2(a,b,c)$.

We begin with the case where $M(\alpha)$ is a lens space.

\begin{lemma} \label{F2-lens}
Suppose that $M$ is a hyperbolic knot manifold which contains an essential twice-punctured torus $F$ of boundary slope $\beta$ and let $\alpha$ be a slope on $\partial M$ such that $M(\alpha)$ is an irreducible small Seifert manifold. If $F$ is a fibre in $M$ and $M(\alpha)$ is a lens space, then $\Delta(\alpha, \beta) \leq 3$.
\end{lemma}

\pf By \cite{L3} $\Delta(\alpha, \beta) \leq 4$ so we suppose that $p = \Delta(\alpha, \beta) = 4$ in order to derive a contradiction. In this case $q$ is odd, so that $L_\alpha = L$ (\ref{comp 2}). Further, $\bar p = 8$ and $\bar q = q$ so $M(\alpha) / \tau_\alpha \cong L(8, q)$. By \cite[Lemma 4.2]{BGZ3}, $L_\alpha$ is either
\vspace{-.2cm}
\begin{itemize}

\item the union of the cores of the two Heegaard solid tori of $L(8, q)$; or

\vspace{.2cm}\item the boundary of an annular spine of a Heegaard solid torus of $L(8, q)$.

\end{itemize}
\vspace{-.35cm}
It follows that each component of $L_\alpha$ carries a generator of $H_1(L(8, q))$. On the other hand,
$L_\alpha$ is isotopic in $L(8, q)$ to the $4$-braid $L$ in $V$ so $L$ must split into a $1$-braid $L_0$ and a $3$-braid $L_1$.
Since the generators of $H_1(L(8, q))$ have order $8$, $L_0$ and $L_1$ carry different generators and so $L_\alpha$ must be the union of the cores of the two Heegaard solid tori of $L(8, q)$.

{\bf Claim 1}. Let $Y_1$ be the exterior $L_1$ in $V$.
Then $Y_1$ is Seifert fibred.

The Dehn filling of $Y_1$ along $\p V$ with the slope $\bar \a$
is a manifold homeomorphic to the exterior of
$L_1$ in $L(8, q)$ and thus is  a solid torus.
On the other hand the Dehn filling of $Y_1$ along $\p V$ with the slope
$\bar\b$ is a manifold homeomorphic the exterior
of $L_1$ in $S^2\times S^1$ and thus is a twice-punctured disk bundle and therefore
Seifert fibred. Since $\D(\bar\a,\bar\b)=8>5$, $Y_1$ cannot be
hyperbolic by \cite[Theorem 1.3]{Go1}.
Since $Y_1$ is the exterior of a single component closed  $3$-braid
in $V$, the monodromy of the corresponding  thrice-punctured disk bundle structure of  $Y_1$ cannot be reducible.
  So it is periodic and $Y_1$ is Seifert fibred.

{\bf Claim 2.} Let $Y$ be the exterior of $L=L_0\cup L_1$ in $V$.
Then $Y$ is hyperbolic.

To prove this claim, suppose for contradiction that
$Y$ is non-hyperbolic. It is irreducible, since $L_0$ and $L_1$ are each
homotopically essential in $V$. Thus  $Y$ is either Seifert fibred or
  contains an
essential torus.
In the latter case,  the monodromy of  $Y$ with respect to the  $4$-punctured disk bundle
structure  is reducible and its  invariant essential curve in a
$4$-punctured disk fiber can only be a circle which
 encloses  the three
punctures coming from $L_1$ but not the  puncture from $L_0$.
Hence  the corresponding essential
torus $T$ separates $Y$ into a thrice-punctured disk
bundle $Y_1$ and a once-punctured annulus bundle
$Y_0$. Obviously $Y_0$ is Seifert fibred and
 $Y_1$ is homeomorphic to the exterior of
$L_1$ in $V$, which is  Seifert fibred by Claim 1.
Therefore $Y$ has no hyperbolic pieces in its JSJ torus decomposition.
Now   the exterior of $\tilde L$ in $M$ is a free double cover of $Y$
and thus has  no hyperbolic pieces in its JSJ torus decomposition.
Thus $M$ cannot be hyperbolic, giving a contradiction.

To finish the proof of the lemma,  we adopt an intersection graph argument as in \S \ref{subsec: 7.2}.
In the present case the argument is  much simpler due to
the relatively large distance of $8$ between the slopes $\bar\a$ and $\bar\b$ on $\p V$.

Let $T_i$ be the torus boundary of $Y$ corresponding to $L_i$, $i=0,1$.
Then $\p Y=T_0\cup T_1\cup\p V$.
The $\bar\a$-filling of $Y$ along $\p V$ is
the exterior of $L$ in $L(8, q)$
and thus  is homeomorphic to $T^2\times [0,1]$ (recall that $L_0$ and $L_1$ are cores
of the two Heegaard solid tori of $L(8,q)$).
Let $A$ be a properly embedded  essential annulus in this manifold
whose  boundary component  $a_1$ on $T_1$ is a meridional slope of $L_1$.
Then the other boundary component $a_0$ of $A$ on $T_0$
is a slope distance $8$ from the meridional slope of $L_0$.
Now isotope $A$ fixing $\p A$ so that $A\cap \p V$ has minimal
number of components and let $Q=Y\cap A$.
Then $Q$ is an essential  $n$-punctured annulus in $Y$.
Since $Y$ is hyperbolic, $n>0$.
Let $b_1, \cdots,  b_n$ be  the components  of $\p Q$ which
lie on $\p V$ and they are indexed in the order they
appear along $\p V$.
Note that $b_1,...,b_n$ are parallel circles in $\p V$
with the slope $\bar \a$.

Let $P$ be a fixed $4$-punctured disk fiber in $Y$. Then $P$ is
essential in $Y$. Denote  the components of $\p P$ by $d, c_0, c_1, c_2, c_3$,
where $d$ lies on $\p V$ and has slope $\bar\b$,
 $c_0$ lies on $T_0$ and has the meridional slope of $L_0$,
and $c_1, c_2, c_3$ lie on $T_1$ and each has the meridional slope of $L_1$.

Isotope $Q$ and $P$ to intersect transversely and
minimally, and consider their intersection graphs
$\G_Q$, $\G_P$ with vertices $a_0, a_1, b_1,\cdots
b_n$ and $d, c_0, c_1, c_2, c_3$ respectively, and
with the arc components of $Q\cap P$ as edges. Since
$a_1$ and $c_1,c_2, c_3$ are meridians of $L_1$ on
$T_1$, they are disjoint and thus have valency $0$. By
construction $a_0$ and $c_0$ have valency $8$, $d$ has
valency $8n$, and $b_1,\cdots, b_n$ have valency  $8$.

The edges of $\G_Q$ and $\G_P$ are  essential arcs in
$P$ and $Q$. By the parity rule,  no edge in $\G_Q$
has  both of its endpoints incident to  $a_0$ and thus
there are exactly $8$ edges in $\G_Q$ incident to
$a_0$, each connecting $a_0$ to  some $b_i$. Call such
edges {\it $ab$-edges}. The remaining $4(n-1)$ edges
$\G_Q$ connect $b_i$'s to $b_i$'s and we call them
{\it $bb$-edges}.

Similarly no edge in $\G_P$ has both of its endpoints
incident to  $c_0$ and thus in $\G_P$ there are
exactly $8$ {\it $cd$-edges} connecting $c_0$ to  $d$.
The remaining $4(n-1)$ edges of $\G_P$ connect $d$ to
$d$ which are called {\it $dd$-edges}.

An edge in one graph is dual to an edge in the other
graph if they come from the same arc component of
$P\cap Q$. Then $ab$-edges are dual to $cd$-edges and
$bb$-edges are dual to $dd$-edges. The parity rule
shows that each $dd$-edge is a positive edge and each
$bb$-edge is negative (i.e. it connects vertices of
opposite signs).

As in Lemma \ref{lem4.8}, $\G_P$ has no Scharlemann cycles consisting of $dd$-edges.
 It is an elementary fact that if $\G_P$ has a parallel family of more than $n/2$ $dd$-edges, then
 it has a Scharlemann cycle of $dd$-edges (an $S$-cycle to be exact; see \cite[Corollary 2.6.7]{CGLS}). Hence
 each parallel family of $dd$-edges has at most $n/2$ members.
 It's easy to see that the reduced graph of $\G_P$ has at most $5$
 $dd$-edges (cf. Figure \ref{bgz5-fig2}). Hence in $\G_P$ there are at most $5n/2$ $dd$-edges.
 On the other hand there are exactly $4(n-1)$ $dd$-edges.
 Hence we have $4(n-1)\leq 5n/2$, which yields $n\leq 2$.

 {\bf Claim 3.}  If    $e_1$ and $e_2$ are two parallel
 $ab$-edges in $\G_Q$, then their dual $cd$-edges $e_1'$ and $e_2'$ cannot be parallel in
 $\G_P$.

 To prove the claim assume, for a contradiction, that $e_1'$ and $e_2'$ are parallel
 in $\G_P$.
 Let $B$ be the disk region in $Q$ between $e_1$ and $e_2$
 and let $B'$ be the disk region in $P$ between $e_1'$ and $e_2'$.
We may assume that $e_1$ and $e_2$ are an innermost such pair, i.e.
 $B$ and $B'$ only intersect in $Y$ along $e_1=e_1'$ and $e_2=e_2'$.
 (Otherwise there is an edge $e\ne e_1, e_2$ in $B$
 whose dual edge $e'$ is contained $B'$, and we may replace $e_2$ by $e$
 and restart the proof.)
 Then  $B$ and $B'$ paste together in $Y$
  forming a properly embedded annulus $E$ in $Y$ connecting
  $\p V$ and $T_0$. We claim that $E$ is essential in $Y$, contrary to Claim 2, thus
  completing the proof of Claim 3.

  Since $E$ connects different components of $\p Y$,
it suffices to show that each component of $\p E$ is
  essential  in $T_0$ or $\p V$.
    To see  this, suppose that $e_1$ and $e_2$ are incident at $b_i$.
  Then the four endpoints of $e_1, e_2$ decompose
  $\p B$ into four arcs: $e_1$, $a$, $e_2$, $b$, where $a$ is
  a sub-arc of $a_0$ and $b$ a sub-arc of $b_i$.
  Similarly the endpoints of $e_1', e_2'$ decompose
  $\p B'$ into four arcs: $e_1'$, $c'$, $e_2'$, $d'$, where $c'$ is
  a sub-arc of $c_0$ and $d'$ a sub-arc of $d$.
  Since $a_0$ and $c_0$ intersect minimally in $T_0$,
  the arcs $a$ and $c'$ paste together in $T_0$ to form a simple essential circle.
  Likewise $b$ and $d'$ paste together in $\p V$ to form a simple essential circle.
  Thus $E$ is essential, contrary to Claim 2.

 {\bf Case $n=1$}.

 In this case there are exactly eight edges in $\G_P$, all being  $cd$-edges,
 and there are exactly eight edges in $\G_Q$, all being $ab$-edges.
It is clear that the eight $ab$-edges in $\G_Q$ are parallel to each other, and it is
also evident that there are at least two
 parallel $cd$-edges in $\G_P$, contrary to Claim 3. Thus $n \ne 1$.

{\bf Case $n=2$}.

 In this case there are exactly twelve edges in $\G_P$, eight being  $cd$-edges
 and four being $dd$-edges,
 and similarly there are exactly twelve edges in $\G_Q$, eight being $ab$-edges
 and four being $bb$-edges.
 Since $n=2$, the four $bb$-edges all connect
 $b_1$ to $b_2$.
 It follows that four of the eight $ab$-edges in $\G_Q$ are incident to $b_1$ and four
 are incident to $b_2$.
 Thus the  reduced graph of $\G_Q$ is one of the three cases (1) (2) (3)
 shown in Figure \ref{bgz5-fig Q}. It follows that there is at least one  family of four parallel $ab$-edges in $\G_Q$,
 which we denote by   $e_1, e_2, e_3, e_4$.
Their dual edges,  $e_1', e_2', e_3', e_4'$,  are $cd$-edges
and it is easy to see that as the valency of $c_0$ in the reduced graph of $\G_P$
is at most $3$, at least two of them are parallel in
$\G_P$. Again we obtain a contradiction to Claim 3 and this rules out the final possibility that $n = 2$.
\qed

\begin{figure}[!ht]
\centerline{\includegraphics{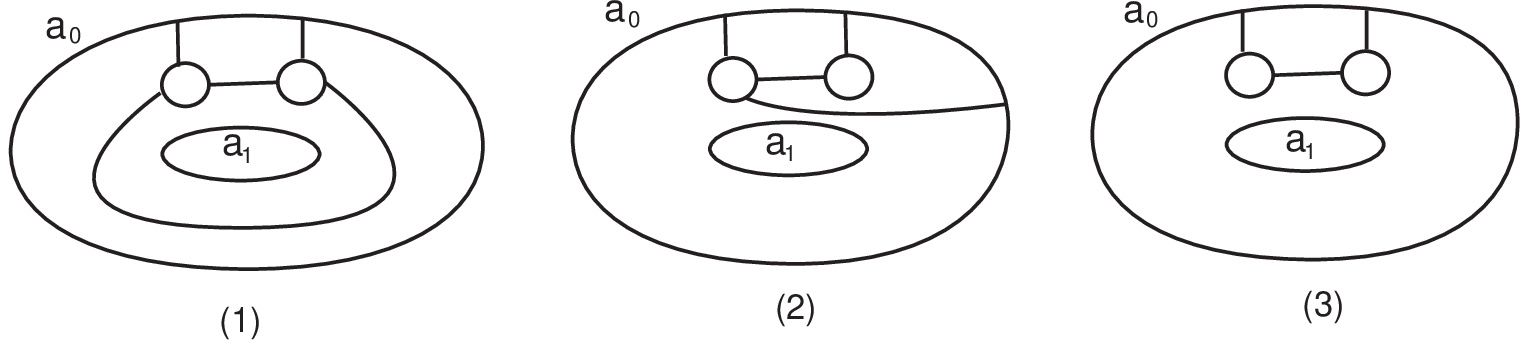}} \caption{ }\label{bgz5-fig Q}
\end{figure}

\begin{lemma} \label{basicreduction3}
Suppose that $M$ is a hyperbolic knot manifold which contains an essential twice-punctured torus $F$ of boundary slope $\beta$ and let $\alpha$ be a slope on $\partial M$ such that $M(\alpha)$ is an irreducible small Seifert manifold. If $F$ is a fibre in $M$, $M(\alpha)$ is not a lens space, and $\tau_\alpha$ reverses the orientations of the Seifert fibres of $M(\alpha)$, then $\Delta(\alpha, \beta) \leq 3$.
\end{lemma}

\pf Suppose that $p = \Delta(\alpha, \beta) \geq 3$. Then $\bar p \geq 3$, so we can argue as in the proof of \cite[Lemma 4.4]{BGZ3} to see that $M(\alpha)$ has base orbifold $S^2(\bar p, \bar p, m)$ where $m \geq 2$. Further, there is an integer $n$ coprime with $m$ such that $L_\alpha$ is isotopic to the closure of an $m/n$ rational tangle in a genus $1$ Heegaard solid torus  $V$ of $M(\alpha) / \tau_\alpha$ as depicted in Figure \ref{bgz5-fig24}. We also have $|L_\alpha| = 1$ if $n$ is odd and $2$ if it is even.

Suppose that $m$ is even. Then $L_\alpha$ is connected, so $L = L_\alpha$ and therefore $q$ is odd (\ref{comp 2}). Further $L $ is homotopically trivial in $L(\bar p, \bar q)$; in fact, $L$ is homotopically trivial in the solid torus $V$ since it is the closure
in $V$ of an $m/n$-rational tangle with $m$ even. But the $4$-braid $L$ represents four times a generator of $H_1(L(\bar p, \bar q))$, so $2p = \bar p > 1$ divides $4$, which is impossible since we have assumed $p \geq 3$. Thus $m$ is odd, and so is at least $3$.

Considering the universal cover of the lens space $L(\bar p, \bar q)$, we obtain two links $\widehat L$ and $\widehat L_\alpha$ in $S^3$ where $\widehat L$, the lift of the image of $L$ in $L(\bar p, \bar q)$, is a $4$-braid and $\widehat L_\alpha$, the lift of $L_\alpha$, is a Montesinos link with $\bar p$ rational tangles each of type $m/n$.

The $2$-fold cover $\Sigma_2(\widehat L_\alpha)$ of $S^3$ branched over $\widehat L_\alpha$ is Seifert with base orbifold a $2$-sphere with $\bar p$ cone points, each of order $m \geq 3$. Thus the Heegard genus of
$\Sigma_2(\widehat L_\alpha))$ is $\bar p-1$ \cite{BoiZie}. (Note that any irreducible horizontal Heegaard surface
 of the Seifert manifold, if one exists, has genus larger than $\bar p-1$.)

If $q$ is odd, then $L_\alpha = L$, so $\widehat L_\alpha$ is
a closed $4$-braid. Since a $2$-sphere in $S^3$
which separates $(S^3, \widehat L_\alpha)$ into two trivial $4$-string tangles
lifts to a genus $3$ Heegaard surface in $\Sigma_2(\widehat L_\alpha)$, the Heegaard genus of
the latter is at most $3$. Thus as $\bar p = 2p$ when $q$ is odd, $2p - 1 = \bar p -1 \leq 3$. But then
$\Delta(\alpha, \beta) = p \leq 2$, contrary to our hypotheses.

Thus $q$ is even, so $p$ is odd and $\bar p = p$.
In this case, $L_\alpha$ is the union of $L$ and the core of the filling torus in
$L(\bar p, \bar q) = V(\bar \alpha)$. Hence there is a genus $1$ Heegaard splitting
$U_1 \cup U_2$ of $S^3$ such that $\widehat L_\alpha$ is the union of the closed
$4$-braid $\widehat L \subset U_1$ and the core of $U_2$.
It follows that $\widehat L_\alpha$ is the closure of a $5$-braid in $S^3$, so
the Heegaard genus of $\Sigma_2(\widehat L_\alpha)$ is at most $4$.
Hence $p - 1 = \bar p - 1 \leq 4$, which gives $p \leq 5$.

Since $p\ge 3$ is odd, $p=3$ or $5$.
We will now eliminate the second case.

Assume $p=5$.
Then $L$ is a 4-braid in the Heegaard solid torus $V\subset L(5,\bar q) = V(\bar\alpha)$,
and $L_\alpha$ is the union of $L$ and a core of the complementary solid torus.
Also, $L_\alpha$ is the closure of the $m/n$-rational tangle ($m$ odd, $n$ even) in
some Heegaard solid torus $W$ in $L(5,\bar q)$, as shown in Figure \ref{bgz5-fig24}.
It follows that each component of $L_\alpha$ is a core of $W$.
In particular, the exterior of $L$ in $L(5,\bar q)$ is a solid torus.

Let $Y$ be the exterior of $L$ in $V$.
Then $\partial Y$ has two components, $\partial V$ and $T_L$, say.
A meridian disk $D$ of $V$ gives rise to a 4-punctured disk $P\subset Y$, with
$\partial P = \partial D\cup c_1 \cup c_2 \cup c_3 \cup c_4$,  where the $c_j$'s are
meridians of $L$.
Note that $\partial D$ has slope $\bar\beta$ on $\partial V$.
Also, since $Y(\bar\alpha)$, the exterior of $L$ in $L(5,\bar q)$, is a solid torus,
there is an essential disk $E\subset Y(\bar\alpha)$ with $\partial E\subset T_L$.
Choosing $E$ to have minimal intersection with the core of the filling solid torus
in $Y(\bar\alpha)$ we get an essential punctured disk $Q\subset Y$, with
$\partial Q = \partial E \cup b_1 \cup \cdots \cup b_n$, where the $b_i$'s are curves
of slope $\bar\alpha$ on $\partial V$.
Since $Y$ is hyperbolic, $n\ge 2$.

Let $\ep$ be the slope of $\partial E$ on $T_L$, and let $\mu$ be the slope on
$T_L$ of a meridian of $L$.

\begin{claim}\label{LemmaOne}
$\Delta (\ep,\mu) =5$.
\end{claim}

\begin{proof}[Proof of Claim \ref{LemmaOne}]
Let $\lambda,\bar\gamma$ be slopes on $T_L$, $\partial V$, respectively, such that
$\Delta (\lambda,\mu) =1$, $\Delta (\bar\gamma,\bar\beta) =1$.
Then $H_1 (Y) \cong \zed\oplus\zed$, generated by $\mu$ and $\bar\gamma$.
Also, in $H_1 (Y)$ we have $\bar\beta = 4\mu$ and $\lambda = 4\bar\gamma +r\mu$
for some $r\in\zed$.
Write $\ep = a\lambda +b\mu$, where $(a,b) =1$.
Then $\ep = a (4\bar\gamma +r\mu) +b\mu = 4a\bar\gamma + (b+ra)\mu\in H_1(Y)$.
Note that $\langle \ep\rangle = \ker (H_1 (T_L) \to H_1 (Y(\bar\alpha)))$.
Since $\bar\alpha = 5\bar\gamma + \bar q\bar\beta = 5\bar\gamma + 4\bar q \mu$,
it follows that there exists $m\in \zed$ such that
\begin{gather*}
4a = 5m\ ,\ \text{ and}\\
b+ra = 4\bar q m\ .
\end{gather*}
The first equation implies $a=5c$, $m=4c$ for some $c\in\zed$, and the second then
implies $b\equiv 0$ (mod~$c$).
Since $(a,b) =1$, $c=\pm1$.
Hence $\Delta (\ep,\mu) = |a| =5$.
\end{proof}

The intersection of $P$ and $Q$ defines in the usual way graphs $\Gamma_P$,
$\Gamma_Q$ in the 2-sphere, where $\Gamma_P$ has vertices $c_1,c_2,c_3,c_4$
and $d_V$ (corresponding to $\partial D$), and $\Gamma_Q$ has vertices
$b_1,b_2,\ldots,b_n$ and $e_L$ (corresponding to $\partial E$).
Thus in $\Gamma_P$ the vertex $d_V$ has valency $5n$ and each $c_j$ has valency~5,
while in $\Gamma_Q$ the vertex $e_L$ has valency $20$ and each $b_i$ has valency~5.
Since $L$ is a braid, all the $c_j$'s have the same sign.
Hence, by the parity rule, $\Gamma_P$ has no $C$-edges, $\Gamma_Q$ has no
$E$-edges, and all $B$-edges of $\Gamma_Q$ are negative.

The computation in the proof of Claim \ref{LemmaOne}
shows that $m = \pm4$ i.e. $\ep = \pm 4\bar\alpha \in H_1 (Y)$.
Thus (with any choice of orientations)
$$\big| \text{ number of positive $b_i$'s $-$ number of negative $b_i$'s }\big| = 4$$
It follows that $n$ is even and $\ge 4$.

\begin{claim}\label{LemmaTwo}
$n=4$.
\end{claim}

\begin{rem}
This is equivalent to saying that $\Gamma_P$ has no $D$-edges.
\end{rem}

\begin{proof}[Proof of Claim \ref{LemmaTwo}]
Suppose $n\ge 6$.
We use the notation and terminology of the argument immediately preceding
Lemma \ref{lem4.12}.

Since each $c_j$ has valency 5, it is easy to see that if $\bar e$ is an edge of
$\bar\Gamma_P$ of type~$O$ then $\lambda (\bar e) \le n-1$.
If $\bar e$ is of type~$N$ then $\lambda (\bar e) \le n$.
Since the valency of $d_V$ is $5n$ we get
$$5n \le k_0 (n-1) + (k-k_0) n + (4-k_0) 5$$
Since $k-k_0 \le 1$ this gives
$$(4-k_0) n \le 20 - 6k_0\ ,$$
which contradicts our assumption that $n\ge 6$.
\end{proof}

\begin{claim}\label{LemmaThree}
$L$ is a $1$-bridge braid in $V$.
\end{claim}

\begin{proof}[Proof of Claim \ref{LemmaThree}]
Recall that in $\Gamma_P$ all the $c_j$'s have the same sign.
Claim \ref{LemmaTwo} implies that all the $b_i$'s have the same sign.
Hence $\Gamma_P$ has only $CD$-edges and $\Gamma_Q$ has only $BE$-edges.
Therefore there are five parallel $BE$-edges in $\Gamma_Q$; see Figure \ref{figC}.
Let $\gamma \subset \partial V$ be the arc in $\partial Q$ shown in Figure \ref{figC}.
Then the bigon faces of $\Gamma_Q$ shown in Figure \ref{figC} allow us to define a
(non-ambient) isotopy of $L$ in $V$ taking it to $\gamma\cup \delta$, where
$\delta$ is the image of $e \cup e'_1$ in $P\subset D$, where we
 treat the vertex $c_1$ as a point in $D$ so that $\d$ is considered as a properly embedded arc in $D$.
See Figure \ref{figD} for a typical situation of the arcs $e_1,e_2,e_3,e_4,e'_1$ in $P$.
This shows that $L$ is a 1-bridge braid in $V$ (with $\d$ as the bridge).
(Note that since $Y$ is hyperbolic $L$ is not a $0$-bridge braid.)
\end{proof}

\begin{figure}[!ht]
\centerline{\includegraphics{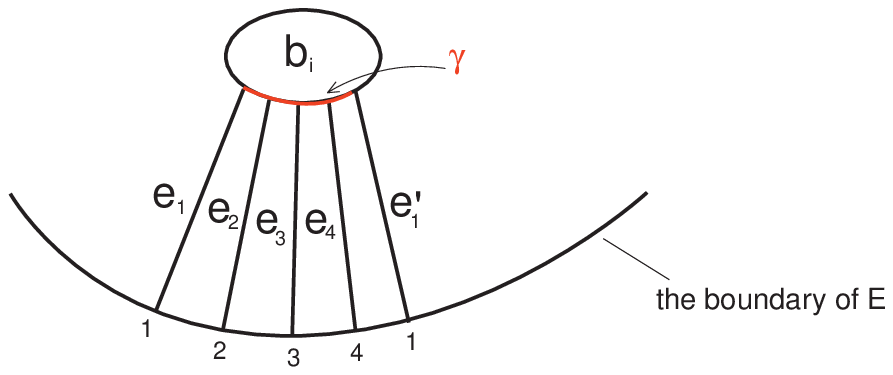}} \caption{}\label{figC}
\end{figure}

\begin{figure}[!ht]
\centerline{\includegraphics{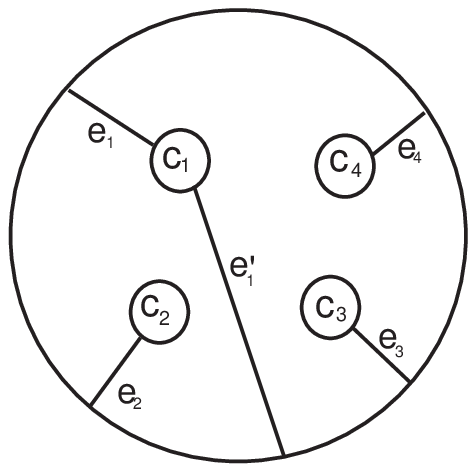}} \caption{}\label{figD}
\end{figure}

We thus have a 1-bridge $L$ in $V$ with a Dehn filling (namely $\bar\alpha$-filling)
on the boundary component $\partial V$ of the exterior $Y$ of $L$ in $V$ that
gives a solid torus.
Such braids are classified in \cite{Wu3}.
In particular \cite[Table 1]{Wu3} shows that there is a unique example (up to
homeomorphism) where $L$ has winding number~4 in $V$:
$L$ is the closure of the braid $\sigma_1 (\sigma_3\sigma_2\sigma_1)^2$.
This braid is conjugate to $\sigma_2\sigma_1\sigma_3\sigma_2\sigma_1^3$, whose
closure is shown in Figure \ref{figE}.
It is clear from this figure that $L$ bounds a M\"{o}bius band in $V$, contradicting
the fact that $Y$ is hyperbolic.

\begin{figure}[!ht]
\centerline{\includegraphics{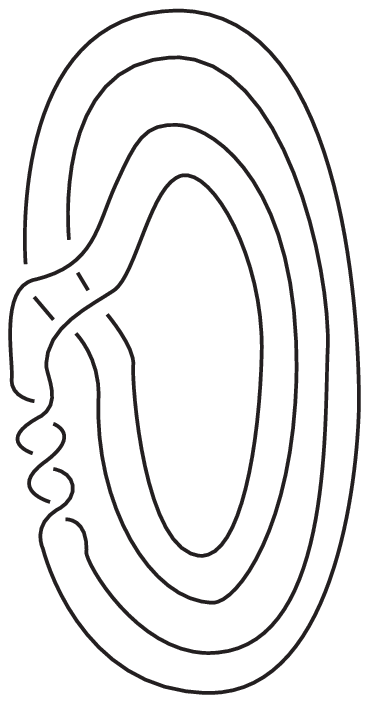}} \caption{}\label{figE}
\end{figure}

This contradiction completes the proof of Lemma~\ref{basicreduction3}.
\qed

To finish the proof of Proposition \ref{prop: F2fibre} we only need to prove the following two lemmas.

\begin{lemma}\label{F2fibre-non-Seifert}
Suppose that $M$ is a hyperbolic knot manifold which contains an essential twice-punctured torus $F$ of boundary slope $\beta$ and let $\alpha$ be a slope on $\partial M$ such that $M(\alpha)$ is an irreducible small Seifert manifold. If $F$ is a fibre in $M$, $M(\alpha)$ is not a lens space, and $\tau_\alpha$ preserves the orientations of the Seifert fibres of $M(\alpha)$, then
 $\Delta(\alpha,\beta)\leq 2$ when $M(\alpha)$ is not a prism manifold.
\end{lemma}

\begin{lemma}\label{prism-case}
Suppose that $M$ is a hyperbolic knot manifold which contains an essential twice-punctured torus $F$ of boundary slope $\beta$ and let $\alpha$ be a slope on $\partial M$ such that $M(\alpha)$ is a prism manifold. If $F$ is a fibre in $M$ and $\tau_\alpha$ preserves the orientations of the Seifert fibres of $M(\alpha)$, then $\Delta(\alpha,\beta)\leq 3$.
\end{lemma}

\begin{proof}[Proof of Lemma \ref{F2fibre-non-Seifert}]
  The argument is similar to that given in \cite[\S 6]{BGZ3}.

Suppose that $\Delta = \Delta (\alpha,\beta) \ge 3$.
Then the branch set $L_\alpha$ is a set of at most three Seifert fibres of $L(\bar p,\bar q)
= M(\alpha)/\tau_\alpha$ with the induced Seifert fibration from $M(\alpha)$ given by
\cite[Lemma~4.3]{BGZ3}.
Since $L$ is a hyperbolic link in $V$, $K_\alpha$ is not contained in $L_\alpha$.
Thus $L_\alpha =L$ so that $q$ is odd and $L(\bar p,\bar q)$ is a lens space of
order $2p = 2\Delta \ge 6$.
As in \cite[\S6.1]{BGZ3}, we
\begin{itemize}
\item define $X$ to be the exterior of $L$ in $V(\bar\alpha) = L(\bar p,\bar q)$
and $Y$ its exterior in $V$;
\item fix a component $K$ of $L$ which is a regular fibre of $L(\bar p,\bar q)$ and
define $T_K$; respectively $T_V$, to be the component of $\partial Y$ corresponding
to $K$, respectively $\partial V$;
\item use a meridian disk of $V$ to construct a 4-punctured disk $P$ properly
embedded in $Y$;
\item construct an essential, separating, vertical annulus $(A,\partial A) \subset (X,T_K)$
which separates $X$ into two components $X_1$ and $X_2$ such that each $X_i$
is Seifert with base orbifold either an annulus with no cone points or a disk with
one cone point, of order at least 3 (\cite[Lemma~6.1]{BGZ3}).
\item construct an $n$-punctured essential annulus $Q$ in $Y$, where $n$ is even,
from an appropriately chosen essential separating vertical annulus in the exterior
of $K$ in $L(\bar p,\bar q)$;
\item define graphs $\Gamma_P$ and $\Gamma_Q$, vertices $d_V$, $c_1,\ldots,c_4$,
$a_1,a_2$, $b_1,\ldots, b_n$, and $D$-edges, $CD$-edges, etc.;
\end{itemize}

In the graph $\Gamma_P$, $d_V$ has valency $2\Delta n\ge 6n$, and each $c_j$
has valency 0 or 2.

The proof closely follows the proof of Case II of Section \ref{sec: twice-punctured semi-fibre}.

First assume $n\ge 4$.
The argument immediately preceding Lemma \ref{lem4.12} then shows that, with the
notation established there,
\begin{align*}
2\Delta n &\le kn + 2 (4-k_0)\\
&\le kn +2(5-k)
\end{align*}
giving
$$(2\Delta -k) n \le 2(5-k)\ .$$
Since $2\Delta \ge 6$, this is a contradiction.

So we may assume $n=2$.
Note that there are no parallel $D$-edges in $\Gamma_P$ and valency $d_V\ge 12$.
We now follow the proof of Lemma \ref{lem4.19}.
By Lemma  \ref{lem4.7}(1), $\ell = 2,3$ or $4$.
\medskip

\noindent $\underline{k=5}$.
The $D$-edges of $\Gamma_P$ are as shown in Figure \ref{figA}.
Applying Lemma \ref{lem4.14}, the face $f$ implies that $c_3$ and $c_4$ have valency~2, and
similarly the face $g$ implies that $c_1$ and $c_2$ have valency~2.
Then $d_V$ has valency~18, implying that $\Delta =9$, a contradiction.
\medskip

\begin{figure}[!ht]
\centerline{\includegraphics{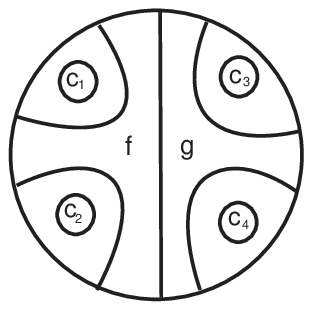}} \caption{}\label{figA}
\end{figure}

\noindent $\underline{k=4}$.
The $D$-edges of $\Gamma_P$ are as shown in Figure \ref{bgz5-fig3}(a) or (b).

In case (a) the proof of Lemma \ref{lem4.19} applies.

In case (b) the proof of Lemma \ref{lem4.19} applies unless $c_1$ and $c_2$ have valency~0.
In this case, $c_3$ and $c_4$ have valency~2 and there are two possibilities for
$\Gamma_P$, shown in Figure \ref{figB}(i) and (ii).
In both cases the faces $f$ and $g$ contradict Lemma \ref{lem4.14}.
\medskip

\begin{figure}[!ht]
\centerline{\includegraphics{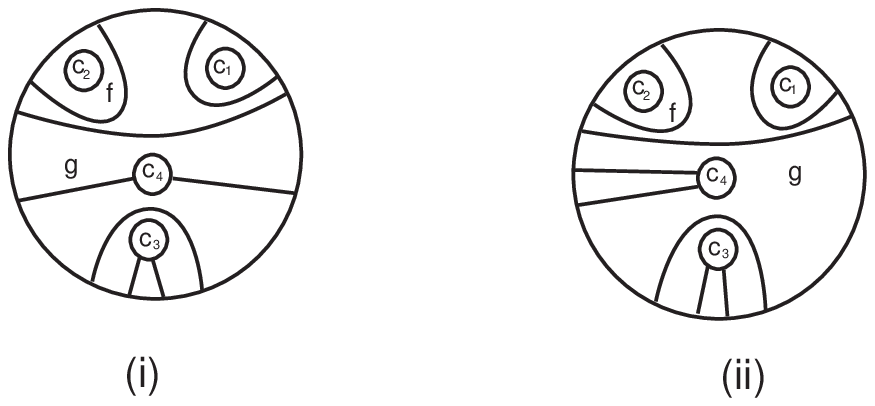}} \caption{}\label{figB}
\end{figure}

\noindent $\underline{k=3}$.
Since $12\le $ valency $d_V = 2k+2\ell$ we have $\ell = 3$ or $4$.
The proof of Lemma \ref{lem4.19} then applies to rule out every case here except that
shown in Figure \ref{bgz5-fig9}(iv).
But in that case $d_V$ has valency~10.
\medskip

\noindent $\underline{k=2}$.
Here $\ell=4$, and the appropriate parts of the proof of Lemma \ref{lem4.19} apply.
More precisely, in case (a) take the part of the proof where $c_3$ and $c_4$
have valency 2.
In case (b), take the part of the proof that starts with the supposition that both
$c_2$ and $c_3$ have valency~2.
Of the possibilities for $\Gamma_P$ shown in Figure \ref{bgz5-fig13}, (vi) and (vii) cannot
occur here since valency $d_V \ge 12$.
The other cases are eliminated as in the proof of Lemma \ref{lem4.19}.
\medskip

\noindent $\underline{k=1}$.
Since $12 \le$ valency $d_V = 2k+2\ell\le 10$, this case is impossible.
\end{proof}

\begin{proof}[Proof of Lemma \ref{prism-case}]
Suppose otherwise that $\D=\D(\alpha,\beta)\geq 4$. Then
by \cite{L2}, $\D=4$.

We may assume, by \cite[Lemma 4.1]{BGZ3}, that the $\tau_\alpha$-invariant Seifert
structure on $M(\alpha)$ is the one whose base orbifold is $S^2(2,2,c)$ for some $c\geq 2$.
 We continue to use the notations established so far in this section.
In the present case  we still have the manifolds $X, T_K, T_V, A, X_1, X_2, Y, P, Q$ and the graphs
$\G_P, \G_Q$ defined as above in this section. The only new situation that possibly arises
 in the present case (which we assume happens since otherwise the argument of Lemma \ref{F2fibre-non-Seifert} works identically)
 is that exactly one of $X_1$ and $X_2$, say $X_1$, is a solid torus whose
singular fibre has order two, and $X_2$ is either a $T^2\times I$ or a
solid torus whose singular fibre has order bigger than two. This assertion follows from the proof
of \cite[Lemma 6.1]{BGZ3}.
Thus  the graph $\G_P$ may contain
$D$-edge $S$-cycles, and the proof of Lemma \ref{lem4.9} yields the following

\begin{claim}\label{prism-claim1}
The bigon face bounded by a $D$-edge $S$-cycle in
$\G_P$ lies on the $X_1$-side of $A$. \qed
\end{claim}

\begin{claim}\label{prism-claim2}
When $n\geq 4$, $\G_P$ cannot have $D$-edge extended  $S$-cycles.
\end{claim}

\begin{proof}[Proof of Claim  \ref{prism-claim2}]
Let $\{e_1, e_2, e_3, e_4\}$ be a $D$-edge extended $S$-cycle with
label sequence  $\{i-1, i, i+1, i+2\}$ (here labels are defined mod $n$).
Let $R_j$ be the bigon face between $e_j, e_{j+1}$ for $j=1,2,3$.
Then $R_2$ is contained in $X_1$ by Claim \ref{prism-claim1}, and $R_1$ and $R_3$   in $X_2$.
Let $H_{j}$  be the component  of  $V_{\bar\a}\setminus \cup_{k=1}^n \widehat b_k$
connecting  $\{\widehat b_{j}, \widehat b_{j+1}\}$ for $j=i-1, i, i+1$.
Let $U_1$ be a regular neighbourhood
of $H_{i}\cup R_2$ in $X_1$, and   $U_2$   a regular neighbourhood
of $H_{i-1}\cup H_{i+1}\cup R_1\cup R_3$ in $X_2$.
Then  $U_1$ is a solid torus and the frontier $E_1$ of $U_1$  in $X_1$ is an annulus
with  winding number
$2$ in $U_1$. Thus $E_1$ is  parallel to $\partial X_1\setminus (\partial U_1\cap A)$ in $X_1\setminus U_1$.
 The manifold   $U_2$ is also a solid torus, the frontier of $U_2$ in $X_2$ is a pair of annuli $E_2', E_2''$, and $E_2'$ is parallel to $E_2''$
in $U_2$.  We may assume that $U_1\cap A$ is equal to a component of $U_2\cap A$.
 Let $U_3$ and $U_4$ be the two components of $X_2\setminus U_2$, and we may assume that
$U_4$ is the one which contains $\partial X_2\setminus A$ and that $E_2''\subset \partial U_4$
(cf. Figure \ref{prism-extended}).
 Note that $U_4$ must be a solid torus in which $E_2''$ is parallel to $\partial U_4\setminus E_2''$
 for otherwise the frontier of $U_4$ in $X$ would be an essential annulus in $X$ which has the same
   boundary as the annulus $A$ in $T_K$ but   has
  less number of intersection components with $V_{\bar\a}$ than $A$.
 Thus $U_3$ is either a torus cross interval or a solid torus
 in which $E_2'$ has winding number larger than two.
 Now let $A'$ be the annulus $E_1\cup E_2'\cup [A\setminus (U_1\cup U_2)]$
 (cf. Figure \ref{prism-extended}).
 Then $A'$ is an essential annulus  in $X$ such that
$\partial A'=\partial A$ in $T_K$ but $A'$  has
  less number of intersection components with $V_{\bar\a}$ than  $A$,
  yielding a final contradiction.
\end{proof}

\begin{figure}[!ht]
\centerline{\includegraphics{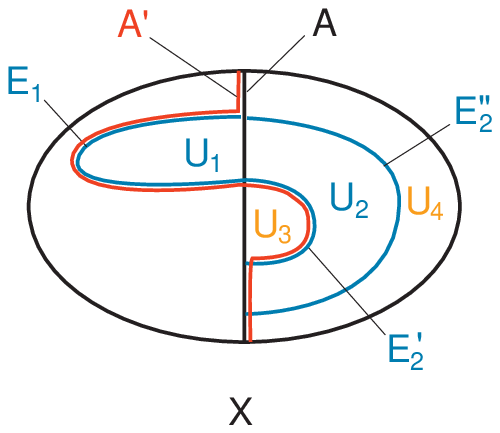}} \caption{}\label{prism-extended}
\end{figure}

Note that Lemma \ref{lem4.11} still holds in the present case.

\begin{claim}\label{prism-claim3}
 $\G_P$ cannot have a family of more than $\frac{n}{2}+1$ parallel $D$-edges.
\end{claim}

\begin{proof}[Proof of Claim  \ref{prism-claim3}]
Otherwise $\G_P$ has either an $D$-edge extended $S$-cycle or two
$S$-cycles with distinct label pairs, contradicting
Claim \ref{prism-claim2} or Lemma \ref{lem4.11} respectively.
\end{proof}

\begin{claim}\label{prism-claim4}
In $\Gamma_Q$, the endpoints of the $AB$-edges incident to $b_i$ are
consecutive around $b_i$, for each fixed $i=1,...,n$.
\end{claim}

\begin{proof}[Proof of Claim  \ref{prism-claim4}]
By the parity rule  all the $AB$-edges incident to a fixed $b_i$
have their other endpoint on a fixed component of $\partial A=\{a_1, a_2\}$.
Also observe that all the $b_i$'s which are connected to
a fixed component of $\{a_1, a_2\}$ by $AB$-edges have the same sign,
and thus there can be no $B$-edges connecting  between these $b_i$'s (because all
$B$-edges are negative). Noticing that every vertex
$b_i$ is incident to some $B$-edges in $\G_Q$, one can now  see that the lemma follows.
(cf. Lemma \ref{lem4.17} and its proof.)\end{proof}

\begin{claim}\label{prism-claim5}
The endpoints with label $i$ on $d_V$ of the $CD$-edges
of $\Gamma_P$ (when non-empty) are consecutive among all the edge-endpoints with label~$i$
on $d_V$, for each fixed $i=1,...,n$. \end{claim}

\begin{proof}[Proof of Claim  \ref{prism-claim5}]
This follows from Claim \ref{prism-claim4} and the condition
that $\D=4$. (cf. Corollary \ref{cor4.18} and its proof.)
\end{proof}

\begin{claim}\label{prism-claim6}
 $n=2$.
\end{claim}

\begin{proof}[Proof of Claim  \ref{prism-claim6}]
Since  there are at most $8$ $CD$-edges, there are at least
$(8 n-8)/2=4 n-4$ $D$-edges in $\G_P$. Thus  there
is a family of parallel $D$-edges in $\G_P$ with at least
$(4 n-4)/5$ edges.
By Claim \ref{prism-claim3},
$(4 n-4)/5\leq\frac{n}{2}+1$.
Thus $n\leq 6$.
So assume for contradiction that $n=6$ or $4$.

If $n=6$, then there are exactly $8$ $CD$-edges and $5$ families of
parallel $D$-edges, each having $4$ edges.
So part of $\G_P$ maybe assumed  as shown in
Figure \ref{prism-a}.
So there is an extended $D$-edge $S$-cycle in $\G_P$, contradicting Claim \ref{prism-claim2}.

\begin{figure}[!ht]
\centerline{\includegraphics{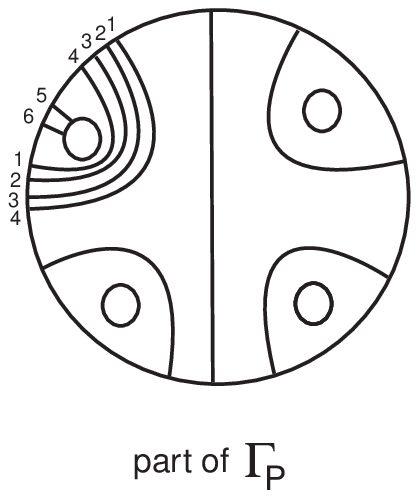}} \caption{}\label{prism-a}
\end{figure}

So  $n=4$. Note that the total weights of all $D$-edges in $\overline \G_P$
is at least $12$ and  each $D$-edge of $\overline \G_P$ has weight at most $3$ by Claim \ref{prism-claim3}.
It follows that $\overline\G_P$ has no $C$-edges.
It also follows that there is an $D$-edge $\bar e$ in $\overline\G_P$ of type $O$ and with weight
$3$.
Let $c_j$ be the single vertex that $\bar e$ cuts off as given in the definition of an $O$-type edge.
We note that there must be two $CD$-edges incident to $c_j$.
For otherwise a part of $\G_P$ is as shown in Figure \ref{prism-b}.
In the figure the face $f$ is bounded by an $D$-edge $S$-cycle and thus
is contained in $X_1$ (which is a solid torus), but the face $g$ is also contained in
$X_1$, yielding  a contradiction since it abuts a puncture $c_i$, which cannot exist
in the solid torus $X_1$.

\begin{figure}[!ht]
\centerline{\includegraphics{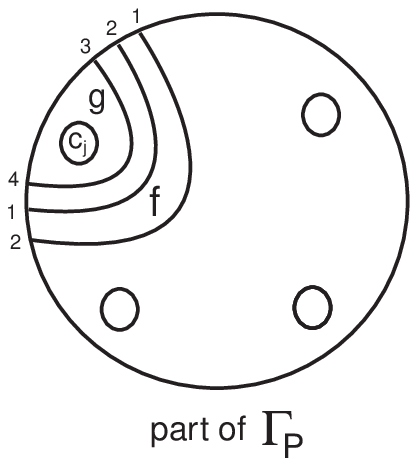}} \caption{}\label{prism-b}
\end{figure}

Hence part of $\G_P$ is as shown in Figure \ref{prism-c}(a). If there is no
$N$-type $D$-edge in $\overline \G_P$, then there are exactly $8$ $CD$-edges and
$12$ $D$-edges in $\G_P$, and $\G_P$  must look like that shown in Figure \ref{prism-c}(b).
But then we get a contradiction with Claim \ref{prism-claim5}, considering the label $1$.
Therefore part of $\G_P$ maybe assumed as shown in Figure \ref{prism-c}(c) where
the edge $e$ has label $4$ at its upper endpoint.

By the parity rule  the lower endpoint of the edge  $e$  has
label $1$ or $3$.
In the former case the part of the graph $\G_P$ on the left-hand side of
the edge $e$  must be as shown in Figure \ref{prism-c}(d) or (e), by considering labels around $d_V$.
But then the face  $g$ shown in both of the subcases  is bounded by an $D$-edge Scharlemann cycle of order $3$
and lies on $X_1$-side, contradicting Lemma \ref{lem4.11}.
In the latter case the  left-hand side of
the  edge $e$ must be as shown in Figure \ref{prism-c}(f).
But then we get a contradiction with Claim \ref{prism-claim5} by looking at the label $4$.
\end{proof}

\begin{figure}[!ht]
\centerline{\includegraphics{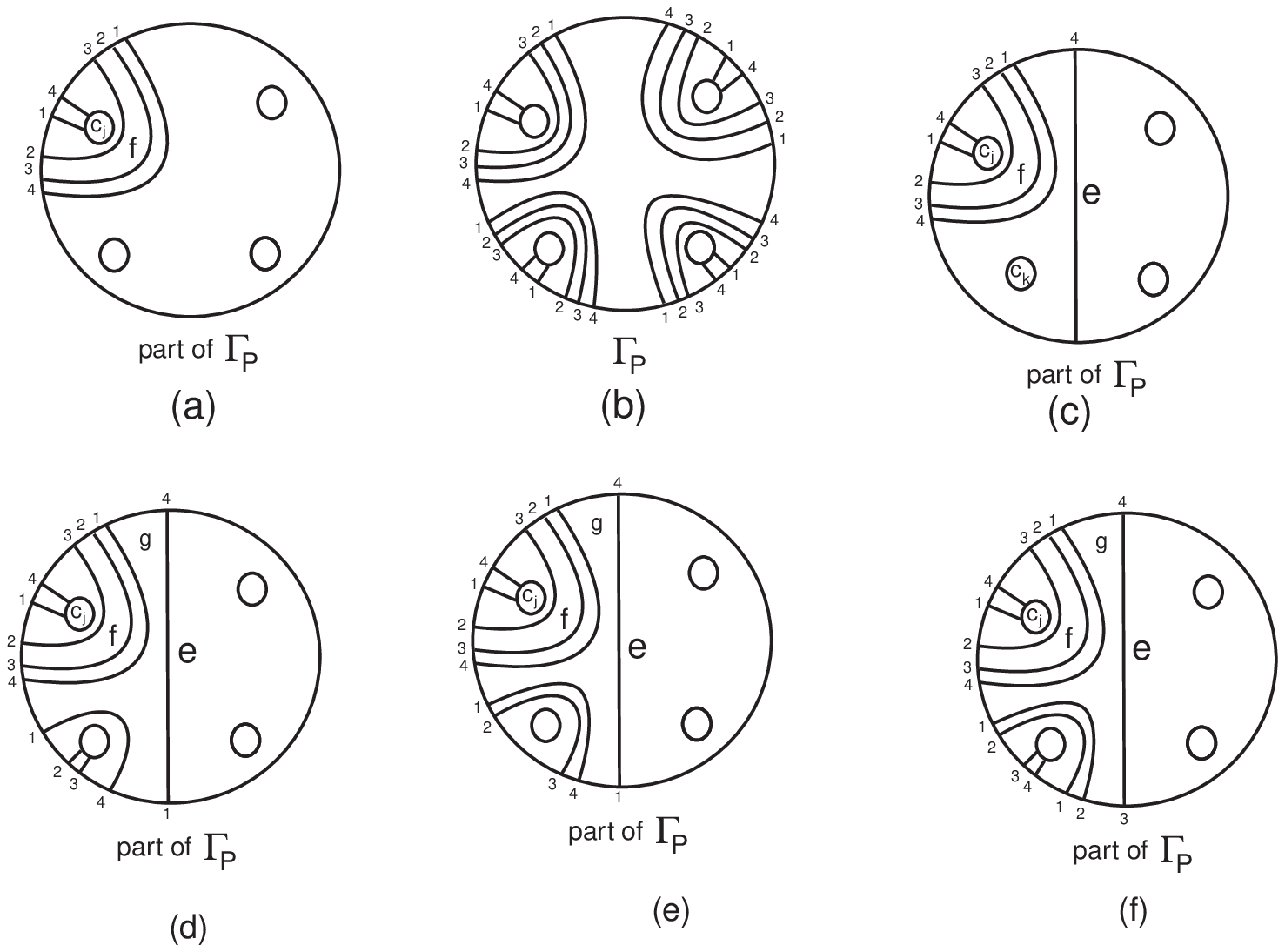}} \caption{}\label{prism-c}
\end{figure}

So $n=2$.
By the proof of Lemma \ref{F2fibre-non-Seifert} we may only consider
the situation when $\G_P$ contains $D$-edge $S$-cycles, i.e.
$\G_P$ contains parallel $D$-edges.
But note that each parallel family of $D$-edges contains at most two edges
by Claim \ref{prism-claim1}.
We have at least $4$ $D$-edges in $\G_P$.

Note that Lemma \ref{lem4.14} is still valid in the present case.

Recall that $k$ is the number of $D$-edges in $\overline \G_P$ and $k\leq 5$.
In the current case $k\geq 2$.
Note that we have assumed that there are at least $k+1$ $D$-edges in $\G_P$.

If $k=4$ or $5$, then $D$-edges in  $\overline\Gamma_P$ are as shown in
Figure \ref{bgz5-fig2} or Figure \ref{bgz5-fig3}, where at least one  edge has weight $2$.
One can easily see from the figures that
 $\Gamma_P$ must   have, in each case,   a disk face bounded by an $D$-edge Scharlemann cycle
  of order $3$ or $4$  that lies in $X_2$.
  Hence $X_2$ is a solid torus by Lemma \ref{lem4.14} and thus
  there are $8$ $CD$-edges in $\G_P$.
But there are at least $5$ $D$-edges in $\G_P$.
So the valency of the vertex $d_V$ would be  at least $18$
in $\G_P$, contradicting to the assumption
that the valency of $d_V$ is $16$ in $\G_P$.

If $k=3$, then the three $D$-edges in  $\overline\Gamma_P$ are as shown in Figure \ref{bgz5-fig6}
(a) (b) (c), with at least one edge  having  weight $2$.
We first consider  Case (a) of Figure \ref{bgz5-fig10}.
If the vertex $c_4$ has valency $2$ in $\G_P$, then $\G_P$ has a disk face
with at least two $d_V$-corners, contained in $X_2$.
Hence $X_2$ is a solid torus and every $c_i$ has valency two in $\G_P$.
Now it is easy to see that Claim \ref{prism-claim5}
is violated. So the valency of $c_4$ is zero, and thus there are at least $5$ $D$-edges
in $\G_P$. Since there are at most $6$ $D$-edges in
$\G_P$, at least two of $c_1,c_2, c_3$ have valency $2$.
Now it is easy to see that we have a contradiction again with Claim \ref{prism-claim5}.
Next we  consider  Case (b) of Figure \ref{bgz5-fig10}.
In this case $\G_P$ has a disk face bounded by an $D$-edge Scharlemann cycle of order 3, contained in $X_2$.  Hence $X_2$ is a solid torus by Lemma \ref{lem4.14} and thus
  every $c_i$ has valency $2$ in $\G_P$.
Now  one can see a contradiction with Claim \ref{prism-claim5}.
Lastly in  Case (c) of Figure \ref{bgz5-fig10} one can easily
see a contradiction with Claim \ref{prism-claim5} once again.

If $k=2$, then the two $D$-edges in  $\overline\Gamma_P$ are as shown in Figure \ref{bgz5-fig10}
(a) (b), each having  weight $2$.
Also there are exactly $8$  $CD$-edges in $\G_P$.
Now it is easy to see that Claim \ref{prism-claim5} is violated in both (a) and (b) cases.
\end{proof}

\section{Further assumptions, reductions, and background material}
\label{sec: second reduction}

The remainder of the paper deals with the cases that $F$ is neither a fibre or semi-fibre.

\begin{lemma} \label{b1 > 1}
If $F$ is neither a fibre nor semi-fibre in $M$ and $b_1(M) \geq 2$, then
$$\Delta(\alpha, \beta) \leq \left\{ \begin{array}{ll} 1 & \hbox{if $M(\alpha)$ is very small} \\ 5 & \hbox{otherwise}   \end{array}  \right.$$
\end{lemma}

\begin{proof}
If $b_1(M) \geq 2$, there is a closed, non-separating, Thurston norm minimizing surface $S$ in the interior of $M$. By \cite[Corollary, page 462]{Ga}, $S$ remains Thurston norm minimizing in $M(\gamma)$ for all but at most one slope $\gamma_0$ on $\partial M$. If $\alpha = \gamma_0$, a result of Wu implies that $\Delta(\alpha, \beta) \leq 1$ (cf. \cite[Proposition 5.1]{BGZ1}), so suppose that $\alpha \ne \gamma_0$. Then $M(\alpha)$ cannot be very small since it contains $S$ as an essential surface. Theorem 1.2 of \cite{BGZ1} now shows that $\Delta(\alpha, \beta) \leq 5$.
\end{proof}

We will apply the results of \cite{BGZ2} in what follows, which requires that more care be taken in the choice of $F$. Here is a list of assumptions we will make in our proof of Theorem \ref{thm: twice-punctured precise} when $F$ is neither a fibre nor semi-fibre. The references which justifiy the sufficiency of these assumptions are listed as well.

\begin{assumps}\label{assumptions 1} $\;$ \\
\vspace{-.4cm}
{\begin{enumerate}

\vspace{-.3cm} \item $b_1(M) = 1$ $(${\rm cf.}~{\rm Lemma \ref{b1 > 1}}$)$.

\vspace{.1cm} \item $M$ does not admit a fibre or semi-fibre which is an essential twice-punctured torus of boundary slope $\beta$ {\rm (cf.~Propositions \ref{prop: F2-semi-fibre} and \ref{prop: F2fibre})}.

\vspace{.1cm} \item $M$ admits no essential once-punctured torus of boundary slope $\beta$ $(${\rm cf.}~\cite[Theorem 1.3]{BGZ3}$)$.

\vspace{.1cm} \item $F$ is chosen according to the assumptions of \cite[\S 2]{BGZ2}.

\vspace{.1cm} \item If $M_F$ is not connected, then it is a union of two genus $2$ handlebodies $(${\rm cf.}~\cite[Proposition 3.3]{BGZ3}$)$.

\end{enumerate}
}
\end{assumps}
\vspace{-.3cm}
Next we provide a summary of the notation and terminology from \cite{BGZ2} that will be used below.

Let $F$ be chosen as above. If $F$ separates $M$ we take $S$ to be $F$. Otherwise we take $S$ to be the frontier of a small radius tubular neighbourhood of $F$ in $M$. Thus $S$ consists of two parallel copies  $F_1, F_2$ of $F$. In either case $S$ splits $M$ into two components $X^+$ and $X^-$. By Assumption \ref{assumptions 1} (2) above, we can suppose that $X^+$ is not an $I$-bundle.

Let $\widehat S$ be a closed surface in $M(\beta)$ obtained by attaching disjoint
meridian disks of the $\beta$-filling solid torus to $S$. Then $\widehat S$ splits $M(\beta)$
into two compact submanifolds $\widehat X^+$ containing $X^+$ and $\widehat X^-$ containing $X^-$,
each having incompressible boundary $\widehat S$.

 It was shown in \cite{BCSZ1} how to construct an immersion $h: Y \to M(\alpha)$ where $Y$ is a disk or torus, a labeled ``intersection" graph $\Gamma_F = h^{-1}(F) \subset Y$, and, for each sign $\epsilon = \pm$, a sequence of characteristic subsurfaces
$$F = \dot{\Phi}_0^\epsilon \supseteq  \dot{\Phi}_1^\epsilon \supseteq  \dot{\Phi}_2^\epsilon \supseteq  \ldots \supseteq  \dot{\Phi}_n^\epsilon \supseteq  \ldots $$
See \cite[\S 5]{BCSZ1} for the definition of the $j$-th characteristic subsurface $\Phi_j^\epsilon \subseteq S$. We shall assume throughout the paper that $\Phi_j^\epsilon$ is neatly embedded in $S$ (\cite[\S 3.1]{BGZ2}). It is characterized up to ambient isotopy by the following property:
$$(*) \left \{ \begin{array}{ll}
\mbox{\rm a large function $f_0:K \to S$ admits an essential homotopy of length $j$ which starts} \\
\mbox{\rm on the $\epsilon$-side of $S$ if and only if it is homotopic in $S$ to a map with image in $\Phi_j^\epsilon$ }
\end{array} \right. $$
See  \cite[Proposition 5.2.8]{BCSZ1}. When $j = 1$, basic Jaco-Shalen-Johannson theory guarantees the existence of an $(I, S^0)$-bundle pair $(\Sigma_1^\epsilon, \Phi_1^\epsilon) \subset (X^\epsilon, S)$. It was shown in \cite[Proposition 4.9]{BGZ2} that we can assume that $(\Sigma_1^\epsilon, \Phi_1^\epsilon)$ is neatly embedded in $(X^\epsilon, S)$ (\cite[\S 3.2]{BGZ2}).

As in \cite[\S 3.2]{BGZ2} we take $\dot{\Phi}_j^\epsilon$ to be the union of the components of $\Phi_j^\epsilon$ which contain some outer boundary components and $\breve{\Phi}_j^\epsilon$ to be the neat subsurface in $S$ obtained from the union of $\dot{\Phi}_j^\epsilon$ and a closed collar neighbourhood of $\partial S - \partial  \dot{\Phi}_j^\epsilon$ in $S - \dot{\Phi}_j^\epsilon$. There are corresponding $I$-bundle pairs  $(\dot{\Sigma}_1^\epsilon, \dot{\Phi}_1^\epsilon)$ and $(\breve \Sigma_1^\epsilon, \breve \Phi_1^\epsilon)$ neatly embedded in $(X^\epsilon, S)$.

A neat subsurface $S_0$ of $S$ is called {\it tight}  if it caps off to a disk in $\widehat S$. Equivalently, $S_0$ is a connected, planar, neat subsurface of $S$ with one inner boundary component.

We use $t_j^\epsilon$ to denote the number of  tight components of $\breve{\Phi}_j^\epsilon$. If $j$ is odd, $t_j^\epsilon$ is even, while if $j$ is even, $t_{j}^+ = t_{j}^-$. See \cite[\S 6]{BGZ2}.

An {\it  intersection graph} $\Gamma_F$ can be constructed in a disk or torus $Y$ from an immersion
\begin{eqnarray} \label{immersion}
h: Y \to M(\alpha)
\end{eqnarray}
(cf. \cite[Section 11]{BGZ2}). The immersion maps the vertices of $\Gamma_F$ to meridian disks of the $\alpha$-filling torus, edges of $\Gamma_F$ to $F$, and faces of $\Gamma_F$ to $X^+$ or $X^-$. For simplicity we shall say a face of $\Gamma_F$ is contained in $X^\epsilon$ if its image under the immersion is contained there. We refer to \cite[Sections 11 and 12]{BGZ2} for terms, notations and basic facts concerning $\Gamma_F$.

\begin{rem}
{\rm Note that  the intersection  graph $\G_F$ in $Y$ has the following property: either $\G_F$ has a connected component which lies in a subdisk of $Y$ or $Y$ is a torus and every face of $\G_F$ is a disk or an annulus.
So we may and shall  assume that either $\G_F$ is a connected graph in a disk or
it is a graph in a torus with only disk faces and/or annulus faces. It turns out that our graph related arguments are  never affected
whether  annulus faces exist or not.}
\end{rem}

\section{The proof of Theorem \ref{thm: twice-punctured precise} when $F$ is non-separating but not a fibre}  \label{non-sep not fibre}

In this section we prove that part of Theorem \ref{thm: twice-punctured precise} dealing with the case that $F$ is non-separating but not a fibre.
We  suppose throughout that Assumptions \ref{assumptions 0} and \ref{assumptions 1} hold.  Note that we can assume that  the components of $\partial F$ are like-oriented on $\partial M$. For  otherwise we can attach a peripheral annulus to $F$ to obtain a closed non-separating genus two surface in $M$. This surface must be incompressible in $M$
since $M$ is hyperbolic. It follows from \cite[Theorem 2.4.3]{CGLS} that $\b$ is a singular slope, which
contradicts  Assumptions \ref{assumptions 0}(4). By construction, $X^- = F\times I$.

We shall prove

\begin{prop}
\label{prop: F2-non-sep non-fibre}
Suppose that $M$ is a hyperbolic knot manifold which contains an essential, non-separating twice-punctured torus $F$ of boundary slope $\beta$ and let $\alpha$ be a slope on $\partial M$ such that $M(\alpha)$ is an irreducible small Seifert manifold. If $F$ is not a fibre and Assumptions \ref{assumptions 0} and \ref{assumptions 1} hold, then
$$\Delta(\alpha, \beta) \leq \left\{ \begin{array}{ll} 4 & \hbox{if $M(\alpha)$ is very small} \\ 5 & \hbox{otherwise}   \end{array}  \right.$$
\end{prop}

We begin with a result whose proof follows from \cite{BGZ2}.

\begin{prop}
\label{prop: non-sep non-fibre background}
Suppose that $F$ is a non-separating, essential, twice-punctured torus of slope $\beta$ in a hyperbolic knot manifold $M$ which completes to an essential torus in $M(\beta)$ but is not a fibre in $M$. Suppose as well that $M(\alpha)$ is an irreducible small Seifert manifold.

$(1)$ If $t_1^+ > 0$, then $\Delta(\alpha, \beta) \leq \left\{ \begin{array}{ll} 3 & \hbox{if $M(\alpha)$ is very small} \\ 4 & \hbox{otherwise}   \end{array}  \right.$

$(2)$ If $t_1^+ = 0$, then

\indent \hspace{.5cm} $(a)$ $\Delta(\alpha, \beta) \leq \left\{ \begin{array}{ll} 4 & \hbox{if $M(\alpha)$ is very small} \\ 6 & \hbox{otherwise}   \end{array}  \right.$

\indent \hspace{.5cm} $(b)$ $M(\beta)_{\widehat F}$ admits a Seifert structure with base orbifold an annulus with one cone point.
\end{prop}

\begin{proof} Assertion (2)(b) of the proposition is a consequence of \cite[Lemma 7.10]{BGZ2}. Assertion (2)(a) follows from \cite[Propositions 13.1 and 13.2]{BGZ2}, as does the general inequality $\Delta(\alpha, \beta) \leq 4$ when $t_1^+ > 0$ claimed in assertion (1).

Assume that $t_1^+ > 0$ and $M(\alpha)$ is very small. Since $t_1^+$ is even and the number of boundary components $F$ is bounded below by $\frac{ t_1^+}{2}$, we have $t_1^+ \in \{2, 4\}$.
As $M(\alpha)$ is very small, the graph $\overline{\Gamma}_S$ is contained in a $2$-disk and so it has a vertex of valency $5$ or less (e.g. see \cite[Proposition 12.2 and Corollary 12.4]{BGZ2}). Hence if $t_1^+ = 4$, then \cite[Inequality 13.0.1]{BGZ2} shows that $\Delta(\alpha, \beta) \leq 2$. Suppose then that $t_1^+ = 2$ and note that $\overline{\Gamma}_S$ has a vertex of valency $3$ or less, for if it doesn't, \cite[Lemma 11.6 and Proposition 11.5(2)]{BGZ2} imply that $\overline{\Gamma}_S$ is rectangular and so is contained in a torus, contrary to our assumptions. But then \cite[Inequality 13.0.1]{BGZ2} shows that $\Delta(\alpha, \beta) \leq 3$, so we are done.
\end{proof}

\subsection{Proof of Proposition \ref{prop: F2-non-sep non-fibre} when $M(\alpha)$ is not very small}
\label{non-sep non-fibre not very small}

By Proposition \ref{prop: non-sep non-fibre background} we can suppose that $t_1^+ = 0$ in this subsection and the base orbifold $S^2(a,b,c)$ of $M(\alpha)$ is hyperbolic.

Recall from \S \ref{sec: second reduction} that $S$ is the frontier of $X^+$ in $M$ and consists of two parallel copies $F_1, F_2$ of $F$.
By Proposition \ref{prop: non-sep non-fibre background} we can assume that $\Delta(\alpha, \beta) = 6$ in order to obtain a contradiction. In this case, the proof of \cite[Proposition 13.2]{BGZ2} shows that the reduced graph $\overline{\Gamma}_S$ is rectangular with every edge having  weight $6$, both $\dot{\Phi}_3^+$ and $\dot{\Phi}_5^-$ consist of a pair of tight components, each a twice-punctured disk, $\dot{\Phi}_5^+$ is a collar on $\partial S$, and so contains no large components.

Let  $b_1,...,b_4$ denote the  components of $\partial S=\partial F_1\cup \partial F_2$ indexed as they appear successively along
$\partial M$ and where $b_1\cup b_3 =  \partial F_1$ and $b_2\cup b_4 = \partial F_2$.
These four circles cut $\partial M$ into four annuli $A_{i,i+1}, i=1,...,4$, such that $\partial A_{i,i+1}=b_i\cup b_{i+1}$ (indexed (mod $4$)). We assume that $\partial X^+ =S\cup A_{2,3}\cup A_{4,1}$.

As in \cite{BGZ3}, an {\it $n$-gon} in $X^+$ means a singular disk $D$ with $\partial
D \subseteq \partial X^+$ such that $\partial D \cap (A_{2,3} \cup A_{4,1})$ is a set of $n$ embedded
essential arcs in $A_{2,3} \cup A_{4,1}$, called the {\it corners} of $D$, and $\partial D\cap S$ is a
set of $n$ singular arcs, called the {\it edges} of $D$. As we go around $\partial D$ in some direction we get
a cyclic sequence of $X_2^{\pm1}$ and {\em $X_4^{\pm1}$-corners}, where
$X_2,X_2^{-1}$ indicate that $\partial D$ is running across $A_{2,3}$ from
2 to 3 or from 3 to 2, respectively, and $X_4,X_4^{-1}$ indicate that
$\partial D$ is running across $A_{4,1}$ from 4 to 1 or 1 to 4, respectively.
In this way $D$ determines a cyclic word $W = W(X_2^{\pm 1},X_4^{\pm1})$,
well-defined up to inversion, and we say that $D$ is {\em of type\/} $W$.
(Thus $D$ is of type $W$ if and only if it is of type $W^{-1}$.)
We emphasize that $W$ is an unreduced word; for example $X_2$ and $X_2X_4X_4^{-1}$
are distinct.

There are no $n$-gons in $X^+$ with $n$ odd (cf. \cite[Lemma 11.6]{BGZ3}).

\begin{lemma}\label{Phi5+NoLarge}
There is no bigon $D$ in $X^+$ whose edges $e_1, e_2$ are essential paths in $(\dot\Phi_5^-, \partial
S)$ and for which the inclusion $(D, e_1 \cup e_2) \to (X^+, \dot\Phi_5^-)$ is essential as a map of
pairs.
\end{lemma}

\pf Suppose otherwise that such a bigon $D$ exists. Then $D$ gives rise to an essential homotopy
between its two edges and  thus the edges of $D$ can be homotoped, relative to their end points, into
$\dot\Phi_1^+$. Then the essential intersection $\dot\Phi_5^- \wedge \dot\Phi_1^+$ contains a
large component and therefore so does $\dot\Phi_6^+ = \tau_+(\dot\Phi_5^- \wedge \dot\Phi_1^+)$,
contrary to the fact  that $\dot\Phi_5^+$ has no large components.
\qed

Recall that $h$ is the $\pi_1$-injective map from the torus $T$ into $M(\alpha)$
which induces the graph $\Gamma_S$ in $T$. For a subset $s$ of $T$ we use $s^*$ to denote its image
under the map $h$.
The image under $h$ of every edge of a rectangular face of $\Gamma_S$ is contained in $\dot\Phi_5^-$.

For notational simplicity, let us write $\dot\Phi_5^- = Q$,  a pair
of twice-punctured disks.
Obviously we have $b_1\cup b_3$ is contained in one component of $Q$
and $b_2\cup b_4$ in the other.

A singular disk $D\subset X^+$ whose edges are contained in $Q$ will
be called a {\em $Q$-disk\/}.
An {\em essential\/} $Q$-disk is a $Q$-disk $D$ such that each edge of $D$
is an essential arc in $Q$.

Note that a $Q$-disk has no 12-, or 14-, or 32- or 34-edge.
This simple fact will be used many times in the rest of the proof
in determining the type $W(D)$ of a $Q$-disk $D$.

\begin{lemma} \label{PropA}
An essential $Q$-$n$-gon, $n\le4$, is a $4$-gon of type $X_2X_4^{-1}X_4X_4^{-1}$
or $X_4X_2^{-1}X_2X_2^{-1}$.
\end{lemma}

\pf
Let $E$ be an essential $Q$-$n$-gon, $n\le 4$.
By the loop theorem (\cite[Theorem 4.10]{He}) we get an essential embedded $Q$-disk $D$, with
$\{\text{corners of }D\} \subset \{\text{corners of }E\}$.
Thus $D$ is a $k$-gon with $k\leq 4$.
We know $k$ must be  even,
and $D$ cannot be a bigon by Lemma \ref{Phi5+NoLarge}.
Hence $D$ is a $4$-gon.

There are three possibilities: $D$ has either
\begin{itemize}
\item[(A)] {\em all $X_2$-corners (or all $X_4$-corners)};
\item[(B)] {\em two $X_2$-corners and two $X_4$-corners};
\item[(C)] {\em one $X_2$-corner and three $X_4$-corners (or vice versa)}.
\end{itemize}

\begin{figure}[!ht]
\centerline{\includegraphics{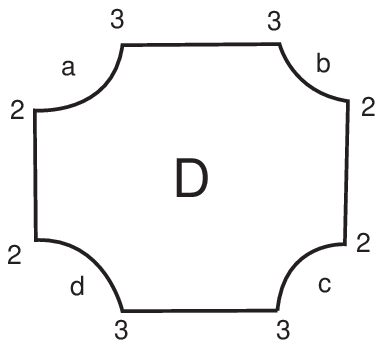}} \caption{ }\label{R1}
\end{figure}

In Case (A),   $W(D)=X_2X_2^{-1} X_2X_2^{-1}$ (or
$X_4X_4^{-1}X_4X_4^{-1}$).
Label the corners of $D$ $a,b,c,d$ as shown in
Figure \ref{R1}. Then $\partial D$ is as shown in Figure \ref{R2}.
Let $Y = \widehat X^-\cup H_{(23)}$.
Note that $\partial Y$  is a surface of genus 2.
We see from Figure \ref{R2} that $\partial D$ is isotopic in $\partial Y$ to a
meridian of $H_{(23)}$, and so bounds a non-separating disk $D'\subset Y$.
Then $D\cup D'$ is a non-separating 2-sphere in $M(\beta)$, a contradiction.

\begin{figure}[!ht]
\centerline{\includegraphics{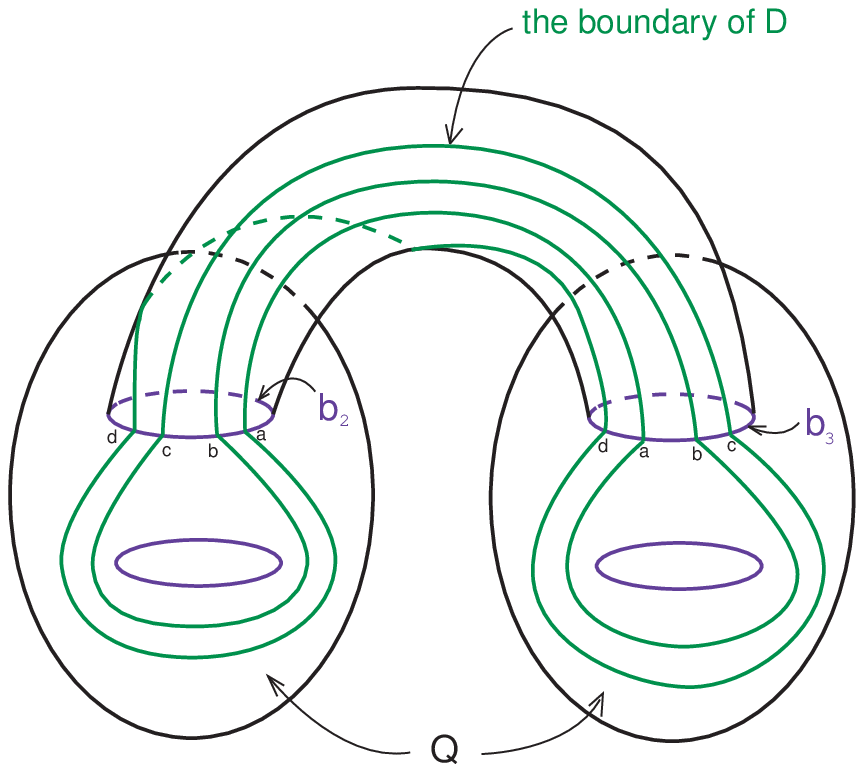}} \caption{ }\label{R2}
\end{figure}

In Case (B), the only possibilities
for $W(D)$ are $X_2X_2^{-1}X_4X_4^{-1}$ and $X_2X_4^{-1}X_2X_4^{-1}$.
In the first case,  $\partial D$ also contains a 1-loop and a 3-loop, which must intersect.
In the second case, let $U = \widehat Q\times I\cup H_{(23)} \cup
H_{(41)} \cup N(D) \subset \widehat X^+$.
Then $U$ is a punctured projective space in $X^+$,  a contradiction.

Hence Case (C) must hold; so suppose
that $D$ has one $X_2$-corner and three $X_4$-corners.
Since $\{\text{corners of }D\} \subset \{\text{corners of }E\}$, $E$
is also a 4-gon with one $X_2$-corner and three $X_4$ corners.
The  possibility for $W(E)$ is  $X_2X_4^{-1} X_4X_4^{-1}$.
This completes the proof of Lemma \ref{PropA} \qed

\begin{lemma}\label{lem6''}
There do not exist disjoint $Q$-disks of types $X_2X_4^{-1}X_4X_4^{-1}$
and $X_4X_2^{-1}X_2X_2^{-1}$.
\end{lemma}

\pf
Let $D_1,D_2$ be $Q$-disks of types $X_2X_4^{-1}X_4X_4^{-1}$ and $X_4X_2^{-1}X_2X_2^{-1}$,
respectively.
Then  $\partial D_1$ contains a 1-loop and $\partial D_2$ contains a 3-loop,
and these must intersect.
\qed

\begin{lemma} \label{PropB}
There cannot be essential $Q$-$4$-gons of both types
$X_2X_4^{-1}X_4X_4^{-1}$ and\break $X_4X_2^{-1}X_2X_2^{-1}$.
\end{lemma}

\pf
Let $E_1,E_2$ be $Q$-disks of types $X_2X_4^{-1}X_4X_4^{-1}$ and $X_4X_2^{-1}X_2X_2^{-1}$
respectively.
By the loop theorem and Lemma \ref{PropA} we get embedded $Q$-disks $D_1$ and
$D_2$ of these types.
By Lemma~\ref{lem6''}, $D_1$ and $D_2$ must intersect; consider an arc of
intersection, coming from the identification of arcs $u_i \subset D_i$,
$i=1,2$.
We may assume that the endpoints of $u_i$ lie on distinct edges of $D_i$,
$i=1,2$.
Then  $u_i$ separates $D_i$ into two disks, $D'_i$ and $D''_i$, say, where
$D'_i$ contains either one or two corners of $D_i$.

If $D'_1$ and $D'_2$ each contain a single corner, and these corners are
distinct, then $D'_1 \cup D'_2$ is a $Q$-bigon with one $X_2$- and one
$X_4$-corner, contradicting Lemma \ref{Phi5+NoLarge}.

If $D'_1$ and $D'_2$ both contain, say, a single $X_2$-corner, then $u_1$ is as
shown in Figure \ref{R3}, which also shows one of the three possibilities for $u_2$.
Since $b_1$ and $b_3$ lie in one component of $Q$, say $Q_1$, and $b_2$
and $b_4$ lie in the other component, say $Q_2$, and each of the arcs $u_1$
and $u_2$ has one endpoint in $Q_1$ and one in $Q_2$, $u_1$ and $u_2$ must
be identified as shown in Figure \ref{R3}.
Then $D_1^* = D''_1 \cup D'_2$ is a $Q$-disk of type $X_2X_4^{-1}X_4X_4^{-1}$
having fewer intersections than $D_1$ with $D_2$.

\begin{figure}[!ht]
\centerline{\includegraphics{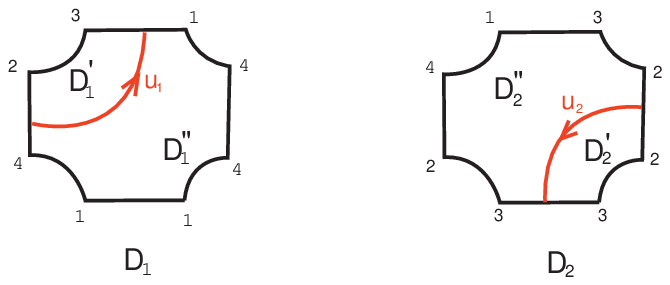}} \caption{ }\label{R3}
\end{figure}

\begin{figure}[!ht]
\centerline{\includegraphics{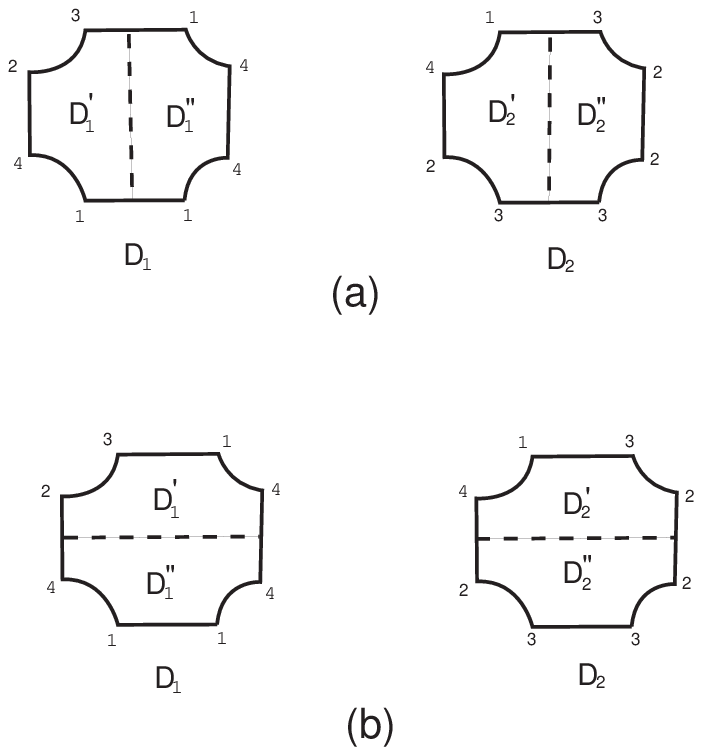}} \caption{ }\label{R4}
\end{figure}

If each of $D'_i$ and $D''_i$ contains two corners, $i=1,2$, the two
possibilities for $u_1$ and $u_2$ are illustrated in Figure \ref{R4}, (a) and (b).
In both cases, $D^*_1 = D''_1 \cup D'_2$ is again a $Q$-disk of type
$X_2X_4^{-1}X_4X_4^{-1}$ having fewer intersections with $D_2$.

Applying the loop theorem to the disk $D^*_1$ constructed above, and
using Lemma \ref{PropA}, we get an embedded $Q$-disk of type $X_2X_4^{-1}X_4X_4^{-1}$
having fewer intersections with $D_2$ than $D_1$.
Continuing, we eventually get disjoint embedded $Q$-disk of types
$X_2X_4^{-1}X_4X_4^{-1}$ and $X_4X_2^{-1}X_2X_2^{-1}$, contradicting Lemma~\ref{lem6''}.
This completes the proof of Lemma \ref{PropB}.
\qed

Since each edge of $\bar\Gamma_S$ has weight 6, consecutive 4-gon corners
of $\Gamma_S$ at a given vertex are distinct.
Hence the total number of $X_2$-corners in the 4-gon faces of $\Gamma_S$ is
the same as the total number of $X_4$-corners.
Since a 4-gon face of $\Gamma_S$ is an essential $Q$-disk, this
contradicts Lemmas \ref{PropA} and \ref{PropB}.

\subsection{Refinements to the very small case when $t_1^+ = 0$}
\label{subsec: refinements nonsep very small}

In this subsection we refine the very small case of Proposition \ref{prop: F2-non-sep non-fibre} when $t_1^+ = 0$.
Assume that this is the case and recall that $M(\beta)_{\widehat F}$ admits a Seifert structure with base orbifold an annulus with one cone point of order $n \geq 2$ (cf. Proposition \ref{prop: non-sep non-fibre background}). Let $T_0$ and $T_1$ be the boundary components of $M(\beta)_{\widehat F}$ and $\phi_0, \phi_1$ the Seifert slope of $M(\beta)_{\widehat F}$ on $T_0, T_1$ respectively. Denote by $f: T_0 \to T_1$ the gluing map which produces $M(\beta)$  and set
$$d = \Delta(f_*(\phi_0), \phi_1)$$

\begin{prop} \label{twice-punctured very small t+=0}
Let $M$ be a hyperbolic knot manifold which contains an essential non-separating twice-punctured torus $F$ of boundary slope $\beta$ and let $\alpha$ be a slope on $\partial M$ such that $M(\alpha)$ is a very small  Seifert manifold. If $t_1^+ =  0$ and $d \ne 1$, then
$$\Delta(\alpha, \beta) \leq \left\{
\begin{array}{ll} 1 & \hbox{if $d = 0$ or $\alpha$ is of $C$-type} \\
2 & \hbox{if $\alpha$ is of $D, T$ or $O$-type, or $(a, b, c)$ is $(2,3,6)$ or $(3,3,3)$} \\
3 & \hbox{if $\alpha$ is of $I$-type or $(a, b, c) = (2,4,4)$} \end{array} \right.$$
\end{prop}

\pf
If $d= 0$, $M(\beta)$ is a Seifert fibred manifold with base orbifold a torus or Klein bottle with one cone point and therefore $\beta$ is a singular slope of a closed essential surface in $M$ (cf. the third paragraph of \S \ref{sec: initial reduction}). Proposition \ref{prop: sing slope exceptional} then shows that $\Delta(\alpha, \beta) \leq 1$.

Assume that $d \geq 2$. By construction,
$$\pi_1(M(\beta)) \cong \langle \pi_1(M(\beta)_{\widehat F}), t : t \gamma t^{-1} = f_*(\gamma) \hbox{ for all } \gamma \in \pi_1(T_0) \rangle$$
Fix bases $\{\phi_j, \phi_j^*\}$ of $\pi_1(T_j)$ ($j = 0,1$). Then the quotient of $\pi_1(M(\beta)_{\widehat F})$ by the fibre class is isomorphic to $\mathbb Z * \mathbb Z/n$ where there are generators $a, b$ of $\mathbb Z, \mathbb Z/n$ respectively so that $\phi_0^*$ is sent to $a$ and $\phi_1^*$ to $ab$.

Write $f_* = \left(\begin{matrix} p & q \\ r & s \end{matrix} \right)$ with respect to the bases $\{\phi_j, \phi_j^*\}$. Since $d = \Delta(f_*(\phi_0), \phi_1)= |r|$, the integers $d$ and $s$ are coprime.

Since $d = \Delta(f_*(\phi_0), \phi_1)$, the quotient of $\pi_1(\widehat F)$ by $\langle \phi_0, \phi_1 \rangle$ is isomorphic to
$\mathbb Z / d$ generated by either $\phi_0^*$ or $\phi_1^*$. Hence if we quotient $\pi_1(M(\beta))$ by the normal
closure of $\langle \phi_0, \phi_1 \rangle$ in $\pi_1(M(\beta))$  we obtain an epimorphism
$$\varphi: \pi_1(M(\beta))  \to  \langle a, b, t : a^d = 1, b^n = 1, (ab)^d = 1, t a t^{-1} = (ab)^s \rangle = \Delta(d,n,d)*_{\psi}$$
where $\psi: \langle a \rangle \xrightarrow{\; \cong \;} \langle ab \rangle, a \mapsto (ab)^s$ (cf. \S \ref{hnn}). If $\beta^* \in \pi_1(\partial M)$ is a dual class to $\beta$, then
$$\varphi (\beta^*) = t a^j t a^l$$
for some integers $j, l$.

Let $Y_0 \subset X_{PSL_2}(\Delta(d,n,d)*_{\psi})$ be a curve constructed in Lemma \ref{lemma: bending homs non-sep}(3) by bending a representation $(\theta, A)$ and define
$$X_0 = (\varphi)^*(Y_0) \subset X_{PSL_2}(M(\beta)) \subset X_{PSL_2}(M)$$
Part (3) of Lemma \ref{lemma: bending homs non-sep} then shows that:
\vspace{-.2cm}
\begin{itemize}

\item $X_0$ is strictly non-trivial;

\vspace{.2cm} \item $\tilde f_{\beta^*}$ has a pole at each ideal point of $X_0$;

\vspace{.2cm} \item if either $n \ne 2$ or $d \not \in\{2,4\}$, then $X_0$ has exactly two ideal points, and therefore $s_{X_0} \geq 2$;

\vspace{.2cm} \item if $n = 2$ and $d \in\{2,4\}$, then $X_0$ has exactly one ideal point, and therefore $s_{X_0} \geq 1$.

\end{itemize}
\vspace{-.2cm}
By (\ref{seminorm distance}) we have
$$\|\alpha\|_{X_{0}} = \Delta(\alpha, \beta) s_{X_{0}}$$
If $\pi_1(M(\alpha))$ is cyclic, then \cite[Proposition 8.1]{BCSZ2}  implies that $\Delta(\alpha, \beta) s_{X_0} = \|\alpha\|_{X_0} = s_{X_0}$ and therefore
$\Delta(\alpha, \beta) = 1$. It also implies that $\Delta(\alpha, \beta) s_{X_0} = \|\alpha\|_{X} \leq 2 s_{X_0} $ if $M(\alpha)$ is a prism manifold since $X_0$ contains strictly irreducible characters. Thus $\Delta(\alpha, \beta) \leq 2$ (cf. \cite[Proposition 8.1]{BCSZ2}).

To deal with the remaining cases, note that by Lemma 5.3 of \cite{BZ1} and Propositions 5.2, 5.3 and 5.4 of \cite{Bo}, if $\rho: \pi_1(M(\alpha)) \to PSL_2(\mathbb C)$ is an irreducible representation, then
$$\mbox{image}(\rho) \cong \left\{
\begin{array}{ll}
T_{12} & \mbox{ if $\alpha$ is of $T$-type} \\
D_3 \mbox{ or } O_{24} & \mbox{ if $\alpha$ is of $O$-type} \\
I_{60} = \Delta(2,3,5) & \mbox{ if $\alpha$ is of $I$-type} \\
D_3 \mbox{ or } T_{12} & \mbox{ if $(a,b,c) = (2,3,6)$} \\
D_2 \mbox{ or } D_4 & \mbox{ if $(a,b,c) = (2,4,4)$}\\
T_{12} & \mbox{ if $(a,b,c) = (3,3,3)$}
\end{array} \right.
$$
Further, in each of these six possibilities at most two such representations have isomorphic images and if two, either $\alpha$ has $I$-type and the image is $I_{60}$ or $(a,b,c) = (2,4,4)$ and the image is $D_4$.
Note, in particular, that the image of $\rho$ is a finite group whose non-trivial elements have order $2, 3, 4$, or $5$.

Assume that $\Delta(\alpha, \beta) > 1$. Then as in the proof of Proposition \ref{prop: very small d not 1} we see that $J_{X_{0}}(\alpha) \ne \emptyset$ and is contained in $X_0^\nu$. Further, the image $\chi_\rho$ of an element of $J_{X_{0}}(\alpha)$ is a smooth point of $X_{PSL_2}(M)$. We know, moreover, that $\rho(\alpha) = \pm I$, and so $\rho$ induces irreducible homomorphism
$$\bar \rho: \pi_1(M(\alpha)) \to PSL(2, \mathbb C)$$
Lemma \ref{lemma: bending homs non-sep}(2) then shows that $n, d \in \{2,3,4,5\}$.

Suppose that either $d \not \in\{2,4\}$ or $d = 4$ and $n > 2$. Then Lemma \ref{lemma: bending homs non-sep}(2) shows that $\theta(\Delta(d, n, d))$ is generated by two elements of order larger than $2$ and so is not conjugate into $\mathcal{N}$. It follows that $X_0$ contains no dihedral characters and by Lemma 5.3 of \cite{BZ1} and Propositions 5.2, 5.3 and 5.4 of \cite{Bo} we have $|J_{X_{0}}(\alpha)| \leq 2$ with equality only if $\alpha$ is of $I$-type. Since $s_{X_0} \geq 2$, Proposition \ref{prop: boundary values}(2) implies that
$$\Delta(\alpha, \beta) \leq 1 + \frac{2|J_{X_{0}}(\alpha)|}{s_{X_0}} \leq 3$$
with equality implying that $\alpha$ is of $I$-type.

Next suppose that $n = 2$ and $d \in\{2,4\}$. That is, $(d, n, d) = (2,2,2)$ or $(4, 2, 4)$. In either case the curve $X_0$ is constructed by bending a representation $(\theta, A)$ where $\theta: \Delta(d,n,d) \to \mathcal{K} \cong \Delta(2,2,2)$ and $A \theta(a) A^{-1} = \theta(ab)$. It is shown in the last paragraph of the proof of Lemma \ref{lemma: bending homs non-sep} that we can suppose that $\theta(a) = \pm \left(\begin{smallmatrix} i & 0 \\ 0 & -i   \end{smallmatrix}\right)$ and $\theta(ab) = \pm \left(\begin{smallmatrix} 0 & 1 \\-1 & 0   \end{smallmatrix}\right)$. Further, $X_0$ is parametrised by the characters of the representations $(\theta, A_x)$ where
$$A_x = \pm \left(\begin{smallmatrix} x & i/2x \\ix & 1/2x   \end{smallmatrix}\right)$$
Suppose that $\mathcal{K}$ is normal in the image of $(\theta, A_x)$. That is, this image is either $D_2, D_4, T_{12}$, or $O_{24}$. Since $A_x \theta(a) A_x^{-1} = \theta(ab)$, $A_x$ must commute with $\theta(b) = \pm \left(\begin{smallmatrix} 0 & i \\ i & 0   \end{smallmatrix}\right)$. The reader will verify that this occurs if and only if $x = \pm \frac{1}{\sqrt{2}}$ or $x = \pm \frac{i}{\sqrt{2}}$ and if it does, the image of $(\theta, A_x)$ is isomorphic to $D_4$. This rules out the possibility that the image of $(\theta, A_x)$ is $D_2$, $T_{12}$, or $O_{24}$ and implies that if the image is conjugate into $\mathcal{N}$, it is $D_4$. It follows
that if $|J_{X_{0}}(\alpha)| \ne \emptyset$, then either $\alpha$ has type $I$ or $(a,b,c) = (2, 4, 4)$. Combined with the previous paragraph, this implies that $\Delta(\alpha, \beta) \leq 2$ when $\alpha$ has $T$-type or $O$-type or $(a,b,c)$ is either $(2,3,6)$ or $(3,3,3)$.

Suppose then that $(a,b,c) = (2,4,4)$, Proposition 5.3 of \cite{Bo} combines with Proposition \ref{prop: boundary values}(2) to show that
$$\Delta(\alpha, \beta) \leq 1 + \frac{|J_{X_{0}}(\alpha)|}{s_{X_0}} \leq 1 + 2 = 3$$
Finally suppose that $\alpha$ has $I$-type. Then $|J_{X_{0}}(\alpha)| \leq 2$ (\cite[Lemma 5.3]{BZ1}) and so as Proposition \ref{prop: non-sep non-fibre background}
shows that
$$\Delta(\alpha, \beta) = 1 + \frac{2|J_{X_{0}}(\alpha)|}{s_{X_0}}$$
is at most $4$,
we have $\Delta(\alpha, \beta) \leq 3$.

The remaining case to  consider is when $d=2$ and $n>2$. In this case we have $s_{X_0}\geq 2$. Further, no representation of $\pi_1(M)$ whose character lies on $X_0$ has image $D_2$ so the reader will verify that $2|J_{X_{0}}(\alpha)|-|\mathcal{N}(\a)|\leq 4$ with equality implying that $\a$ is of $I$-type.
Proposition \ref{prop: boundary values}(2) now gives the distance bounds in Proposition \ref{twice-punctured very small t+=0}.
This completes the proof.
\qed

\section{Algebraic and embedded $n$-gons in $X^\epsilon$}
\label{sec: n-gons}

We prove several results which will be used in the remainder of the paper. First we fix some
notation.

The twice punctured essential torus $F$ separates $M$ into two components $X^\epsilon$, $\e\in\{+,-\}$.
  Let  $b_1, b_2$ denote the components of
$\partial F$. These two circles cut $\partial M$ into two annuli $B^\epsilon$, $\e\in\{+,-\}$,
 such that $\partial X^\epsilon  = F \cup B^\epsilon$.
Let $\widehat b_1$ and $\widehat b_2$ be two disjoint meridian disks of the filling solid torus $V_\b$
bounded by $b_1$ and $b_2$. These two disks cut $V_\b$ into two
$3$-balls $H^\epsilon$, $\e\in\{+,-\}$, such that $\partial H^\e=B^\e\cup \widehat b_1\cup \widehat b_2$.
Recall that $\widehat F = F\cup \widehat b_1\cup \widehat b_2$ is an incompressible torus in $M(\beta)$
and let $\widehat X^\e=X^\e\cup H^\epsilon$. Here $H^\epsilon$ can be considered as an $2$-handle attached to $X^\epsilon$ along $B^\epsilon$.

Fix a sign $\epsilon$.  An {\it $n$-gon} $D$ in $X^\epsilon$ is a singular disk $D$ with $\partial D \subseteq \partial X^\epsilon=F\cup B^\epsilon$ such that
$\partial D \cap B^\epsilon$ is a set of $n$ embedded
essential arcs in $B^\epsilon$, called the {\it corners} of $D$, and $\partial D\cap F$ is a
set of $n$ singular paths, called the {\it edges} of $D$. As we go around $\partial D$ in some direction we get
a cyclic sequence of $Y^{\pm 1}$ where
$Y,Y^{-1}$ indicate that $\partial D$ is running across $B^\epsilon$ from
$b_1$ to $b_2$ or from $b_2$ to $b_1$, respectively.
In this way $D$ determines an unreduced cyclic word $W = W(Y^{\pm 1})$,
well-defined up to inversion. We say that $D$ is an {\it algebraic $m$-gon} if the absolute value of the exponent sum of
$W$ is $m$. We say that $D$ is {\it essential} if the map $(D,  \partial D) \to (X^\e, \p X^\e)$ is essential as a map of pairs. For instance, an algebraic $m$-gon is essential if $m > 0$.

We use the term {\it monogon, bigon} or {\it trigon} for $n$-gon when $n = 1,2$ or $3$ respectively. We call an $n$-gon an {\it algebraic monogon, algebraic bigon} or {\it algebraic trigon} if it is an algebraic $m$-gon for $m = 1, 2$ or $3$.

\begin{lemma}
\label{embedded n-gons}
Suppose that $D$ is an embedded $n$-gon in $X^\epsilon$ for some $n \leq 3$.

$(1)$ $D$ is not a monogon.

$(2)$ If $D$ is  an algebraic $n$-gon, then $t_1^\epsilon = 0$ and $\widehat X^\epsilon$ is Seifert fibred over $D^2$ with two cone points, one of order $n$.

$(3)$  If $n = 3$ and $D$ is an algebraic monogon, then there is an incompressible separating annulus $(A_\epsilon, \partial A_\epsilon) \subset (X^\epsilon, F)$ which is the frontier of a trefoil knot exterior $Q \subset \widehat X^\epsilon$ such that

\indent \hspace{3mm} $(a)$ $Q \cap \widehat F$ is an $\widehat F$-essential annulus whose slope on $\partial Q$ is the meridional slope of $Q$;

\indent \hspace{3mm} $(b)$ if $k_\e$ is a core arc of the handle $H^\e$, then $(k_\e, \partial k_\e)$ can be isotoped in $(Q, \partial Q)$ to lie as \\ \indent \hspace{9mm}  a core arc of an essential vertical annulus of $Q$.
\end{lemma}

\begin{proof}
If $D$ is a monogon, its union with $H^\epsilon$ has a $3$-ball regular neighbourhood across which we could isotope the essential torus $\widehat F$ into $M$, which is impossible. Thus (1) holds, so that $n = 2$ or $3$.

Suppose that the embedded $n$-gon $D$ is also an algebraic $n$-gon. Let $e_1, \ldots , e_n \subset F$ be the edges of $D$. If $\hat b_1 \cup \hat b_2 \cup e_1 \cup \ldots \cup e_n$ is contained in a $2$-disk $D_0 \subset \widehat F$, then a regular neighbourhood of $D_0 \cup D \cup H^\epsilon$ in $M(\beta)$ is a once-punctured lens space of order $n$. But then as $\widehat F$ is essential in $M(\beta)$, the latter would be reducible, contrary to Assumption \ref{assumptions 0}(6). We claim that $\hat b_1 \cup \hat b_2 \cup e_1 \cup \ldots \cup e_n$ is contained in an $\widehat F$-essential annulus $A$ as depicted in Figure \ref{s2or3-gon}. This is obvious when $n=2$. To see it is the case when $n=3$,
we just need to note that each of the edges $e_1, e_2,e_3$ connects $b_1$ and $b_2$ and that if
the endpoints of the corners of $D$ at $b_1$  occur in clockwise (respectively  anticlockwise) order,
 then  their other endpoints at $b_2$
occur in anticlockwise (respectively clockwise) order.

Now define $Q$ to be a regular neighbourhood of $A \cup H^\e\cup D$ in $\widehat X^\epsilon$. The reader will verify that $\pi_1(Q) \cong \langle t, x : tx^n\rangle$ where $t$ generates $\pi_1(A)$ and $x$ corresponds to the core of $H^\epsilon$. Thus $Q$ is a solid
torus in which $A$ has winding number $n$. Let $A_\epsilon$ be the frontier of $Q$ in $X^\epsilon$.
Then $A_\epsilon$ is an annulus of  winding number $n$ in $Q$.

Since $\widehat X^\e - \mbox{int}(Q)$ is contained in $M$ and has boundary a torus containing an $\widehat F$-essential annulus,
it must be also  a solid torus. Furthermore $A_\epsilon$ has winding number larger than one in $\widehat X^\e - \mbox{int}(Q)$ as otherwise
$\widehat X^\epsilon$ would be a solid torus and thus $\widehat F$ would be compressible.
Thus $A_\epsilon$ is an essential annulus in $X^\epsilon$. It follows that $\widehat X^\epsilon$ is Seifert fibred over $D^2$ with two cone points, one of order $n$.

Recall from \S \ref{sec: second reduction} that $t_1^\e$ is an even positive integer bounded above by $2$. If $t_1^\e = 2$, then $\breve{\Phi}_1^\e$ is a collar on $\partial F$, so there are no essential homotopies in $(X^\e, F)$ of large maps to $F$. Let $K_\beta$ be the core of the $\beta$-filling solid torus in $M(\beta)$ and $k_\e = K_\beta \cap H^\e$. Then $k_\e$ is a core arc of $H^\e$ and by construction, the pair $(k_\e, \partial k_\e)$ is isotopic in $(Q, A)$ to a transverse arc of $A_\e$. Thus $k_\e$ is a transverse arc in an essential annulus $A'$ properly embedded in $\widehat X^\e$. Consider a tubular neighbourhood $N'$ of $A'$ containing $H^\e \setminus (\widehat b_1 \cup \widehat b_2)$ in its interior and set $N'_0 = N' \cap M$. Then $N_0' \cap F$ is a disjoint union of two once-punctured tori $A_1^0$ and $A_2^0$ and $(N_0', A_1^0 \sqcup A_2^0) \cong (A_1^0 \times I, A_1^0 \times \partial I)$. But then there is an essential homotopy $(X^\e, F)$ of the large subsurface $A_1^0$ of $F$, a contradiction. Thus $t_1^\e = 0$, which completes the proof of (2).

\begin{figure}[!ht]
\centerline{\includegraphics{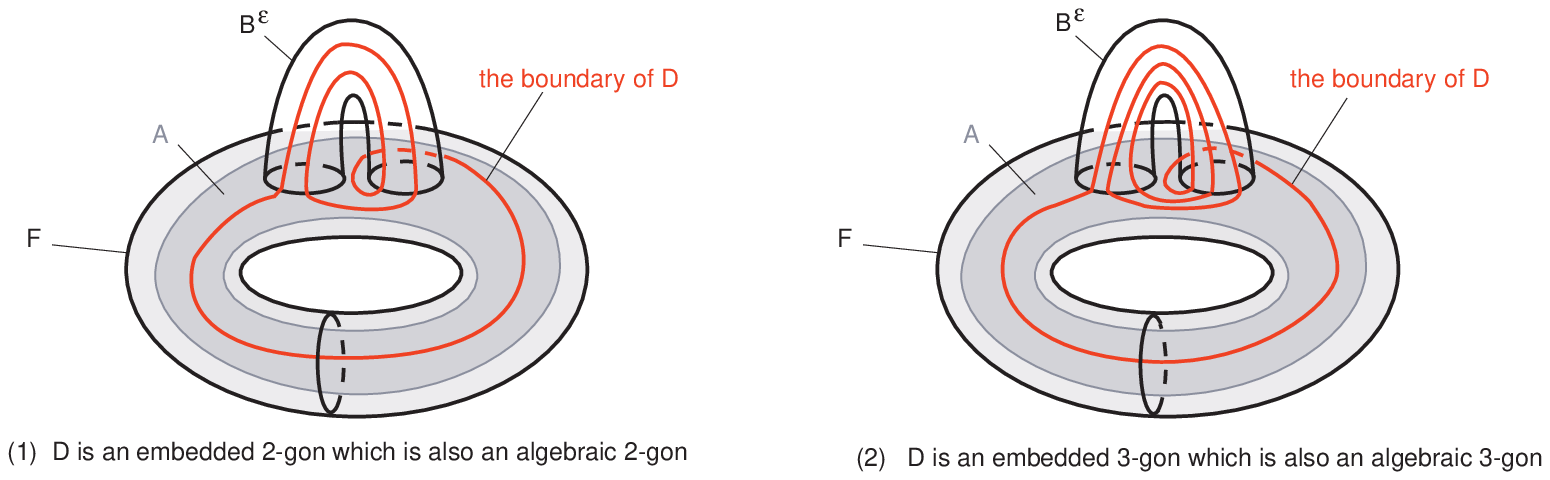}}
\caption{ }\label{s2or3-gon}
\end{figure}

Finally suppose that $n = 3$ and $D$ is an algebraic monogon. Then we may assume that the edges,
denoted $e_1,e_2, e_3$, and the corners, denoted $u, v, w$,  of $D$ are as shown  in Figure \ref{negative trigon}(1),
where an endpoint of an edge (or a corner) of $D$ is labeled $1$ or $2$ if it lies in $b_1$ or $b_2$.
For each of $i=1,2$, the `loop' $\widehat b_i\cup e_i$ cannot be contained in a disk in $\widehat F$, for otherwise
$(D, \partial D)\subset (X^\e, \partial X^\e)$ can be isotoped into an embedded monogon, contradicting part (1) of the lemma.
Hence the two loops $\widehat b_1\cup e_1$ and $\widehat b_2\cup e_2$ are essential and parallel
in $\widehat F$. So the edges of $D$ are contained in the annulus $A$ shown in Figure \ref{negative trigon} (2).

\begin{figure}[!ht]
\centerline{\includegraphics{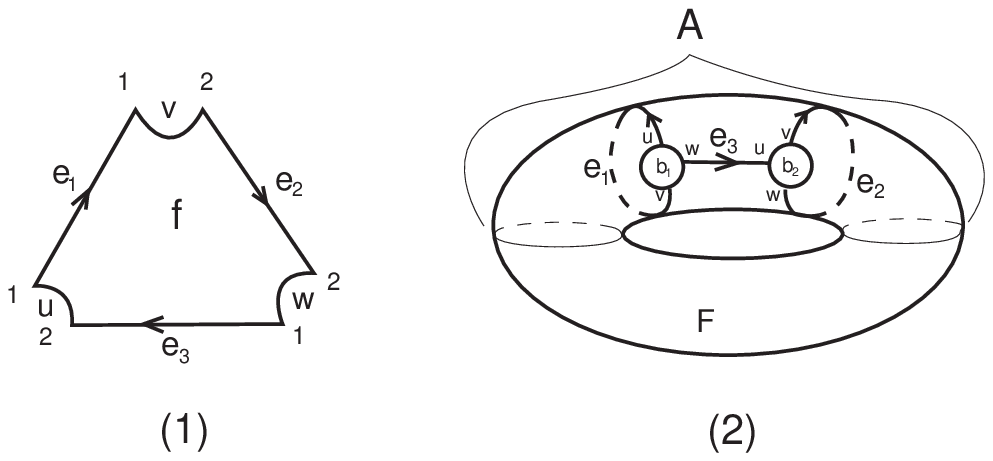}}
\caption{ }\label{negative trigon}
\end{figure}

\begin{figure}[!ht]
\centerline{\includegraphics{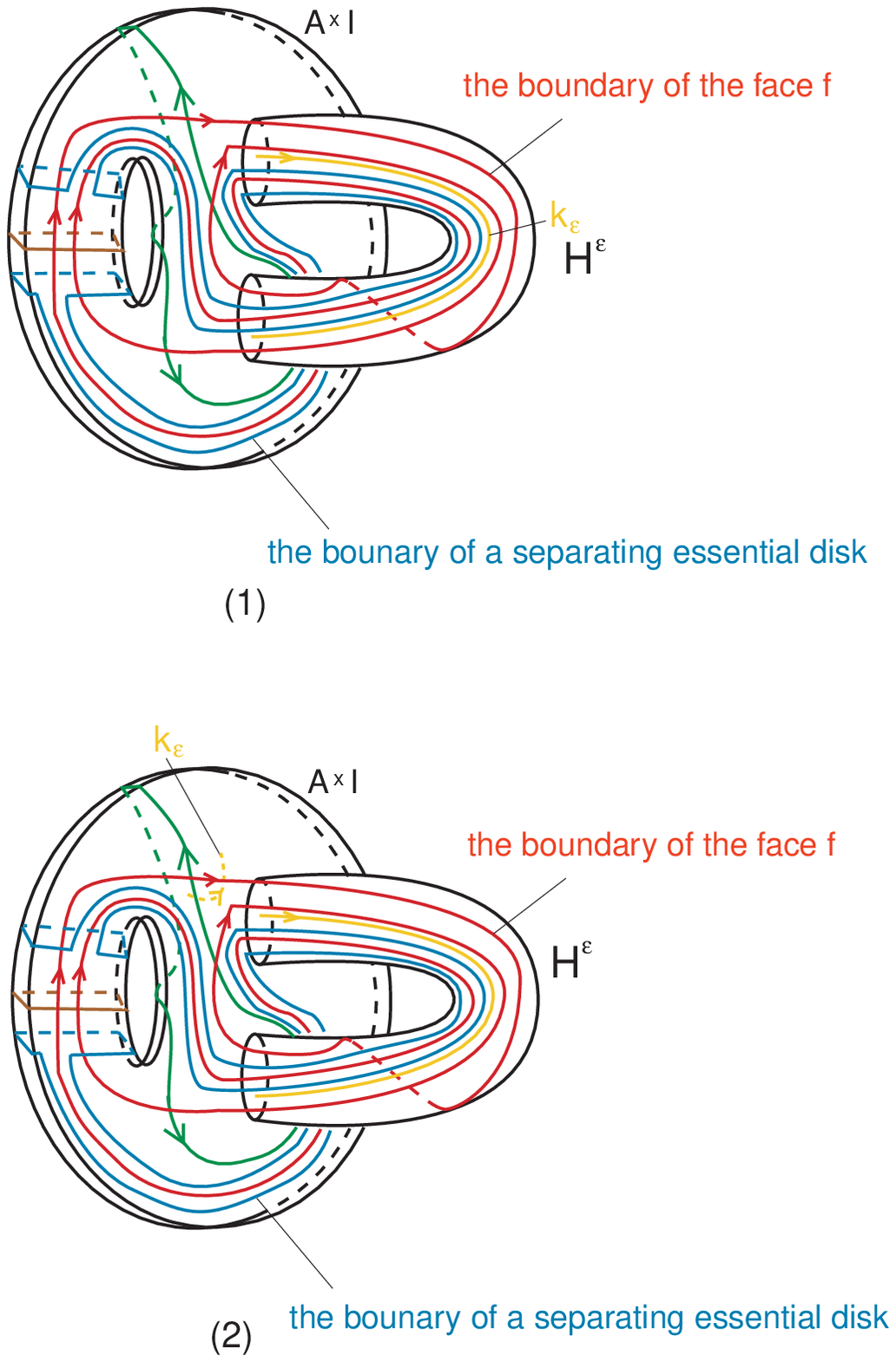}}
\caption{ }\label{trefoil}
\end{figure}

Now define $Q$ to be a regular neighbouhood of $A\cup H^\e\cup D$ in $\widehat X^\epsilon$.
Choose a fat base point in $\widehat F$ which contains $b_1\cup b_2 \cup e_3$. Let $x$ be the element of $\pi_1(\widehat X^\epsilon)$ represented by an essential arc of the annulus
$B^\epsilon$ running from $b_1$ to $b_2$ and let $t$ be the element represented by the edge $e_1$ oriented as  indicated in Figure \ref{negative trigon}(2). Then $Q$ has the following fundamental group:
$$\pi_1(Q) = \langle t, x : txtx^{-2} \rangle$$
Note that $\pi_1(Q)$ injects into $\pi_1(\widehat X^\epsilon)$ since
the frontier of $Q$ in $\widehat X^\epsilon$ is an incompressible annulus
in $(\widehat X^\epsilon, \widehat F)$.
Let $y=tx$, then $\pi_1(Q) = \langle x,y : y^2=x^3 \rangle$ and we see that $Q$ is homeomorphic to the
 trefoil knot exterior with the loop $t$ as a meridian.
 In particular the loop $t$ is distance $1$ from a regular fibre of $Q$ in $\partial Q$.
A more explicit geometric illustration of the situation is given in part (1) of Figure \ref{trefoil}, where $A\times I\cup H^\epsilon$ is a genus two handlebody, the brown loop and the green loop bound two disjoint disks which cut the handle body into a $3$-ball,
the face $f$ is attached to the handlebody along the red loop which intersects the green loop three times with the same sign and intersects the brown curve twice with the same sign.

In Figure \ref{trefoil}, the blue loop is the boundary of an essential separating disk $D_*$ of the handlebody.
The disk $D_*$ separates the handlebody into two solid tori. Let  $V$ be the one
which contains the green loop.
The red loop intersects the blue loop twice (with opposite signs).
Let $e$ be the part of the red curve in $\partial V$ and let $E=e\times I$ be a regular
neighbourhood of $e$ in $\partial V - \partial D_*$. Observe that  $D_*\cup E$ is a vertical essential annulus of $Q$.
Now one can see from Figure \ref{trefoil} that the arc $(k_\e, \partial k_\e)$ can be isotoped in $(Q, \partial Q)$
to lie in a position as shown in part (2) of Figure \ref{trefoil}.
This completes the proof.
\end{proof}

\begin{cor} \label{trigon implies trefoil exterior}
Suppose that $t_1^\epsilon = 0$. If there is a trigon face of $\Gamma_F$ lying in $X^\epsilon$ which is an algebraic monogon, then $\widehat X^\epsilon$ is a trefoil exterior.
\end{cor}

\begin{proof}
Suppose that $f$ is a trigon face of $\Gamma_F$ lying in $X^\epsilon$ which is an algebraic monogon. The hypothesis that $t_1^\epsilon = 0$ implies that $\widehat X^\e$ admits a Seifert structure with base orbifold a disk with two cone points. We must show that the orders of these cone points are $2$ and $3$.

Combining Lemma \ref{embedded n-gons}(1) and the loop theorem with respect to the (mod $2$) intersection with the class of $b_1$ in $H_1(\partial \widehat X^\e)$ (\cite[Theorem 4.10]{He}), we see that there is an embedded trigon in $X^\epsilon$. If it is an algebraic monogon we are done by Lemma \ref{embedded n-gons}(3). If it is an algebraic trigon, then the base orbifold of $\widehat X^\epsilon$ has a cone point of order $3$ by Lemma \ref{embedded n-gons}(2). Now apply the loop theorem to $f$ with respect to the (mod $3$) intersection with the class of $b_1$. The result is either an embedded trigon which is an algebraic monogon, so we are done, or an embedded bigon which is also an algebraic bigon, in which case the base orbifold of $\widehat X^\epsilon$ has a cone point of order $2$ (Lemma \ref{embedded n-gons}(2)), so we are done.
\end{proof}

\section{The proof of Theorem \ref{thm: twice-punctured precise} when $F$ separates but is not a semi-fibre and $t_1^+ + t_1^-  >  0$}
\label{sec: t1+ + t1- > 0}

The goal of this section is to prove Theorem \ref{thm: twice-punctured precise} in the case that $F$ separates, is not a semi-fibre, and $t_1^+ + t_1^-  >  0$. More particularly we show:

\begin{prop}
\label{prop: t1+ + t1- > 0}
Suppose that Assumptions \ref{assumptions 0} and \ref{assumptions 1}  hold, $F$ is separating, and $t_1^+ + t_1^-  >  0$. Then
$\Delta(\alpha, \beta) \leq 4$.
\end{prop}

We begin with a lemma.

\begin{lemma}
\label{constraints when some t1e > 0}
Suppose that $t_1^+ + t_1^- > 0$.

$(1)$ If  $t_1^+ + t_1^- > 2$, then
$\Delta(\alpha, \beta) \leq \left\{ \begin{array}{ll} 2 & \hbox{if $M(\alpha)$ is very small} \\ 3 & \hbox{otherwise}   \end{array}  \right.$

$(2)$ Suppose that $t_1^+ + t_1^- = 2$, say  $t_1^\epsilon = 2$ and $ t_1^{-\epsilon} =  0$.

$(a)$ $X^\epsilon$ admits no essential bigons or embedded trigons which are algebraic trigons.

$(b)$  $\Delta(\alpha, \beta) \leq \left\{ \begin{array}{ll} 5 & \hbox{if $M(\alpha)$ is very small} \\ 6 & \hbox{otherwise} \end{array}  \right.$

$(c)$ If $\Delta(\alpha, \beta) \geq \left\{ \begin{array}{ll} 4 & \hbox{if $M(\alpha)$ is very small} \\ 5 & \hbox{otherwise}   \end{array}  \right.$, then  $\widehat X^\epsilon$ is the union of a trefoil exterior $Q$ and a solid torus $U$ along an annulus $A_\epsilon = \partial Q \cap \partial U$ which is meridional  in $Q$ and of winding number $2$ or more in $U$.

\end{lemma}

\pf Recall from \S \ref{sec: second reduction} that $t_1^+$ and $t_1^-$ are even integers satisfying $0 \leq t_1^+, t_1^- \leq 2$. If $t_1^+ = t_1^- = 2$, then $\breve{\Phi}_1^+ = \breve{\Phi}_1^-$ is a collar on $\partial F$, so there are no essential homotopies in $(M, F)$ of positive length of large maps to $F$. Hence no pair of distinct edges in $\Gamma_F$ are parallel (\cite[Proposition 11.3]{BGZ2}) and therefore the valency of each vertex of $\overline{\Gamma}_F$ equals that of the vertices of $\Gamma_F$, which is $2 \Delta(\alpha, \beta)$. If $M(\alpha)$ is very small, $\overline{\Gamma}_F$ is contained in a disk so that there is a vertex of valency at most $5$. Hence $\Delta(\alpha, \beta) \leq 2$. Otherwise $\overline{\Gamma}_F$ is contained in a torus and so has a vertex of valency at most $6$. Thus $\Delta(\alpha, \beta) \leq 3$, which proves (1).

Assume next that $t_1^+ + t_1^- = 2$. Without loss of generality we can suppose that $t_1^+ = 2$ and $t_1^- = 0$. Since $t^+_1=2$, $\breve{\Phi}_1^+$ is a collar on $\partial F$, so there are no essential homotopies of large maps in $(X^+, F)$. Hence there are no essential bigons in $X^+$. Lemma \ref{embedded n-gons}(2) implies that there are no embedded trigons in $X^+$ which are algebraic trigons, so (2)(a) holds.

Since there are no essential bigons in $X^+$, all bigon faces of $\Gamma_F$ lie in $X^-$. It follows that the weight of each edge of  $\overline{\Gamma}_F$ is no more than $2$. Since the valency of any vertex of $\Gamma_F$ is $2\Delta(\alpha, \beta)$ it follows that $\Delta(\alpha, \beta)$ is bounded above by the valency of any vertex $v$ of $\overline{\Gamma}_F$ with equality if and only if the weight of each edge incident to $v$ is $2$. When the immersion surface $Y$ containing $\overline{\Gamma}_F$ is a torus (cf. (\ref{immersion})), there is a vertex of $\overline{\Gamma}_F$ of valency $6$ or fewer, so $ \Delta(\alpha, \beta) \leq 6$. When it is a disk, for instance when $M(\alpha)$ is very small, there is a vertex of $\overline{\Gamma}_F$ of valency $5$ or fewer, so $\Delta(\alpha, \beta) \leq 5$. Thus (2)(b) holds.

Now  assume that the hypotheses of (2)(c) hold. Then from the previous paragraph we see that for each vertex $v$ of $\overline{\Gamma}_F$,
\begin{equation}
\label{eqn: valency bound}
\mbox{valency}_{\overline{\Gamma}_F}(v) \geq \Delta(\alpha, \beta) \geq \left\{ \begin{array}{ll} 4 & \hbox{if $M(\alpha)$ is very small} \\ 5 & \hbox{otherwise}   \end{array}  \right.
\end{equation}

\begin{claim}
\label{claim: trigon face}
$\Gamma_F$ has a trigon face $f$ which lies in $X^+$.
\end{claim}

\begin{proof}[Proof of Claim \ref{claim: trigon face}]
Let $\varphi_j(v)$ denote the number of corners of $j$-gons incident to a vertex $v$ of $\Gamma_F$ and set
$$\displaystyle \mu(v) = \varphi_2(v)+\frac{\varphi_3(v)}{3}$$
If there is a vertex $v$ of $\Gamma_F$ such that $\mu(v) > 2 \Delta(\alpha, \beta) - 4$, then our hypotheses combine with \cite[Proposition 12.2 (1)]{BGZ2} to show that either $v$ has valency $4$ and $\varphi_3(v) \geq 1$ or $v$ has valency $5$ and $\varphi_3(v) \geq 4$. In either case, our assumed lower bound on $\Delta(\alpha, \beta)$ implies that $\mbox{valency}_{\overline{\Gamma}_F}(v) = \Delta(\alpha, \beta)$. But as we remarked above, this implies that each edge of $\overline{\Gamma}_F$ incident to $v$ has weight $2$. This implies that each of the faces of $\overline{\Gamma}_F$ incident to $v$ lie in $X^+$. In particular, the fact that $\varphi_3(v) > 0$ implies that there is a trigon face of $\Gamma_F$ lying in $X^+$.

Assume next that $\mu(v) \leq 2 \Delta(\alpha, \beta) - 4$ for each $v$. Then Corollary 12.4 of \cite{BGZ2} implies that the immersion surface is a torus and $\mu(v) = 2 \Delta(\alpha, \beta) - 4$ for each $v$. Proposition 12.2 of \cite{BGZ2} then implies that each vertex of $\overline{\Gamma}_F$ has valency $5$ or $6$.

If $\Delta(\alpha, \beta) = 6$ then each vertex has valency $6$, each edge of $\overline{\Gamma}_F$ has weight $2$, and \cite[Proposition 12.2]{BGZ2} implies that each face of $\overline{\Gamma}_F$ is a trigon. Thus $\overline{\Gamma}_F$ is a hexagonal triangulation and any of its faces can serve as $f$.

Finally suppose that $\Delta(\alpha, \beta) = 5$. If there is a vertex $v$ of valency $5$ the edges of $\overline{\Gamma}_F$ incident to $v$ have weight $2$ and $\varphi_3(v) = 3$ (\cite[Proposition 12.2]{BGZ2}), so we are done. If all vertices of $\overline{\Gamma}_F$ have valency $6$, then all its faces are trigons and each vertex is incident to exactly four edges of weight $2$, from which the existence of $f$ follows. This completes the proof of Claim  \ref{claim: trigon face}
\end{proof}

Apply the loop theorem to replace a trigon face $f$ of $\Gamma_F$ lying in $X^+$ by an
embedded monogon, bigon or trigon $(D, \partial D)\subset (X^+, \partial X^+)$, with each edge of $D$ an
essential arc in $(F, \partial F)$. Lemma \ref{embedded n-gons}(1) shows that $D$ is not a monogon while as we noted above, the condition that $t_1^+ = 2$ implies that it is not a bigon. Thus it is a trigon. Lemma \ref{embedded n-gons}(2) shows that it is an algebraic monogon.

According to Lemma \ref{embedded n-gons}(3), there is an incompressible separating annulus $(A_+, \partial A_+) \subset (X^+, F)$ which is the frontier of a trefoil knot exterior $Q \subset \widehat X^+$ such that

\indent \hspace{3mm} $(a)$ $Q \cap \widehat F$ is an $\widehat F$-essential annulus whose slope on $\partial Q$ is the meridional slope of $Q$;

\indent \hspace{3mm} $(b)$ if $k_+$ is a core arc of the handle $H^+$, then $(k_+, \partial k_+)$ can be isotoped in $(Q, \partial Q)$ to lie as \\ \indent \hspace{9mm}  a core arc of an essential vertical annulus of $Q$.

If $A_+$ is inessential in $(X^+, F)$, then $Q$ is isotopic to $\widehat X^+$ in $M(\beta)$. But then, as in the proof of Lemma \ref{embedded n-gons}(2), $t_1^+=0$, which is a contradiction. Thus $A_+$ is essential in $(X^+, F)$. On the other hand, $A_+ \cup  (F \setminus Q \cap F)$ is a separating torus in $X^+ \subset M$ which contains an $\widehat F$-essential annulus. Thus it bounds a solid torus $U \subset X^+$ such that $U \cap Q = A_+$ and $A_+$ has winding number $2$ or more in $U$. It follows that $\widehat X^+$ is the union of a trefoil exterior with a cable space in such a way that  the cable slope is identified with the meridional slope of the trefoil. This completes the proof of Lemma \ref{constraints when some t1e > 0}.
\qed

\begin{lemma}
\label{lemma: homotope face into Q}
Suppose that Assumptions \ref{assumptions 0} and \ref{assumptions 1} hold, $F$ separates, $t_1^+= 2, t_1^-  = 0$, and
$$\Delta(\alpha, \beta) \geq \left\{ \begin{array}{ll} 4 & \hbox{if $M(\alpha)$ is very small} \\ 5 & \hbox{otherwise}   \end{array}  \right.$$
If $f$ is a trigon face of $\Gamma_F$ contained in $X^+$, then there is a homotopy $H: Y \times I \to M(\alpha)$ of $h$ to a new immersion $h'$ satisfying
\vspace{-.2cm}
\begin{itemize}

\item $H$ has support in an arbitrarily small neighbourhood in $Y$ of a compact subset of $f$ disjoint from its corners;

\vspace{.2cm} \item $(h')^{-1}(F) = h^{-1}(F)$;

\vspace{.2cm} \item $h'(f) \subset Q$ where $Q \subset \widehat X^+$ is the trefoil exterior described in Lemma \ref{constraints when some t1e > 0}(2).

\end{itemize}
\end{lemma}

\begin{proof}
Lemmas \ref{embedded n-gons}(3) and \ref{constraints when some t1e > 0}(3) imply that
\vspace{-.2cm}
\begin{itemize}

\item  $\widehat X^+$ is the union of a trefoil exterior $Q$ and a solid torus $U$ along an essential annulus
$$(A_+, \partial A_+) = (\partial Q \cap \partial U, \partial Q \cap \partial U \cap F)$$
which is meridional in $\partial Q$ and of winding number $2$ or more in $U$;

\vspace{.2cm} \item $Q \cap \widehat F$ is an $\widehat F$-essential annulus $A$ containing $\partial F$ and $H^+ \subset Q - A_+$;

\vspace{.2cm} \item $U \cap F$ is an $\widehat F$-essential annulus $E$.

\end{itemize}

First homotope $h$, preserving $h^{-1}(F)$, in a small neighbourhood of $f$ in $Y$ disjoint from its corners so that $h|_{f}$ is transverse to $A_+$.
A loop component of $f \cap h^{-1}(A_+)$ lies in $\hbox{int}(f) \subset \hbox{int}(X^+)$ and so as
$\pi_2(Q, A_+) = \pi_2(U, A_+) = 0$, any such component could be removed
by a homotopy with support in the interior of $f$. In particular, the homotopy is fixed on $h^{-1}(F)$.

Next suppose there is an arc component of $f \cap h^{-1}(A_+)$ whose
end-points lie on an edge of $f$. Let $a$ be such an arc which is
outermost in $f$ and $D_0$ a disk in $f$ whose boundary consists
of $a$ and an arc $b$ in the edge of $f$ containing $\partial a$. Then
$h|_{(D_0, a \cup b)}$ represents an element of $\pi_2(Q, \partial Q)$ or
$\pi_2(U, \partial U)$ which is necessarily zero. To see why the latter holds, note that
the long exact homotopy sequence of the pair $(Q, \partial Q)$
shows that $\pi_2(Q, \partial Q) = 0$ while that of $(U, \partial U)$ yields a short exact sequence
of abelian groups
$$0 \to \pi_2(U, \partial U) \xrightarrow{\; \partial \;} \pi_1(\partial U) \to \pi_1(U) \to 0,$$
from which we see that $\pi_2(U, \partial U) \cong \mathbb Z$ is generated by a meridional disk
of $U$.

Suppose that $h|_{(D_0, a \cup b)}$ represents a non-zero homotopy class
in its group. Then $h(D_0) \subset U$ and the fundamental class of $\partial D_0$
is sent by $h_*$ to a non-zero multiple of the meridional class in $H_1(\partial U)$.
If $h(\partial a)$ is contained in a single component of $\partial A_+$, then $h|_{D_0}$ is
homotopic to a map $(D_0, \partial D_0) \to (U, E)$ and so
$h|_{(D_0, a \cup b)}$ lies in the image of $0 = \pi_2(U, E) \to \pi_2(U, \partial U)$,
a contradiction. Thus the two points of $\partial a = \partial b$ are sents into distinct components of $\partial A_+$ by $h$. It follows that
$h|_a$ is homotopic (rel $\partial a$) to a transverse arc in $A_+$ and
$h|_b$ is homotopic (rel $\partial b = \partial a$) to a transverse arc in the annulus $E$. But then $h|_{\partial D_0}$
is homotopic to a simple closed curve in $\partial U$ which must be meridional. By construction, this curve
has algebraic intersection $\pm 1$ with the core of $A_+$. On the other hand, the absolute value of this intersection is
the winding number of $A_+$ in $U$, which is therefore $1$, a contradiction.

Suppose next that $h|_{(D_0, a \cup b)}$ represents zero in
$\pi_2(W, \partial W)$ where $W$ is $Q$ or $U$ as the case may be.
Then we can remove $a$ from $f \cap h^{-1}(A_+)$ by a homotopy
preserving $h^{-1}(F)$ with support in an arbitrarily small neighbourhood of $D_0$ in $Y$. More precisely, the
long exact sequence in homotopy of the pair $(W, \partial W)$ shows that $h|_{\partial D_0}$ is homotopically
trivial in $\partial W$. Hence its algebraic intersection with either component of $\partial A_+$ is zero. It follows that
$\partial b$ is contained in a single component of $\partial A_+$ and therefore can be homotoped (rel $\partial b$)
into that component. Hence we can homotope $h$, preserving $h^{-1}(F)$ and with support in an arbitrarily small neighbourhood of $b$
in $Y$ so as to replace $a$ by a closed component of $f \cap h^{-1}(A_+)$ contained in the interior of $f$. The latter can be removed as in the
the first paragraph of this proof.

Finally suppose that there is an arc component of $f \cap h^{-1}(A_+)$ whose
end-points lie on distinct edges of $f$. Choose a cornermost such arc
$a$. Then $a$ cobounds a disk $D_0$ in $f$ with arcs $b, c, d$ where
$h(D_0) \subset Q$, $h(b \cup d) \subset A$, $h(c)$ runs once
over the annulus $B^+$, and the arcs are oriented so that
the product $a * b * c * d$ is well-defined and represents a fundamental
class of $\partial D_0$.

Choose the base point for $\pi_1(Q)$ to be $p_0 = h(d \cap a) \in A_+ \cap F \subset A$ and recall the presentation
$\pi_1(Q; p_0) = \langle t, x : txtx^{-2} \rangle$ described above. By construction, we can suppose that
$t \in \pi_1(A; p_0) \leq \pi_1(Q; p_0)$. On the other hand, $x$ may be assumed to be represented by
$h(d)^{-1} * h(c)^{-1} * \gamma$ where $\gamma$ is a path in $A$ from $h(b \cap c)$ to $h(a \cap d)$. Then
the existence of the disk $D_0$ implies that $h(d)^{-1} * h(c)^{-1} * \gamma \simeq h(a) * h(b) * \gamma$ rel $\{0, 1\}$.
But then $x = [h(a) * h(b) * \gamma] \in \pi_1(\partial Q; p_0) \leq \pi_1(Q; p_0)$ which would imply that
$\pi_1(\partial Q; p_0) \to \pi_1(Q; p_0)$ is surjective, a contradiction.
Thus this case does not arise, which completes the proof of the lemma.
\end{proof}

\begin{lemma}
\label{lemma: two trigons pos edge}
Suppose that Assumptions \ref{assumptions 0} and \ref{assumptions 1}  hold, $F$ is separating, $t_1^+= 2, t_1^-  = 0$, and
$$\Delta(\alpha, \beta) \geq \left\{ \begin{array}{ll} 4 & \hbox{if $M(\alpha)$ is very small} \\ 5 & \hbox{otherwise}   \end{array}  \right.$$
Then $\overline{\Gamma}_F$ contains no pair of trigon faces $f_1, f_2$ which share a positive edge $e$ of weight $2$.
\end{lemma}

\begin{proof}
We use the notation from the previous lemma and note that it implies that we can assume $h(f_1) \cup h(f_2) \subset Q$.

Let $e_i$ be the edge of $f_i$ parallel to $e$ and let $R$ be the bigon of $\Gamma_F$ cobounded by $e_1$ and $e_2$.
Since $e$ is a positive edge, $R$ is an algebraic bigon. By construction, $\partial R \subset \partial X^- - E$. Applying the loop theorem with respect to the (mod $3$) intersection with $b_1$ yields an embedded $n$-gon $D$ with boundary contained in $\partial X^- - E$ where $n \leq 2$ such that $D$ is an algebraic $m$-gon for some $m \not \equiv 0$ (mod $3$). Lemma \ref{embedded n-gons}(1) shows that $n = 2$ and so $m = 2$ as well.

There is a disk $D'$ properly embedded in $H^-$ which contains $k_- = K_\beta \cap H^-$ as a properly embedded arc and which intersects $B^-$ in the two arcs $D \cap B^-$. Then $D \cup D'$ is a M\"{o}bius band properly embedded in $\widehat X^-$ which contains $k_-$ as a properly embedded essential arc. It has a solid torus regular neighbourhood, $V$ say, which intersects $\widehat F$ in an essential annulus $E' \subset \widehat F- E$ whose interior contains $\partial F$ and which has winding number $2$ in $V$. Up to isotopy, we can assume that $\partial E' = \partial A_+$.

The annulus $A_- = \partial V \setminus E'$ is properly embedded and separating in $X^-$ and as $t_1^- = 0$, it cobounds a solid torus $V' \subset X^-$
with $E$ in which $A_-$ has winding number two or more. But then $U \cup V$ would be a Seifert manifold over the $2$-disk with two cone points which is contained in the hyperbolic manifold $M$, which is impossible. Thus  $\overline{\Gamma}_F$ contains no pair of trigon faces $f_1, f_2$ which share a positive edge $e$.
\end{proof}

\begin{lemma}
\label{lemma: two trigons negative edge}
Suppose that Assumptions \ref{assumptions 0} and \ref{assumptions 1}  hold, $F$ is separating, $t_1^+= 2, t_1^-  = 0$, and
$$\Delta(\alpha, \beta) \geq \left\{ \begin{array}{ll} 4 & \hbox{if $M(\alpha)$ is very small} \\ 5 & \hbox{otherwise}   \end{array}  \right.$$
Then $\Gamma_{\overline{F}}$ contains no pair of trigon faces $f_1, f_2$ which share a negative edge $e$ of weight $2$.
\end{lemma}

\begin{proof}
We use the notation of Lemma \ref{lemma: homotope face into Q} and note that it implies that we can assume $h(f_1) \cup h(f_2) \subset Q$.

Let $e_i$ be the edge of $f_i$ parallel to $e$ and let $R$ be the bigon of $\Gamma_F$ cobounded by $e_1$ and $e_2$. Let
$\sigma_i$ be the essential loop $h|{(e_i, \partial e_i)}: (e_i, \partial e_i) \to (F, \partial F)$. Without loss of generality we can assume that $b_i$ is the origin of $\sigma_i$. We can also assume that the images of $\sigma_1, \sigma_2$ are contained in $\dot{\Phi}_1^-$.

If $\sigma_i$ is $\widehat F$-inessential, then the relation associated to $f_i$ (cf. \S 15 of \cite{BGZ2}) implies that the class of $x \in \pi_1(Q) \leq \pi_1(\widehat X^+)$ (cf. the proof of Lemma \ref{embedded n-gons}(3)) is peripheral (cf. \cite[Corollary 15.2]{BGZ2}), which is false. Thus $\sigma_1$ and $\sigma_2$ represent $\widehat F$-essential loops and therefore $\partial R$ is an essential loop in $\partial X^- - E$. The loop theorem produces an embedded bigon $(D, \partial D) \subset (X^-, \partial X^-)$. The proof of Lemma \ref{lemma: two trigons pos edge} shows that $D$ is not an algebraic bigon, so intersects $F$ in disjoint loops, one based at $b_1$ and one at $b_2$. Since $F$ is incompressible, each of the loops is $\widehat F$-essential and thus is isotopic in $\widehat F$ into $\partial E$.

The union of $D$ and an appropriate choice of a product region in $B^-$ between the two arcs of $D \cap B^-$
yields an embedded annulus $A^-$ in $X^-$ whose boundary is contained in $A$ and
is isotopic in $A$ to $\partial A = \partial E$. After an isotopy in $X^-$, we can suppose that $\partial A_- = \partial E$. Since
$(k_-, \partial k_-)$ is isotopic in $(\widehat X^-, \widehat F)$ into the chosen product region in $B^-$, $A_-$ is an essential annulus
in $X^-$.

If $A_-$ is non-separating in $\widehat X^-$, then $\widehat X^-$ is a twisted $I$-bundle over the Klein bottle and
$A_- \cup A_+$ is an embedded non-separating Klein bottle or torus in the hyperbolic manifold $M$,
which is impossible. But if $A_-$ is separating in $\widehat X^-$, it splits the latter into two solid tori in each of which
$A_-$ has winding number $2$. One of these solid tori, $V$ say, is contained in $X^-$ and has boundary $A_- \cup E$.
Then $V \cup U$  is a Seifert manifold over a $2$-disk with two cone points contained in $M$, which is impossible.
\end{proof}

\begin{proof}[Proof of Proposition \ref{prop: t1+ + t1- > 0}]
Suppose otherwise that $\Delta(\alpha, \beta) \geq 5$ and, without loss of generality, that $t_1^+ = 2, t_1^- =  0$.
We saw in the proof of Lemma \ref{constraints when some t1e > 0}(2) that the condition $t_1^+ = 2$ implies that each edge of $\overline{\Gamma}_F$ has weight at most $2$ and each vertex has valency at least $5$. Further, if $v$ is a vertex of $\overline{\Gamma}_F$, then $\mbox{valency}_{\overline{\Gamma}_F}(v) \geq \Delta(\alpha, \beta)$ with equality if and only if each edge incident to $v$ has weight $2$.

Corollary 12.4 of \cite{BGZ2} shows that there is a vertex $v$ of $\Gamma_F$ for which $\mu(v) \geq 2 \Delta(\alpha, \beta) - 4$. Proposition 12.2 of that paper then shows that $5 \leq \mbox{valency}_{\overline{\Gamma}_F}(v) \leq 6$. Hence if $k$ is the number of weight $2$ edges incident to $v$, then
$$10 \leq 2 \Delta(\alpha, \beta) = \mbox{valency}_{\Gamma_F}(v) = 2k + (\mbox{valency}_{\overline{\Gamma}_F}(v)  - k) \leq 6 + k$$
Hence $k \geq 4$.

If $\mu(v) > 2 \Delta(\alpha, \beta) - 4$, Proposition 12.2 of \cite{BGZ2} shows that $v$ has valency $5$ with at least four trigon faces incident to it, two of which share a weight $2$ edge as $k \geq 4$, contrary to Lemmas \ref{lemma: two trigons pos edge} and \ref{lemma: two trigons negative edge}.

On the other hand if $\mu(v)= 2 \Delta(\alpha, \beta) - 4$, then Proposition 12.2 of \cite{BGZ2} shows that $v$ has valency $5$ with at least four trigon faces incident to it or $6$ with six trigon faces incident to it. In either case the fact that $k \geq 4$ implies that there are a pair of trigons incident to $v$ which share a weight $2$ edge, contrary to Lemmas \ref{lemma: two trigons pos edge} and \ref{lemma: two trigons negative edge}. This final contradiction shows that $\Delta(\alpha, \beta) \leq 4$.
\end{proof}

\section{Background for the proof of Theorem \ref{thm: twice-punctured precise} when $F$ separates and $t_1^+ = t_1^-=0$}
\label{sec: background sep not semifibre}

Throughout the the remainder of the paper we assume that $F$ separates, is not a semi-fibre, and $t_1^+ = t_1^-=0$.

Recall that when $t_1^+ = t_1^-=0$, $\widehat X^+$ and $\widehat X^-$ are Seifert fibred over the disk with two cone points \cite[Proposition 7.4]{BGZ2}. Further, if $\widehat X^\epsilon$ has base orbifold $D^2(p_\epsilon,q_\epsilon)$, $2 \leq p_\epsilon \leq q_\epsilon$,  then
\begin{itemize}
\vspace{-.1cm} \item $2 < p_\epsilon \leq  q_\epsilon$ if and only if $\widehat{\dot \Phi}_1^\epsilon$ is the union of two once-punctured $\widehat F$-essential annuli which are vertical in $\widehat X^\epsilon$.

\vspace{.3cm} \item $2 = p_\epsilon < q_\epsilon$ if and only if $\widehat{\dot \Phi}_1^\epsilon$ is a twice-punctured $\widehat F$-essential annulus which is vertical in $\widehat X^\epsilon$.

\vspace{.3cm} \item $p_\epsilon = q_\epsilon = 2$ if and only if $X^\epsilon$ is a twisted $I$ bundle.

\end{itemize}
Without loss of generality, we assume that $(p_-, q_-) \leq (p_+, q_+)$ lexicographically. That is,
\begin{itemize}
\vspace{-.1cm} \item $p_- \leq p_+$ and if $p_- = p_+$, then $q_- \leq q_+$.
\end{itemize}
Since $F$ is not a semi-fibre, $(p_+, q_+) \ne (2,2)$. An analysis of the possible essential tori in $M(\beta)$ then shows that it is not the union of two twisted $I$-bundles over the Klein bottle (cf. \cite[Corollary 7.6]{BGZ2}). Further, under the added assumption that $X^-$ is not a twisted $I$-bundle, $(p_-, q_-) \ne (2,2)$, if $M(\beta)$ contains a Klein bottle then $M(\beta)$ is Seifert with base orbifold $S^2(p_+,q_+,p_-,q_-)$ where $p_+ = p_- = 2$.

\subsection{The case that $M(\beta)$ is Seifert}
Let $\phi_\epsilon$ denote the slope on $\widehat F$ of a Seifert fibre  of $\widehat X^\epsilon$
corresponding to the Seifert structure whose base orbifold is $D^2(p_\e,q_\e)$ and set
$$d = \Delta(\phi_+, \phi_-).$$
When $\widehat X^-$ is a twisted $I$-bundle over the Klein bottle we let $\phi_-'$ denote the slope  on $\widehat F$ of
a Seifert fibre of $\widehat X^-$
corresponding to the Seifert structure whose base orbifold
is a M\"{o}bius band. It is well-known that
$$\D(\phi_-,\phi_-')=1.$$
Set
$$d' = \Delta(\phi_+, \phi_-')$$
and orient $\phi_+, \phi_-$, and $\phi_-'$ so that
$$d = \phi_+ \cdot \phi_- \; \hbox{ and } d' = \phi_+ \cdot \phi_-'.$$
If $d = 0$ or $X^-$ is a twisted $I$-bundle and $d' = 0$, then there are Seifert fibre structures on $\widehat X^+$ and $\widehat X^-$ which piece together to give one on $M(\beta)$.

Conversely, if $M(\beta)$ is a Seifert fibre space, the separating essential torus $\widehat F$ cannot be horizontal as this would imply that $M(\beta)$ has non-orientable Euclidean base orbifold
$P^2(2,2)$ or the Klein bottle $K$. In either case, $M(\beta)$ would be a union of two twisted $I$-bundles over the Klein bottle, contrary to what we noted above. Thus $\widehat F$ is vertical and this implies that there are Seifert structures on $\widehat X^+$ and $\widehat X^-$ which coincide on $\widehat F$. It follows that either $\phi_+ = \phi_-$ (i.e. $d = 0$) or $\widehat X^-$ is a twisted $I$-bundle and $\phi_+ = \phi_-'$ (i.e. $d' = 0$).

\begin{prop}
\label{prop: d=0}
Suppose that Assumptions \ref{assumptions 0} and \ref{assumptions 1} hold, $F$ is separating but not a semi-fibre, and $t_1^+ = t_1^-  = 0$.
If $d = 0$, then $\Delta(\alpha, \beta) \leq 1$.
In particular, Theorem \ref{thm: twice-punctured precise} holds.
\end{prop}

\begin{proof}
When $d = 0$, $M(\beta)$ is   Seifert fibreed with base orbifold $S^2(p_+,q_+,p_-,q_-)$. Since  $(p_+, q_+) \ne (2,2)$, $S^2(p_+,q_+,p_-,q_-)$ is a hyperbolic $2$-orbifold and so   \cite[Theorem 1.7]{BGZ1} implies that $\beta$ is a singular slope of a closed essential surface in $M$. Consequently, $\Delta(\alpha, \beta) \leq 1$ by Proposition \ref{prop: sing slope exceptional}.
\end{proof}

\begin{prop}
\label{prop: d'=0}
Suppose that Assumptions \ref{assumptions 0} and \ref{assumptions 1} hold, $F$ is separating but not a semi-fibre, and $t_1^+ = t_1^-  = 0$, $p_- = q_- = 2$ $($i.e. $X^-$ is a twisted $I$-bundle$)$, $d' = 0$, and $\Delta(\alpha, \beta)$ is even. Then $\Delta(\alpha, \beta) \leq 2$. In particular, Theorem \ref{thm: twice-punctured precise} holds.
\end{prop}

\begin{proof}
By assumption, $X^-$ is a twisted $I$-bundle over a once-punctured Klein bottle and hence there is a $2$-fold cover of $\widetilde M \to M$ which restricts to the cover $F \times I \to X^-$ on the $-$-side of $F$ and the trivial double cover on the $+$-side of $F$. Since $m \equiv 2$ (mod $4$) the boundary of $\widetilde M$ is connected.

Now $\beta$ lifts to a slope $\beta'$ on $\partial \widetilde M$ whose associated filling admits a $2$-fold cover $\widetilde M(\beta') \to M(\beta)$. Since $d' = 0$, $M(\beta)$ admits a Seifert fibre structure with base orbifold $P^2(p_+, q_+)$ and the reader will verify that $\widetilde M(\beta')$ admits a Seifert fibre structure with base orbifold $S^2(p_+, q_+, p_+, q_+) \ne S^2(2,2,2,2)$. Hence $\beta'$ is a singular slope of some closed essential surface $S \subseteq \widetilde M$ (\cite[Theorem 1.7]{BGZ1}).

Since the distance of $\alpha$ to $\beta$ is even, $\alpha$ also lifts to a slope $\alpha'$ on $\partial \widetilde M$ with the associated filling a $2$-fold cover of $M(\alpha)$. Hence $\widetilde M(\alpha')$ admits a Seifert fibre structure with base orbifold a $2$-sphere with three or four cone points.

Suppose that $\Delta(\alpha, \beta) \geq 4$. It's easy to see that the distance between $\alpha'$ and $\beta'$ is $\Delta(\alpha, \beta)/2 \geq 2$. Hence as $\beta'$ is a singular slope for $S$, $S$ is incompressible in $\widetilde M(\alpha')$. As $\widetilde M$ is hyperbolic, $S$ cannot be a torus and therefore must be horizontal in $\widetilde M(\alpha')$. It cannot be separating as the base orbifold of $\widetilde M(\alpha')$ is orientable. Thus it is non-separating. But then Proposition \ref{prop: sing slope exceptional} implies the distance between $\alpha'$ and $\beta'$ is at most $1$, and therefore $\Delta(\alpha, \beta) = 2\Delta(\alpha', \beta') \leq 2$, a contradiction. Thus $\Delta(\alpha, \beta) \leq 3$, and as this distance is even, $\Delta(\alpha, \beta) \leq 2$.
\end{proof}

\subsection{Edge weights}

\begin{lemma}
\label{lemma: edge wts}
Suppose that $t_1^+ = t_1^- = 0$ and $\Delta(\alpha, \beta) > 3$.

$(1)$ If $X^-$ is not a twisted $I$-bundle, then,

$(a)$ the edges of $\overline{\Gamma}_F$ have weight at most $3$ and a negative edge of weight $3$ determines an $\widehat F$-inessential loop;

$(b)$ if there is an edge of weight $3$ then $p_+ = p_- = 2$;

$(c)$ a negative edge of $\overline{\Gamma}_F$ incident to a trigon has weight at most $2$, and if $2$, then $d = 1$.

$(2)$ If  $X^-$ is a twisted $I$-bundle, then the faces of $\overline{\Gamma}_F$ which lie in $X^-$ can be assumed to be bigons and its edges can be assumed to have even weight. Further,

$(a)$ the edges of $\overline{\Gamma}_F$ have weight at most $4$;

$(b)$ a negative edge of $\overline{\Gamma}_F$ incident to a trigon has weight at most $2$.

\end{lemma}

\begin{proof}
(1)(a) Proposition \ref{prop: d=0} shows that $d \geq 1$, so when $X^-$ is not a twisted $I$-bundle, $\dot{\Phi}_2^+ \cong \dot{\Phi}_2^- \cong \dot{\Phi}_1^+ \wedge \dot{\Phi}_1^-$ consists of tight components (\cite[Lemma 9.2]{BGZ2}). Hence $\dot{\Phi}_3^+ = \dot{\Phi}_3^- = \emptyset$ (\cite[Proposition 9.3]{BGZ2}), so the weight of each edge is at most $3$. Further, any negative edge of weight $3$ represents loop contained in a tight component of $\dot{\Phi}_2^\epsilon$ for some $\epsilon$, so is $\widehat F$-inessential. This proves (1)(a).

For (1)(b), suppose that $p_\epsilon > 2$. Then $\dot{\Phi}_1^\epsilon$ consists of two once-punctured annuli, so  each edge of $\Gamma_F$ which lies in $\dot{\Phi}_1^\epsilon$ is negative and $\widehat F$-essential. Thus if an edge $\bar e$ of $\overline{\Gamma}_F$ has weight $3$, it is negative and its lead edge is $\widehat F$-essential, contrary to (1)(a). Thus $p_+ = p_- = 2$.

(1)(c) Let $\bar e$ be a negative edge of $\overline{\Gamma}_F$ incident to a trigon $f$ lying in $X^\e$ and denote by $e_1$ the edge of $f$ parallel to $\bar e$. Then \cite[Corollary 15.2]{BGZ2} shows that the loop in $\widehat F$ corresponding to $e_1$ is $\widehat F$-essential. If the weight of $\bar e$ is $3$, $\tau_{-\e}(h(e_1))$ is an $\widehat F$-essential loop contained in $\dot{\Phi}_1^+$ and in $\dot{\Phi}_1^-$. But then $\phi_+ = \phi_-$, contrary to the assumption that $ d \ne 0$. Hence the weight of $\bar e$ can be at most $2$, which proves the first part of (1)(c).

For the second part of (1)(c), suppose that $\bar e$ has weight $2$ and recall from the proof of Lemma \ref{embedded n-gons}(3) that $\pi_1(\widehat X^\e)$ is generated by $\pi_1(\widehat F)$ and a class $x$ represented by a loop which is the concatenation of an essential arc in $B^\e$ and a path in $F$. Let $\g$ denote the class of the loop $h(e_1)$ in $\pi_1(\widehat X^\e)$. Since $X^{-\e}$ is not a twisted $I$-bundle and the weight of $\bar e$ is $2$, $\g = \phi_{-\e}^k$ for some integer $k$. The relation associated to $f$ (cf. \S 15 of \cite{BGZ2}) shows that the image $x$ in $\pi_1(\widehat X^\e)/\langle \langle \g \rangle \rangle = \pi_1(\widehat X^\e)/\langle \langle \phi_{-\e}^k \rangle \rangle$ is contained in the image of $\pi_1(\widehat F)$, so $\pi_1(\widehat X^\e)/\langle \langle \phi_{-\e}^k \rangle \rangle$ is abelian. On the other hand, $\pi_1(\widehat X^\e)/\langle \langle \phi_{-\e}^k \rangle \rangle$ admits the presentation $\langle u, v : u^{p_\e}, v^{q_\e}, (uv)^{kd} \rangle$ which is abelian only if $kd = \pm 1$. Thus $k = \pm 1$ and $d = 1$.

(2) If $X^-$ is a twisted $I$-bundle, its base is a once-punctured Klein bottle $K_0$. We use $K_0$ to construct an intersection graph $\Gamma_{K_0}$ as in \S 2 of \cite{BCSZ1}. Though surfaces were assumed orientable in \cite{BCSZ1}, the construction goes through in our context to yield a graph $\Gamma_{K_0}$ such that the image of each of its edges in $K_0$ is an essential path in $(K_0, \partial K_0)$. There is an associated intersection graph $\Gamma_F$ obtained by doubling each edge of $\Gamma_{K_0}$. In this case, each face of $\Gamma_F$ contained in $X^-$ is a bigon and each edge of $\overline{\Gamma}_F$ has even weight.

(2)(a) If the weight of an edge $\bar e$ is larger than $4$, then
\cite[Corollary 11.4 and Proposition 9.4(3)]{BGZ2} show that the weight of $\bar e$ is $6$, $\dot{\Phi}_1^+$ is an $\widehat F$-essential twice-punctured annulus, $\dot{\Phi}_3^+$ is the union of two $\widehat F$-essential once-punctured annuli, and $\phi_+ = \phi_-'$ so that $M(\beta)$ is Seifert fibred with base orbifold $P^2(2, n)$ for some $n \geq 3$.
Let $e_1,...,e_6$ be the parallel successive edges represented by $\bar e$. Since $e_2$ lies in $\dot{\Phi}_4^+ = \dot{\Phi}_3^+$, $\bar e$ is a negative edge representing an $\widehat F$-essential loop.

Let $R$ be the bigon face between $e_3$ and $e_4$, which lie in $\dot{\Phi}_1^+$,
and apply the loop theorem  to the singular disk
$h|_R:(R, \partial R)\ra (X^-, \partial X^-)$ to obtain an embedded disk $(D, \partial D)\subset (X^-,\partial X^-)$, which must be a bigon
 whose two edges $d_1$ and $d_2$ are contained in  $\dot{\Phi}_1^+$ and two corners of $D$ are essential arcs in the annulus $B^-$.

If $D$ is an algebraic bigon, then the union of $D$ and a product region in $B^-$ between the two corners of $D$ yields an
embedded M\"{o}bius band in $X^-$ whose boundary has slope  $\phi_-$ (since the boundary of any embedded M\"{o}bius band in $\widehat X^-$
 has slope $\phi_-$).
But the boundary of this M\"{o}bius band is contained in $\dot{\Phi}_1^+$ and thus has slope $\phi_+=\phi_-'$. But then $\phi_- =  \phi_-'$,
which is false.

So $D$ is an algebraic $0$-gon. In this case, the union of $D$ and an appropriate choice of a product
region in $B^-$ between the two corners of $D$ yields an embedded essential annulus $A^-$ in $X^-$
whose boundary is contained in $\dot{\Phi}_1^+$ and is isotopic to the boundary of the annulus $\widehat{\dot{\Phi}}_1^+$.
 Since the boundary slope of $A^-$ in $\widehat F$ is $\phi_+=\phi_-'$, $A^-$ is a non-separating annulus in $\widehat X^-$.
 On the other hand we know that the boundary of
 $\widehat{\dot{\Phi}}_1^+$ bounds a separating essential annulus $A^+$ in $X^+$.
  So up to isotopy the two annuli  $A^-$ and $A^+$ can be pieced  together to form a Klein bottle in $M$,
  which contradicts the fact that $M$ is hyperbolic.

(2)(b) follows as in the proof of \cite[Lemma 19.6]{BGZ2}.
\end{proof}

\begin{defn}
\label{def: maxl weight}
{\rm We say that an edge of $\overline{\Gamma}_F$ has {\it maximal weight} if its weight is $3$ when $X^-$ is not a twisted $I$-bundle and $4$ when $X^-$ is a twisted $I$-bundle.}
\end{defn}

\subsection{The trace of an essential annulus in a face of $\Gamma_F$}
\label{subsec: trace face}

In this subsection we assume that $X^\e$ is not a twisted $I$-bundle. In particular,
$$q_\epsilon \geq 3.$$

Up to isotopy, there are two choices for an essential annulus $A$ properly embedded in $(X^\epsilon, F)$ which cobounds a solid torus in $X^\epsilon$ with an $\widehat F$-essential annulus in $F$, one for which $A$ has winding number $q_\epsilon \geq 3$ in the solid torus and one for which it has winding number $p_\epsilon \geq 2$. Throughout the rest of the paper we use
\vspace{-.2cm}
\begin{itemize}

\item $A_\epsilon$ to denote the essential annulus of winding number $q_\epsilon$, $U$ to denote the associated solid torus, $E = U \cap F$ to denote the associated $\widehat F$-essential annulus, and $c_1, c_2$ to denote the boundary components of $A_\epsilon$;

\vspace{.2cm} \item $A_\epsilon'$ to denote the essential annulus of winding number $p_\epsilon$, $V$ to denote the associated solid torus, $G = V \cap F$ to denote the associated $\widehat F$-essential annulus, and $c_1', c_2'$ to denote the boundary components of $A_\epsilon'$.

\end{itemize}
We can write $X^\epsilon = U \cup P \cup V$ where $P$ is the product of a once-punctured annulus with an interval and $P$ has frontier $A_\epsilon \cup A_\epsilon'$ in $X^\epsilon$. Up to isotopy we have
$$(\dot{\Sigma}_1^\epsilon, \dot{\Phi}_1^\epsilon) = \left\{
\begin{array}{cl}
(P, P \cap F) & \mbox{ if } p_\epsilon > 2 \\
(P \cup V, (P \cup V) \cap F) & \mbox{ if } p_\epsilon = 2.
\end{array} \right.$$
Note that
$$(\widehat P, A_\epsilon, A_\epsilon') \cong (A_\epsilon \times I, A_\epsilon \times \{0\}, A_\epsilon \times \{1\})$$
in such a way that $P$ is the exterior of $a_0 \times \{\frac12\}$ where $a_0$ is a transverse arc of $A_\epsilon$.

Index the boundary components of $A_\epsilon$ and $A_\epsilon'$ so that $b_i, c_i$, and $c_i'$ lie in the same component of $P \cap F$.

It is clear that $W = V \cup P =  X^\epsilon \setminus U$ is a genus $2$ handlebody and $\widehat W = W \cup H^\epsilon$ is a solid torus in which $A_\epsilon \subset \partial \widehat W$ has winding number $p_\epsilon$.
Then
$$(\widehat{\dot{\Sigma}_1^\epsilon}, \widehat{\dot{\Phi}_1^\epsilon}) = \left\{
\begin{array}{cl}
(\widehat P, \widehat P \cap \widehat F) & \mbox{ if } p_\epsilon > 2 \\
& \\
(\widehat W, \widehat W \cap \widehat F) & \mbox{ if } p_\epsilon = 2.
\end{array} \right.$$
In either case, $(\widehat P, \widehat P \cap \widehat F)$ is a regular neighbourhood of $A_\epsilon$ in $(\widehat W, \widehat W \cap \widehat F)$.

Recall the immersion $h$ of a disk or torus $Y$ into $M(\alpha)$ which yields the graph $\Gamma_F$. Fix a face $f$ of $\Gamma_F$ lying in $X^\epsilon$.
We can homotope $h$ to a new immersion such that
\vspace{-.2cm}
\begin{itemize}

\item the homotopy has support in an arbitrarily small neighbourhood in $Y$ of a compact subset of $f$ disjoint from its corners;

\vspace{.2cm} \item the homotopy preserves $h^{-1}(F), h^{-1}(X^+), h^{-1}(X^-)$ and therefore $\Gamma_F$;

\vspace{.2cm} \item the restriction of the new immersion to $f$ is transverse to $A_\epsilon$ and further, $f \cap h^{-1}(A_\epsilon)$ has the minimal number of components among all such immersions.

\end{itemize}
We say that $f$ is in {\it minimal position} if $h|_f$ satisfies the last of these conditions and assume that this is the case for the remainder of this subsection.

Label an endpoint of an arc component of $h^{-1}(A_\epsilon) \cap f$ by $i$ if its image under $h$ lies in $c_i$.

\begin{lemma}
\label{lemma: A epsilon}
Suppose that  $X^\e$ is not a twisted $I$-bundle and that $f$ is a minimally positioned face of $\Gamma_F$ lying in $X^\e$.

$(1)$ There are no closed components of $f \cap h^{-1}(A_\epsilon)$.

$(2)$ No arc component of $f \cap h^{-1}(A_\epsilon)$ has both of its endpoints lying on the same edge of $f$.

$(3)$ If $e$ is an edge of $f$ and there are endpoints of arcs of $f \cap h^{-1}(A_\epsilon)$ which are successive on $e \cap h^{-1}(A_\epsilon)$ and have the same label, then they cobound a subarc $e_0$ of $e$ which is mapped by $h$ into $F - E \subset W$.
\end{lemma}

\begin{proof}
(1) Fix an innermost closed component $c$ of $f \cap h^{-1}(A_\epsilon)$ and let $f_0 \subset f$ be the disk which it bounds. Then $h|f_0$ represents an element of $\pi_2(U, A_\epsilon)$ or of $\pi_2(W, A_\epsilon)$ depending on which side of $A_\epsilon$ $h(f_0)$ lies on. Since the homomorphisms $\pi_1(A_\epsilon) \to \pi_1(U), \pi_1(A_\epsilon) \to \pi_1(W)$ are injective and $\pi_2(U) \cong \pi_2(W) \cong \{0\}$, the groups $\pi_2(U, A_\epsilon)$ and $\pi_2(W, A_\epsilon)$ are both $\{0\}$. Hence we can homotope $h$ with support in an arbitrarily small neighbourhood of $f_0$ to eliminate $c$. The minimality condition on $h|f$ implies that there can be no such component $c$ of $f \cap h^{-1}(A_\epsilon)$. Thus $f \cap h^{-1}(A_\epsilon)$ contains no loop components.

(2) Since $A_\epsilon$ has winding number $2$ or more in $U$ and in $\widehat W$, the result follows as in the proof of a similar situation described in the proof of Lemma \ref{lemma: homotope face into Q}.

(3) Suppose that $e$ is an edge of $f$ and there are two arc components of $f \cap h^{-1}(A_\epsilon)$ whose endpoints are successive on $e$ and have the same label. Let $e_0$ be the subarc of $e$ bounded by these endpoints. Then  $h(e_0)$ is a path contained in the annulus $E$ or
$F\setminus E$.
If $h(e_0)$ is contained in $E$, it is disjoint from one of the components of $c_1 \cup c_2 = \partial E$.
Using the fact that $E$ strong deformation retracts onto either of its boundary components, we could homotope $h|_f$
with support in an arbitrarily  small neighbourhood of $e_0$ to eliminate the two endpoints  from $e \cap h^{-1}(A_\epsilon)$.
Minimality shows that such a situation does not arise.
\end{proof}

We know from Lemma \ref{lemma: A epsilon} that the components of $f \cap h^{-1}(A_\epsilon)$ are arcs
connecting different edges of $f$. Call an arc component of $f \cap h^{-1}(A_\epsilon)$ a {\it corner-arc}
if it is parallel to a corner of $f$, i.e. the arc joins adjacent edges of $f$. We call an arc component of $f \cap h^{-1}(A_\epsilon)$
a {\it cross-arc} if it joins edges of $f$ which are not adjacent.

Since $A_\epsilon$ separates $X^\epsilon$ into two components, the arcs of $f \cap h^{-1}(A_\epsilon)$ decompose $f$ into subsurfaces, called {\it tiles}, each of which is mapped by $h$ into $U$ or $W$. Adjacent tiles lie on different sides of $A_\e$. Colour a tile red if it is mapped into $U$ by $h$ and call it an $n$-gon if its boundary contains $n$
arcs from $f \cap h^{-1}(A_\epsilon)$.

\begin{cor}
\label{cor: edge intersection}
Suppose that  $X^\e$ is not a twisted $I$-bundle and that $f$ is a minimally positioned face of $\Gamma_F$ lying in $X^\e$.
If the intersection of a red tile with an edge $e$ of $f$ is non-empty, the intersection is a closed arc $e_0$ contained in the interior of $e$. Further, the labels at the endpoints of $e_0$ are different.
\qed
\end{cor}

\begin{lemma}
\label{lemma: corner-arcs}
Suppose that  $X^\e$ is not a twisted $I$-bundle and that $f$ is a minimally positioned face of $\Gamma_F$ lying in $X^\e$.

$(1)$ If $a$ is a corner-arc of  $f \cap h^{-1}(A_\epsilon)$, then its endpoints have different labels.

$(2)$ No two corner-arcs in $f \cap h^{-1}(A_\epsilon)$ are parallel in $f$. Equivalently, there are no red tile bigons which intersect
adjacent edges of $\partial f$. In particular, a trigon face of $\Gamma_F$ lying in $X^\e$ contains no red tile bigons.

\end{lemma}

\pf (1) If the endpoints of $a$ have the same label, say $1$, then $a$ can be homotoped, with its endpoints fixed,
 into $c_1\subset F$. So in turn, the corner of $f$ parallel to $a$ can be homotoped, with its endpoints fixed,
 into $F$, which is impossible.

(2) Suppose otherwise and choose $a_1, a_2$ from among the arcs of $f \cap h^{-1}(A_\epsilon)$ parallel to a given corner
such that $a_1 \cup a_2$ lie in the boundary of a red tile $D$. By part (1), the endpoints of each $a_i$ have different labels
and by Lemma \ref{lemma: A epsilon}(3) the labels of endpoints of $a_1$ and $a_2$
alternate $1, 2, 1, 2$ around $\partial D$. This implies that  the algebraic intersection in $\partial U$ between the loop $h|_{\partial D}$ and the circle $c_1$ is $2$. Now applying the loop theorem to the singular disk $h|_D$ in $U$, we get an emdedded disk $D_*$ in $U$
such that the geometric intersection number in $\p U$ between $\p D_*$ and $c_1$ is either $1$ or $2$,
which in turn implies that  $A_\e$ has winding number $1$ or $2 < q_\epsilon$ in  $U$, a contradiction.
This completes the proof.
\qed

\begin{cor}
\label{cor:3-gon}
Suppose that  $X^\e$ is not a twisted $I$-bundle and that $f$ is a minimally positioned trigon face of $\Gamma_F$ lying in $X^\e$. Then $f$ contains at most one red tile and if one, it is a trigon. If there is one, then $q_\epsilon = 3$.
\end{cor}

\begin{proof}
Suppose that $f$ contains a red tile. Lemma \ref{lemma: corner-arcs}(2) shows that it cannot be a bigon, so it is a trigon. Then $\partial D$ consists of three corner-arcs $a_1, a_2, a_3$, one for each corner of $f$, and three subarcs $e_1^0, e_2^0, e_3^0$ of $\partial f$, one for each boundary edge of $f$. The labels of endpoints
of $a_1, a_2, a_3$ alternate $1, 2, 1, 2, 1, 2$ around $\partial D$ and so as $h|_D :(D, \partial D) \to (U, \partial U)$, the algebraic intersection in $\partial U$ between the loop $h|_{\partial D}$ and $c_1$ is $3$, which in turn by the loop theorem (as in the proof of Lemma \ref{lemma: corner-arcs} (2))
implies that $q_\epsilon$, the winding number of $A_\e$ in $U$, is $3$.
\end{proof}

\begin{cor}\label{cor:4-gon}
Suppose that  $X^\e$ is not a twisted $I$-bundle and that $f$ is a minimally positioned quad face of $\Gamma_F$ lying in $X^\e$. If $e_1,e_2$ are two adjacent edges of $f$ and two or more red tiles are incident to $e_1$, then there is at most one red tile incident to $e_2$.
\end{cor}

\pf No red tile bigon connects adjacent edges by Lemma \ref{lemma: corner-arcs}(2) and so as there are at least two
red tiles incident to $e_1$, each of their boundaries contains a cross-arc. It follows that if there is a red tile incident to $e_2$,
it must be a trigon with boundary arcs connecting $e_1, e_2$ and the other edge $e_3$ of $f$ adjacent to $e_2$. But then any other
red tile incident to $e_2$ would have to be a bigon, which we have already ruled out.
\qed

\begin{lemma}
\label{parallel cross-arcs}
Suppose that  $X^\e$ is not a twisted $I$-bundle and that $f$ is a minimally positioned face of $\Gamma_F$ lying in $X^\e$.
Let $\{a_1,...,a_n\}$ be any  set of parallel adjacent cross-arcs of  $f \cap h^{-1}(A_\epsilon)$.

$(1)$ Every $a_i$ has the same label at its endpoints or every $a_i$ has different
labels at its endpoints.

$(2)$ If the $a_i$ have different labels at their endpoints, then $n \leq 2$ and if $2$, the two arcs are contained in different
red tiles.

$(3)$ If $e$ is an edge of  $\Gamma_F$ incident to $f$ which has maximal weight in $\overline{\Gamma}_F$, the $a_i$ connect the two edges of $f$ adjacent to $e$, and the labels at the endpoints of $a_i$ are the same, then there is a red tile incident to $e$.

\end{lemma}

\pf
(1) Since $A_\epsilon$ is an essential annulus in $X^\epsilon$ and $\pi_2(X^\epsilon, F) = 0$, a
properly embedded arc in $A_\epsilon$ is essential as a map of pairs to $(A_\epsilon, \partial A_\epsilon)$ if
and only if it is essential as a map of pairs to $(X^\epsilon, F)$. It follows that for each $i$, the map
$h_i = h|: (a_i, \partial a_i) \to (X^\epsilon, F)$ is essential as a map of pairs if and only if the labels at the endpoints
of $a_i$ are distinct. Since any two of the $h_i$ are homotopic as maps to $(X^\epsilon, F)$, the first assertion of the lemma follows.

(2) For the second assertion suppose $n \geq 2$ and there is a pair of adjacent arcs  $a_i$ and $a_{i+1}$ which
lie in the boundary of some red tile $D$, necessarily a bigon.
It follows from Lemma \ref{lemma: A epsilon}(3) that the labels of the endpoints
of $a_i$ and $a_{i+1}$ alternate around $\partial D$ as  $1,2,1,2$, which implies that $D$ is an
essential disk in the solid torus $U$ whose algebraic intersection number
with $c_1$ is $2$. But then $c_1$ has winding number $2$ in $U$, a contradiction.

(3) Suppose that the $a_i$ closest to $e$ is $a_1$ and that there are no red tiles incident to $e$. Then $a_1$ cobounds a disk $D$ in $f$ with
$e$, the two corners incident to $e$, and two subarcs $e_0, e_0'$ of the edges of $\Gamma_F$ incident to $e$. Since the labels at the endpoints of $a_1$ are the same, we can homotope $h|_{e_0 * a_1 * e_0'}$ (rel $\partial$) to have image in $F$. But then the disk $D$ provides an essential homotopy of $e$ in $X^\epsilon$ contrary to the fact that it has maximal weight. Thus there must be a red tile incident to $e$.
\qed

\section{Recognizing the figure eight knot exterior}
\label{sec: recogn the fig 8}
In this section we describe how work of Martelli, Petronio and Roukema can be used to recognize when $M$ is the figure eight knot exterior. This will be used below to handle the proof of Theorem \ref{thm: twice-punctured precise} in the case that $d=1$.

\subsection{Exceptional Dehn fillings of the minimally twisted chain link}

In \cite{MP}, Martelli and Petronio classified the non-hyperbolic Dehn fillings of
the ``magic manifold" $M_3$, which is the exterior of the hyperbolic chain link of three components in $S^3$.
(This link appears as $6_1^3$ in Appendix C of Rolfsen's book \cite{Rlf}.) We remark that $M_3$ is the manifold $N$ of the first two sections of the paper.
We've changed notation for convenience.

The following result is contained in \cite[Corollary A.6]{MP}.

\begin{prop} {\rm (Martelli, Petronio)}
Let $M$ be a hyperbolic knot manifold obtained by Dehn filling $M_3$
along  two of
its boundary components.
If  $M(\beta)$ is a toroidal manifold  and $M(\alpha)$ is a Seifert fibred manifold
such that $\D(\alpha,\beta)>5$, then $M$ is the figure eight knot exterior.
\qed\end{prop}

In \cite{Rou}, Roukema classified all  non-hyperbolic  Dehn fillings on
the link exterior $M_5$ of the minimally twisted  chain link of five components
in $S^3$ (shown in Figure \ref{5link}).

\begin{figure}[!ht]
\centerline{\includegraphics{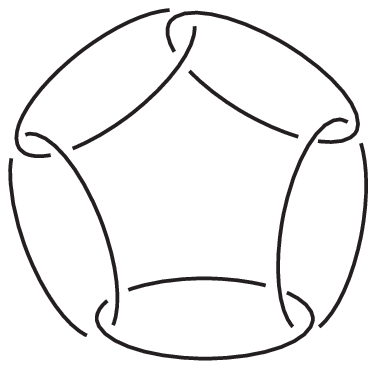}} \caption{ }\label{5link}
\end{figure}

In particular, the following result is contained in \cite[Theorem 3 and Theorem 4]{Rou}.

\begin{prop} {\rm (Roukema)}
Suppose that  $M$ is  a hyperbolic knot manifold obtained by Dehn filling $M_5$ along
 four of
its boundary components but $M$ cannot be obtained by Dehn filling
$M_3$ along two of its boundary components.
Then $\D(\alpha,\beta)\leq 4$ for any two non-hyperbolic  filling  slopes $\alpha, \b$ of $M$.
\qed\end{prop}

Combining the latter two propositions, one gets

\begin{cor}
\label{MPR}
Suppose that  $M$ is  a hyperbolic knot manifold obtained by Dehn filling $M_5$ along
 four of
its boundary components. If  $M(\beta)$ is a toroidal manifold  and $M(\alpha)$ is a Seifert fibred manifold
such that $\D(\alpha,\beta)>5$, then $M$ is the figure eight knot exterior.
\qed\end{cor}

\subsection{Essential annuli in $\widehat X^\epsilon$ and the figure eight knot exterior}
We prove two lemmas which provide sufficient topological conditions for $M$ to be the figure eight knot exterior.

Let $K_\beta$ be the core of the filling solid torus in forming $M(\beta)$.
We noted above that for each $\epsilon$, $\widehat X^\epsilon$ admits a Seifert structure over $D(p_\epsilon, q_\epsilon)$. The $(I, S^0)$-bundle pair $(\Sigma_1^\epsilon, \Phi_1^\epsilon)$ extends to an $(I, S^0)$-bundle pair $(\widehat{\Sigma_1^\epsilon}, \widehat{\Phi_1^\epsilon})$ in which $k_\epsilon = K_\beta \cap \widehat{\dot\Sigma_1^\epsilon}$ is an $I$-fibre.

\begin{figure}[!ht]
\centerline{\includegraphics{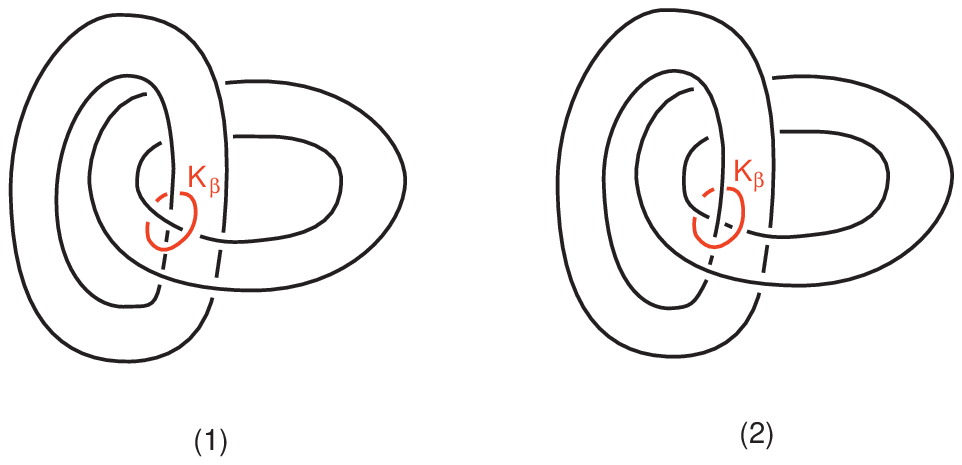}} \caption{ }\label{standard}
\end{figure}

\begin{lemma}
\label{nice annuli}
Suppose that $\Delta(\alpha, \beta) > 5$ and that for each $\epsilon$ there is an embedded essential annulus $A^\epsilon$ in $\widehat X^\epsilon$ such that
\vspace{-.2cm}
\begin{itemize}

\item $\partial A^\epsilon$ has slope $\phi_\epsilon$ in $\widehat F$;

\vspace{.2cm} \item $k_\epsilon$ is an essential arc in $A^\epsilon$;

\vspace{.2cm} \item each boundary component of $A^+$ intersects each boundary component of $A^-$ transversely and
exactly once.

\end{itemize}
\vspace{-.2cm}
Then $M$   can be obtained by Dehn filling four boundary components
of the minimally twisted $5$-chain link exterior and thus it is the figure eight knot exterior.
\end{lemma}

\pf Note that $A^\epsilon$ separates $\widehat X^\epsilon$ into two solid tori
in which $A^\epsilon$ has winding numbers $p_\e, q_\epsilon$ respectively
and the core circles of the two solid tori are singular fibres of the Seifert
fibred space $\widehat X^\epsilon$.
Let $Z^\epsilon$ be the exterior of the singular fibres  in  $\widehat X^\epsilon$, and let $Z$ be the connected submanifold of $M(\beta)$ which is the union of $Z^+$ and $Z^-$ in $M(\beta)$ meeting along $\widehat F$.
Then $Z^\epsilon$  is homeomorphic to $Q^\e\times S^1$ where
$Q^\epsilon$ is a twice punctured disk.
Call $\widehat F$ the outer boundary  component  of $Z^\epsilon$ and called the other two boundary  components of $Z^\epsilon$
inner components. Correspondingly the component of $\partial Q^\epsilon$ contained in $\widehat F$ is called the outer component and the other two components of $\partial Q^\epsilon$ are call inner components.
The choice of $Q^\epsilon$ in the product structure of $Z^\e=Q^\e\times S^1$ is not unique. By Dehn twisting along a vertical annulus in $Z^\epsilon$ disjoint from $A^\epsilon$ which connects $\widehat F$ to an inner boundary component of $Z^\e$, we may assume that $Q^+$ has been chosen so that its  outer boundary component
has slope $\phi_-$ in $\widehat F$ and similarly we may assume
that $Q^-$ has been chosen  so that its  outer boundary component
has slope $\phi_+$ in $\widehat F$.
It follows that $Z$ is homeomorphic to the exterior of ``the untwisted double Hopf link'' in $S^3$
such that the torus $\widehat F$ becomes
a torus in $S^3$  which bounds a trivial solid torus $V^\epsilon$  in $S^3$ on each  side of $\widehat F$   and which separates  the two pairs of
parallel components of the double Hopf link. Furthermore  the twice punctured disk  $Q^\epsilon$ in $Z^\epsilon$
caps off in $V^\epsilon$ to a standard meridian disk of $V^\epsilon$.
The annulus $A^\epsilon$ is   a vertical essential annulus in $Z^\epsilon$ and it separates the two inner components of $\partial Z^\epsilon$. In fact $A^\e=\d_\e\times S^1$ where $\d_\epsilon$ is a proper  arc in $Q^\epsilon$
which separates the two inner components of $\partial Q^\epsilon$.
The arc $k_\epsilon$ in $A^\epsilon$ is an essential arc.
By Dehn twisting along a vertical annulus $\s_\e\times S^1$ in $Q^\epsilon$ where $\s_\epsilon$ is a proper arc in $Q^\epsilon$ connecting the two inner boundary  components of $Q^\epsilon$ and intersecting $\d_\epsilon$ exactly once,
we may assume that $K_\b=k_+\cup k_-$ is as shown in part (1) of Figure \ref{standard} (cf. the red coloured component).
Let $Y$ be the exterior of $K_\b$ in $Z$, i.e. $Y$ is the exterior of the $5$-component link
shown in Figure \ref{standard} (1).
Considering Dehn surgery  on the component $K_\b$ (the red coloured component) in $S^3$, we see that
$Y$ is homeomorphic to the exterior of
the $5$-component link in $S^3$  shown in part (2) of  Figure \ref{standard}.
The latter link is the minimally twisted chain link of $5$ components.
 \qed

\begin{rem}{\rm This argument was used in \cite{GL2}.}
\end{rem}

For the exterior $N$ of a link $L$ in $S^3$ with components $K_0,K_1,\cdots, K_{n-1}$,
let $T_0, T_1, \cdots, T_{n-1}$ be the corresponding ordered boundary tori of $N$.
 A slope on each $T_i$  will be expressed as $(p, q)$ with respect to the standard
(meridian, longitude) coordinates, where $p, q$ are relative prime
integers and $q\geq 0$, e.g. $(1,0)$ is the meridional slope, $(0, 1)$
is the canonical longitude. Let $(0,0)$ denote the empty slope. Dehn
filling of $N$ will be denoted by $N[(p_0, q_0), (p_1, q_1),\cdots,
(p_{n-1}, q_{n-1})]$, meaning each $T_i$ is assigned the slope
$(p_i, q_i)$, possibly empty. For example when $n=3$, $N[(0,0),
(-2,1), (2,3)]$ means leaving  $T_0$ unfilled and filling $T_1, T_2$
with slopes $(-2,1), (2, 3)$ respectively. When $N$ is clear, we may
simply denote  the filling by $[(p_0, q_0), (p_1, q_1),\cdots,
(p_{n-1}, q_{n-1})]$. For each $k$, $1\leq k\leq n$, a $k$-filling
of $N$ means  filling $N$ along $k$ boundary tori with nonempty
slopes  and leaving  the rest of the boundary tori unfilled. For example
$N[(0,0), (-2,1), (2,3)]$ is a $2$-filling of $N$.

\begin{lemma}\label{lemma:nice annuli2}
Suppose that $\Delta(\alpha, \beta) > 5$, $X^-$ is a twisted $I$-bundle  and  $d=1$.
 Suppose that for each $\epsilon$ there is an embedded essential annulus $A^\epsilon$ in $\widehat X^\epsilon$ such that
\vspace{-.2cm}
\begin{itemize}

\item $\partial A^+$ has slope $\phi_+$ in $\widehat F$ and $\partial A^-$ has slope $\phi_-'$ in $\widehat F$;

\vspace{.2cm} \item $k_\epsilon$ is an essential arc in $A^\epsilon$;

\vspace{.2cm} \item each boundary component of $A^+$ intersects each boundary component of $A^-$ transversely and
exactly once.

\end{itemize}
\vspace{-.2cm} Then $M$   can be obtained by Dehn filling the
boundary components $T_1,\cdots T_6$ of the exterior of the
7-component link $\{K_0,K_1, \cdots, K_6\}$  in $S^3$  shown in
{\rm Figure \ref{bgz5-7-component link}(2)} such that the slopes on $T_3$
and $T_4$ are both $(-1, 1)$.  Furthermore $M$ is the figure eight knot
exterior.
\end{lemma}

\begin{proof}
As in the proof of Lemma \ref{nice annuli},
let $Z^\epsilon$ be the exterior of the singular fibres  in  $\widehat X^\epsilon$
with respect to the Seifert structure
whose base orbifold is $D^2(p_\e,q_\e)$, and let $Z$ be the connected
submanifold of $M(\beta)$ which is the union of $Z^+$ and $Z^-$ in $M(\beta)$
meeting along $\widehat F$.
Then $Z^\epsilon$  is homeomorphic to $Q^\e\times S^1$ where
$Q^\epsilon$ is a twice punctured disk.
Since $d=1$,  we have, as in the proof of Lemma \ref{nice annuli}, that
  $Z$ is homeomorphic to the exterior of ``the double Hopf link'' in $S^3$
such that the torus $\widehat F$ becomes a torus in $S^3$  which
bounds a trivial solid torus $V^\epsilon$  in $S^3$ on each  side of
$\widehat F$   and which separates  the two pairs of parallel
components of the double Hopf link and that  the twice punctured
disk  $Q^\epsilon$ in $Z^\epsilon$ caps off in $V^\epsilon$ to a
standard meridian disk of $V^\epsilon$. Figure \ref{bgz5-d=d'=1}
shows  $Z^-$ cut open along $Q^-$. Note that $\widehat
X^-$ is an annulus bundle over $S^1$ with $A^-$ as a fibre. The part
of $A^-$ in $Z^-$ is a twice punctured annulus. Figure
\ref{bgz5-d=d'=1} shows the part of $A^-$ in  the cut-open $Z^-$
(the surface  coloured green). Here we may assume that $A^-$ is of
the form shown in Figure \ref{bgz5-d=d'=1} inside  $Z=Z^+\cup Z^-$
(i.e. after the specific choice of $Q^-$ for $Z^-$) since there is
only one twisted I-bundle over the Klein bottle up to homeomorphism.

\begin{figure}[!ht]
\centerline{\includegraphics{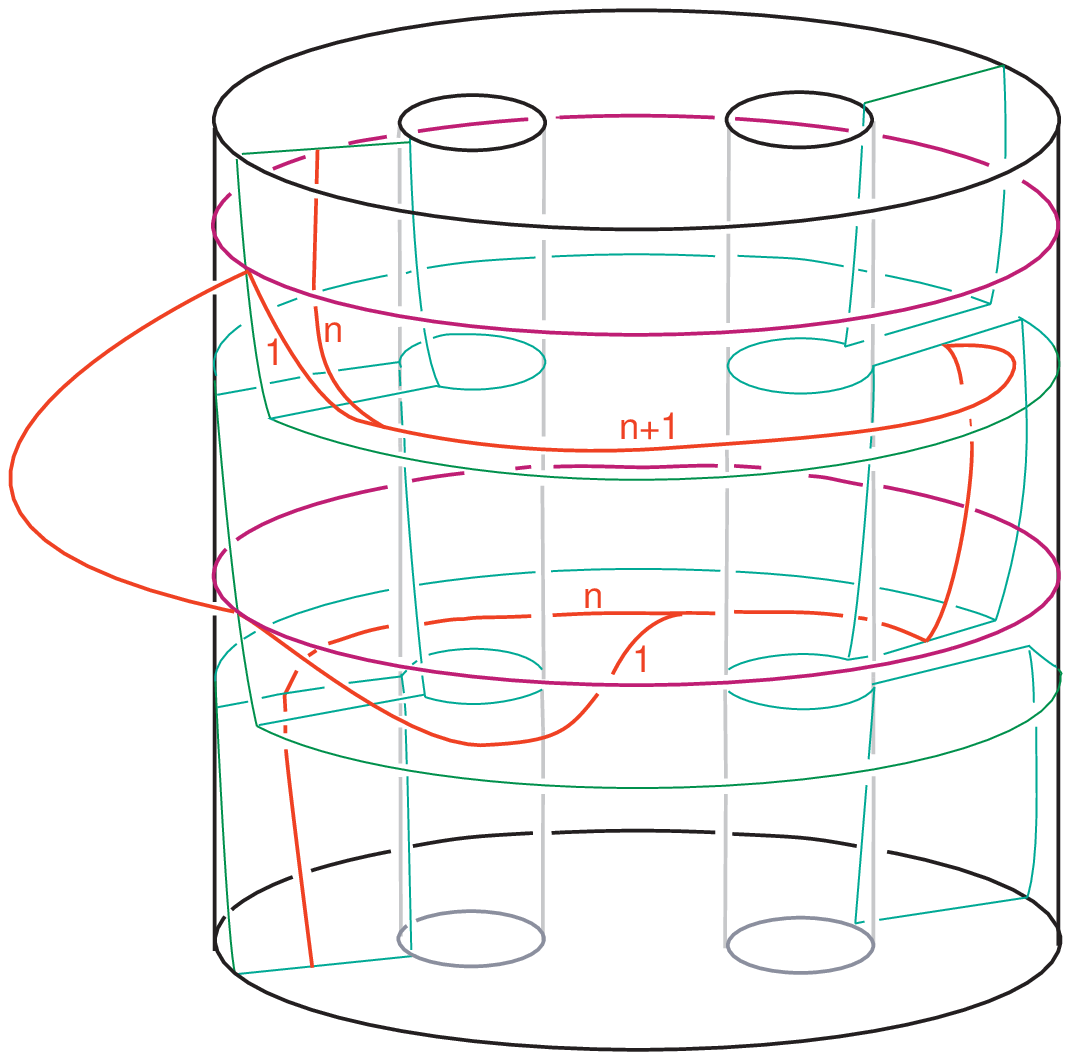}}
\caption{}\label{bgz5-d=d'=1}
\end{figure}

Also as in the proof of Lemma \ref{nice annuli}, we have that  the
annulus $A^+$ separates the two inner components of $\partial Z^+$,
$\partial A^+$  is as shown in Figure \ref{bgz5-d=d'=1} (the two
horizontal circles coloured purple), and up to Dehn twists and
isotopy, the arc $k_+$ in $A^+$ is  transverse to the $S^1$ fibres of
$A^+$ and is disjoint from a copy of $Q^+$ (shown in Figure
\ref{bgz5-d=d'=1} as the  red coloured arc that lies outside $Z^-$).
The arc $k_-$ is an essential transverse arc in $A^-$ which may wrap
along the $S^1$ factor of $A^-$ many times. In Figure
\ref{bgz5-d=d'=1} $k_-$ is illustrated as a branched curve in $A^-$
(a branch with label $n$ means $n$ parallel  arcs in $A^-$). To
eliminate the wraps of $k_-$ in $A^-$, we drill out two knots $K_5,
K_6$ in $Z^-$ which are chosen as follows: $\partial A^-$ cuts  the torus
$\widehat F$ into two annuli, push the center circle of each of
these two annuli slightly into the interior of $Z^-$ but disjoint
from $A^-$, the resulting knots are $K_5$ and $K_6$. Let $Y^-$ be
the exterior of $K_5$ and $K_6$ in $Z^-$, and let $T_5$ and $T_6$ be
the torus boundary components of $Y^-$ corresponding to $K_5$ and
$K_6$. There is an obvious  essential annulus $A$ in $Y^-$
connecting $T_5$ and $T_6$ which intersects $A^-$ in a single circle
which is essential in both $A$ and $A^-$. Twisting along this
annulus $A$ (which does not change the homeomorphism type of $Y^-$,
which fixes $\widehat F$ point-wise and which preserves $A^-$
set-wise) together with some isotopy,
 we may simplify $k_-$ in $A^-$ so that it does not fully wrap around the $S^1$ factor of $A^-$ and
 is disjoint from a copy of $Q^-$.
Now let $N$ be the exterior  in $S^3$ of the $7$-component link
shown in part (1) of Figure \ref{bgz5-7-component link} where the
red coloured component $K_0$  is $K_\b=k_+\cup k_-$, the two green
coloured components are $K_5$ and $K_6$ and the remaining four components
form the double Hopf link. Then by the construction of $N$, $M$ can
be obtained by Dehn filling the six boundary components $T_1, \cdots,
T_6$ of $N$, leaving $T_0$ unfilled.

Note that the twisted $I$-bundle over the Klein bottle $\widehat
X^-$ can be  recovered  from $Z^-$ by Dehn filling the two inner
boundary tori $T_3$ and $T_4$ of $Z^-$ both with slope $(-2,1)$
since the twisting along the annulus $A$ above fixes $T_3$ and $T_4$
point-wise and thus does not change filling slopes on these two
tori.

Simplifying this link as illustrated in Figure \ref{bgz5-7-component
link} (which does not change the homeomorphism type of $N$), we get
the link shown in Figure \ref{bgz5-7-component link} part (2). By the
Kirby-Rolfsen surgery calculus, the filling slope for the new $T_3$
and $T_4$  becomes $(-1, 1)$. This proves the first conclusion of
the lemma.

\begin{figure}[!ht]
\centerline{\includegraphics{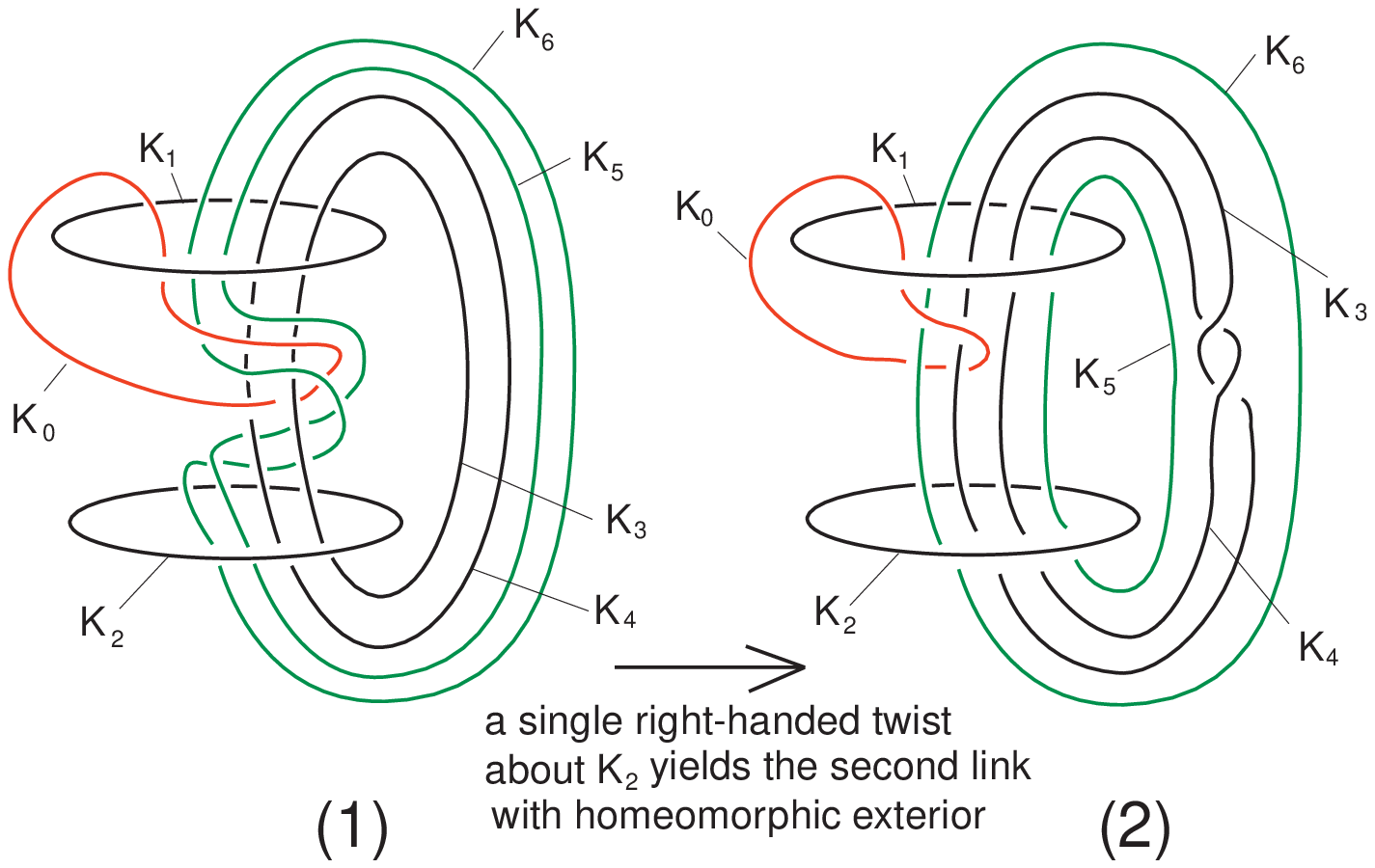}} \caption{ }\label{bgz5-7-component link}.
\end{figure}

To prove the second conclusion of the lemma, let $N$ be the exterior
of the ordered $7$-component link in $S^3$ shown in part (2) of
Figure \ref{bgz5-7-component link}, which can be checked with SnapPy
to be hyperbolic.
 Our manifold $M$ is supposed to be a hyperbolic filling of $N$ leaving
$T_0$ unfilled, i.e.
$$M=N[(0,0), (p_1,q_1), (p_2, q_2), (-1,1), (-1,1),(p_5,q_5), (p_6,q_6)]$$
 for some nonempty  slopes $(p_1,q_1), (p_2, q_2), (p_5,q_5), (p_6, q_6)$.
Let $ Y$  be the manifold  $$ N[(0,0),(0,0),(0,0),(-1,1),(-1,1),(0,0),(0,0)].$$
There is an annulus $ E $ in Y joining $T_5$  and $T_6$ (since $ \widehat X^- $ is an $A^-$-bundle
over $ S^1$, $K_5$ and $K_6$  bound an annulus in $\widehat X^-$
 disjoint from the annulus $ A^-$  which contains $k_-$). A regular neighbourhood $Z_*$  of
$E \cup T_5 \cup T_6$  in $Y$  is homeomorphic to $ P \times S^1$, where $P$  is a twice punctured disk.
 Let $ W = Y \setminus Z_*$, and let $T$  be the torus
$W \cap Z_*$. Our manifold $M$  is obtained by Dehn filling $Y $ along $T_1$,
$T_2$, $T_5$  and $T_6$: $ M = \widehat W \cup_{T} \widehat Z_*$, say,
where $\widehat W$ is the result of  Dehn filling $W$ along $T_1$
and $T_2$  with slopes $ (p_1,q_1)$  and $(p_2,q_2)$, and $\widehat Z_*$ is $Z_*$
Dehn filled along $T_5$  and $T_6 $ with slopes $(p_5,q_5) $ and $(p_6,q_6)$.
Since $M$  is hyperbolic, $T$  compresses in either $\widehat {Z_*}$  or
$\widehat W$. In the first case, $\widehat Z_*$ must be a solid torus, and in the second case,
$\widehat W  \cong M \# V$, where $V$ is a solid torus and $V \cup \widehat Z_* \cong S^3$. In the first case,
$\widehat Z_*$ can be obtained from $Z_*$  by doing the trivial Dehn filling $(1,0)$ on $T_6$ and some filling on $T_5$, with slope
$(p'_5,q'_5)$, say. In the second case, we can do a similar Dehn filling on $Z_*$ to get a solid torus $U$
 such that $ V \cup_{T} U \cong S^3$. Thus in both cases $ M = N[(0,0),(p_1,q_1),(p_2,q_2),(-1,1),(-1,1),(p'_5,q'_5),(1,0)]$.
 That is, if $N_{*}$ is the exterior of the $6$-component link
  in $S^3$ shown in part (1) of Figure \ref{bgz5-5-component link},
  then $M=N_*[(0,0), (p_1,q_1), (p_2, q_2), (-1,1), (-1,1), (p_5',q_5')]$.
   Using the Kirby-Rolfsen surgery calculus  we can eliminate the component $K_4$
    and get the ordered
   $5$-component link $\{J_0, J_1, J_2, J_3, J_4\}$ in $S^3$ shown in part (2) of Figure \ref{bgz5-5-component link}
   so that  $M=N_{\#}[(0,0), (m_1, n_1), (m_2, n_2), (m_3,n_3),(0,1)]$
   for non-empty slopes $(m_1, n_1), (m_2, n_2), (m_3,n_3)$,
   where $N_{\#}$ is the exterior of the ordered $5$-component link $\{J_0, J_1, J_2, J_3, J_4\}$ in $S^3$.

\begin{figure}[!ht]
\centerline{\includegraphics{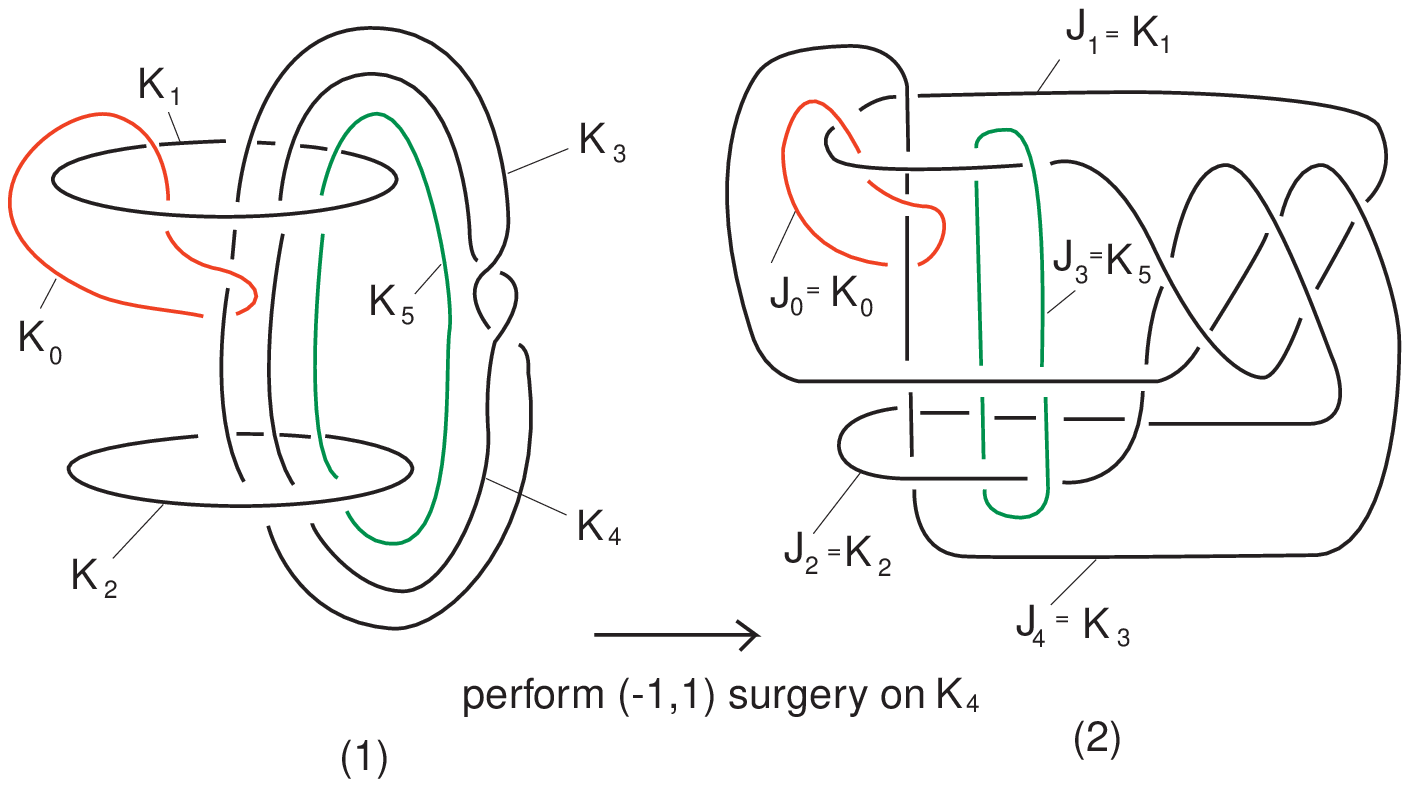}} \caption{ }\label{bgz5-5-component link}.
\end{figure}

The five component link $\{J_0, J_1, J_2, J_3, J_4\}$ in $S^3$ is
hyperbolic and we have reduced the proof of the lemma to the proof
of the following

{\bf Claim}. If $M$ can be obtained by Dehn filling $N_{\#}$ along
$T_1,..., T_4$ with slope $(0, 1)$ on $T_4$, then $M$ is the figure
8 knot exterior.

It is enough to show that if a $4$-filling $N_{\#}[(0,0), (m_1,
n_1), (m_2, n_2), (m_3, n_3), (0, 1)]$ is a hyperbolic manifold
which has two non-hyperbolic filling slopes  on $T_0$ distance at
least $6$ apart, then it is either the figure eight knot exterior, or
$Wh[(0,0), (-2,1)]$ or $Wh[(0,0),(5, 1)]$, where $Wh$ is the
exterior of the Whitehead link in the $3$-sphere\footnote{For the rest of this section we've replaced the Whitehead link exterior $Wh$ (cf. Figure \ref{2 and 3 chain link}) with its mirror image as a matter of convenience.}, because for each of
$Wh[(0,0), (-2,1)]$ and  $Wh[(0,0),(5, 1)]$, if two non-hyperbolic
filling slopes are distance at least $6$ apart, they are both
toroidal filling slopes.

To this end,  we apply the Python codes {\it
find\_exceptional\_fillings.py}
 and
 {\it search\_-geometric\_solutions}
 written by Martelli \cite{Ma}, used on SnapPy written by Culler, Dunfield,
   Goemer and Weeks \cite{CDGW}, to  the $5$-component link exterior
$N_{\#}$. Note that when running {\it
find\_exceptional\_fillings.py} on the $5$-component link exterior
$N_{\#}$ we only need to find those fillings of  $N_{\#}$ which
either have slope $(0, 1)$ on $T_4$ or leave $T_4$ unfilled (this
can be done by adding a finite set of well chosen $1$-fillings into the
exclude.py file in Martelli's Python package). We get two lists as the outputs
from running the code {\it find\_exceptional\_fillings.py}: a list of
candidate isolated non-hyperbolic fillings on $N_{\#}$ and a list of
candidate hyperbolic fillings on $N_{\#}$,   here an isolated
non-hyperbolic filling means a non-hyperbolic filling whose  proper
sub-fillings are all hyperbolic fillings.

We first need to check if the fillings in the second list are
indeed hyperbolic fillings, which is accomplished by running {\it
search\_-geometric\_solutions} on it. It turns out the
fillings in the second list are all $5$-fillings (there are $862$ of
them) and they are all hyperbolic. (To reduce the running time, we
first set the max degree in the code to $3$, at which 809 fillings
are confirmed to be hyperbolic, then set the max degree at $5$, at
which another twenty fillings are confirmed, then degree  $7$
confirming twenty-six more fillings. The remaining seven fillings are confirmed at
degree $11$.)

We now deal with the first list.
We don't have to confirm if all the fillings listed are indeed
non-hyperbolic.
We treat them as possible isolated non-hyperbolic fillings.
The good thing is that the list contains all real isolated non-hyperbolic fillings.
The list contains eight $1$-fillings, twenty-seven $2$-fillings, sixty-three $3$-fillings, one hundred and seventy $4$-fillings and three hundred and one
$5$-fillings.

If a $k$-filling of $N_{\#}$, $1\leq k\leq 3$, is an isolated non-hyperbolic filling but
 contains hyperbolic pieces, it might extend to infinitely many hyperbolic
 $4$-fillings of $N_{\#}$ with $T_0$ unfilled, which could potentially contribute
  many new non-hyperbolic
 slopes on $T_0$ which have not occurred in the fillings of the list.
 Therefore such cases, if they exist, are potentially difficult to deal with.
 Fortunately such  a case does not happen for our manifold $N_{\#}$.
 There are some $k$-fillings, $1\leq k\leq 3$, which are isolated non-hyperbolic
 containing hyperbolic pieces. But after adding the slope  $(0,1)$
 on $T_4$ to such $k$-filling (if $T_4$ is unfilled  there), it still contains a hyperbolic piece
 only  when  its slope on $T_0$ is non-empty and
 therefore this $k$-filling contributes at most one
 non-hyperbolic slope on $T_0$.
 Similarly we don't need to worry about isolated non-hyperbolic $4$-fillings
 as each of them either contributes only one slope on $T_0$ or
 can not yield $M$.

 Now we combine fillings in the first list (exclude those with empty slope on $T_0$)
 into the list of maximal subgroups  so that
 in each subgroup any two fillings are compatible in the sense
 that their slopes on each of $T_1, T_2, T_3, T_4$ agree
 unless one or both of  them are empty.
 (The new list, obtained using  a simple Python code, as well as the two lists mentioned earlier
will be posted along with the paper on authors' web
pages).
 Each such subgroup contains  at least one 5-filling
 which by restriction  gives a hyperbolic $4$-filling leave $T_0$ unfilled
 (this hyperbolic $4$-filling is unique for this subgroup).
 Slopes on $T_0$ from the elements in this subgroup
 contain all possible non-hyperbolic filling slopes for
 the hyperbolic $4$-filling.
 So we just need to calculate the  distance
 between such  slopes on $T_0$ for each subgroup (which  can be done
 using a simple Python code).
 It turns out that if the distance between two slopes on $T_0$
 from  a subgroup is at least $6$, then one of the following
 events holds:

 1) The corresponding hyperbolic $4$-filling for this subgroup
 is the figure eight knot exterior (occurs for $4$ subgroups) or
 $Wh[(0,0), (5,1)]$ (occurs for $4$ subgroups) or
 $Wh[(0,0), (-2,1)]$ (occurs for $8$ subgroups), which can be verified by SnapPy.

 2) One of the two slopes is
 $(0, 1)$.
 But this slope is
 not a non-hyperbolic slope of the corresponding
 $4$-filling of the subgroup, which can be checked  using {\it search\_geometric\_solutions}.
 (Such instance occurs for $11$ subgroups).
 Such case  happens because
 $[(0,1), (0,0), (0,0), (0,0), (0,0)]$
 is a non-hyperbolic filling which contains a hyperbolic piece.

 3) The distance is $7$  and is realized on the unique pair
  $[(-3,4), (2,1), (3,1), (-2, 1), (0, 1)]$ and\newline
 $[(1,1), (2,1), (0, 0), (0, 0), (0, 0)]$.
  But $[(1,1), (2,1), (3,1), (-2, 1), (0, 1)]$ is
 hyperbolic.
  This happens because
 $[(1,1), (2,1), (0,0), (0,0), (0,0)]$
 is an isolated  non-hyperbolic filling which contains hyperbolic piece.

 The proof of the second conclusion of the lemma is now finished.
 \end{proof}

\subsection{The combinatorics of $\overline{\Gamma}_F$ and the figure eight knot exterior}
We show how combinatorial conditions on $\overline{\Gamma}_F$ guarantee the existence of
annuli as in the previous subsection.

\begin{prop}
\label{prop: trigon with max weight edge}
Suppose that $d \ne 0$, $\Delta(\alpha, \beta) \geq 6$, and $f$ is a trigon face  of $\overline \Gamma_F$ contained in $X^\epsilon$
with a maximal weight edge $\bar e$. Then there are annuli $A^+$ and $A^-$ with the properties given in Lemma \ref{nice annuli}.
Therefore $M$ is the figure eight knot exterior.
\end{prop}

\begin{proof}
Since $f$ is a trigon, $X^\epsilon$ is not a twisted $I$-bundle and by Lemma \ref{lemma: edge wts}, $\bar e$ is a positive edge.

Let $e_1 ,e_2,e_3$ be parallel adjacent edges of $\Gamma_F$ corresponding to $\bar e$ where
$e_1$ is an edge of $f$ and $e_2$ is the edge lying between $e_1$ and $e_3$.
Since $\bar e$ is a positive edge of weight $3$ or more, $\widehat{\dot\Sigma_1^+}$ is a twisted $I$-bundle over a M\"{o}buis band and $\widehat{\dot\Sigma_1^-}$ is either a twisted $I$-bundle over a Mobuis band or
a twisted $I$-bundle over a Klein bottle.

\begin{figure}[!ht]
\centerline{\includegraphics{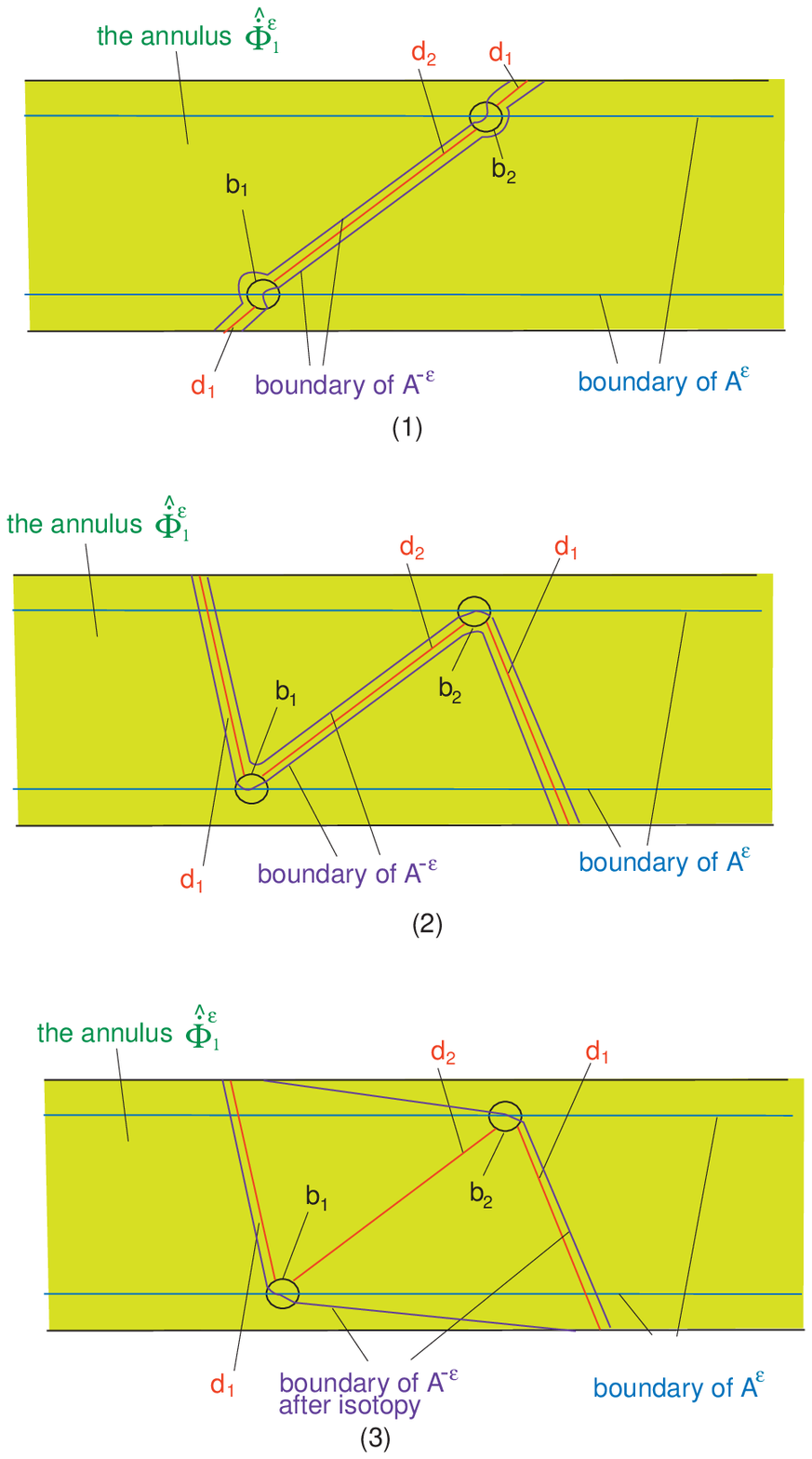}} \caption{ }\label{annuli1}
\end{figure}

After a homotopy of $h$ we can assume that the face $f$ of $\Gamma_F$ is minimally positioned (cf. \S \ref{subsec: trace face}).
Recall that $A_\e$ is the frontier of $\widehat{\dot\Sigma_1^\e}$ in $\widehat X^\e$
with boundary components $c_1$ and $c_2$,
$\widehat P=A_\e\times I$ is a product region in $\widehat{\dot\Sigma_1^\e}$
and $k_\e = K_\beta \cap H^\epsilon$ is an essential arc in $A^\e = A_\e\times\{1/2\}$.
By Corollary \ref{cor:3-gon}, $e_1$ can be homotoped in $F$ with its two endpoints fixed so that
it intersects each $c_i$ at most once.
Further, since the bigon face between $e_2$ and $e_3$ lies in $X^\e$,
$e_2$ can be homotoped in $F$ with its two endpoints fixed
so that it is contained in the interior of $\dot\Phi_1^\e$.

Recall that $\partial X^{-\e}=F\cup B^{-\e}$ and $b_1$, $b_2$ are components of $\partial F=\partial B^{-\e}$.

If we apply the loop theorem to the algebraic bigon $R$ between $e_1$ and $e_2$ with respect to
the algebraic intersection with $b_1$, we obtain an embedded disk $D$ in $X^{-\e}$ such that $\partial D$ has non-zero algebraic intersection
with $b_1$. In particular $\partial D$ must intersect $b_1$. The loop theorem implies that
we may assume that $\partial D\cap (B^{-\e}\cup c_1\cup c_2)$ is contained in
$h(\partial R)\cap (B^{-\e}\cup c_1\cup c_2)$ (cf. \cite[Theorem 4.10]{He}). Lemma \ref{embedded n-gons}(1) shows that
$D$ cannot be a monogon, so it is an algebraic bigon. Each of the two components of $F \cap \partial D$ are arcs connecting
$b_1$ and $b_2$ so that one of them, $d_1$ say, intersects each $c_i$ at most once, and the other, $d_2$ say, is contained in the interior of $\dot\Phi_1^\e$, and thus is disjoint from each $c_i$.

Note that $D$ is part of an embedded M\"{o}bius band in $X^{-\e}$ obtained by attaching a product region in the annulus
$B^{-\e}$ between the two corners of $D$ (cf. the proof of Lemma \ref{lemma: two trigons pos edge}).
The boundary of this M\"{o}bius band has slope $\phi_{-\e}$ and $c_i$ has slope $\phi_\e$.
Our hypothesis that $d \ne 0$ implies that $\phi_+\ne \phi_-$,
so the boundary of the M\"{o}bius band intersects each $c_i$ exactly once and the intersection is transverse. (So $d = 1$.)
That is, the edge $d_1$ intersects each $c_i$ transversely exactly once.

Recall $\hat b_1$ and $\hat b_2$ are the disks attached to $F$ to form $\widehat F$.
If we consider  $\hat b_1$, $\hat b_2$ as points
then $d_1\cup \hat b_1\cup d_2\cup \hat b_2$ is a loop in $\widehat F$  of slope $\phi_{-\e}$
which intersects $c_i$ transversely exactly once. We may now isotope $D$ in $X^{-\e}$ so that its edges
become transverse to the $S^1$-fibres of the annulus $\widehat{\dot\Phi_1^\e}$.
Figures \ref{annuli1}(1) and (2) illustrate the possibilities for
the edges $d_1$ and $d_2$ of $D$ up to Dehn twists fixing their endpoints along a $S^1$-fibre of the annulus $\widehat{\dot\Phi_1^\e}$.
If the case of Figure \ref{annuli1}(1) occurs, the resulting algebraic bigon $D$ yields a corresponding embedded M\"{o}bius band $B$ in $X^{-\e}$
such that the frontier of a suitably  chosen regular neighbourhood of $B$ in $\widehat X^{-\e}$ is an essential annulus $A^{-\e}$ which contains $k_{-\e}$ as an essential arc. Moreover each component of $\partial A^{-\e}$ intersects each component of $\partial A^\e$
exactly once. The boundary of $A^{-\e}$ is illustrated in Figure \ref{annuli1}(1).

If the case of Figure \ref{annuli1}(2) occurs, the resulting algebraic bigon $D$ also yields a corresponding embedded M\"{o}bius band $B$ in $X^{-\e}$
such that the frontier of a suitably  chosen regular neighbourhood of $B$ in $\widehat X^{-\e}$ is an essential annulus $A^{-\e}$ which contains $k_{-\e}$ as an essential arc, as illustrated in
Figure \ref{annuli1}(2). But at this stage each component of $\partial A^{-\e}$ intersects
each component of $\partial A^\e$ exactly twice, one transverse and one tangent. We may now isotope
$A^{-\e}$ in $\widehat{X^{-\e}}$, fixing the arc $k_{-\e} \subset A^{-\e}$, so that at the end of the isotopy
each component of $\partial A^{-\e}$ intersects each component of $\partial A^\e$ exactly once, and transversely, as indicated in Figure \ref{annuli1}(3).
\end{proof}

\begin{prop}
\label{prop: quad with adjacent max weight edges}
If $\Delta(\alpha, \beta) \geq 6$ and
there is  a quad face $f$ of $\overline \Gamma_F$ with adjacent edges of maximal weight, then
there are annuli $A^+$ and $A^-$ with the properties given in Lemma \ref{nice annuli}. Therefore $M$ is the figure eight knot exterior.
\end{prop}

\pf Suppose that $f$ lies in $X^\e$ and note that $X^\e$ cannot be a twisted $I$-bundle as in this case the only faces of $\Gamma_F$ lying in $X^-$
are bigons. We may assume that $f$ is minimally positioned (cf. \S \ref{subsec: trace face}).

By  Corollary \ref{cor:4-gon}, $f$ has an edge $\bar e$ of maximal weight which intersects each $c_i$ at most once.
Let $e_1 ,e_2,e_3$ be the parallel adjacent edges in the family of edges represented by  $\bar e$ with
$e_1$ being the edge of $f$ and $e_2$ being the one between $e_1$ and $e_3$.
Note that we may assume that $e_2$ is contained in the interior of $\dot{\Phi}^\e_1$ and thus is disjoint from
each $c_i$.

Let $R$ be the bigon face between $e_1$ and $e_2$. Note that $\partial R$ is an essential loop in $\partial X^{-\e}$ since each of the edges $e_1$ and $e_2$ is an essential arc in $F$,
so we can apply the loop theorem to the singular disk $R$ which is the bigon face between $e_1$ and $e_2$.
We obtain a properly embedded disk $(D, \partial D)$ in $(X^{-\e}, \partial X^{-\e})$ such that $\partial D$ is an essential loop in $\partial X^{-\e}$
which intersects each $c_i$ transversely and in at most one point. Since $F$ is incompressible and $X^{-\e}$ does not contain a monogon, $D$ is a bigon.

If $D$ is an algebraic $0$-gon, its edges are disjoint $\widehat F$-essential loops and therefore their algebraic intersection with each $c_i$
is even, contrary to what we have deduced. Thus it is an algebraic bigon. Now proceed as in the proof of Proposition \ref{prop: trigon with max weight edge} to see that the existence of such an embedded disk $D$ implies that
there are annuli $A^+$ and $A^-$  with the properties given in Lemma \ref{nice annuli}. Therefore $M$ is the figure eight knot exterior.
\qed

\begin{prop}
\label{prop: two trgons sharing a positive edge}
Suppose that $\Delta(\alpha, \beta)=6$ and $d=1$.
If $\overline \Gamma_F$
has two trigon faces $f_1$ and $f_2$  sharing a common positive edge $\bar e$ of weight $2$,
then  $M$ is the figure eight knot exterior.
\end{prop}

\pf
Let $e_1$ and $e_2$ be the two parallel edges of $\Gamma_F$ represented by $\bar e$.
Note that $f_1$ and $f_2$ lie on the same side of $F$, say $X^\e$, and as they are trigons, $X^\epsilon$ is not a twisted $I$-bundle.
We may assume that each of $e_1$ and $e_2$ intersects each $c_i$ exactly once by Corollaries \ref{cor: edge intersection}(2) and \ref{cor:3-gon}.

The bigon $R$ containing $e_1$ and $e_2$ in its boundary is an $S$-bigon. That is, $e_1$ and $e_2$ form an $S$-cycle.
Applying the loop theorem to the singular disk $R$ with respect to (mod $2$) intersection with $b_1$ yields an embedded $S$-bigon $D$ in $X^{-\e}$ the union of whose two edges intersects each $c_i$ at most twice.
As in the proof of Proposition \ref{prop: trigon with max weight edge}, the bigon yields a M\"{o}bius band
in $X^{-\e}$ whose boundary  has slope $\phi_{-\e}$ and intersects
each $c_i$ at most twice. Since $d=1$, we see that the boundary of the M\"{o}bius band
intersects each $c_i$ exactly once and the intersection is transverse.
As in the proof of Proposition \ref{prop: trigon with max weight edge}, such a M\"{o}bius band will
yield an embedded essential annulus $A^{-\e}$ which together with $A^\e$, chosen as in the proof
of Proposition \ref{prop: trigon with max weight edge}, satisfy the properties of Lemma \ref{nice annuli}
and thus $M$ is the figure eight knot exterior.
\qed

\begin{prop}
\label{prop: adjacent trigons}
Suppose that $\Delta(\alpha, \beta)=6$, $d=1$  and $X^-$ is a twisted $I$-bundle.
If $\overline \Gamma_F$
has two trigon faces $f_1$ and $f_2$  sharing a common edge $\bar e$,
then  $M$ is the figure eight  knot exterior.
\end{prop}

\pf Since $X^-$ is a twisted $I$-bundle, both $f_1$ and $f_2$ lie in $X^+$.
By Proposition \ref{prop: trigon with max weight edge}, we may assume that
$\bar e$ has weight $2$.

Let $e_1$ and $e_2$ be the two parallel edges of $\Gamma_F$ represented by $\bar e$.
By Corollary \ref{cor:3-gon}, we may assume that each of
$e_1$ and $e_2$ intersects each $c_i$ at most once.
Now applying the loop theorem to the bigon face $R$ of $\Gamma_F$ containing $e_1$ and $e_2$
in its boundary to obtain an embedded bigon $D$ in $X^{-}$.
The boundary of $D$ intersects each $c_i$ transversely and in at most two points.

{\bf Case 1}.  $D$ is an algebraic $0$-gon.

In this case, the edges of $D$ are disjoint $\widehat F$-essential loops and therefore their union intersects each $c_i$ transversely
in at most two points. Since these two loops are homologous, either both are disjoint from the $c_i$ or both intersect each $c_i$ exactly once.

As in the third paragraph of the proof of Lemma \ref{lemma: two trigons negative edge},
the edges of $D$ form a pair of $\widehat F$-essential loops which are the boundary of an essential annulus $A^-$ in
$X^-$. Thus they have slope $\phi_-$ or $\phi_-'$. Since $d=1$ and $d'\ne 0$ (cf. Proposition \ref{prop: d'=0}),
each edge of $D$, which is an $\widehat F$-essential loop in $F$, intersects each $c_i$ exactly once and the intersection is transverse. In particular,
in the case that the edges of $D$ have slope $\phi_-'$, we have $d' = 1$. The reader will verify that $A^-$ and $A^+$ can be positioned
to have the properties listed in the statement of Lemma \ref{nice annuli} when $\partial A^-$ has slope $\phi_-$, or
those listed in Lemma \ref{lemma:nice annuli2} when $\partial A^-$ has slope $\phi_-'$.
Hence $M$ is the figure eight knot exterior.

{\bf Case 2}.  $D$ is an algebraic bigon.

The argument is similar to that given in the proof of Proposition \ref{prop: trigon with max weight edge}.
\qed

The proof of Proposition \ref{prop: adjacent trigons} actually shows  the following result.

\begin{cor}
\label{cor: parallel edges intersecting c_i at most twice}
Suppose that $\Delta(\alpha, \beta)=6$, $d=1$, and $X^-$ is a twisted $I$-bundle.
If $\Gamma_F$ has a bigon face in $X^-$ the union of whose two edges are incident to at most
two red tiles, then $M$ is the figure eight knot exterior.
\qed\end{cor}

\begin{prop}\label{prop: 5-gon with adjacent max weight edges}
If $\Delta(\alpha, \beta)=6$, $d=1$, $X^-$ is a twisted $I$-bundle, and
there is  a $5$-gon face $f$ of $\overline \Gamma_F$ with adjacent edges of weight $4$, then $M$ is the figure eight knot exterior.
\end{prop}

\pf Note  that $f$ lies in $X^+$, so we may assume that $f$ is minimally positioned (cf. \S \ref{subsec: trace face}).
Let $e_1$ and $e_2$ be adjacent edges of $f$ in $\Gamma_F$ which have maximal weight. Suppose that we can show that there are at most two red tiles incident to
one of these edges, $e_1$ say. Since $q_+ \geq 3$, no bigon face of $\Gamma_F$ lying in $X^+$ contains a red tile, so we can apply Corollary \ref{cor: parallel edges intersecting c_i at most twice} to the bigon face of $\Gamma_F$ incident to $e_1$ to complete the proof.

Suppose that there are three red tiles incident to each of $e_1$ and $e_2$. Since $q_+ \geq 3$, there are no red tile bigons connecting
$e_1$ to either of its adjacent edges in $\partial f$, and a similar statement holds for $e_2$. It follows that for $i = 1, 2$, a red tile incident to $e_i$
contains a cross-arc in its boundary running between $e_i$ and the edge $e'$ of $f$ which is adjacent to neither $e_1$ nor $e_2$ (cf. Figure \ref{bgz5-5gon}).
By Lemma \ref{parallel cross-arcs}, each cross-arc connecting $e_1$ and $e'$ has the same labels at its endpoints, and the same holds for the cross-arcs connecting $e_2$ and $e'$. But then the restriction of $h$ to each of these arcs is homotopic (rel endpoints) to a path in $F$ and so we can construct a monogon in $X^+$ (cf. Figure \ref{bgz5-5gon}); the sub-disk of $f$ between the cross-arcs $a$ and $a'$ with a corner at $v$
 is a monogon), which gives a contradiction.
\qed

\begin{figure}[!ht]
\centerline{\includegraphics{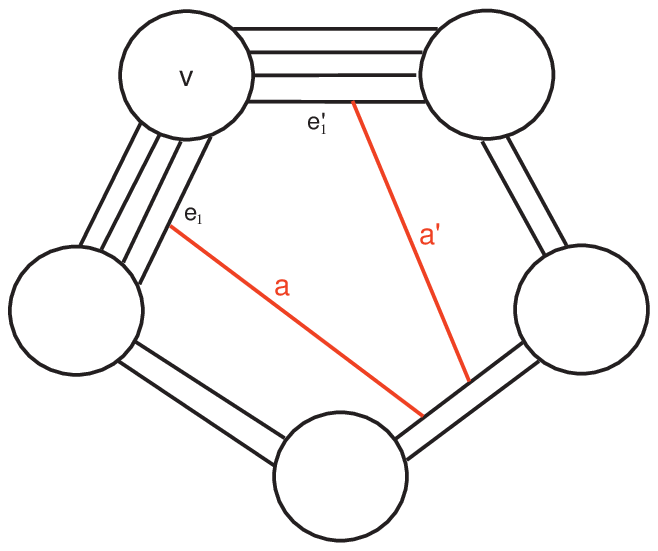}} \caption{ }\label{bgz5-5gon}
\end{figure}

\section{Completion of the proof of Theorem \ref{thm: twice-punctured precise} when $\Delta(\alpha, \beta) \geq 7$}
 \label{sec: Delta 7 or more}

 To complete of the proof of Theorem \ref{thm: twice-punctured precise} when $\Delta(\alpha, \beta) \geq 7$, we must prove the following proposition.

\begin{prop}
\label{prop: delta at least 7}
Suppose that Assumptions \ref{assumptions 0} and \ref{assumptions 1} hold, $F$ is separating but not a semi-fibre, and $t_1^+ = t_1^-  = 0$.
If $\Delta(\alpha, \beta) \geq 7$, then $M$ is the figure eight knot exterior.
\end{prop}

\pf
Since the edges of $\overline{\Gamma}_F$ have weight at most $4$ (Lemma \ref{lemma: edge wts}), for each vertex $v$ of $\Gamma_F$ we have an inequality
$$14 \leq 2 \Delta(\alpha, \beta) = \mbox{valency}_{\Gamma_F}(v) \leq 4 \; \mbox{valency}_{\overline{\Gamma}_F}(v),$$
which shows that the valency of $v$ in $\overline\Gamma_F$ is at least $4$.

First suppose that there is a vertex $v$ of $\Gamma_F$ such that $\mu(v) > 2 \Delta(\alpha, \beta) - 4$. Propoition 12.2 of \cite{BGZ2} then shows that the valency of $v$ is at most $5$ and
$$\varphi_3(v) \geq \left\{\begin{array}{ll} 1 & \mbox{ if valency}(v) = 4 \\ 4 & \mbox{ if valency}(v) = 5 \end{array} \right.$$
Consequently, the condition that $\Delta(\alpha, \beta) \geq 7$ implies that there is a trigon face of $\overline{\G}_F$ incident to $v$ with an edge of maximal weight and therefore $M$ is the figure eight knot exterior by Proposition  \ref{prop: trigon with max weight edge}.

Thus we may assume that $\mu(v) \leq 2 \Delta(\alpha, \beta) - 4$ for each $v$. By Corollary 12.4 of \cite{BGZ2}, $\mu(v) =2 \Delta(\alpha, \beta) - 4$ for each $v$ and then Proposition 12.5 of \cite{BGZ2} implies that if $v$ is a vertex of $\overline{\Gamma}_F$, then either
\vspace{-.2cm}
\begin{itemize}

\item $\mbox{valency}(v) = 4 = \varphi_4(v)$, or

\vspace{.2cm} \item $\mbox{valency}(v) = 5$, $\varphi_3(v) = 3$, and $\varphi_4(v) = 2$, or

\vspace{.2cm} \item $\mbox{valency}(v) = 6 = \varphi_3(v)$.

\end{itemize}
Hence if $\overline{\Gamma}_F$ has a vertex of valency $5$ or $6$, our assumption that $\Delta(\alpha, \beta) \geq 7$ implies that it is incident to a trigon face with an edge of maximal weight, so we are done by Proposition  \ref{prop: trigon with max weight edge}.

Assume then that $\mbox{valency}(v) = 4 = \varphi_4(v)$ for each vertex $v$. Then $\overline{\Gamma}_F$ is rectangular and as the weight of the edges is $4$ or less, $\Delta(\alpha, \beta)$ is either $7$ or $8$. It follows that each vertex $v$ of $\overline{\Gamma}_F$ is incident to at least three edges of weight $4$, so there is a quad face of $\overline{\Gamma}_F$ with adjacent edges of maximal weight. Proposition \ref{prop: quad with adjacent max weight edges}  then shows that $M$ is the figure eight knot exterior.
\qed

\section{Completion of the  proof of Theorem \ref{thm: twice-punctured precise} when $X^-$ is not a twisted $I$-bundle }
\label{sec: delta = 6 not twisted}
In this section we complete of the proof of Theorem \ref{thm: twice-punctured precise} when $X^-$ is not a twisted $I$-bundle.

\begin{prop}
\label{prop: X^- is twisted I bundle}
Suppose that Assumptions \ref{assumptions 0} and \ref{assumptions 1}  hold, $F$ is separating but not a semi-fibre, and $t_1^+ = t_1^-  = 0$.
If $X^-$ is not a twisted $I$-bundle, then $\Delta(\alpha, \beta) \leq 5$.
\end{prop}

\pf
Since $X^-$ is not a twisted $I$-bundle, $M$ is not the figure eight knot exterior.

Suppose that $\D(\alpha,\beta) \geq 6$ and note that $\D(\alpha,\beta)=6$ by Proposition \ref{prop: delta at least 7}.
Proposition \ref{prop: d=0} implies that $d > 0$. The weight of each edge in $\overline\Gamma_F$ is at most $3$ by Lemma \ref{lemma: edge wts}
and therefore $\overline \Gamma_F$ has no valency $3$ vertices and each edge incident to a valency $4$ vertex is of maximal weight.

If there is a vertex $v$ of $\Gamma_F$ such that $\mu(v) > 2 \Delta(\alpha, \beta) - 4$,
Proposition  12.2(1) of \cite{BGZ2} implies that $\overline\Gamma_F$ has a trigon face with an edge of weight $3$.
But then $M$ would be the figure eight knot exterior by Proposition \ref{prop: trigon with max weight edge}, a contradiction.

Corollary  12.4 of \cite{BGZ2} now shows that $\mu(v) = 2 \Delta(\alpha, \beta) - 4$ for each vertex  $v$
and by Proposition  12.5 of that paper, for each vertex $v$ of $\overline{\Gamma}_F$ either
\begin{itemize}

\item $\mbox{valency}(v) = 4 = \varphi_4(v)$, or

\vspace{.2cm} \item $\mbox{valency}(v) = 5$, $\varphi_3(v) = 3$, and $\varphi_4(v) = 2$, or

\vspace{.2cm} \item $\mbox{valency}(v) = 6 = \varphi_3(v)$.

\end{itemize}
Applying Propositions \ref{prop: trigon with max weight edge} and \ref{prop: quad with adjacent max weight edges} we see that
$\overline\Gamma_F$ has no vertices of valency $4$ or $5$. It follows that $\overline{\Gamma}_F$ is hexagonal. Each of its edges has weight $2$
by Proposition \ref{prop: trigon with max weight edge} while Proposition 11.2 of \cite{BGZ2} shows that some edge $\bar e$ of $\overline{\Gamma}_F$ is negative.
Then $d = 1$ by Lemma \ref{lemma: edge wts}(1)(c) and as every trigon face of
$\overline\Gamma_F$ has a positive edge we can apply Proposition \ref{prop: two trgons sharing a positive edge}
to see that $M$ is the figure eight knot exterior. This final contradiction
completes the proof.
 \qed

\section{Completion of the  proof of Theorem \ref{thm: twice-punctured precise} when $\D(\alpha,\beta)=6$ and $d = 1$}
\label{sec: delta = 6 d=1}

We suppose that $F$ is separating, $\D(\alpha,\beta)=6$, $t_1^+ = t_1^-  = 0$, and $d = 1$ in this section.
By Proposition \ref{prop: X^- is twisted I bundle} we can assume that
$X^-$ is a twisted $I$-bundle and therefore each edge of $\overline\Gamma_F$ has weight $2$ or $4$ by Lemma \ref{lemma: edge wts}.
We prove,

\begin{prop}
\label{prop: Delta = 6, d= 1}
Suppose that Assumptions \ref{assumptions 0} and \ref{assumptions 1} hold, $F$ is separating but not a semi-fibre, and $t_1^+ = t_1^-  = 0$.
If $\D(\alpha,\beta)=6$ and $d = 1$, then $M$ is the figure eight knot exterior.
\end{prop}

The proof of this proposition follows from Lemmas \ref{lemma: not homogeneous} and \ref{lemma: homogeneous} below.
In order to set these lemmas up, recall that $0 \leq \chi(Y) = \sum_v \chi(v)$ where the sum is over the vertices of $\overline{\Gamma}_F$ and
$$\chi(v) = 1 - \frac{\mbox{valency}(v)}{2} + \sum_{v \in f} \frac{\chi(f)}{|\partial f|}$$

\begin{lemma}
\label{lemma: not homogeneous}
Suppose that $X^-$ is a twisted $I$-bundle, $\Delta(\alpha, \beta) = 6$,  and $d=1$. If $\overline\Gamma_F$ has a vertex $v$ with $\chi(v)>0$,  then $M$ is the figure eight knot exterior.
\end{lemma}

\pf
Consider a vertex $v$ of $\overline{\Gamma}_F$ such that $\chi(v) > 0$. Then
$$0 < \chi(v) = 1 - \frac{\mbox{valency}(v)}{2} + \sum_{v \in f} \frac{\chi(f)}{|\partial f|} \leq 1 - \frac{\mbox{valency}(v)}{2} + \frac{\mbox{valency}(v)}{3}  \leq 1 - \frac{\mbox{valency}(v)}{6}$$
so that $\mbox{valency}(v) \leq 5$. On the other hand, the assumption that $\Delta(\alpha, \beta) = 6$ implies that the vertices of $\overline{\Gamma}_F$ have valency $3$ or more. Thus the valency of $v$ is either $3, 4$, or $5$.

\begin{case}
$v$ has valency $5$.
\end{case}

The reader will verify that $\chi(v) > 0$ is equivalent to requiring that $\varphi_3(v) \geq 4$. Then each edge incident to $v$ is incident to a trigon face of $\overline{\Gamma}_F$. On the other hand, as $\D(\alpha,\beta)=6$ and $v$ has valency $5$, there is at least one edge incident to $v$
having weight $4$. Proposition \ref{prop: trigon with max weight edge} then shows that $M$ is the figure eight knot exterior.

\begin{case}
$v$ has valency $4$.
\end{case}

Since $\Delta(\alpha, \beta)  = 6$, two edges incident to $v$ have weight $2$ and two have weight $4$. The condition that $\chi(v) > 0$ implies that $\varphi_3(v)\geq 1$.
By Proposition \ref{prop: trigon with max weight edge}, we may assume that $\varphi_3(v)=1$
and the weights of the edges incident to $v$ are $2, 2, 4, 4$ in cyclic order.
The  condition that $\chi(v) > 0$ then implies that either $\varphi_4(v) = 3$ or $\varphi_4(v) = 2$ and $\varphi_5(v) = 1$. In the former case there is a quad face incident to $v$ which has adjacent edges of maximal weight, and so Proposition \ref{prop: quad with adjacent max weight edges} shows that $M$ is the figure eight knot exterior. In the latter case there is either a quad face or a $5$-gon face incident to $v$ with adjacent edges of weight $4$.
We may now apply Propositions \ref{prop: quad with adjacent max weight edges} and \ref{prop: 5-gon with adjacent max weight edges}
to see that $M$ is the figure eight knot exterior.

\begin{case}
$v$ has valency $3$.
\end{case}

Since $\Delta(\alpha, \beta)  = 6$, each of the edges incident to $v$ has weight $4$.
The condition that $\chi(v) > 0$ shows that at least one of the faces incident to $v$ is a $3$-gon or  $4$-gon or $5$-gon.
So we may apply Propositions   \ref{prop: trigon with max weight edge}, \ref{prop: quad with adjacent max weight edges}
and \ref{prop: 5-gon with adjacent max weight edges} respectively
to see that $M$ is the figure eight knot exterior.
\qed

\begin{lemma}
\label{lemma: homogeneous}
Suppose that $X^-$ is a twisted $I$-bundle, $\Delta(\alpha, \beta) = 6$, and $d=1$. If $\chi(v) \leq 0$ for each vertex $v$ of $\overline\Gamma_F$,  then $M$ is the figure eight knot exterior.
\end{lemma}

\pf
Since $0 \leq \chi(Y) = \sum_v \chi(v)$, we have $\chi(v)=0$ for each vertex $v$.
One can easily  verify that for any vertex $v$ with  $\chi(v)=0$, one of the following four cases
holds:
\vspace{-.2cm}
\begin{itemize}

\item $\mbox{valency}(v) = 3$ and the number of edges of the three faces incident to $v$ are $(4,6, 12), (4,8,8),$ $(5,5,10)$, or $(6,6,6)$;

\vspace{.2cm} \item $\mbox{valency}(v) = 4 = \varphi_4(v)$;

\vspace{.2cm} \item $\mbox{valency}(v) = 5$, $\varphi_3(v) = 3$, and $\varphi_4(v) = 2$;

\vspace{.2cm} \item $\mbox{valency}(v) = 6 = \varphi_3(v)$.

\end{itemize}
\vspace{-.2cm}
By Proposition \ref{prop: adjacent trigons}, we may assume that $\overline\Gamma_F$ has no vertices of valency $5$ or $6$.

Suppose that $\overline\Gamma_F$ has a vertex $v$ of valency $3$. Then each of the edges incident to $v$ have weight $4$, so
by Propositions \ref{prop: trigon with max weight edge}, \ref{prop: quad with adjacent max weight edges}
and \ref{prop: 5-gon with adjacent max weight edges} we may assume that $\varphi_6(v)=3$.
It follows that $v$ cannot be connected to a valency $4$ vertex by an edge of $\overline\Gamma_F$. In
particular, the vertices of $\overline{\Gamma}_F$ which cobound an edge with $v$ have valency $3$.
It follows that if $f$ is a $6$-gon faces incident to $v$, each of its edges has weight $4$.

Since $q_+ \geq 3$, no bigon face of $\Gamma_F$ lying in $X^+$ contains a red tile and there are no red tile bigons connecting
adjacent edges of $f$. It follows that if an edge $e$ of $f$ has at least two red tiles incident to it, each of these tiles has a
cross-arc in its boundary connecting $e$ to another edge of $f$. But then it's easy to see that some edge of $f$ is incident to
at most one red tile. Corollary  \ref{cor: parallel edges intersecting c_i at most twice} then shows that $M$ is the figure eight knot exterior.

Finally suppose that $\overline\Gamma_F$ has no valency $3$ vertices. In this case, $\overline\Gamma_F$ is rectangular
(\cite[Proposition 11.5]{BGZ2}) and the weights of edges at any of its vertices alternate $2$, $4$, $2$, $4$
(Proposition \ref{prop: quad with adjacent max weight edges}). By Corollary \ref{cor: parallel edges intersecting c_i at most twice} we may assume that every edge of weight $4$ intersects cross-arcs and thus there are no cross-arcs incident to a weight $2$ edge. It follows that the two edges of a bigon face corresponding to a weight $2$ edge is incident to at most two red tiles. Corollary \ref{cor: parallel edges intersecting c_i at most twice} then implies that $M$ is the figure 8 exterior.
\qed

\section{The case that $F$ separates but not a semi-fibre, $t_1^+ = t_1^- = 0$, $d \ne 1$, and $M(\alpha)$ is very small}
\label{sec: sep not semifibre very small}
In this section we suppose that $F$ is separating, though not a semi-fibre, $t_1^+ = t_1^- = 0$, $d \ne 1$, and $M(\alpha)$ is a small Seifert manifold which is very small.
We use character variety methods to prove the following proposition.

\begin{prop}
\label{twice-punctured very small t+t-=0}
Suppose that Assumptions \ref{assumptions 0} and \ref{assumptions 1} hold, $F$ is separating but not a semi-fibre, $t_1^+ = t_1^-  = 0$, and $d \ne 1$.
If $M(\alpha)$ is a very small Seifert manifold, then
$$\Delta(\alpha, \beta) \leq \left\{
\begin{array}{ll} 1 & \hbox{if $d = 0$ or $M(\alpha)$ is of $C$-type} \\
2 & \hbox{if $\alpha$ is of $D$-type} \\
3 & \hbox{if $\alpha$ is of $T$-type, $O$-type or $I$-type, or $(a,b,c) = (2,3,6), (2,4,4)$ or $(3,3,3)$}
\end{array} \right.$$
\end{prop}

The case that $d = 0$ is handled in Proposition \ref{prop: d=0}, so we suppose that $d \geq 2$ below.

Recall that $\widehat X^\epsilon$ has base orbifold $D^2(p_\epsilon, q_\epsilon)$ where $2 \leq p_\epsilon \leq q_\epsilon$. Write $\pi_1(D^2(p_\epsilon, q_\epsilon)) \cong \mathbb Z/p_\epsilon * \mathbb Z/q_\epsilon$ and choose generators  $a_\epsilon$ of $\mathbb Z/p_\epsilon$ and $b_\epsilon$ of $\mathbb Z/q_\epsilon$ for which $a_\epsilon b_\epsilon$ generates $\pi_1(\partial D^2(p_\epsilon, q_\epsilon))$. There is an epimorphism
$$\pi_1(M(\beta)) \stackrel{\varphi}{\longrightarrow} \Delta(p_+, q_+, d) *_{\psi} \Delta(p_-, q_-, d)$$
where $\varphi$ is the quotient by $\langle \langle \phi_+, \phi_- \rangle \rangle_{\pi_1(M(\beta))}$ and $\psi: \varphi(\pi_1(\widehat F)) \to \varphi(\pi_1(\widehat F))$ is the isomorphism determined by the gluing map $\partial \widehat X^+ \to \partial \widehat X^-$. The reader will verify that for either $\epsilon$, $\varphi(\pi_1(\widehat F))$ is the copy $\mathbb Z/d$ contained in $\Delta(p_\e, q_\e, d)$ corresponding to $(a_\e b_\e)^d$.

If $\beta^* \in \pi_1(\partial M)$ is a dual class to $\beta$, then  $\varphi(\beta^*) = a_+(a_+b_+)^k a_-$ for some integer $k$.

\begin{lemma}
\label{lemma: bending curves}
Suppose that $t_1^+ = t_1^- =  0$ and $d \geq 2$. If $d' > 1$ is a divisor of $d$, there is a non-trivial curve $X_{d'} \subset X_{PSL_2}(M(\beta)) \subset X_{PSL_2}(M)$ such that

$(1)$ $\tilde f_{\beta^*}$ has a pole at each ideal point of $X_{d'}$;

$(2)(a)$ $X_{d'}$ has one ideal point if $(p_-, q_-, d') = (2,2,2)$, and therefore $s_{X_{d'}} \geq 1$;

$(b)$ $X_{d'}$ has two ideal points if $(p_-, q_-, d') \ne (2,2,2)$, and therefore $s_{X_{d'}} \geq 2$.

$(3)$ $X_{d'}$ is not strictly non-trivial if and only if $(p_+, q_+, d') = (2,4,4)$ and $(p_-, q_-, d')$ is either $(2,2,4)$ or $(2,4,4)$.

\end{lemma}

\begin{proof}
If $d' > 1$ is a positive integer dividing $d$ we can use the surjective composition of homomorphisms
$$\pi_1(M(\beta)) \xrightarrow{\; \varphi \;} \Delta(p_+, q_+, d) *_{\mathbb Z / d} \Delta(p_-, q_-, d) \to \Delta(p_+, q_+, d') *_{\mathbb Z / d'} \Delta(p_-, q_-, d')$$
to construct a curve $X_{d'} \subset X_{PSL_2}(M(\beta)) \subset X_{PSL_2}(M)$. Lemma \ref{lemma: bending homs sep}(3) shows that assertion (1) of the lemma holds and further, that $X_{d'}$ has two ideal points unless $(p_\epsilon, q_\epsilon, d) = (2,2,2)$ for some $\epsilon$, and otherwise one. Since $(p_+, q_+) \ne (2,2)$, assertion (2) of the lemma holds. Lemma \ref{lemma: bending homs sep}(2) shows that $X_{d'}$ is not strictly non-trivial if and only if $(p_\epsilon, q_\epsilon, d') \in \{(2,2,d'), (2,4,4)\}$ for both $\epsilon$. Since $(p_+, q_+) \ne (2,2)$, the latter condition is equivalent to $(p_+, q_+, d') = (2,4,4)$ and $(p_-, q_-, d') \in \{(2,2,4), (2,4,4)\}$. Thus assertion (3) of the lemma holds.
\end{proof}

\begin{proof}[Proof of Proposition \ref{twice-punctured very small t+t-=0}]
As mentioned above, the case $d = 0$ is handled by Proposition \ref{prop: d=0}. Suppose then that $d \geq 2$.

If $d' > 1$ is a divisor of $d$ and $X_{d'} \subset X_{PSL_2}(M)$ is a non-trivial curve obtained as in Lemma \ref{lemma: bending curves}, then
$$\Delta(\alpha, \beta) s_{X_{d'}} = \|\alpha\|_{X_{d'}}$$
by (\ref{seminorm distance}) and further, $s_{X_{d'}} \ne 0$. Hence if $\pi_1(M(\alpha))$ is cyclic, then $\|\alpha\|_{X_{d'}} = s_{X_{d'}}$ by \cite[Proposition 8.1]{BCSZ2}, so $\Delta(\alpha, \beta) = 1$.

Next suppose that $M(\alpha)$ is a prism manifold. That is, $\alpha$ is of $D$-type. If $X_{d'}$ is strictly non-trivial, then
$\Delta(\alpha, \beta) s_{X_{d'}} = \|\alpha\|_{X} \leq 2 s_{X_{d'}}$ (\cite[Proposition 8.1]{BCSZ2}), so $\Delta(\alpha, \beta) \leq 2$ .
On the other hand, if $X_{d'}$ is not strictly non-trivial, then $(p_+, q_+, d') = (2,4,4)$ and $(p_-, q_-, d')$ is $(2,2,4)$ or $(2,4,4)$ by Lemma \ref{lemma: bending curves}(3). In either case, $X_2$ is a strictly non-trivial curve, so $\Delta(\alpha, \beta) \leq 2$ as above.

Finally suppose that $M(\alpha)$ is neither $C$-type nor $D$-type and recall from the proof of Proposition \ref{twice-punctured very small t+=0} that if $\rho: \pi_1(M(\alpha)) \to PSL_2(\mathbb C)$ is an irreducible representation, then
$$\mbox{image}(\rho) \cong \left\{
\begin{array}{ll}
T_{12} & \mbox{ if $\alpha$ is of $T$-type or $(a,b,c) = (3,3,3)$} \\
D_3 \mbox{ or } O_{24}& \mbox{ if $\alpha$ is of $O$-type} \\
I_{60}& \mbox{ if $\alpha$ is of $I$-type} \\
D_3 \mbox{ or } T_{12} & \mbox{ if $(a,b,c) = (2,3,6)$} \\
D_2 \mbox{ or } D_4 & \mbox{ if $(a,b,c) = (2,4,4)$}
\end{array} \right.
$$
Further, in each of these five possibilities at most two such representations have isomorphic images and if two, either $\alpha$ has $I$-type and the image is $I_{60}$ or $(a,b,c) = (2,4,4)$ and the image is $D_4$.
Note, in particular, that the image of $\rho$ is finite with elements of order at most $5$.

Suppose that $(p_-, q_-, d') \ne (2,2,2)$, so that $s_{X_{d'}} \geq 2$ and no representation with character lying on $X_{d'}$ has image isomorphic to $D_2$.
If $X_{d'}$ is strictly non-trivial, we can apply Proposition \ref{prop: boundary values}(2) to see that $\Delta(\alpha, \beta) \leq 2$ when $\alpha$ is of $T$-type, $O$-type or $(a, b, c)$ is either $(2,3,6), (2,4,4)$, or $(3,3,3)$ and that $\Delta(\alpha, \beta) \leq 3$ when it is of $I$-type.

If $(p_-, q_-, d') \ne (2,2,2)$ and $X_{d'}$ is not strictly non-trivial, then as in the prism manifold case we know that $(p_+, q_+, d') = (2,4,4)$ and $(p_-, q_-, d')$ is either $(2,2,4)$ or $(2,4,4)$. In either case, $X_2$ is a strictly non-trivial curve with $s_{X_2} \geq 2$ and as in the previous paragraph, $\Delta(\alpha, \beta) \leq 2$ when $\alpha$ is of $T$-type, $O$-type or $(a, b, c)$ is either $(2, 3, 6), (2, 4, 4)$, or $(3,3,3)$, and that $\Delta(\alpha, \beta) \leq 3$ when it is of $I$-type.

Next suppose that $(p_-, q_-, d') = (2,2,2)$. In this case $X_{d'}=X_2$ is strictly non-trivial and the image of each representation whose character lies on $X_2$ contains a copy of $\Delta(2,2,2)$ as a proper subgroup, which excludes the possibility that the image is $D_2$ or $D_3$.  It then follows from Proposition \ref{prop: boundary values}(2) that $\Delta(\alpha, \beta) \leq 3$ when $\alpha$ is of $T$-type or $O$-type or $(a, b, c)$ is either $(2, 3, 6), (2, 4, 4)$ or $(3,3,3)$.

Finally suppose that $\alpha$ has $I$-type. Proposition \ref{prop: boundary values} shows that to obtain the inequality $\Delta(\alpha, \beta) \leq 3$, it suffices to prove $|J_{X_2}(\alpha)| \leq 1$.

Suppose then that $|J_{X_2}(\alpha)| = 2$. There cannot be another non-trivial curve $X_*\subset X_{PSL_2}(M(\beta)) \subset X_{PSL_2}(M)$ as otherwise $|J_{X_*}(\alpha)| = \emptyset$, so $\Delta(\alpha, \beta) = 1$ (Proposition \ref{prop: boundary values}). In particular, since
killing $\pi_1(\widehat F)$ in $\pi_1(M(\beta))$ yields $\mathbb Z/  \gcd(p_+, q_+)  * \mathbb Z/  2 $, we must have $\gcd(p_+, q_+) = 1$.
Also, the discussion in \S \ref{subsec: prod cyclics} shows that $(p_+, q_+) = (2,3)$, so $\widehat X^+$ is a trefoil exterior and $(p_+, q_+, d) = (2,3, 2)$.
We claim that there are no $I$-type filling slopes in this situation.

The assumption that $(p_-, q_-, d') = (2,2,2)$ implies that $X^-$ is a twisted $I$-bundle (\cite[Proposition 7.5]{BGZ2})
and therefore $X^-$ is a twisted $I$-bundle over the Klein bottle. Then $H_1(\widehat X^-) \cong \mathbb Z \oplus \mathbb Z/2$ where the $\mathbb Z$ factor is generated by a class $\xi$ such that $2 \xi = [\phi_-]$ and the $\mathbb Z/2$ factor is generated by $[\phi_-']$.

Now if $[\phi_+] = m [\phi_-] + n [\phi_-']$, then $|n| = \Delta(\phi_+, \phi_-) = d = 2$, so $[\phi_+] = 2m \xi \pm 2 [\phi_-'] \in H_1(\widehat X^-)$. Hence $[\phi_+]$ maps to zero in $H_1(\widehat X^-; \mathbb Z/2)$. On the other hand, as the fibre class in a trefoil knot exterior, $[\phi_+]$ maps to zero in $H_1(\widehat X^+; \mathbb Z/2)$. Consideration of the Mayer-Vietoris sequence for $M(\beta) =  \widehat X^+ \cup \widehat X^-$ then shows that the first (mod $2$) Betti number of $M(\beta)$ is $2$. Hence that of $M$ is at least $2$, and so that of $M(\alpha)$ is at least $1$. In particular, $\alpha$ cannot be an $I$-type filling slope (cf. \cite[page 117]{BZ1}), contrary to our assumptions.

We conclude that $|J_{X_2}(\alpha)| \leq 1$ and therefore $\Delta(\alpha, \beta) \leq 3$.
\end{proof}

\section{The case that $F$ separates but is not a semi-fibre, $t_1^+ = t_1^- = 0$, $d>1$, and $M(\alpha)$ is not very small}
\label{sec: sep not semifibre not very small}

In this final section we complete the proof of Theorem \ref{thm: twice-punctured precise} with the following proposition.

\begin{prop}
\label{twice-punctured not very small t+t-=0}
Suppose that Assumptions \ref{assumptions 0} and \ref{assumptions 1} hold, $F$ is separating but not a semi-fibre, $t_1^+ = t_1^-  = 0$, and $d > 1$.
If $M(\alpha)$ is a small Seifert manifold but not very small, then $\Delta(\alpha, \beta) \leq 5$.
\end{prop}
We suppose that the hypotheses of the proposition hold for the remainder of this section.
Since $d > 1$, $M$ is not the figure eight knot exterior, and so $\Delta(\alpha,\beta) \leq 6$ by Proposition \ref{prop: delta at least 7}.
We assume below that $\Delta(\alpha,\beta)=6$ in order to derive a contradiction. Proposition \ref{prop: X^- is twisted I bundle} then implies that
$X^-$ is a twisted $I$-bundle.

\subsection{Involutions on $M$ and its fillings}
\label{subsec: involutions}

For each $\epsilon \in \{+,-\}$, recall the essential annulus $A^{\epsilon}$ which separates $\widehat X^{\e}$ into a union of two solid tori which contains $k_\e = K_\b\cap \widehat X^{\e}$
as an essential arc. Recall as well that $\widehat X^+$ is Seifert fibred over $D^2(p_+, q_+)$ with
$p_+ \geq 3$ and $q_+ \geq 2$, and $\widehat X^-$ is Seifert fibred over $D^2(2, 2)$.

\begin{figure}[!ht]
\centerline{\includegraphics{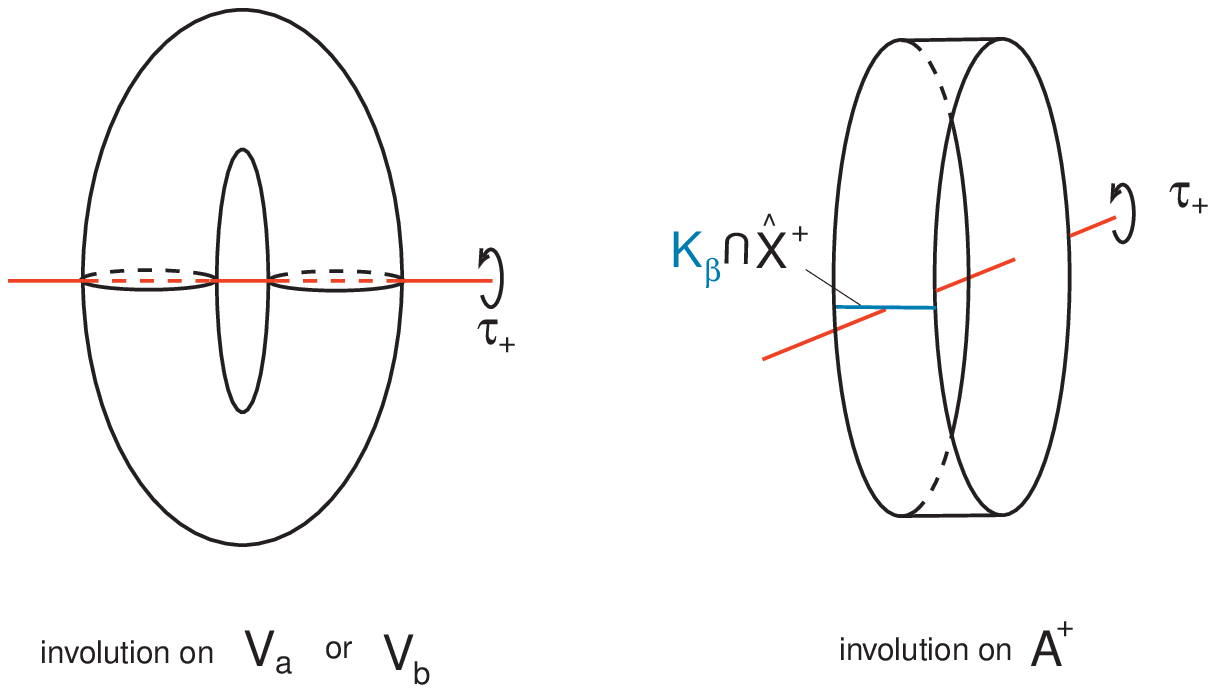}} \caption{ }\label{involution}
\end{figure}

Let $V_a$ and $V_b$ be the two solid tori which arise as the components of $\widehat X^+ \setminus A^+$ where $A^+$  has winding number $p_+$ in $V_a$ and $q_+$ in $V_b$.

There is an involution $\tau_+$ on $\widehat X^+$ under which each of $V_a,V_b, A^+, k_+$ is invariant. More precisely,
\vspace{-.7cm}
\begin{itemize}

\item the restriction of $\tau_+$ to both $V_a$ and $V_b$ is a standard involution of a solid torus whose fixed point set is a pair of arcs each contained in
a meridional disk of the solid torus;

\vspace{.2cm} \item the restriction of $\tau_+$ to $A^+$ is a $\pi$-rotation with two fixed points (cf. Figure \ref{involution});

\vspace{.2cm} \item the restriction of $\tau_+$ to $k_+$ is a rotation about a fixed point (cf. Figure \ref{involution});

\vspace{.2cm} \item the restriction of $\tau_+$ to $\widehat F$ is a hyperelliptic involution which exchanges the two points $\partial k_+$;

\vspace{.2cm} \item the quotient space of $\widehat X^+$ under $\tau_+$ is a $3$-ball whose branched set is shown
in Figure \ref{involution2}(1).

\end{itemize}
There is an analogous involution $\tau_-$ on $\widehat X^-$ with quotient space as shown in Figure \ref{involution2}(2), at least up to homeomorphism.

\begin{figure}[!ht]
\centerline{\includegraphics{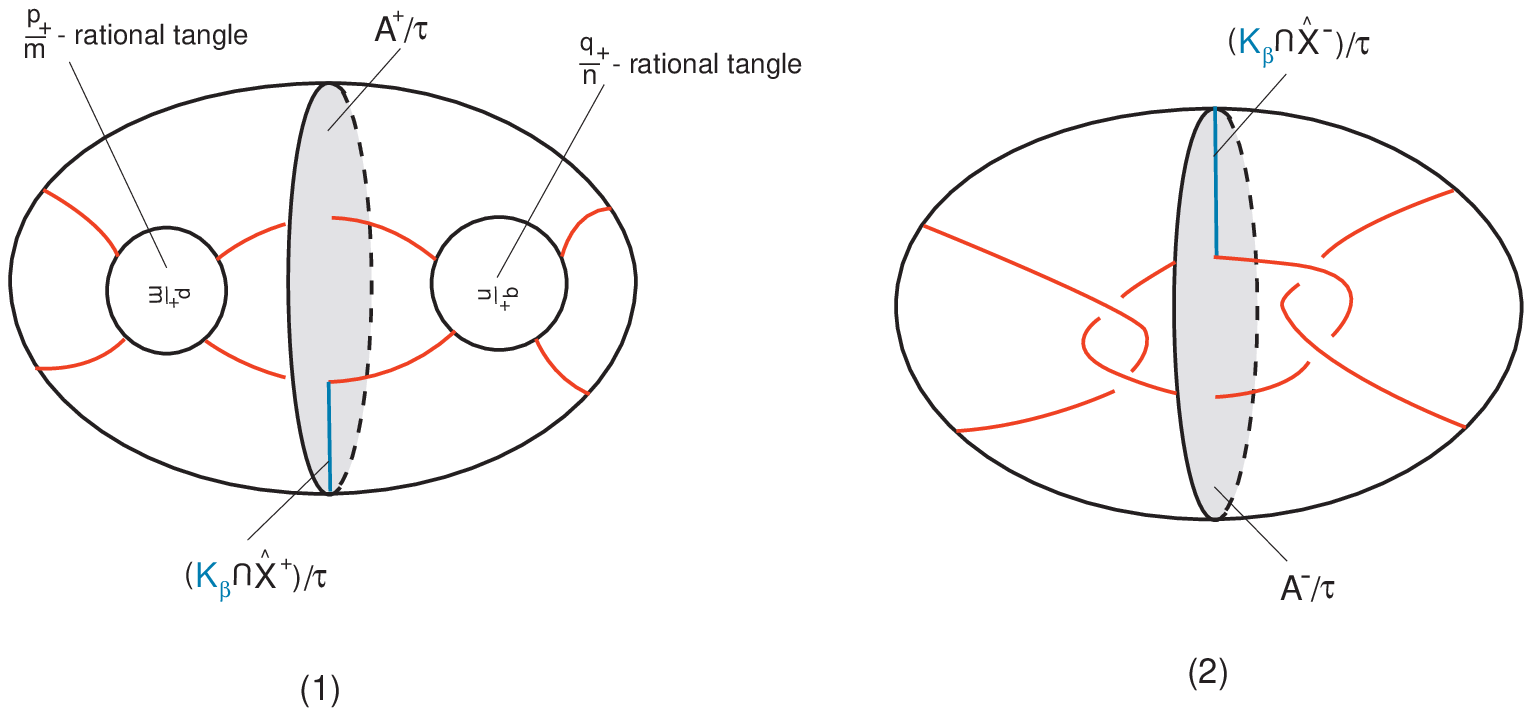}} \caption{ }\label{involution2}
\end{figure}

Let $f: \partial \widehat X^+ \to \partial \widehat X^-$ be the gluing map yielding $M(\beta) = \widehat X^+ \cup \widehat X^-$. Then
as both $\tau_-|_{\widehat F}$ and $f \circ \tau_+|_{\widehat F} \circ f^{-1}$ are hyperelliptic involutions exchanging the two points
$\partial k_+ = \partial k_-$, there is a homeomorphism $g$ of $\widehat X^-$ isotopic (rel $\partial k_-$) to the identity such that
$$f \circ (\tau_+|_{\partial  \widehat X^+}) \circ f^{-1} = (g \circ \tau_- \circ g^{-1})|_{\partial  \widehat X^-}$$
Hence, up to replacing $\tau_-$ by $g \circ \tau_- \circ g^{-1}$, we may assume that the restriction of $\tau_+$ to $(\partial  \widehat X^+, k_+)$
equals that of $\tau_-$ to $(\partial  \widehat X^-, k_-)$. Let
$$\tau_\beta: (M(\beta), K_\beta) \to (M(\beta), K_\beta)$$
be the involution obtained by gluing
$\tau_+$ and $\tau_-$ together along $\widehat F$.

We may assume that the filling solid torus $V_\b$ is  $\tau_\beta$-invariant
and thus obtain an involution $\t$ of $M=M(\beta)\setminus V_\b$.
The quotient spaces $M(\beta)/\tau_\beta$ and $M/\t$ and their branched sets are as shown in Figure \ref{involution3}.
The involution $\t$ extends to an involution $\t_\g$ on any $\g$-filling of $M$ along $\partial M$ and the branched set
in $M(\g)/\t_\g$ is obtained by filling $M/\t$ with a corresponding rational tangle along $\partial M/\t$.
It follows that $M(\g)$ is a $2$-fold cover of $S^3$ branched along a link $L_\g$.

\begin{figure}[!ht]
\centerline{\includegraphics{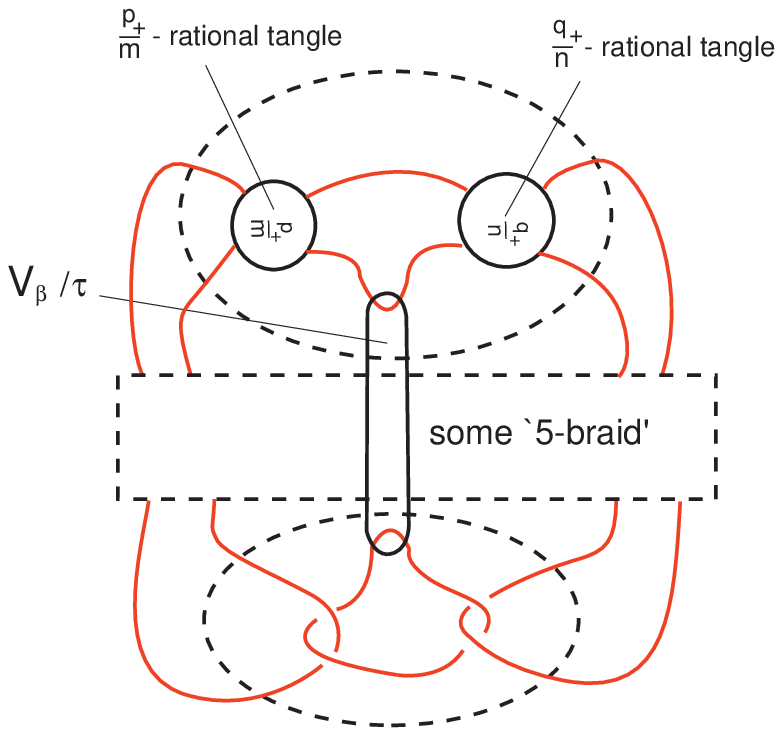}} \caption{ }\label{involution3}
\end{figure}

\subsection{Branched sets of the involutions on small Seifert manifolds}
Suppose that $W$ is Seifert fibred with base orbifold $S^2(a,b,c)$ where $a,b,c \geq 2$,
but is not a prism manifold, which is the $2$-fold cover of $S^3$ branched over a link $L$.
By the orbifold theorem we may assume that the associated covering involution, $\theta$ say,
preserves the Seifert structure on $W$. We distinguish two cases. If $\theta$ reverses the
orientations of the Seifert fibres of $W$ then $L$ is a Montesinos link (\cite{Mo}). If it preserves their
orientations, the Seifert structure on $W$ quotients to one on $S^3$ for which $L$ is a union of fibres.
In the former case we say that $L$ is a Montesinos link $L$ of {\it type} $(a,b,c)$. In the latter case we say that $L$ is a
{\it Seifert link}.

Our goal in this subsection is to describe more precisely the possibilities for $L$ when it is a Seifert link. We begin with some standard examples
of Seifert links and the base orbifolds of their branched covers.

\begin{exa}
{\rm
If $L$ is a $(p,q)$ torus knot, then its $2$-fold branched cover has base orbifold
$$S^2(a,b,c) \cong \left\{ \begin{array}{ll} S^2(2,p,q) & \hbox{if $pq$ is odd} \\
S^2(\frac{p}{2} ,q,q) & \hbox{if $p$ is even} \\  S^2(\frac{q}{2} ,p,p) & \hbox{if $q$ is even}\end{array}
\right.$$
}
\end{exa}

\begin{exa}
{\rm
If $L=U \cup K \subset S^3$ is a two component link where $U$ is an unknot and $K$ is a $(p,q)$-cable of $U$ (i.e. $K$
is isotopic in $S^3 - U$ into the boundary of a regular neighbourhood $N(U)$ of $U$ as a $(p,q)$-torus knot with winding number $q$
in $N(U)$), then we say $K$ is a {\it $(p,q)$-torus knot with respect to $U$}. In this case, the linking number between $U$ and $K$ is $p$. The $2$-fold branched cover of $L$ has base orbifold
$$S^2(a,b,c) \cong \left\{ \begin{array}{ll} S^2(q,p,p) & \hbox{if $q$ is odd} \\
S^2(2,p,2q) & \hbox{if $q$ is even}\end{array} \right.$$
}
\end{exa}

\begin{exa}
{\rm
If $L=U_0\cup U_1\cup K \subset S^3$ is a three component link where $U_0 \cup U_1$
is a Hopf link and $K$ is a $(p,q)$-torus knot which is isotopic in $S^3 - (U_0 \cup U_1)$
into  the boundary of a regular neighbourhood $N(U_0)$ of $U_0$,  then we say that $K$ is a {\it $(p,q)$-torus knot with respect to
$U_0\cup U_1$}. The $2$-fold branched cover of $L$ has base orbifold
$$S^2(a,b,c) \cong S^2(2,2p,2q) $$
}
\end{exa}

The following lemma is a consequence of \cite[Lemma 4.3]{BGZ3} and
the list of the base orbifolds of the $2$-fold branched covers of the Seifert links in the three examples above.

\begin{lemma}
\label{lemma: branched link}
Let $W$ and $\theta$ be as above. Suppose that $\theta$ preserves the orientations of the Seifert fibres of $W$ and endow $W/\t$
with the induced Seifert structure. Assume as well that $W/\theta=S^3$ and let $L_\theta \subset W/\t \cong S^3$ be the branched set of $\theta$.

$(1)$ $L_\theta$ is a union of at most three Seifert fibres of $S^3$. Moreover at least one component of $L_\theta$ is a regular fibre.

$(2)$ If $|L_\theta|=1$, then either
   \newline
\indent \hspace{4.5mm} $(i)$ $(a,b,c)=(2, p,q)$ where  $p$ and $q$ are odd, $\gcd(p,q)=1$, and  $L_\theta$ is a
 $(p,q)$ torus knot, or
   \newline
\indent \hspace{4.5mm} $(ii)$ $(a,b,c)=(p/2, q, q)$ where $p$ is even, $\gcd(p, q)=1$, and $L_\theta$ is a
   $(p,q)$-torus knot.

$(3)$ If $|L_\theta|=2$, then either \newline
\indent \hspace{4.5mm} $(i)$ $(a,b,c)=(2, p, 2q)$ where $q$ is even,  $\gcd(p,q)=1$, and
  $L_\theta$ is the union of an unknot $U$ \\
   \indent \hspace{1cm} and a $(p,q)$-torus knot with respect to $U$,
  or \newline
\indent \hspace{4.5mm} $(ii)$ $(a,b,c)=(q,p,p)$ where $q$ is odd, $\gcd(p,q)=1$, and $L_\theta$ is the union of an unknot $U$ \\
   \indent \hspace{1cm} and
   a $(p,q)$-torus knot with respect to $U$.

$(4)$ If $|L_\theta|=3$, then $(a,b,c)=(2,2p, 2q)$ where $\gcd(p,q)=1$ and $L_\theta$ is the union of
 a Hopf link and a $(p,q)$-torus knot $K$ with respect to $L_\theta - K$.
   \qed
\end{lemma}

\subsection{The proof of Proposition \ref{twice-punctured not very small t+t-=0}}
Suppose that $M(\alpha)$ is Seifert fibred with base orbifold $S^2(a,b,c)$ where $a,b,c \geq 2$,
but is not a prism manifold. We know that with respect to the involution $\t_\a$ constructed in \S \ref{subsec: involutions},
$M(\alpha)$ is a $2$-fold branched cover of $(S^3, L_\alpha)$ where $L_\a$ is either
a Montesinos link or a link whose exterior is Seifert fibred.

As $\D(\alpha,\beta) = 6$ is even,  $L_\alpha$ contains a trivial component $U_0$, as shown in Figure \ref{involution5}(1),
where $L_\a - U_0$ is a Montesinos link whose $2$-fold branched cover is obtained by
Dehn filling $\widehat X^+$ with a slope of distance $d$ from $\phi_+$.
Thus $L_\a - U_0$ is a Montesinos link of type $(p_+, q_+,d)$ (cf. Figure \ref{involution5}(2)).
As $d > 1$, the fundamental group of the $2$-fold branched cover of $L_\a - U_0$ is not cyclic, so
$L_\a - U_0$ is not a $2$-bridge link, and therefore $L_\a$
cannot be a Montesinos link of three cyclically composed rational tangles. Hence $L_\a$ must be a Seifert link of two or three components. In particular,
$\tau_\alpha$ can be assumed to preserve the orientations of the Seifert fibres of $M(\alpha)$, so that Lemma \ref{lemma: branched link} applies to our situation. Then $L_\a - U_0$ has either one or two components.

\setcounter{case}{0}

\begin{case}
$K = L_\a - U_0$ is a knot
\end{case}

If $K = L_\a - U_0$ is a knot, it is both a torus knot and a Montesinos knot of type $(p_+, q_+,d)$. Therefore
$K$ is either the $(-2, 3, 3)$-pretzel knot or the $(-2,3,5)$-pretzel knot
(\cite{Oe}). Equivalently,
$K$ is either the $(3,4)$-torus knot or, respectively, the $(3,5)$-torus knot. Further,
the linking number between $U_0$ and $K$ is an odd number (see Figure \ref{involution5}(2) and
apply \cite[Remark 12.7]{BuZi}).

First suppose that $K=T(3,4)$. Then $K$ is a $(3,4)$-torus knot with respect to $U_0$.
Lemma \ref{lemma: branched link} then shows that $M(\alpha)$ has base orbifold $S^2(2,8,3)$. On the other hand,
the $2$-fold branched cover of $K$ has base orbifold $S^2(p_+, q_+, d) \cong S^2(2,3,3)$, so $(p_+, q_+,d)$ is either $(2,3,3)$ or
$(3,3,2)$.

If $(p_+, q_+,d) = (2,3,3)$, there is an epimorphism $\pi_1(M(\beta)) \stackrel{\varphi}{\longrightarrow} \Delta(2,3, 3) *_{\psi} \Delta(2, 2, 3)$ and as in Lemma \ref{lemma: bending curves}, we can build a strictly non-trivial curve $X_0$ in $X_{PSL_2}(M(\beta)) \subset X_{PSL_2}(M)$ which contains only irreducible characters and for which $s_{X_0} \geq 2$. Further, $\tilde f_{\alpha}$ has a pole at each of ideal point of $X_0$. By \cite[Proposition 3.2]{BeBo}, $\Delta(2,8,3)$ has exactly four irreducible characters and therefore Proposition \ref{prop: boundary values}(2) implies that $\Delta(\alpha, \beta) \leq 1 + \frac{8}{2} = 5$, contrary to our assumptions.

If $(p_+, q_+,d) = (3,3,2)$, there is an epimorphism $\pi_1(M(\beta)) \stackrel{\varphi}{\longrightarrow} \Delta(3,3, 2) *_{\psi} \Delta(2, 2, 2)$ and as above we can build a strictly non-trivial curve $X_0$ in $X_{PSL_2}(M(\beta)) \subset X_{PSL_2}(M)$ which contains only irreducible characters and for which $s_{X_0} \geq 1$. Further, $\tilde f_{\alpha}$ has a pole at each of ideal point of $X_0$. Two of the four irreducible characters of $\Delta(2,8,3)$ correspond to the discrete faithful representation $\Delta(2,8,3) \to PSL_2(\mathbb R) \to PSL_2(\mathbb C)$ and to the quotient map to $\Delta(2,2,3)$. Neither of these can lie on $X_0$, so Proposition \ref{prop: boundary values}(2) implies that $\Delta(\alpha, \beta) \leq 1 + \frac{4}{1} = 5$, contrary to our assumptions. Thus $K \ne T(3,4)$.

Next suppose that $K=T(3,5)$. By Lemma \ref{lemma: branched link}, the base orbifold of $M(\alpha)$ is either $S^2(3,3,5)$ or $S^2(3,5,5)$. Further, the base orbifold of the $2$-fold branched cover of $K$ is $S^2(2,3,5)$, so $(p_+, q_+,d)=(2,3,5)$ or $(2,5,3)$ or $(3,5,2)$. In any event, we can construct a strictly non-trivial curve $X_0 \subset X_{PSL_2}(M)$ as above. In fact, since there are two conjugacy classes of elements of order $5$ in $PSL_2(\mathbb C)$, we can construct disjoint, strictly non-trivial curve $X_0, X_1 \subset X_{PSL_2}(M)$. Since $\Delta(3,3,5)$ and $\Delta(3,5,5)$ are generated by elements of odd order,
no representation whose character lies on these curves can have image in $\mathcal{N}$. Hence if $s_{X_j} = 1$ for some $j$, Proposition \ref{lemma: factors}(3)  implies that $\Delta(\alpha, \beta)$ is odd, contrary to assumption. Thus $s_{X_j} \geq 2$ for both $j$. Set $X = X_0 \cup X_1$. Then $\|\cdot\|_X = \|\cdot\|_{X_0} + \|\cdot\|_{X_1}$ so that $s_X = s_{X_0} + s_{X_1} \geq 4$. Since $\Delta(3,3,5)$ has exactly four irreducible characters and $\Delta(3,5,5)$ exactly eight (\cite[Proposition 3.2]{BeBo}), we see that $\Delta(\alpha, \beta) \leq 1 + \frac{16}{4} = 5$, a contradiction. We have therefore ruled out the possibility that $L_\a - U_0$ is a knot.

\begin{figure}[!ht]
\centerline{\includegraphics{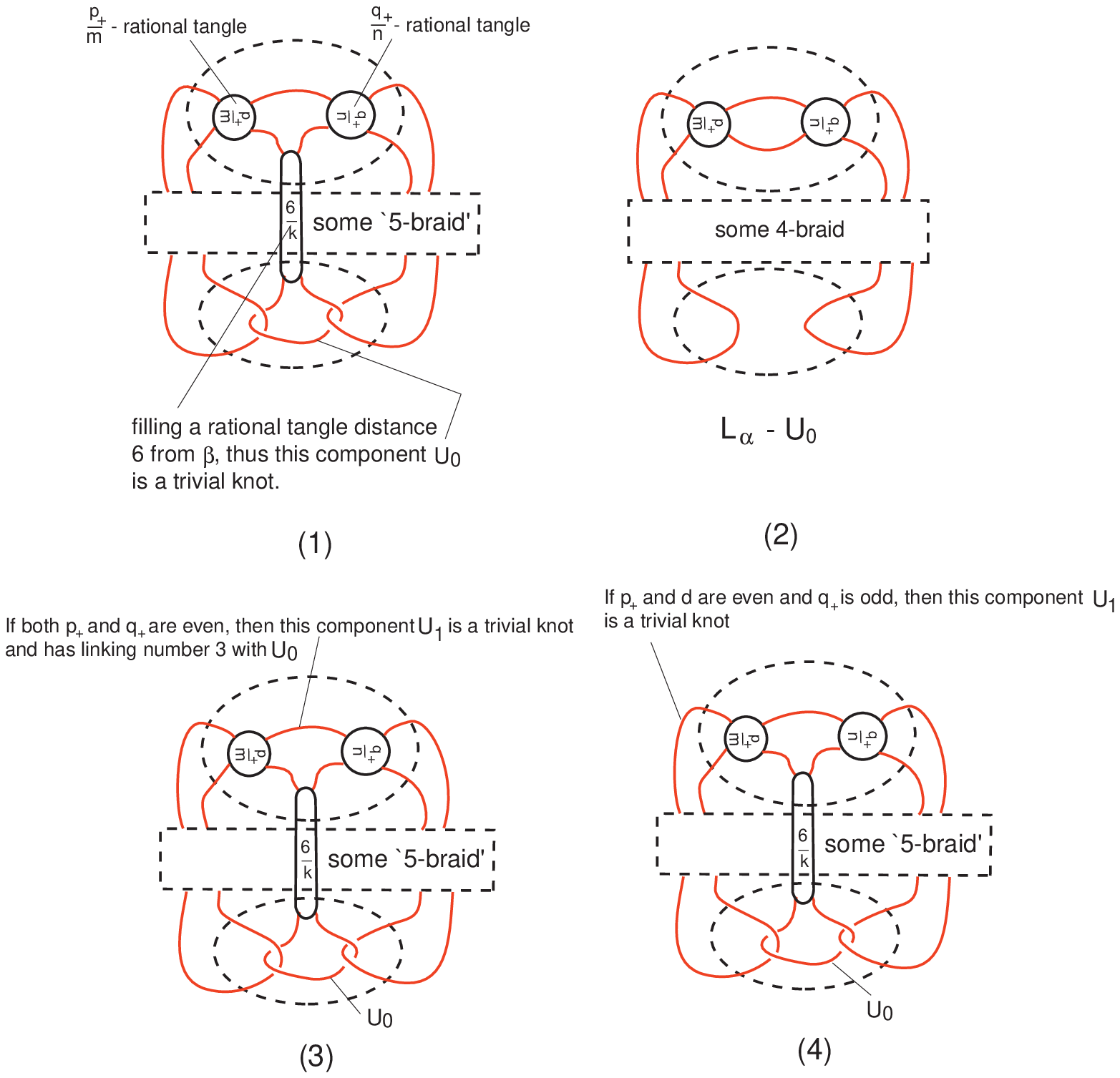}} \caption{ }\label{involution5}
\end{figure}

\begin{case}
$L_\a - U_0$ is a two component link
\end{case}

If $L_\a - U_0$ is a two component link, then $L_\a - U_0=U_1\cup K_1$ where $U_1$ is a trivial knot
and $K_1$ is both a $2$-bridge knot and a torus knot. Thus $K_1$ is a $(2,n)$ torus knot for some odd integer $n$. By Lemma \ref{lemma: branched link}(3), $M(\alpha)$ is Seifert fibred over $S^2(2, 4, 2q)$.

Now $K$ is either a $(2,n)$ or $(n,2)$ torus knot with respect to $U_1$. Hence the base orbifold $S^2(p_+, q_+,d)$ of the $2$-fold branched cover of $U_1 \cup K_1$ is either $S^2(2,4,n)$ or $S^2(2,2,n)$, so exactly two of $p_+, q_+$ and $d$ are even. Now $d$ cannot be odd, as otherwise $p_+$ and $q_+$ would both be even. Inspection of Figure \ref{involution5}(3) then shows that the linking number between $U_0$ and $U_1$ is $3$, which is impossible since $U_1\cup U_0$ is a Hopf link. We assume then that $p_+$ and $d$ are even. The case that $q_+$ and $d$ are even is treated similarly.

Since $q_+ = n$ odd, we know that $(p_+, q_+,d)$ is either $(2,q_+,2), (2,q_+,4)$ or $(4,q_+,2)$. Further, $U_1$ is as illustrated in Figure \ref{involution5}(4) and $K_1$ is a $(2,q_+)$-torus knot. We can also see from the figure that the linking number between $U_0$ and $K_1$ is even and therefore, $K_1$ is a $(q_+,2)$-torus knot with respect to $U_1$. Hence the base orbifold $S^2(p_+, q_+,d)$ of the $2$-fold branched cover of $U_1 \cup K_1$ is $S^2(2,4, q_+)$, so $(p_+, q_+,d)=(2, q_+, 4)$ or $(4,q_+,2)$.

If $d = 4$, we can build $\lfloor \frac {\;q_+}{2} \rfloor$ strictly non-trivial disjoint curves $X_0$ as above, each with $s_{X_0} \geq 2$ (cf. \S \ref{subsec: prod cyclics}). Hence if $X$ is the union of these curves, $s_X \geq q_+ -1$. Similarly if $d = 2$, we can build $2 \lfloor \frac {\;q_+}{2} \rfloor = q_+-1$ strictly non-trivial disjoint curves $X_0$ as above, each with $s_{X_0} \geq 1$. Hence if $X$ is the union of these curves, we again have $s_X \geq q_+ -1$. By Lemma \ref{lemma: branched link}(2), $M(\alpha)$ is Seifert fibred over $S^2(2, 4, 2q_+)$ and by \cite[Proposition 3.2]{BeBo}), $\Delta(2,4,2q_+)$ has exactly $q_+$ irreducible characters of representations with values in $PSL_2(\mathbb C)$. Hence, $\Delta(\alpha, \beta) \leq 1 + \frac{2q_+}{q_+ -1} \leq 5$, contrary to our assumptions. This rules out the possibility that $L_\a - U_0$ is a two component link and completes the proof of Proposition \ref{twice-punctured not very small t+t-=0}.

\section{Proof of Theorem \ref{thm: very small cases}}\label{sec: very small cases}

The first proposition deals with the case where $M(\alpha)$ is of $C$- or $D$-type.

\begin{prop}\label{prop: C and D types} Let $M$ be a hyperbolic knot manifold with slopes $\alpha$ and $\beta$ on $\partial M$ such that $M(\beta)$ is toroidal and $M(\alpha)$ is a very small Seifert manifold.

$(1)$ If $M(\alpha)$ is of $C$-type then $\Delta(\alpha,\beta) \le 4$.

$(2)$ If $M(\alpha)$ is of $D$-type then \newline
$\;\;\;\;\;\;(a)$ $\Delta(\alpha,\beta) \le 4$, and
\newline
$\;\;\;\;\;\;(b)$ if $M(\beta)$ is a torus semi-bundle then $\Delta(\alpha,\beta) \le 3$.
\end{prop}

\pf  (1) If $M(\alpha)$ is $S^3$ then $\Delta(\alpha,\beta) \le 2$ by \cite{GL1}, if $M(\alpha)$ is $S^1 \times S^2$ then $\Delta(\alpha,\beta) \le 3 $ by
\cite{Oh} and \cite{Wu2}, and if $M(\alpha)$ is a lens space then $\Delta(\alpha,\beta) \le 4$ by \cite{L3}.

(2) (a) If $M$ is a hyperbolic knot manifold such that $M(\beta)$ is toroidal and $M(\alpha)$ contains a Klein bottle, it follows from \cite{V} (if $\Delta(\alpha,\beta) > 5$) and \cite{L2} (if $\Delta(\alpha,\beta) = 5$), that $M(\alpha)$ is toroidal. If $M(\alpha)$ is of $D$-type then it contains a Klein bottle but is atoroidal. Hence $\Delta(\alpha,\beta) \le 4$.
Part (b) is  Proposition \ref{prop: semi very small} (5).
\qed

\begin{prop}\label{prop: toroidal vs very small} Let $M$ be a hyperbolic knot manifold which contains an $m$-punctured torus $F$ with boundary slope $\beta$, that is not a fibre in $M$ if $m \ge 3$. Let $\alpha$ be a slope on $\partial M$ such that $M(\alpha)$ is a very small Seifert manifold. Then

$(1)$ $\Delta(\alpha,\beta) \le 5$, and

$(2)$ if $\Delta(\alpha,\beta) = 5$, then either
\newline
$\;\;\;\;\;(a)$ $m = 1$, $F$ is not a fibre, $(M;\alpha,\beta) \cong (Wh(-3/2);-5,0)$, and $M(\alpha)$ has base orbifold $S^2(2,3,3)$, or
\newline
$\;\;\;\;\;\;(b)$ $F$ is separating in $M$ though not a semi-fibre.\end{prop}

\pf  This follows from Propositions 13.1 and 13.2 and Theorem 2.7 of \cite{BGZ2}, Theorem 1.3 of \cite{BGZ3},
Theorem \ref{thm: twice-punctured precise} and Proposition \ref{prop: F2-non-sep non-fibre}, and the fact that the figure eight knot exterior admits no small Seifert  Dehn fillings which are very small.
\qed

Theorem \ref{thm: very small cases} follows from Propositions \ref{prop: C and D types},
\ref{prop: semi very small} and \ref{prop: toroidal vs very small}.

%%%%%%%%%%%%%%%%%%%%%%%%%%%%%%%%%%%%%%%%%%%
\def\bysame{$\underline{\hskip.5truein}$}
%%%%%%%%%%%%%%%%%%%%%%%%%%%%%%%%%%%%%%%%%%%

\end{document}